\documentclass[a4paper, abstract=true, headings=small]{scrartcl}

\usepackage[utf8]{inputenc}
\usepackage[T1]{fontenc}
\usepackage{amsmath}
\usepackage{amssymb}
\usepackage{amsthm}
\usepackage{mathrsfs}
\usepackage{dsfont}
\usepackage{array}
\usepackage{graphicx}
\usepackage{tensor}
\usepackage[dvipsnames]{xcolor}
\usepackage{verbatim}
\usepackage{enumitem}
\usepackage{etoolbox} 
\usepackage{constants}

\usepackage{color}
\definecolor{purple}{cmyk}{0.75,0.90,0,0}
\definecolor{DB}{rgb}{0.07,0.0,0.5}
\definecolor{DG}{rgb}{0.0,0.37,0.07}%
\definecolor{DR}{rgb}{0.37,0,0.07}

\usepackage[english]{babel}
\usepackage[hyperindex=true,colorlinks=true, urlcolor={DR}, linkcolor={DR}, menucolor={DR}, citecolor={DR}, anchorcolor={DR},linktoc=all]{hyperref}
\PassOptionsToPackage{unicode}{hyperref}
\PassOptionsToPackage{naturalnames}{hyperref}

\usepackage[abbrev,backrefs]{amsrefs}
\usepackage[theoremfont, osf]{newpxtext}
\usepackage[slantedGreek]{newpxmath}

\renewcommand{\d}{\mathrm d} 
 
\newcommand{\Z}{\varmathbb{Z}} 
 
\newcommand{\R}{\varmathbb{R}} 
\newcommand{\N}{\varmathbb{N}}

\newcommand{\Sp}{\varmathbb{S}}

\newcommand{\B}{\varmathbb{B}}
\newcommand{\Ha}{\mathcal{H}} 
\newcommand{\Mc}{\mathcal{M}}
\newcommand{\Nc}{\mathcal{N}}
\newcommand{\Cc}{\mathcal{C}}
\newcommand{\Ec}{\mathcal{E}}

\newcommand{\Kc}{\mathcal{K}}

\newcommand{\Bc}{\mathcal{B}}
\newcommand{\Uc}{\mathcal{U}}
\newcommand{\Tc}{\mathcal{T}}
\newcommand{\Rc}{\mathcal{R}}

\newcommand{\Sc}{\mathcal{S}}

\newcommand{\Yc}{\mathcal{Y}}
\newcommand{\ulrho}{\underline{\rho}}
\newcommand{\olrho}{\overline{\rho}}
\newcommand{\ulmu}{\underline{\mu}}
\newcommand{\olmu}{\overline{\mu}}
\newcommand{\ulr}{\underline{r}}
\newcommand{\olr}{\overline{r}}
\newcommand{\op}{\mathrm{op}}
\newcommand{\sm}{\mathrm{sm}}
\newcommand{\thck}{\mathrm{th}}
\newcommand{\ext}{\mathrm{ex}}
\newcommand{\sh}{\mathrm{sh}}

\newcommand{\VMO}{\textnormal{VMO}}
\DeclareMathOperator{\id}{id}

\DeclareMathOperator{\Supp}{Supp}
\DeclareMathOperator{\supp}{supp}

\DeclareMathOperator{\dist}{dist} 
\DeclareMathOperator{\card}{card}

\DeclareMathOperator{\diam}{diam}
\DeclareMathOperator{\jac}{jac}
\DeclareMathOperator{\Dist}{Dist}
\DeclareMathOperator{\Int}{Int}

\AtEndEnvironment{exemple}{\quad\null\hfill\qedsymbol} 

\numberwithin{equation}{section}
\counterwithin{figure}{section}

\theoremstyle{plain}
\newtheorem{thm}{Theorem}[section]
\newtheorem{definition}[thm]{Definition}
\newtheorem*{definition*}{Definition}
\newtheorem{lemme}[thm]{Lemma}
\newtheorem*{lemme*}{Lemma}
\newtheorem{prop}[thm]{Proposition}
\newtheorem*{prop*}{Proposition}
\newtheorem*{thm*}{Theorem}

\newtheorem*{cor*}{Corollary}

\newtheorem*{affirmation*}{Affirmation}
\theoremstyle{definition}

\newtheorem*{exo*}{Exercise}

\newtheorem*{exemple*}{Example}

\newtheorem*{problem*}{Problem}
\theoremstyle{remark}

\newtheorem*{rmq*}{Remark}

\title{A complete answer to the strong density problem in Sobolev spaces with values into compact manifolds}
\author{Antoine Detaille}
\date{}

\begin{document}

\maketitle

\begin{abstract}
	We consider the problem of strong density of smooth maps in the Sobolev space \( W^{s,p}(Q^{m};\Nc) \), where \( 0 < s < +\infty \), \( 1 \leq p < +\infty \), \( Q^{m} \) is the unit cube in \( \R^{m} \), and \( \Nc \) is a smooth compact connected Riemannian manifold without boundary.
	Our main result fully answers the strong density problem in the whole range \( 0 < s < +\infty \): the space \( \Cc^{\infty}(\overline{Q}^{m};\Nc) \) is dense in \( W^{s,p}(Q^{m};\Nc) \) if and only if \( \pi_{[sp]}(\Nc) = \{0\} \).
	This completes the results of Bethuel (\( s=1 \)), Brezis and Mironescu (\( 0 < s < 1 \)), and Bousquet, Ponce, and Van Schaftingen (\( s = 2 \), \( 3 \), \dots).
	We also consider the case of more general domains \( \Omega \), in the setting studied by Hang and Lin when \( s = 1 \).
\end{abstract}

\section{Introduction}
\label{sect:intro}

We address here the question of the density of smooth maps in Sobolev spaces \( W^{s,p}(\Omega;\Nc) \) of maps \emph{with values into a compact manifold \( \Nc \)}.
Here and in the sequel, \( 1 \leq p < +\infty \) and \( 0 < s < +\infty \).
Recall the following well-known fundamental result in the theory of classical \emph{real-valued} Sobolev spaces: 
if \( \Omega \subset \R^{m} \) is a sufficiently smooth open set, then \( \Cc^{\infty}(\overline{\Omega};\R) \) is dense in \( W^{s,p}(\Omega;\R) \).
The reader may consult, for instance,~\cite{brezis_functional_analysis} or~\cite{willem} for a proof in the case where \( \Omega \) is a smooth domain,
or~\cite{adams_sobolev_spaces} in the case where \( \Omega \) satisfies the weaker \emph{segment condition}.
Here, 
\[
	\Cc^{\infty}(\overline{\Omega}) 
	= \{u_{\vert\Omega}\mathpunct{:} u \in \Cc^{\infty}(\R^{m})\}.
\]

More difficult is the analogue question of the density of smooth maps in Sobolev spaces when the target \( \Nc \) is a manifold.
In what follows, we let \( \Nc \) be a smooth compact connected Riemannian manifold without boundary, isometrically embedded in \( \R^{\nu} \).
The latter assumption is not restrictive, since we may always find such an embedding provided that we choose \( \nu \in \N \) sufficiently large; see~\cite{Nash54} and~\cite{Nash56}.
The natural analogue question is whether \( \Cc^{\infty}(\overline{\Omega};\Nc) \) is dense in \( W^{s,p}(\Omega;\Nc) \).
Here, the space \( W^{s,p}(\Omega;\Nc) \) is the set of all maps \( u \in W^{s,p}(\Omega;\R^{\nu}) \) such that \( u(x) \in \Nc \) for almost every \( x \in \Omega \).
Due to the presence of the manifold constraint, \( W^{s,p}(\Omega;\Nc) \) is in general not a vector space, but it is nevertheless a metric space endowed with the distance defined by 
\[
    d_{W^{s,p}(\Omega)}(u,v) = \lVert u-v \rVert_{W^{s,p}(\Omega)}.
\]
The space \( \Cc^{\infty}(\overline{\Omega};\Nc) \) is defined analogously as the set of all \( \Cc^{\infty}(\overline{\Omega};\R^{\nu}) \) maps taking their values into \( \Nc \).

We note that the usual technique for proving density of smooth maps, relying on regularization by convolution, is not applicable in this context, since in general it does not preserve the constraint that the maps take their values into \( \Nc \).
In the range \( sp \geq m \), however, density always holds.
Indeed, in this range, Sobolev maps are continuous, or belong to the set \( \VMO \) of functions with \emph{vanishing mean oscillation}.
One may therefore proceed as in the classical case, \emph{via} regularization and nearest point projection onto \( \Nc \); see~\cite{SU_boundary_regularity} and~\cite{BN_degree_BMO_I}.

The case \( sp < m \) is way more delicate.
Schoen and Uhlenbeck~\cite{SU_boundary_regularity} were the first to observe that density may fail in this range, due to the presence of topological obstructions.
More precisely, they showed that the map \( u \colon \B^{3} \to \Sp^{2} \) defined by \( u(x) = \frac{x}{\lvert x \rvert} \) may not be approximated by smooth functions in \( W^{1,2}(\B^{3};\Sp^{2}) \).
This was subsequently generalized by Bethuel and Zheng~\cite{BZ_density}*{Theorem~2} and finally by Escobedo~\cite{escobedo_some_remarks}, leading to the conclusion that \( \Cc^{\infty}(\overline{\Omega};\Nc) \)
is never dense in \( W^{s,p}(\Omega;\Nc) \) when \( \pi_{[sp]}(\Nc) \neq \{0\} \).
Here, \( \pi_{\ell}(\Nc) \) is the \( \ell \)-th homotopy group of \( \Nc \), and \( [sp] \) denotes the integer part of \( sp \).
For further use, we note that the condition \( \pi_{[sp]}(\Nc) = \{0\} \) means that every continuous map \( f \colon \Sp^{[sp]} \to \Nc \) may be extended to a continuous map \( g \colon \B^{[sp]+1} \to \Nc \).

A natural question is whether the condition \( \pi_{[sp]}(\Nc) = \{0\} \) is also sufficient for the density of \( \Cc^{\infty}(\overline{\Omega};\Nc) \) in \( W^{s,p}(\Omega;\Nc) \).
A remarkable result of Bethuel~\cite{bethuel_approximation} asserts that, when \( s = 1 \), \( 1 \leq p < m\), and \( \Omega \) is a cube, the condition \( \pi_{[sp]}(\Nc) \neq \{0\} \) is \emph{the only obstruction} to strong density of \( \Cc^{\infty}(\overline{\Omega};\Nc) \) in \( W^{1,p}(\Omega;\Nc) \).
Bethuel's result has been extended to other values of \( s \) and \( p \), but not all (see below).
Our first main result provides a \emph{complete} generalization of Bethuel's result (covering all values of \( s \) and \( p \)).

\begin{thm}
    \label{thm:density_smooth_functions}
    If \( sp < m \) and \( \pi_{[sp]}(\Nc) = \{0\} \), then \( \Cc^{\infty}(\overline{Q}^{m};\Nc) \) is dense in \( W^{s,p}(Q^{m};\Nc) \).
\end{thm}

Here, \( Q^{m} \) denotes the unit cube in \( \R^{m} \).
The case of more general domains is more involved since the topology of the domain also comes into play, as it was first noticed by Hang and Lin~\cite{HL_topology_of_sobolev_mappings_II}.
We investigate this question in Section~\ref{sect:density_smooth_maps}, establishing counterparts of Theorem~\ref{thm:density_smooth_functions} when the domain is a smooth bounded open set, 
or even a smooth compact manifold of dimension \( m \); see Theorems~\ref{thm:density_smooth_maps_smooth_domain} and~\ref{thm:density_smooth_maps_manifold} below.
    
When \( \pi_{[sp]}(\Nc) \neq \{0\} \), density fails, and a natural question in this context is whether one can find a suitable substitute for the class \( \Cc^{\infty}(\overline{Q}^{m};\Nc) \).
This is indeed the case provided that we replace smooth functions on \( \overline{\Omega} \) by functions that are smooth on \( \overline{\Omega} \) except on some singular set whose dimension depends on \( [sp] \).
This direction of research also originates in Bethuel's paper~\cite{bethuel_approximation}.
(For subsequent results, see below.)
    
We define the class \( \Rc_{i}(\Omega;\Nc) \) as the set of maps \( u \colon \overline{\Omega} \to \Nc \) which are smooth on \( \overline{\Omega} \setminus T \), where \( T \) is a finite union of \( i \)-dimensional planes, and such that for every \( j \in \N_{\ast} \) and \( x \in \Omega \setminus T \),
\[
    \lvert D^{j}u(x) \rvert \leq C\frac{1}{\dist(x,T)^{j}}
\]
for some constant \( C > 0 \) depending on \( u \) and \( j \).
We establish the density of the class \( \Rc_{m-[sp]-1} \) in the full range \( 0 < s < +\infty \) when \( \Omega = Q^{m} \).

\begin{thm}
\label{thm:density_class_R}
    If \( sp < m \), then \( \Rc_{m-[sp]-1}(Q^{m};\Nc) \) is dense in \( W^{s,p}(Q^{m};\Nc) \).
\end{thm}

We mention that, in some sense, the class \( \Rc_{m-[sp]-1}(Q^{m};\Nc) \) is the best dense class in \( W^{s,p}(Q^{m};\Nc) \) one can hope for.
More precisely, the singular set cannot be taken of smaller dimension: the class \( \Rc_{i}(Q^{m};\Nc) \) with \( i < m-[sp]-1 \) is never dense in \( W^{s,p}(Q^{m};\Nc) \) if \( \pi_{[sp]}(\Nc) \neq \{0\} \); see the discussion in~\cite{BPVS_density_higher_order}.

In addition to its own importance, Theorem~\ref{thm:density_class_R} is crucial in establishing Theorem~\ref{thm:density_smooth_functions}.
In Section~\ref{sect:density_class_R}, we explain how to deal with more general domains.
We show that Theorem~\ref{thm:density_class_R} has a valid counterpart on bounded domains \( \Omega \) that merely satisfy the segment condition or when the domain is instead a smooth compact manifold of dimension \( m \); see Theorems~\ref{thm:density_class_R_segment_condition} and~\ref{thm:density_class_R_manifold} below.

Theorems~\ref{thm:density_smooth_functions} and~\ref{thm:density_class_R} where known for some values of \( s \) and \( p \).
As mentioned above, the case \( s = 1 \) was established by Bethuel in his seminal paper~\cite{bethuel_approximation}.
Progress was then made by Brezis and Mironescu~\cite{BM_density_in_Wsp} and by Bousquet, Ponce, and Van Schaftingen~\cite{BPVS_density_higher_order}.
Using an \emph{ad hoc} method based on \emph{homogeneous extension}, Brezis and Mironescu were able to completely solve the case \( 0 < s < 1 \).
On the other hand, Bousquet, Ponce, and Van Schaftingen introduced several important tools that are tailored to higher order Sobolev spaces, which allowed them to give a full answer to the strong density problem in the case \( s = 2 \), \( 3 \), \( \dots \)
Their approach incorporates and adapts major concepts from Bethuel's proof for \( s = 1 \) (among which the \emph{method of good and bad cubes}, which lies at the core of the proof) and from Brezis and Li~\cite{BL_topology_sobolev_spaces}.
It turns out that this approach in~\cite{BPVS_density_higher_order} extends to noninteger values of \( s \).
This is the main contribution of our paper.
Other special cases were obtained by Bethuel and Zheng~\cite{BZ_density}, Escobedo~\cite{escobedo_some_remarks}, Haj\l asz~\cite{hajlasz}, Bethuel~\cite{bethuel_approximation_trace}, Rivière~\cite{Riviere_dense_subsets}, Bousquet~\cite{bousquet_topological_singularities}, Mucci~\cite{mucci_strong_density_results}, and Bousquet, Ponce, and Van Schaftingen~\cites{BPVS_density_simply_connected,BPVS_fractional_density_simply_connected}.
However, the case where \( s > 1 \) is not an integer and \( \Nc \) is a general manifold is not covered by these contributions and is the main novelty of Theorem~\ref{thm:density_smooth_functions}.


The method of homogeneous extension used in~\cite{BM_density_in_Wsp} to settle the case where \( 0 < s < 1 \) was shown by the authors themselves not to work when \( s = 1 \); see~\cite{BM_density_in_Wsp}*{Lemma~4.9}.
On the contrary, as we explained above, the approach in~\cite{BPVS_density_higher_order} can be adapted to handle noninteger values of \( s \).
It is our goal here to explain in detail this adapted construction, introducing the modifications and new ideas that are required to make it suitable for the fractional order setting.
This does not only prove the density of smooth maps in the remaining case where \( s > 1 \) is not an integer, but it also provides a unified proof covering the full range \( 0 < s < +\infty \),
including the case \( 0 < s < 1 \) originally treated \emph{via} a different approach.

This paper is organized as follows.
In Sections~\ref{sect:opening} to~\ref{sect:thickening}, we develop the tools that we need to prove Theorem~\ref{thm:density_class_R}, following the approach in~\cite{BPVS_density_higher_order} and extending the auxiliary results to the noninteger case. 
With these tools at hand, we proceed in Section~\ref{sect:density_class_R} with the proof of Theorem~\ref{thm:density_class_R}.
For the sake of simplicity, we first deal with the model case \( \Omega = Q^{m} \), before explaining how to handle more general domains.
In Section~\ref{sect:density_smooth_maps}, we present the proof of Theorem~\ref{thm:density_smooth_functions} and the counterpart of Theorem~\ref{thm:density_smooth_functions} in general domains.
The proofs rely on an additional tool presented in Section~\ref{sect:shrinking}.

Before delving into technicalities, we start by presenting in Section~\ref{sect:sketch_proof} an overview of the proof of Theorems~\ref{thm:density_class_R} and~\ref{thm:density_smooth_functions} and the main tools that it relies on.
Section~\ref{sect:sketch_proof} also gathers the main useful definitions and basic auxiliary results used throughout the paper.

\subsection*{Acknowledgements}

I am deeply grateful to Petru Mironescu and Augusto Ponce for introducing me to this beautiful topic, for their constant support and many helpful suggestions to improve the exposition.
I especially thank Petru Mironescu for long discussions concerning the paper, and Augusto Ponce for sharing and discussing with me the preprint~\cite{BPVS_screening}.

\section{Definitions and sketch of the proof}
\label{sect:sketch_proof}

From now on, we write \( s = k + \sigma \) with \( k \in \N \) and \( \sigma \in [0,1) \).
We recall that the Sobolev space \( W^{k,p}(\Omega) \) is the set of all \( u \in L^{p}(\Omega) \) such that for every \( j \in \{1,\dots,k\} \),
the weak derivative \( D^{j}u \) belongs to \( L^{p}(\Omega) \).
This space is endowed with the norm defined by
\[
    \lVert u \rVert_{W^{k,p}(\Omega)} = \lVert u \rVert_{L^{p}(\Omega)} + \sum_{j=1}^{k} \lVert D^{j}u \rVert_{L^{p}(\Omega)}.
\]
When \( \sigma \in (0,1) \), the fractional Sobolev space \( W^{\sigma,p}(\Omega) \) is the set of all measurable maps \( u \colon \Omega \to \R \) such that \( \lvert u \rvert_{W^{\sigma,p}(\Omega)} < +\infty \),
where the \emph{Gagliardo seminorm} \( \lvert \cdot \rvert_{W^{\sigma,p}(\Omega)} \) is defined by 
\[
    \lvert u \rvert_{W^{\sigma,p}(\Omega)}
    =
    \biggl(\int_{\Omega}\int_{\Omega} \frac{\lvert u(x)-u(y) \rvert^{p}}{\lvert x-y \rvert^{m+\sigma p}}\,\d x\d y\biggr)^{\frac{1}{p}}.
\]
It is endowed with the norm 
\[
    \lVert u \rVert_{W^{\sigma,p}(\Omega)} = \lVert u \rVert_{L^{p}(\Omega)} + \lvert u \rvert_{W^{\sigma,p}(\Omega)}.
\]
When \( \sigma \in (0,1) \) and \( k \geq 1 \), the Sobolev space \( W^{s,p}(\Omega) \) is the set of all \( u \in W^{k,p}(\Omega) \) such that \( D^{k}u \in W^{\sigma,p}(\Omega) \), endowed with the norm
\[
    \lVert u \rVert_{W^{s,p}(\Omega)} = \lVert u \rVert_{W^{k,p}(\Omega)} + \lvert D^{k}u \rvert_{W^{\sigma,p}(\Omega)}.
\]
When working specifically with the Gagliardo seminorm, we shall often consider implicitly that \( \sigma \neq 0 \).
We also mention that here, we consider \( L^{p}(\Omega) \) maps as measurable functions \( u \colon \Omega \to \R \) (and not classes of functions), i.e., we do not identify two maps that are almost everywhere equal.
As we will see, this will be of importance in the course of Section~\ref{sect:opening}.

Throughout the paper, we make intensive use of decompositions of domains into suitable families of cubes.
For this purpose, we introduce some notation.
Given \( \eta > 0 \) and \( a \in \R^{m} \), we denote by \( Q^{m}_{\eta}(a) \) the cube of center \( a \) and radius \( \eta \) in \( \R^{m} \), the radius of a cube being half of the length of its edges.
When \( a = 0 \), we abbreviate \( Q^{m}_{\eta}(0) = Q^{m}_{\eta} \).
We also abbreviate \( Q^{m}_{1} = Q^{m} \).

A \emph{cubication} \( \Kc^{m}_{\eta} \) of radius \( \eta > 0 \) is any subset of \( Q^{m}_{\eta} + 2\eta\Z^{m} \).
Given \( \ell \in \{0,\dots,m\} \), the \emph{\( \ell \)-skeleton} of \( \Kc^{m}_{\eta} \) is the set \( \Kc^{\ell}_{\eta} \) of all faces of dimension \( \ell \) of all cubes in \( \Kc^{m}_{\eta} \).
A \emph{subskeleton} of dimension \( \ell \) of \( \Kc^{m}_{\eta} \) is any subset of \( \Kc^{\ell}_{\eta} \).
Given a skeleton \( \Sc^{\ell} \), we denote by \( S^{\ell} \) the union of all elements of \( \Sc^{\ell} \), that is, 
\[
    S^{\ell} = \bigcup_{\sigma^{\ell} \in \Sc^{\ell}} \sigma^{\ell}.
\]

Given a skeleton \( \Sc^{\ell} \), the \emph{dual skeleton} of \( \Sc^{\ell} \) is the skeleton \( \Tc^{\ell^{\ast}} \) of dimension \( \ell^{\ast} = m-\ell-1 \) 
consisting in all cubes of the form \( \sigma^{\ell^{\ast}} + a - x \), where \( \sigma^{\ell^{\ast}} \in \Sc^{\ell^{\ast}} \), 
\( a \) is the center and \( x \) a vertex of a cube of \( \Sc^{m} \) with \( x \in \sigma^{\ell^{\ast}} \).
The dimension \( \ell^{\ast} \) is the largest possible so that \( S^{\ell} \cap T^{\ell^{\ast}} = \varnothing \).
Here, 
\[
	T^{\ell^{\ast}} = \bigcup_{\sigma^{\ell^{\ast}} \in \Tc^{\ell^{\ast}}} \sigma^{\ell^{\ast}}.
\]

For further use, we note that \( S^{\ell^{\ast}} \) is a homotopy retract of \( S^{m} \setminus T^{\ell^{\ast}} \); see e.g.~\cite{white_infima_energy_functionals}*{Section~1} or~\cite{BPVS_density_simply_connected}*{Lemma~2.3}.

Given a map \( \upPhi \colon \R^{m} \to \R^{m} \), the \emph{geometric support of \( \upPhi \)} is defined by
\[	
\Supp\upPhi = \overline{\{x \in \R^{m} \mathpunct{:} \upPhi(x) \neq x\}}.
\]
This should not be confused with the \emph{analytic support} of a map \( \varphi \colon \R^{m} \to \R \), defined by
\[
\supp\varphi = \overline{\{x \in \R^{m} \mathpunct{:} \varphi(x) \neq 0\}}.
\]

We now present the sketch of the proof of Theorem~\ref{thm:density_class_R}.
We also include graphical illustrations of the various constructions involved in the proof, with \( m = 2 \) and \( [sp] = 1 \).
As we explained in the introduction, we follow the approach of Bousquet, Ponce, and Van Schaftingen~\cite{BPVS_density_higher_order},
and we provide the necessary tools and ideas to adapt their method to the fractional setting.
Let \( u \in W^{s,p}(Q^{m};\Nc) \).
For the sake of simplicity, we assume that \( u \) is defined in a neighborhood of \( \overline{Q}^{m} \).
The starting point is Bethuel's concept of \emph{good cubes} and \emph{bad cubes} that we now present.
Let \( \Kc^{m}_{\eta} \) be a cubication of \( Q^{m} \), that is, \( K^{m}_{\eta} = Q^{m} \).
Here, \( \eta > 0 \) is such that \( 1/\eta \in \N_{\ast} \).
(Actually, for technical reasons, we will need to work on a cubication of a slightly larger cube than \( Q^{m} \), but for this informal exposition, let us stick to a cubication of \( Q^{m} \) for the sake of simplicity.)
We fix \( 0 < \rho < \frac{1}{2} \) and define the family \( \Ec^{m}_{\eta} \) of all \emph{bad cubes} as the set of cubes \( \sigma^{m} \in \Kc^{m}_{\eta} \) such that 
\begin{equation}
\label{eq:def_bad_cubes_sketch}
    \frac{1}{\eta^{m-sp}}\lVert Du \rVert_{L^{sp}(\sigma^{m}+Q^{m}_{2\rho\eta})}^{sp} > c
    \quad 
    \text{if \( s \geq 1 \), or}
    \quad 
    \frac{1}{\eta^{m-sp}}\lvert u \rvert_{W^{s,p}(\sigma^{m}+Q^{m}_{2\rho\eta})}^{p} > c
    \quad 
    \text{if \( 0 < s < 1 \),}
\end{equation}
where \( c > 0 \) is a small parameter to be determined later on.
The remaining cubes are the \emph{good cubes}.
From now on, we shall assume that we are in the case \( s \geq 1 \), in order to avoid having to distinguish two cases.
(The case \( 0 < s < 1 \) is similar.)
The condition defining the good cubes ensures that \( u \) \emph{does not oscillate too much} on such cubes.
On the contrary, one cannot control the behavior of \( u \) on bad cubes, but we can show that there are not too many of them.
Indeed, each bad cube contributes with at least \( c\eta^{\frac{m}{sp}-1} \) to the energy of \( u \), which limits the number of such cubes.

On Figure~\ref{fig:good_and_bad_cubes}, one finds a possible decomposition of \( Q^{2} \) in \( 16 \) cubes, which corresponds to \( \eta = \frac{1}{4} \).
Here, the three cubes in red are bad cubes, while green cubes are good cubes.
For technical reasons that will become clear later on, it is useful to work on a set slightly larger than the union of bad cubes.
We therefore let \( \Uc^{m}_{\eta} \) be the set of all cubes in \( \Kc^{m}_{\eta} \) that intersect some bad cube in \( \Ec^{m}_{\eta} \).
This fact is ignored in our graphical illustrations, which are drawn as if \( \Uc^{m}_{\eta} = \Kc^{m}_{\eta} \).
This allows us to keep readable pictures with large cubes.
Nevertheless, the reader should keep in mind that all constructions explained below are actually performed not only on the red cubes, but also on all green cubes adjacent to them, and that decompositions could possibly consist in many small cubes.

\begin{figure}[ht]
	\centering
	\includegraphics[page=16]{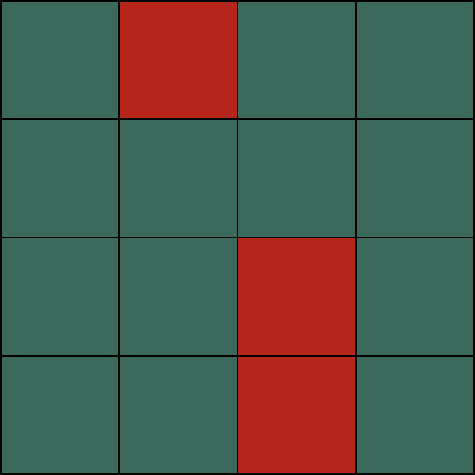}
	\caption{Good and bad cubes}
	\label{fig:good_and_bad_cubes}
\end{figure}

We now turn to the construction of the maps in \( u \) in the class \( \Rc_{m-[sp]-1} \) approximating \( u \).
The first tool is the \emph{opening}, which is explained in Section~\ref{sect:opening}.
This technique originates in the work of Brezis and Li~\cite{BL_topology_sobolev_spaces} about the topology of Sobolev spaces of maps between manifolds.
We \emph{open} the map \( u \) in order to obtain a map \( u^{\op}_{\eta} \) which, on a neighborhood of the \( [sp] \)-skeleton \( U^{[sp]}_{\eta} \),
is constant on the \( (m-[sp]) \)-dimensional cubes orthogonal to cubes in \( \Uc^{[sp]}_{\eta} \).
Therefore, on this neighborhood, the map \( u^{\op}_{\eta} \) behaves locally as a function of \( [sp] \)-variables.
But since \( sp \geq [sp] \), this means that, on this region, \( u^{\op}_{\eta} \) is actually a \( \VMO \) function.
The map \( u^{\op}_{\eta} \) is obtained by modifying \( u \) on a slightly larger neighborhood of \( U^{[sp]}_{\eta} \), and the construction does 
not increase too much the energy of \( u \) on this neighborhood.

On Figure~\ref{fig:opening_around_bad_cubes}, one finds an illustration of the opening procedure when \( [sp] = 1 \).
The map \( u \) is opened on the blue region, where it therefore satisfies \( \VMO \) estimates.

\begin{figure}[ht]
	\centering
	\includegraphics[page=17]{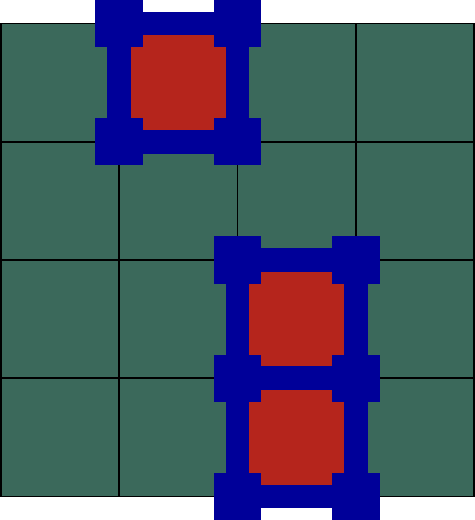}
	\caption{Opening around the \( 1 \)-skeleton of bad cubes}
	\label{fig:opening_around_bad_cubes}
\end{figure}

The next step is to smoothen the map \( u^{\op}_{\eta} \).
Given a mollifier \( \varphi \in \Cc^{\infty}_{c}(B^{m}_{1}) \) and \( r > 0 \), the usual convolution product is defined as 
\[
    \varphi_{r} \ast u(x) = \int_{B^{m}_{1}}\varphi(z)u(x+rz)\,\d z.
\]
Here we rely on the method of \emph{adaptive smoothing}, whose principle is to allow the convolution parameter to depend on the point where the convolution is evaluated.
This technique was made popular by the work of Schoen and Uhlenbeck~\cite{SU_regularity_theory_harmonic_maps}, where it was used in the study of the regularity of harmonic maps with values into a manifold.

More precisely, given \( \psi \in \Cc^{\infty}(Q^{m}) \), we let 
\[
    \varphi_{\psi} \ast u(x) = \int_{B^{m}_{1}} \varphi(z)u(x+\psi(x)z)\,\d z.
\]
To pursue the proof, we choose a suitable map \( \psi_{\eta} \in \Cc^{\infty}(B^{m}_{1}) \), whose construction depends on \( \eta \) and will be explained later on,
and we define \( u^{\sm}_{\eta} = \varphi_{\psi_{\eta}} \ast u^{\op}_{\eta} \).

This convolution procedure guarantees that the resulting map \( u^{\sm}_{\eta} \) is smooth, 
but has the drawback that it need no longer take its values into \( \Nc \), since the convolution product is in general not compatible with non convex constraints.
We therefore need to estimate the distance between \( u^{\sm}_{\eta} \) and \( \Nc \).
By straightforward computations, we write 
\[
    \lvert u^{\sm}_{\eta}(x) - u^{\op}_{\eta}(z) \rvert
    \leq
    \Cl{cst:fixed_sketch1}\fint_{Q^{m}_{\psi_{\eta}(x)}} \lvert u^{\op}_{\eta}(y)-u^{\op}_{\eta}(z) \rvert \,\d y.
\]
Averaging over all \( z \in Q^{m}_{\psi(x)}(x) \), since \( u^{\op}_{\eta}(z) \in \Nc \), we deduce that 
\[
    \dist{(u^{\sm}_{\eta}(x),\Nc)}
    \leq 
    \Cr{cst:fixed_sketch1}\fint_{Q^{m}_{\psi_{\eta}(x)}}\fint_{Q^{m}_{\psi_{\eta}(x)}} \lvert u^{\op}_{\eta}(y)-u^{\op}_{\eta}(z) \rvert \,\d y\d z.
\]
Here we see the usefulness of the opening construction performed at the previous step: since \( u^{\op}_{\eta} \) is a \( \VMO \) function 
close to \( U^{[sp]}_{\eta} \), the right-hand side of the above estimate may be made arbitrarily small in this region provided that we choose \( \psi_{\eta}(x) \) sufficiently small.
On the good cubes, we pursue the estimate by invoking the Poincaré--Wirtinger inequality to write 
\begin{equation}
\begin{aligned}
\label{eq:first_Poincare_Wirtinger}
    \dist{(u^{\sm}_{\eta}(x),\Nc)}
    &\leq 
    \C \frac{1}{\psi_{\eta}(x)^{\frac{m}{sp}-1}}\lVert Du^{\op}_{\eta} \rVert_{L^{sp}(Q^{m}_{\psi_{\eta}(x)}(x))} \\
    &\leq 
    \C \frac{1}{\psi_{\eta}(x)^{\frac{m}{sp}-1}}\lVert Du \rVert_{L^{sp}(Q^{m}_{\psi_{\eta}(x)}(x))}.
\end{aligned}
\end{equation}
If we choose \( \psi_{\eta}(x) \) of order \( \eta \), then on the right-hand side of~\eqref{eq:first_Poincare_Wirtinger}, we find precisely the energy of \( u \) which is controlled on the good cubes.
Therefore, choosing suitably the constant \( c > 0 \) in~\eqref{eq:def_bad_cubes_sketch}, on the good cubes, \( u^{\sm}_{\eta} \) will be \( \delta \)-close to \( \Nc \), for some given arbitrarily small number \( \delta > 0 \).
To summarize, we are invited to choose the convolution parameter very small on bad cubes, near the \( [sp] \)-skeleton, and of order \( \eta \) on good cubes.
Between those two regimes, we need a transition region in order to allow \( \psi_{\eta} \) to change of magnitude, which is precisely the reason to introduce both families \( \Uc^{m}_{\eta} \) and \( \Ec^{m}_{\eta} \) instead of working directly on bad cubes.
The precise way to perform this construction is explained in Section~\ref{sect:adaptive_smoothing}, and gathering the estimates on good and bad cubes, we conclude that \( u^{\sm}_{\eta} \) is close to \( \Nc \) on the good cubes, and on the part of bad cubes close to the \( [sp] \)-skeleton.

It therefore remains to deal with the part of bad cubes far from the \( [sp] \)-skeleton, where we have no control on the distance between \( u^{\sm}_{\eta} \) and \( \Nc \) (which corresponds to the red region in Figure~\ref{fig:opening_around_bad_cubes}).
This is the purpose of the last tool we need, which is called \emph{thickening}.
The method is inspired from the use of homogeneous extension by Bethuel in the case \( s = 1 \).
We illustrate the idea when \( s = 1 \) and \( m-1 < p < m \).
Given a map \( v \in \Cc^{\infty}(\overline{Q}^{m}) \), we define \( w \) on \( Q^{m} \) by 
\[
    w(x) = v\Bigl(\frac{x}{\lvert x \rvert_{\infty}}\Bigr).
\]
Here we recall that \( \lvert \cdot \rvert_{\infty} \) stands for the \( \infty \)-norm in \( \R^{m} \), defined for \( x = (x_{1},\dots,x_{m}) \in \R^{m} \) by \( \displaystyle \lvert x \rvert_{\infty} = \max_{1 \leq i \leq m} \lvert x_{i} \rvert \).
Using radial integration, we see that \( w \in W^{1,p}(Q^{m}) \) and 
\[
    \lVert Dw \rVert_{L^{p}(Q^{m})}^{p}
    \leq
    \C\lVert Dv \rVert_{L^{p}(\partial Q^{m})}^{p}\int_{0}^{1} r^{m-p-1}\,\d r
    \leq 
    \C\lVert Dv \rVert_{L^{p}(\partial Q^{m})}^{p}.
\]
Here we use the assumption \( p < m \).
Hence, \( w \) is a \( W^{1,p}(Q^{m}) \) map that depends only on the values of \( v \) on \( \partial Q^{m} \).
We may iterate this construction on faces by downward induction on the dimension to construct a map which only depends on the values of \( v \) 
on the \( [p] \)-skeleton of \( Q^{m} \).

Two major difficulties arise when we try to adapt this construction to general Sobolev maps and spaces.
First, it requires to work with slices of Sobolev maps on sets of zero measure.
But more importantly, gluing such constructions on two cubes sharing a common face is a delicate matter.
This is already the case when \( s = 1 \) if \( p < m-1 \) at the interface between a good and a bad cube, since the resulting maps do not coincide on the whole common face, 
and gets worse when \( s > 1+\frac{1}{p} \) as the derivatives do not match at the interface.
We bypass this difficulty by working with a more involved version of homogeneous extension, the \emph{thickening} procedure.

Let \( \Tc^{[sp]^{\ast}}_{\eta} \) denote the dual skeleton of \( \Uc^{[sp]}_{\eta} \).
The homogeneous extension, in the form presented above, associates with a map \( v \colon U^{[sp]}_{\eta} \to \R^{\nu} \) a map \( w \colon U^{m}_{\eta} \setminus T^{[sp]^{\ast}}_{\eta} \to \R^{\nu} \), and this map is, in general, discontinuous on \( T^{[sp]^{\ast}}_{\eta} \).
The map \( w \) may be written as \( w = v \circ \upPhi^{\mathrm{he}} \), where \( \upPhi^{\mathrm{he}} \colon U^{m}_{\eta} \setminus T^{[sp]^{\ast}}_{\eta} \to U^{[sp]}_{\eta} \) is a Lipschitz map.
Instead, the thickening procedure associates with a map \( v \colon U^{[sp]}_{\eta}+Q^{m}_{\delta} \to \R^{\nu} \) (for some \( \delta > 0 \) sufficiently small) a map \( w \colon U^{m}_{\eta} \setminus T^{[sp]^{\ast}}_{\eta} \to \R^{\nu} \), which, again, is in general singular on the set \( T^{[sp]^{\ast}}_{\eta} \).
The map \( w \) is obtained from \( v \) as \( w = v \circ \upPhi^{\thck} \), where \( \upPhi^{\thck} \colon U^{m}_{\eta} \setminus T^{[sp]^{\ast}}_{\eta} \to U^{[sp]}_{\eta}+Q^{m}_{\delta} \) is a \emph{smooth} map.
Working with the neighborhood \( U^{[sp]}_{\eta}+Q^{m}_{\delta} \) instead of the skeleton \( U^{[sp]}_{\eta} \) is the key idea to avoid working with slices of Sobolev maps, and more importantly, to be able to choose \( \upPhi^{\thck} \) smooth, which, in turn, is crucial to ensure that composition with \( \upPhi^{\thck} \) preserves higher order Sobolev regularity.

The detailed construction, devised in~\cite{BPVS_density_higher_order}*{Section~4}, is explained in Section~\ref{sect:thickening}, and we apply it to modify the map \( u^{\sm}_{\eta} \) on \( U^{m}_{\eta} \) 
to a map \( u^{\thck}_{\eta} \) whose values on \( U^{m}_{\eta} \) only depend on the values of \( u^{\sm}_{\eta} \) near \( U^{[sp]}_{\eta} \),
while not increasing too much the energy of the map on \( U^{m}_{\eta} \).
Therefore, the map \( u^{\thck}_{\eta} \) is close to \( \Nc \) on the whole \( Q^{m} \setminus T^{[sp]^{\ast}}_{\eta} \), 
which makes possible to project it back onto \( \Nc \) relying on the nearest point projection \( \upPi \).
Since the map \( u^{\sm}_{\eta} \) is smooth, the map \( u^{\thck}_{\eta} \) is smooth on \( Q^{m} \setminus T^{[sp]^{\ast}}_{\eta} \), and we will show that 
the singularities created on \( T^{[sp]^{\ast}}_{\eta} \) by the thickening are sufficiently mild so that \( u^{\thck}_{\eta} \) belongs to the class \( \Rc_{m-[sp]-1}(Q^{m};\R^{\nu}) \).

One finds an illustration of the thickening procedure on Figure~\ref{fig:thickening_around_centers_of_bad_cubes}.
The values of \( u \) on the dark blue region are propagated into the light blue region.
This process creates point singularities on the centers of bad cubes, which are represented by the intersection of all the black lines.

\begin{figure}[ht]
	\centering
	\includegraphics[page=18]{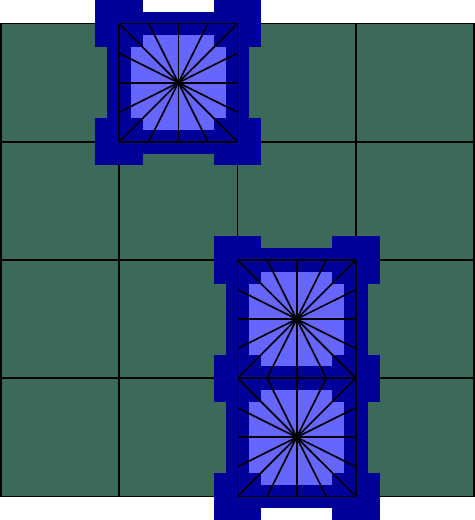}
	\caption{Thickening around the centers of bad cubes}
	\label{fig:thickening_around_centers_of_bad_cubes}
\end{figure}

The maps \( u_{\eta} = \upPi \circ u^{\thck}_{\eta} \) therefore belong to \( \Rc_{m-[sp]-1}(Q^{m};\Nc) \), and they are actually the approximations of \( u \) that we were looking for.
The only step that is required to obtain this conclusion is to show the convergence \( u_{\eta} \to u \) in \( W^{s,p} \) as \( \eta \to 0 \).
This is done in Section~\ref{sect:density_class_R}, and amounts to a careful combination of the estimates obtained at each step of the construction.
Except for the adaptive smoothing, all the modifications performed on \( u \) are localized in a neighborhood of \( U^{m}_{\eta} \).
The main ingredient to reach the conclusion \( u_{\eta} \to u \) is therefore the fact that there are not too many bad cubes, and that actually the measure of the union of all bad cubes decays at a sufficiently high rate.

The density of the class \( \Rc \) being established, we may then move to the density of smooth maps under the assumption \( \pi_{[sp]}(\Nc)=\{0\} \).
For this, it suffices to show that maps \( u_{\eta} \) of the class \( \Rc_{m-[sp]-1}(Q^{m},\Nc) \) as constructed in the first part of the proof above may be approximated by smooth maps.
Under the assumption \( \pi_{[sp]}(\Nc)=\{0\} \), for any given arbitrarily small number \( \delta > 0 \), one may find a smooth map \( u^{\ext}_{\delta} \) such that \( u^{\ext}_{\delta} \) coincides with \( u_{\eta} \) everywhere on \( Q^{m} \) except on \( T^{[sp]^{\ast}}_{\eta}+Q^{m}_{\delta} \).
This is explained in Section~\ref{sect:density_smooth_maps}, in connection with the notion of \emph{extension property} introduced by Hang and Lin~\cite{HL_topology_of_sobolev_mappings_II}.

The map \( u^{\ext}_{\delta} \) allows us to remove the singularities of \( u_{\eta} \), but this topological construction does not allow to conclude that \( u^{\ext}_{\delta} \) is close to \( u_{\eta} \) with respect to the \( W^{s,p} \) distance, since \( u^{\ext}_{\delta} \) could have arbitrarily large energy on the set \( T^{[sp]^{\ast}}_{\eta}+Q^{m}_{\delta} \) where it differs from \( u_{\eta} \).
To overcome this issue, we use a scaling argument to obtain a better extension.
Again, we illustrate the method on the model case where \( s = 1 \) and \( m-1 < p < m \).
Assume that \( v \in W^{1,p}(Q^{m}) \) and that \( w \in \Cc^{\infty}(\overline{Q}^{m}) \) coincides with \( v \) on \( Q^{m} \setminus Q^{m}_{\delta} \), where \( 0 < \delta < \frac{1}{2} \).
Given \( 0 < \tau < 1 \), we define \( w_{\tau} \) on \( Q^{m} \) by 
\[
    w_{\tau}(x) = 
    \begin{cases}
        w(x) & \text{if \( x \in Q^{m} \setminus B^{m}_{2\delta} \),} \\
        w\bigl(\frac{x}{\tau}\bigr) & \text{if \( x \in B^{m}_{\tau\delta} \),} \\
        w\Bigl(\frac{x}{\lvert x \rvert}\bigl(\frac{1}{2-\tau}(\lvert x \rvert-\tau\delta)+\delta\bigr)\Bigr) & \text{if \( x \in B^{m}_{2\delta} \setminus B^{m}_{\tau\delta} \).}
    \end{cases}
\]
This corresponds to shrinking \( w \) from \( B^{m}_{\delta} \) to \( B^{m}_{\tau\delta} \) while keeping it unchanged on \( Q^{m} \setminus B^{m}_{2\delta} \),
filling the transition region by linear interpolation.
By a change of variable, we estimate 
\begin{align*}
    \lVert Dw_{\tau} \rVert_{L^{p}(Q^{m})}^{p}
    &= 
    \lVert Dw_{\tau} \rVert_{L^{p}(B^{m}_{\tau\delta})}^{p} + \lVert Dw_{\tau} \rVert_{L^{p}(Q^{m}\setminus B^{m}_{\tau\delta})}^{p} \\
    &\leq 
    \Cl{cst:fixed_sketch2}\tau^{m-p}\lVert Dw \rVert_{L^{p}(B^{m}_{\delta})}^{p} + \Cl{cst:fixed_sketch3}\lVert Dw \rVert_{L^{p}(Q^{m}\setminus B^{m}_{\delta})}^{p}.
\end{align*}
Since \( v = w \) on \( Q^{m} \setminus Q^{m}_{\delta} \), we deduce that 
\[
    \lVert Dw_{\tau} \rVert_{L^{p}(Q^{m})}^{p}
    \leq 
    \Cr{cst:fixed_sketch2}\tau^{m-p}\lVert Dw \rVert_{L^{p}(B^{m}_{\delta})}^{p} + \Cr{cst:fixed_sketch3}\lVert Dv \rVert_{L^{p}(Q^{m}\setminus B^{m}_{\delta})}^{p}.
\]
Choosing \( \tau \) sufficiently small -- depending on \( \delta \) and on \( w \) -- we may therefore make so that 
\[
    \lVert Dw_{\tau} \rVert_{L^{p}(Q^{m})}^{p}
    \leq 
    \C\lVert Dv \rVert_{L^{p}(Q^{m})}^{p}.
\]
In Section~\ref{sect:shrinking}, we explain the technique of \emph{shrinking}, which is actually a more involved version of this scaling argument, devised in~\cite{BPVS_density_higher_order}*{Section~8} to handle lower order skeletons and higher order regularity.

An illustration of this idea is available on Figure~\ref{fig:shrinking_around_centers_of_bad_cubes}.
The point singularities in Figure~\ref{fig:thickening_around_centers_of_bad_cubes} have been patched with a topological extension, which has been shrinked into the small region in gray to obtain a map with controlled energy.

\begin{figure}[ht]
	\centering
	\includegraphics[page=19]{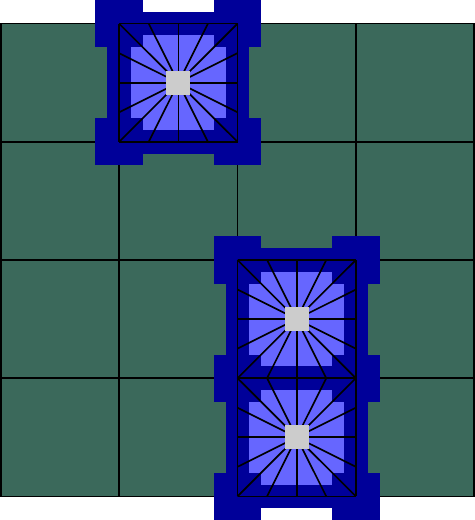}
	\caption{Shrinking around the centers of bad cubes}
	\label{fig:shrinking_around_centers_of_bad_cubes}
\end{figure}

This allows to proceed with the proof of Theorem~\ref{thm:density_smooth_functions} in Section~\ref{sect:density_smooth_maps}.
The strategy is exactly the same as in the model example above: we start with the smooth extension \( u^{\ext}_{\delta} \) provided by topological arguments, 
we shrink it to a map \( u^{\sh}_{\delta,\tau} \), and we use carefully the estimates available for shrinking to choose the parameter \( \tau > 0 \) 
in order to obtain a better extension with control of the energy.
As \( \delta \to 0 \), this provides an approximation of \( u_{\eta} \) by smooth maps with values into \( \Nc \), which is enough to prove Theorem~\ref{thm:density_smooth_functions}
since we already obtained the density of the class \( \Rc \).

\resetconstant

After this sketch of our proofs, we move to the detailed construction of the different tools that were described above.
The proofs being rather long and technical, we hope that this informal presentation will help the reader to identify and keep in mind the purpose and the intuition behind 
each construction when studying the details of the reasoning.

We end this section with two lemmas that will be used repeatedly in the sequel.
Most of our constructions on cubications will be built blockwise: we start from a building block defined on a cube, and we glue copies of this block on each cube of the skeleton to obtain a map defined on the whole skeleton.
When establishing Sobolev estimates for such constructions, integer order estimates on the skeleton are readily obtained from corresponding estimates on each cube by additivity of the integral.
On the contrary, the Gagliardo seminorm is not additive due to its nonlocal nature.
We bypass this obstruction by relying on the lemmas below.

\begin{lemme}
\label{lemma:fractional_additivity}
    Let \( \delta > 0 \) and let \( \Omega = \bigcup_{i \in I} \Omega_{i} \), 
    where \( I \) is finite or countable and \( \Omega_{i} \subset \R^{m} \) for every \( i \in I \).
    Set \( \Omega_{i,\delta} = \{ x \in \Omega\mathpunct{:} \dist(x,\Omega_{i}) < \delta\} \).
    For every \( u \colon \Omega \to \R \) measurable, one has
    \[
        \lvert u \rvert_{W^{\sigma,p}(\Omega)}^{p}
        \leq 
        \sum_{i \in I} \lvert u \rvert_{W^{\sigma,p}(\Omega_{i,\delta})}^{p}
            + C\delta^{-\sigma p}\lVert u \rVert_{L^{p}(\Omega)}^{p}
    \]
    for some constant \( C > 0 \) depending on \( m \), \( \sigma \), and \( p \).
\end{lemme}

This lemma acts as a replacement for the additivity for the Gagliardo seminorm.
Similar kind of results were already present in the work of Bourdaud concerning the continuity of the composition operator on Sobolev or Besov spaces; see e.g.~\cite{bourdaud93} and the references therein.
The price to pay to have a decomposition of the Gagliardo seminorm is that we need some margin of security between the different parts of the domain on which we split the energy,
and that an additional term involving the \( L^{p} \) norm of the map under consideration shows up, which deteriorates as the margin of security shrinks.
In the sequel, Lemma~\ref{lemma:fractional_additivity} will often be employed by taking the \( \Omega_{i} \) to be rectangles, 
which therefore suggests to have at our disposal estimates on rectangles slightly larger than the \( \Omega_{i} \).
Here we use the term \emph{rectangle} to denote any product of \( m \) intervals with non-empty interior.
We reserve the word \emph{cube} for the case where all the intervals have the same length.

\begin{proof}
    Let \( x \), \( y \in \Omega \).
    By assumption, either \( x \), \( y \in \Omega_{i,\delta} \) for some \( i \in I \), or \( \lvert x-y \rvert \geq \delta \).
    Otherwise stated,
    \[
        \Omega \times \Omega
        \subset 
        \{(x,y) \in \Omega \times \Omega \mathpunct{:} \lvert x-y \rvert \geq \delta\} \cup \bigcup_{i \in I} \Omega_{i,\delta} \times \Omega_{i,\delta}.
    \]
    Therefore,
    \[
        \lvert u \rvert_{W^{\sigma,p}(\Omega)}^{p}
        \leq
        \sum_{i \in I} \lvert u \rvert_{W^{\sigma,p}(\Omega_{i,\delta})}^{p}
        + \int_{\{(x,y) \in \Omega \times \Omega\mathpunct{:} \lvert x-y \rvert \geq \delta\}} \frac{\lvert u(x)-u(y) \rvert^{p}}{\lvert x-y \rvert^{m+\sigma p}}\,\d x\d y.
    \]
    We estimate
    \begin{multline*}
        \int_{\{(x,y) \in \Omega \times \Omega\mathpunct{:} \lvert x-y \rvert \geq \delta\}} \frac{\lvert u(x)-u(y) \rvert^{p}}{\lvert x-y \rvert^{m+\sigma p}}\,\d x\d y \\
        \leq
        2^{p}\int_{\Omega}\lvert u(x) \rvert^{p}\biggl(\int_{\R^{m} \setminus B^{m}_{\delta}(x)} \frac{1}{\lvert x-y \rvert^{m+\sigma p}}\,\d y\biggr)\,\d x 
        \leq 
        \C \delta^{-\sigma p}\int_{\Omega} \lvert u(x) \rvert^{p}\,\d x.
    \end{multline*}
    This completes the proof of the lemma.
    \resetconstant
\end{proof}

It is also possible to obtain a replacement for the additivity for the Gagliardo seminorm without a term involving the \( L^{p} \) norm of the map under consideration.
The price to pay is that such an estimate only applies on finite decompositions, hence not covering the case where an infinite number of sets is involved.
If \( Q \subset \R^{m} \) is a rectangle, then \( \lambda Q \) is the rectangle having the same center as \( Q \) and sidelengths multiplied by \( \lambda \).

\begin{lemme}
\label{lemma:finite_fractional_additivity}
    Let \( 0 < \lambda < 1 \) and \( Q \subset \R^{m} \) be a rectangle.
    For every \( \Omega \subset \R^{m} \) such that \( Q \setminus \lambda Q \subset \Omega \) and every \( u \colon \Omega \to \R \) measurable, we have 
    \[
        \lvert u \rvert_{W^{\sigma,p}(\Omega)}
        \leq 
        C\Bigl(\lvert u \rvert_{W^{\sigma,p}(\Omega \cap Q)} + \lvert u \rvert_{W^{\sigma,p}(\Omega \setminus \lambda Q)}\Bigl)
    \]
    for some constant \( C > 0 \) depending on \( m \), \( \sigma \), \( p \), \( \lambda \), and the ratio between the largest and the smallest side of \( Q \).
\end{lemme}

Lemma~\ref{lemma:finite_fractional_additivity} is inspired from~\cite{MVS_uniform_boundedness_principles}*{Lemma~2.2}, and we follow their proof.
At the core of the argument lies a very classical averaging argument, which was already present in the proof of Besov's lemma; see e.g.~\cite{adams_sobolev_spaces}*{Proof of Lemma~7.44}.
A similar idea is also used in the proof of Morrey's embedding.
This type of argument will be used in multiple occasions in this paper.

We note that the constant \( C \) necessarily diverges to \( +\infty \) as \( \lambda \to 1 \).
Moreover, one cannot deduce an improved version of Lemma~\ref{lemma:fractional_additivity} without the \( L^{p} \) norm by applying Lemma~\ref{lemma:finite_fractional_additivity} inductively, 
since the constant \( C \) is actually larger than \( 1 \).
Hence, we may iterate the lemma to obtain an estimate for a decomposition into a finite number of sets, but the constant depends on the number of sets.

\begin{proof}
    We start by writing 
    \[
        \lvert u \rvert_{W^{\sigma,p}(\Omega)}^{p}
        \leq 
        \lvert u \rvert_{W^{\sigma,p}(\Omega \cap Q)}^{p} + \lvert u \rvert_{W^{\sigma,p}(\Omega \setminus \lambda Q)}^{p}
        + 2\int_{\Omega \cap \lambda Q}\int_{\Omega \setminus Q} \frac{\lvert u(x)-u(y) \rvert^{p}}{\lvert x-y \rvert^{m+\sigma p}}\,\d y\d x.
    \]
    Now we use the average estimate 
    \begin{multline}
    \label{eq:average_estimate_finite_subadditivity}
        \int_{\Omega \cap \lambda Q}\int_{\Omega \setminus Q} \frac{\lvert u(x)-u(y) \rvert^{p}}{\lvert x-y \rvert^{m+\sigma p}}\,\d y\d x \\
        \leq 
        2^{p-1}\biggl(\int_{\Omega \cap \lambda Q}\int_{\Omega \setminus  Q}\fint_{Q \setminus \lambda Q} \frac{\lvert u(x)-u(z) \rvert^{p}}{\lvert x-y \rvert^{m+\sigma p}}\,\d z\d y\d x \\
            + \int_{\Omega \cap \lambda Q}\int_{\Omega \setminus Q}\fint_{Q \setminus \lambda Q} \frac{\lvert u(z)-u(y) \rvert^{p}}{\lvert x-y \rvert^{m+\sigma p}}\,\d z\d y\d x\biggr).
    \end{multline}
    Let \( c > 0 \) be the length of the smallest side of \( Q \).
    Since \( \lvert x-y \rvert \geq c(1-\lambda) \) whenever \( x \in \lambda Q \) and \( y \in \Omega \setminus Q \), first integrating with respect to \( y \) in the first term on the right-hand side of~\eqref{eq:average_estimate_finite_subadditivity}, we find 
    \begin{multline*}
        \int_{\Omega \cap \lambda Q}\int_{\Omega \setminus  Q}\fint_{Q \setminus \lambda Q} \frac{\lvert u(x)-u(z) \rvert^{p}}{\lvert x-y \rvert^{m+\sigma p}}\,\d z\d y\d x \\
        \leq 
        \C\frac{1}{c^{\sigma p}(1-\lambda)^{\sigma p}}\int_{\Omega \cap \lambda Q}\fint_{Q \setminus \lambda Q} \lvert u(x)-u(z) \rvert^{p}\,\d z\d x.
    \end{multline*}
    Now we observe that \( \lvert Q \setminus \lambda Q \rvert \geq (1-\lambda^{m})c^{m} \) and that \( \lvert x-z \rvert \leq \Cl{cst:ratio_lengths_Q} c \) for \( x \in \lambda Q \) and \( z \in Q \setminus \lambda Q \).
    Here \( \Cr{cst:ratio_lengths_Q} \) depends on the ratio between the largest and the smallest side of \( Q \).
    This allows us to conclude that 
    \begin{multline*}
        \int_{\Omega \cap \lambda Q}\int_{\Omega \setminus  Q}\fint_{Q \setminus \lambda Q} \frac{\lvert u(x)-u(z) \rvert^{p}}{\lvert x-y \rvert^{m+\sigma p}}\,\d z\d y\d x \\
        \leq 
        \C\frac{1}{(1-\lambda)^{\sigma p}(1-\lambda^{m})}\int_{\Omega \cap \lambda Q}\int_{Q \setminus \lambda Q} \frac{\lvert u(x)-u(z) \rvert^{p}}{\lvert x-z \rvert^{m+\sigma p}}\,\d z\d x
        \leq 
        \C\lvert u \rvert_{W^{\sigma,p}(\Omega \cap Q)}^{p}.
    \end{multline*}

    For the second term in the right-hand side of~\eqref{eq:average_estimate_finite_subadditivity}, we start by noting that if \( x \in \lambda Q \) and \( y \in \partial (rQ) \) for some \( r \geq 1 \), then \( \lvert x-y \rvert \geq c(r-\lambda) \).
    On the other hand, if \( y \in \partial (rQ) \) and \( z \in Q \setminus \lambda Q \), then 
    \[
    	\lvert y-z \rvert 
    	\leq 
    	\Cl{cst:ratio_lengths_Q_bis}c(r+1)
    	=
    	\Cr{cst:ratio_lengths_Q_bis}c\frac{r+1}{r-\lambda}(r-\lambda)
    	\leq 
    	\Cl{cst:yz_vs_xy} c(r-\lambda),
    \]
    where \( \Cr{cst:ratio_lengths_Q_bis} \) depends on the ratio between the largest and the smallest side of \( Q \).
    Therefore, for any \( x \in \lambda Q \), \( y \in \Omega \setminus Q \), and \( z \in Q \setminus \lambda Q \), we have \( \lvert y-z \rvert \leq \Cr{cst:yz_vs_xy}\lvert x-y \rvert \).
    Hence, we obtain
    \begin{multline*}
        \int_{\Omega \cap \lambda Q}\int_{\Omega \setminus Q}\fint_{Q \setminus \lambda Q} \frac{\lvert u(z)-u(y) \rvert^{p}}{\lvert x-y \rvert^{m+\sigma p}}\,\d z\d y\d x \\
        \leq 
        \C\frac{\lambda^{m}}{1-\lambda^{m}}\int_{\Omega \setminus Q}\int_{Q \setminus \lambda Q} \frac{\lvert u(z)-u(y) \rvert^{p}}{\lvert y-z\rvert^{m+\sigma p}}\,\d z\d y
        \leq 
        \C\lvert u \rvert_{W^{\sigma,p}(\Omega \setminus \lambda Q)}^{p}.
    \end{multline*}
    Gathering the estimates for both terms in the right-hand side of~\eqref{eq:average_estimate_finite_subadditivity} yields the conclusion.
\resetconstant
\end{proof}

\section{Opening}
\label{sect:opening}

This section is devoted to the opening procedure.
We follow the approach of Bousquet, Ponce, and Van Schaftingen~\cite{BPVS_density_higher_order}*{Section~2}, who adapted to higher order regularity a construction of Brezis and Li~\cite{BL_topology_sobolev_spaces}.
The main result of this section is the following fractional counterpart of~\cite{BPVS_density_higher_order}*{Proposition~2.1}, which contains the opening construction.
Recall that we write \( s = k + \sigma \), with \( k \in \N \) and \( \sigma \in [0,1) \).
Note carefully that the map \( \upPhi \) constructed below \emph{depends on} the map \( u \in W^{s,p} \) it is composed with.

\begin{prop}
\label{prop:main_opening}
    Let \( \Omega \subset \R^{m} \) be open, \( \ell \in \{0,\dots,m-1\} \), \( \eta > 0 \), \( 0 < \rho < \frac{1}{2} \), and \( \Uc^{\ell} \) be a subskeleton of \( \R^{m} \) of radius \( \eta \) such that
    \( U^{\ell} + Q^{m}_{2\rho\eta} \subset \Omega \).
    For every \( u \in W^{s,p}(\Omega;\R^{\nu}) \), there exists a smooth map \( \upPhi \colon \R^{m} \to \R^{m} \) such that
    \begin{enumerate}[label=(\roman*)]
        \item\label{item:first_geometric_main_opening} for every \( d \in \{0,\dots,\ell\} \) and for every \( \sigma^{d} \in \Uc^{d} \), \( \upPhi \) is constant on the \( (m-d) \)-dimensional cubes of radius \( \rho\eta \) which are orthogonal to \( \sigma^{d} \);
        \item\label{item:second_geometric_main_opening} \( \Supp \upPhi \subset U^{\ell} + Q^{m}_{2\rho\eta} \) and \( \upPhi(U^{\ell} + Q^{m}_{2\rho\eta}) \subset U^{\ell} + Q^{m}_{2\rho\eta} \);
        \item\label{item:first_estimates_main_opening} \( u \circ \upPhi \in W^{s,p}(\Omega;\R^{\nu}) \), and moreover, for every \( \omega \subset \Omega \) such that \( U^{\ell}+Q^{m}_{2\rho\eta} \subset \omega \),
            the following estimates hold:
            \begin{enumerate}[label=(\alph*)]
            	\item\label{item:first_estimate_main_opening_frac_low} if \( 0 < s < 1 \), then
	            	\[
		            	\eta^{s}\lvert u\circ\upPhi \rvert_{W^{s,p}(\omega)} 
		            	\leq 
		            	C\Bigl(\eta^{s}\lvert u \rvert_{W^{s,p}(\omega)} + \lVert u \rVert_{L^{p}(\omega)} \Bigr);
	            	\]
                \item\label{item:first_estimate_main_opening_integer} if \( s \geq 1 \), then for every \( j \in \{1,\dots,k\} \),
                    \[
                        \eta^{j}\lVert D^{j}(u\circ\upPhi) \rVert_{L^{p}(\omega)} \leq C\sum_{i=1}^{j}\eta^{i}\lVert D^{i}u \rVert_{L^{p}(\omega)};
                    \]
                \item\label{item:first_estimate_main_opening_frac_high} if \( s \geq 1 \) and \( \sigma \neq 0 \), then for every \( j \in \{1,\dots,k\} \),
                    \[
                        \eta^{j+\sigma}\lvert D^{j}(u\circ\upPhi) \rvert_{W^{\sigma,p}(\omega)} \leq C\sum_{i=1}^{j}\Bigl(\eta^{i}\lVert D^{i}u \rVert_{L^{p}(\omega)} 
                        + \eta^{i+\sigma}\lvert D^{i}u \rvert_{W^{\sigma,p}(\omega)}\Bigr);
                    \]
                \item\label{item:first_estimate_main_opening_all} for every \( 0 < s < +\infty \),
                    \[
                        \lVert u \circ \upPhi \rVert_{L^{p}(\omega)} \leq C \lVert u \rVert_{L^{p}(\omega)};
                    \]
            \end{enumerate}
        \item\label{item:second_estimates_main_opening} for every \( \omega \subset \Omega \) such that \( U^{\ell}+Q^{m}_{2\rho\eta} \subset \omega \), the following estimates hold:
            \begin{enumerate}[label=(\alph*)]
            	\item\label{item:second_estimate_main_opening_frac_low} if \( 0 < s < 1 \), then
	            	\[
	            	\eta^{s}\lvert u\circ\upPhi - u \rvert_{W^{s,p}(\omega)} 
	            	\leq 
	            	C\Bigl(\eta^{s}\lvert u \rvert_{W^{s,p}(U^{\ell} + Q^{m}_{2\rho\eta})} + \lVert u \rVert_{L^{p}(U^{\ell} + Q^{m}_{2\rho\eta})} \Bigr);
	            	\]
                \item\label{item:second_estimate_main_opening_integer} if \( s \geq 1 \), then for every \( j \in \{1,\dots,k\} \),
                    \[
                        \eta^{j}\lVert D^{j}(u\circ\upPhi)-D^{j}u \rVert_{L^{p}(\omega)} \leq C\sum_{i=1}^{j}\eta^{i}\lVert D^{i}u \rVert_{L^{p}(U^{\ell} + Q^{m}_{2\rho\eta})};
                    \]
                \item\label{item:second_estimate_main_opening_frac_high} if \( s \geq 1 \) and \( \sigma \neq 0 \), then for every \( j \in \{1,\dots,k\} \),
                    \begin{multline*}
                        \eta^{j+\sigma}\lvert D^{j}(u\circ\upPhi)-D^{j}u \rvert_{W^{\sigma,p}(\omega)} \leq C\sum_{i=1}^{j}\Bigl(\eta^{i}\lVert D^{i}u \rVert_{L^{p}(U^{\ell} + Q^{m}_{2\rho\eta})} \\
                        + \eta^{i+\sigma}\lvert D^{i}u \rvert_{W^{\sigma,p}(U^{\ell} + Q^{m}_{2\rho\eta})}\Bigr);
                    \end{multline*}
                \item\label{item:second_estimate_main_opening_all} for every \( 0 < s < +\infty \),
                    \[
                        \lVert u \circ \upPhi - u \rVert_{L^{p}(\omega)} \leq C\lVert u \rVert_{L^{p}(U^{\ell}+Q^{m}_{2\rho\eta})};
                    \]
            \end{enumerate}
    \end{enumerate}
    for some constant \( C > 0 \) depending on \( m \), \( s \), \( p \), and \( \rho \).
\end{prop}

Recall that \( \Supp\upPhi \) denotes the geometric support of \( \upPhi \), defined as
\[
	\Supp\upPhi = \overline{\{x \in \R^{m}\mathpunct{:} \upPhi(x) \neq x\}}.
\]

Crucial to the proof of Theorem~\ref{thm:density_class_R} are the estimates in~\ref{item:first_estimates_main_opening} with \( \omega = U^{\ell}+Q^{m}_{2\rho\eta} \).
They imply that the opening procedure does not increase too much the energy of the map \( u \) where it is modified.
Proposition~\ref{prop:main_opening} will be used in the proof of Theorem~\ref{thm:density_class_R} in order to prove that a map can be opened by paying the price of an arbitrarily small increase of the norm.

The map \( \upPhi \) will be constructed blockwise: for every \( d \in \{0,\dots,\ell\} \) and every \( \sigma^{d} \in \Uc^{d} \), we construct an opening map \( \upPhi_{\sigma^{d}} \) around the face \( \sigma^{d} \), and then we suitably combine those maps together to yield the desired map \( \upPhi \).
The construction of the building block \( \upPhi_{\sigma_{d}} \) is performed in Proposition~\ref{prop:block_opening} below.
Before giving a precise statement, we first introduce, for the convenience of the reader, some additional notation.

The construction of the map \( \upPhi \) provided by Proposition~\ref{prop:block_opening} involves four parameters \( 0 < \ulrho < \ulr < \olr < \olrho < 1 \).
These parameters being fixed, we introduce the rectangles
\begin{equation}
\label{eq:def_rectangles_opening}
    \begin{split}
    Q_{1} = Q_{1,\eta} = Q^{d}_{(1-\olrho)\eta} \times Q^{m-d}_{\ulrho\eta},
    \quad
    &Q_{2} = Q_{2,\eta} = Q^{d}_{(1-\olr)\eta} \times Q^{m-d}_{\ulr\eta}, \\
    \quad
    Q_{3} = Q_{3,\eta} = Q^{d}_{(1-\ulr)\eta} \times Q^{m-d}_{\olr\eta}
    \quad
    \text{and}&
    \quad
    Q_{4} = Q_{4,\eta} = Q^{d}_{(1-\ulrho)\eta} \times Q^{m-d}_{\olrho\eta}.
    \end{split}
\end{equation}
The rectangle \( Q_{1} \) is the place where the opening construction is actually performed: the map \( \upPhi \) only depends on the first \( d \) variables on \( Q_{1} \).
The rectangle \( Q_{2} \) contains the support of the map \( \upPhi \), that is, \( \upPhi \) coincides with the identity outside of \( Q_{2} \).
The region between \( Q_{1} \) and the exterior of \( Q_{2} \) serves as a transition region between both regimes.

From now on we shall keep using the notation \( Q_{1} \), \dots, \( Q_{4} \) for the sake of conciseness and because it makes more apparent the inclusion relations between the four rectangles:
observe that \( Q_{1} \subset Q_{2} \subset Q_{3} \subset Q_{4} \).
The dependence with respect to the parameters \( \ulrho \), \( \ulr \), \( \olr \), \( \olrho \), and \( \eta \) will be implicit.

\begin{prop}
\label{prop:block_opening}
    Let \( d \in \{0,\dots,m-1\} \), \( \eta > 0 \), and \( 0 < \ulrho < \ulr < \olr < \olrho < 1 \).
    For every \( u \in W^{s,p}(Q_{4};\R^{\nu}) \), 
    there exists a smooth map \( \upPhi \colon Q_{4} \to Q_{4} \) such that
    \begin{enumerate}[label=(\roman*)]
        \item\label{item:form_block_opening} \( \upPhi(x',x'') = (x', \zeta(x)) \) for every \( x = (x',x'') \in Q_{4} \), 
            where \( \zeta \colon Q_{4} \to Q^{m-d}_{\olrho\eta} \) is smooth;
        \item\label{item:block_opening_constant} for every \( x' \in Q^{d}_{(1-\olrho)\eta} \)\,, \( \upPhi \) is constant on \( \{x'\} \times Q^{m-d}_{\ulrho\eta} \);
        \item\label{item:support_block_opening} \( \Supp \upPhi \subset Q_{2} \) and \( \upPhi(Q_{2}) \subset Q_{2} \);
        \item\label{item:estimates_block_opening} \( u \circ \upPhi \in W^{s,p}(Q_{3};\R^{\nu}) \), and moreover, the following estimates hold:
        	\begin{enumerate}
        		\item\label{item:estimate_block_opening_0s1} if \( 0 < s < 1 \), then
        			\[
	        			\lvert u\circ\upPhi \rvert_{W^{s,p}(Q_{3})} 
	        			\leq 
	        			C\lvert u \rvert_{W^{s,p}(Q_{4})};
        			\]
        		\item\label{item:estimate_block_opening_sge1_integer} if \( s \geq 1 \), then for every \( j \in \{1,\dots,k\} \),
        			\[
	        			\eta^{j}\lVert D^{j}(u\circ\upPhi) \rVert_{L^{p}(Q_{3})} 
	        			\leq 
	        			C\sum_{i=1}^{j}\eta^{i}\lVert D^{i}u \rVert_{L^{p}(Q_{4})};
        			\]
        		\item\label{item:estimate_block_opening_sge1_frac} if \( s \geq 1 \) and \( \sigma \neq 0 \), then for every \( j \in \{1,\dots,k\} \),
        			\[
	        			\eta^{j+\sigma}\lvert D^{j}(u\circ\upPhi) \rvert_{W^{\sigma,p}(Q_{3})} 
	        			\leq 
	        			C\sum_{i=1}^{j}\Bigl(\eta^{i}\lVert D^{i}u \rVert_{L^{p}(Q_{4})} 
	        			+ \eta^{i+\sigma}\lvert D^{i}u \rvert_{W^{\sigma,p}(Q_{4})} \Bigr);
        			\]
        		\item\label{item:estimate_block_opening_all} for every \( 0 < s < +\infty \),
        			\[
	        			\lVert u \circ \upPhi\rVert_{L^{p}(Q_{3})} \leq C\lVert u \rVert_{L^{p}(Q_{4})};
        			\]
        	\end{enumerate}
    \end{enumerate}
    for some constant \( C > 0 \) depending on \( m \), \( s \), \( p \), \( \ulrho \), \( \ulr \), \( \olr \), and \( \olrho \).
\end{prop}

We comment on the domains involved in the estimates of item~\ref{item:estimates_block_opening} in Proposition~\ref{prop:block_opening} above.
We need estimates on the rectangle \( Q_{3} \) instead of the smaller rectangle \( Q_{2} \) containing the support of \( \upPhi \), in order to have enough room to apply Lemmas~\ref{lemma:fractional_additivity} and~\ref{lemma:finite_fractional_additivity} as substitutes for the additivity of the integral when proving the fractional estimates in Proposition~\ref{prop:main_opening}.
Moreover, we only control the energy on \( Q_{3} \) by the energy on the larger rectangle \( Q_{4} \) due to the averaging process involved in the proof of Proposition~\ref{prop:block_opening}, as we will see later on.

Taking Proposition~\ref{prop:block_opening} as granted for the moment, we proceed with the proof of Proposition~\ref{prop:main_opening}.
Before providing a detailed rigorous proof, we sketch the argument.

We first open the map \( u \) around each vertex of \( \Uc^{0} \) by applying Proposition~\ref{prop:block_opening} with \( d = 0 \) and using parameters \( \olrho = 2\rho \) and \( \ulrho = \rho_{0} < 2\rho \).
This produces a map \( u^{0} \) which is constant on cubes of radius \( \rho_{0}\eta \) around each vertex of \( \Uc^{0} \).
We next open the map \( u^{0} \) around each edge of \( \Uc^{1} \) using Proposition~\ref{prop:block_opening} with \( d = 1 \), \( \olrho = \rho_{0} \), and \( \ulrho = \rho_{1} < \rho_{0} \).
One may see that the geometric supports of the building blocks around each face do not overlap, so that we may glue them together to obtain a well-defined map on the whole \( \Omega \).
This construction yields a map \( u^{1} \) which is constant on all \( (m-1) \)-cubes of radius \( \rho_{1}\eta \) which are orthogonal to the edges of \( \Uc^{1} \), provided that they lie at distance at least \( \rho_{0}\eta \) from the endpoints of the edges.
But the map \( u^{1} \) is constructed from the map \( u^{0} \) which was constant on the cubes of radius \( \rho_{0}\eta \) centered at the vertices of \( \Uc^{0} \).
Hence we conclude that the map \( u^{1} \) is constant on all \( (m-1) \)-cubes of radius \( \rho_{1}\eta \) which are orthogonal to the edges of \( \Uc^{1} \).
We then pursue this construction by induction until we reach the desired dimension, which yields a map \( \upPhi \) as in Proposition~\ref{prop:main_opening}.

An illustration of this construction on one cube for \( m = 2 \) and \( \ell = 1 \) is presented in Figure~\ref{fig:opening}.
On the left part of the figure, one sees the result of opening around vertices.
The map \( u \) becomes constant on the dark blue squares, and is left unchanged on the white region, the light blue region serving as a transition.
The central part of the figure shows the opening step around edges.
The map \( u \) becomes constant on the segments orthogonal to the edges of \( \Uc^{1} \) that are sufficiently far from the vertices, some of which being represented in black.
The regions involved in the construction at the previous step, when opening around the vertices, are depicted in light colors, to show how all the regions are located relatively to each other.
One sees that the opening regions around vertices and edges connect perfectly.
The right part of the figure shows the combination of both steps.
The map \( u \) becomes constant on all segments orthogonal to the edges of \( \Uc^{1} \).

\begin{figure}[ht]
	\centering
	~
	\hfill
	\includegraphics[page=1]{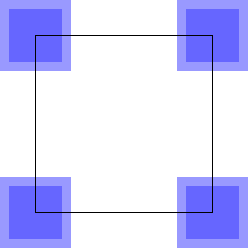}
	\hfill
	\includegraphics[page=2]{figures_strong_density.pdf}
	\hfill
	\includegraphics[page=3]{figures_strong_density.pdf}
	\hfill
	~
	\caption{Opening for \( m = 2 \) and \( \ell = 1 \)}
	\label{fig:opening}
\end{figure}

The construction sketched above is strongly inspired by~\cite{BPVS_density_higher_order}*{Proposition~2.1}, but nevertheless significantly different from~\cite{BPVS_density_higher_order}*{Proposition~2.1}.
Indeed, in our approach, at each step of the iterative process, the sets on which we apply opening around each face (the building blocks of our global construction) do not overlap. 
Hence, gathering the constructions made around each face yields a globally well-defined map on the whole \( \Omega \), regardless of the map we started with at the beginning of the step.
For instance, the map \( u^{1} \) described in the above sketch is well-defined on the whole \( \Omega \), regardless of the form of the map \( u^{0} \).
On the other hand, the construction in~\cite{BPVS_density_higher_order} relies on the fact that, at the \( d \)-th step of the iterative process, we work with a map that has already been opened around the \( i \)-faces for \( i < d \). 
Indeed the constructions made at step \( d \) are not compatible near the lower dimensional faces, where they overlap.
Our approach simplifies the proof of Sobolev estimates, especially in the fractional case where one needs some margin of security to apply Lemmas~\ref{lemma:fractional_additivity} and~\ref{lemma:finite_fractional_additivity}, but also for the case of integer order estimates (already treated in~\cite{BPVS_density_higher_order}).

\begin{proof}[Proof of Proposition~\ref{prop:main_opening}]
    As announced, we construct a family of maps \( (\upPhi_{d})_{0 \leq d \leq \ell} \) by induction.
    For the convenience of notation we set \( \upPhi_{-1} = \id \).
    Assuming that the maps \( \upPhi_{-1} \), \dots, \( \upPhi_{d-1} \) have already been constructed, we set \( u^{d} = u \circ \upPhi_{-1} \circ \cdots \circ \upPhi_{d-1} \).
    Let \( (\rho_{d})_{0 \leq d \leq \ell} \), \( (\ulr_{d})_{0 \leq d \leq \ell} \), and \( (\olr_{d})_{0 \leq d \leq \ell} \) be decreasing sequences such that
    \[
        \rho = \rho_{\ell} < \ulr_{\ell} < \olr_{\ell} < \rho_{\ell-1} < \cdots < \rho_{d} < \ulr_{d} < \olr_{d} < \rho_{d-1} < \cdots <  \rho_{0} < \ulr_{0} < \olr_{0} < 2\rho.
    \]
    For every \( d \in \{0,\dots,\ell\} \) and every \( \sigma^{d} \in \Uc^{d} \), there is an isometry \( T_{\sigma^{d}} \) of \( \R^{m} \) mapping \( Q^{d}_{\eta} \times \{0\}^{m-\ell} \) onto \( \sigma^{d} \).
    \emph{Via} this isometry, we apply Proposition~\ref{prop:block_opening} to \( u^{d} \) around \( \sigma^{d} \) 
    with parameters \( \ulrho = \rho_{d} \), \( \ulr = \ulr_{d} \), \( \olr = \olr_{d} \) and \( \olrho = \rho_{d-1} \) -- with the convention that \( \rho_{-1} = 2\rho \) --
    in order to obtain a map \( \upPhi^{\sigma^{d}} \colon T_{\sigma^{d}}(Q_{4}) \to T_{\sigma^{d}}(Q_{4}) \)
    such that, for every \( x' \in \sigma^{d} \) with \( \dist{(x',\partial\sigma^{d})} > \rho_{d-1} \), \( \upPhi^{\sigma^{d}} \) is constant on the cube orthogonal to \( \sigma^{d} \) of radius \( \rho_{d}\eta \)
    passing through \( x' \).
    We then define \( \upPhi_{d} \colon \R^{m} \to \R^{m} \) by
    \[
        \upPhi_{d}(x) = 
        \begin{cases}
            \upPhi^{\sigma^{d}}(x) & \text{if \( x \in T_{\sigma^{d}}(Q_{4}) \),} \\
            x & \text{otherwise.}
        \end{cases}
    \]
    This map is well-defined since \( \Supp\upPhi^{\sigma^{d}} \subset T_{\sigma^{d}}(Q_{2}) \) and 
    \( T_{\sigma^{d}_{1}}(Q_{2}) \cap T_{\sigma^{d}_{2}}(Q_{2}) = \varnothing \) if \( \sigma^{d}_{1} \neq \sigma^{d}_{2} \).
    Finally, we set \( \upPhi = \upPhi_{0} \circ \cdots \circ \upPhi_{\ell} \).

    By induction and using the definition of the maps \( \upPhi^{\sigma^{d}} \) provided by Proposition~\ref{prop:block_opening}, we observe that \( \upPhi \) satisfies properties~\ref{item:first_geometric_main_opening} and~\ref{item:second_geometric_main_opening}.
    We now turn to properties~\ref{item:first_estimates_main_opening} and~\ref{item:second_estimates_main_opening}.
    Let \( U^{\ell}+Q^{m}_{2\rho\eta} \subset \omega \subset \Omega \).
    We notice that it suffices to prove property~\ref{item:first_estimates_main_opening} with \( \upPhi \) replaced by \( \upPhi_{d} \) and \( u \) replaced by \( u^{d} \), as one may then conclude by induction.

    We start with the estimates for integer order derivatives.
    Let \( j \in \{1,\dots,k\} \) and \( d \in \{0,\dots,\ell\} \).
    By the additivity of the integral, we have
    \[
        \lVert D^{j}(u^{d}\circ\upPhi_{d}) \rVert_{L^{p}(\omega)}^{p}
        \leq
        \sum_{\sigma^{d} \in \Uc^{d}} \lVert D^{j}(u^{d}\circ\upPhi_{d}) \rVert_{L^{p}(T_{\sigma^{d}}(Q_{3}))}^{p} 
            + \lVert D^{j}(u^{d}\circ\upPhi_{d}) \rVert_{L^{p}\bigl(\omega\setminus\bigcup\limits_{\sigma^{d}\in\Uc^{d}}T_{\sigma^{d}}(Q_{3})\bigr)}^{p}.
    \]
    Since \( \Supp\upPhi^{\sigma^{d}} \subset T_{\sigma^{d}}(Q_{2}) \subset T_{\sigma^{d}}(Q_{3}) \), we find
    \[
        \lVert D^{j}(u^{d}\circ\upPhi_{d}) \rVert_{L^{p}\bigl(\omega\setminus\bigcup\limits_{\sigma^{d}\in\Uc^{d}}T_{\sigma^{d}}(Q_{3})\bigr)}
        =
        \lVert D^{j}u^{i} \rVert_{L^{p}\bigl(\omega\setminus\bigcup\limits_{\sigma^{d}\in\Uc^{d}}T_{\sigma^{d}}(Q_{3})\bigr)},
    \]
    while the estimate given by Proposition~\ref{prop:block_opening} yields
    \[
        \eta^{j}\lVert D^{j}(u^{d}\circ\upPhi_{d}) \rVert_{L^{p}(T_{\sigma^{d}}(Q_{3}))}
        =
        \eta^{j}\lVert D^{j}(u^{d}\circ\upPhi^{\sigma^{d}}) \rVert_{L^{p}(T_{\sigma^{d}}(Q_{3}))} 
        \leq
        \C\sum_{i=1}^{j}\eta^{i}\lVert D^{i}u^{d} \rVert_{L^{p}(T_{\sigma^{d}}(Q_{4}))}.
    \]
    Combining both above estimates and using the fact that the number of overlaps between one given set of the form \( T_{\sigma^{d}}(Q_{4}) \) and all the other such sets is bounded
    from above by a number depending only on \( m \), we deduce that
    \[
        \eta^{j}\lVert D^{j}(u^{d}\circ\upPhi_{d}) \rVert_{L^{p}(\omega)} \leq \C\sum_{i=1}^{j}\eta^{i}\lVert D^{i}u^{d} \rVert_{L^{p}(\omega)}
        \quad
        \text{for every \( j \in \{1,\dots,k\} \)}.
    \]
    Since \( \Supp\upPhi \subset U^{\ell}+Q^{m}_{2\rho\eta} \), the estimate~\ref{item:second_estimate_main_opening_integer} of point~\ref{item:second_estimates_main_opening} follows directly from estimate~\ref{item:first_estimate_main_opening_integer} of point~\ref{item:first_estimates_main_opening}
    using again the additivity of the integral.
    The estimates for the \( L^{p} \) norm of \( u \circ \upPhi \) (estimates~\ref{item:first_estimate_main_opening_all}) are proven similarly.

    The estimates for the Gagliardo seminorm are proved similarly, replacing the additivity of the integral by Lemma~\ref{lemma:fractional_additivity}.
    Indeed, if \( k \geq 1 \), this lemma ensures that 
    \begin{multline*}
        \lvert D^{j}(u^{d}\circ\upPhi_{d}) \rvert_{W^{\sigma,p}(\omega)}^{p}
        \leq
        \sum_{\sigma^{d} \in \Uc^{d}} \lvert D^{j}(u^{d}\circ\upPhi_{d}) \rvert_{W^{\sigma,p}(T_{\sigma^{d}}(Q_{3}))}^{p} \\
            + \lvert D^{j}(u^{d}\circ\upPhi_{d}) \rvert_{L^{p}(\omega\setminus\Supp\upPhi_{d})}^{p}
            + \C\eta^{-\sigma p}\lVert D^{j}(u^{d} \circ \upPhi_{d}) \rVert_{L^{p}(\omega)}^{p}
        \quad 
        \text{for every \( j \in \{1,\dots,k\} \).}    
    \end{multline*}
    We note that here, we made use of the fact that the distance between the support of the map provided by Proposition~\ref{prop:block_opening} and the complement of \( Q_{3} \) 
    is bounded from below by a constant multiple of \( \eta \).
    Estimate~\ref{item:first_estimate_main_opening_frac_high} of point~\ref{item:first_estimates_main_opening} then follows as for the integer order estimate.
    To obtain the estimate~\ref{item:second_estimate_main_opening_frac_high} for point~\ref{item:second_estimates_main_opening}, we observe that actually \( \dist{(\Supp\upPhi,\omega \setminus (U^{\ell}+Q^{m}_{2\rho\eta}))} \) 
    is bounded from below by a constant multiple of \( \eta \).
    We conclude by making again use of Lemma~\ref{lemma:fractional_additivity} along with the integer order estimate that we already obtained.
    Indeed, we have 
    \begin{multline*}
        \lvert D^{j}(u \circ \upPhi)-D^{j}u \rvert_{W^{\sigma,p}(\omega)}^{p}
        \leq 
        \lvert D^{j}(u \circ \upPhi)-D^{j}u \rvert_{W^{\sigma,p}(U^{\ell}+Q^{m}_{2\rho\eta})}^{p} \\
            + \lvert D^{j}(u \circ \upPhi)-D^{j}u \rvert_{W^{\sigma,p}(\omega \setminus \Supp\upPhi)}^{p}
            + \eta^{-\sigma p}\lVert D^{j}(u \circ \upPhi)-D^{j}u \rVert_{L^{p}(\omega)}^{p}.
    \end{multline*}
    The first term is upper bounded using the triangle inequality and estimate~\ref{item:first_estimate_main_opening_frac_high} of~\ref{item:first_estimates_main_opening}, the second one vanishes by definition of the geometric support,
    and the third one is the integer order term that we already estimated (item~\ref{item:second_estimate_main_opening_integer} of~\ref{item:second_estimates_main_opening}).
    The case \( 0 < s < 1 \) is handled in the same way, replacing \( D^{j}u \) by \( u \).
    \resetconstant
\end{proof}

We now turn to the proof of Proposition~\ref{prop:block_opening}.
Consider some fixed Borel map \( u \colon Q^{m} \to \Nc \) such that \( u \in W^{s,p}(Q^{m}) \).
In order to prove that there exists some map \( \upPhi \) (\emph{depending on \( u \)}) such that \( u \circ \upPhi \in W^{s,p}(Q^{m};\Nc) \) along with the corresponding estimates, it will be convenient to rely on some genericity arguments using the framework of \emph{Fuglede maps}, in a formalism developed by Bousquet, Ponce, and Van Schaftingen in~\cite{BPVS_screening}; our presentation is limited to the tools instrumental in our proofs.
The results below are taken from~\cite{BPVS_screening}, sometimes with slight modifications.
Nevertheless, we reproduce the proofs here for the convenience of the reader.

We start with the following lemma, see \cite{BPVS_screening}*{Proposition~2.1}, suited for \( L^{p} \) regularity, which gives a criterion to detect a family of maps \( \gamma \) such that composition with \( \gamma \) is compatible with \( L^{p} \) convergence. 

\begin{lemme}
    \label{lemma:opening_fuglede}
    Let \( (X,\mathcal{X},\mu) \) be a measure space, \( u \colon X \to \R \) a measurable map which does not vanish \( \mu \)-almost everywhere, and \( (u_{n})_{n \in \N} \) a sequence of maps in \( L^{p}(X,\mu) \) such that \( u_{n} \to u \) in \( L^{p}(X,\mu) \).
    There exists a summable function \( w \colon X \to [0,+\infty] \) satisfying \( \int_{X} w\,\d\mu > 0 \) and a subsequence \( (u_{n_{i}})_{i \in \N} \) such that for every measure space \( (Y,\mathcal{Y},\lambda) \) and every measurable map 
    \( \gamma \colon Y \to X \) satisfying \( w \circ \gamma \in L^{1}(Y,\lambda) \), we have \( u_{n_{i}} \circ \gamma \in L^{p}(Y,\lambda) \), 
    \[
        u_{n_{i}} \circ \gamma \to u \circ \gamma
        \quad 
        \text{in \( L^{p}(Y,\lambda) \)},
    \]
    and
    \[
        \int_{Y} \lvert u \circ \gamma \rvert^{p} \,\d\lambda
        \leq
        2\frac{\int_{Y} w \circ \gamma\,\d\lambda}{\int_{X} w\,\d\mu}\int_{X} \lvert u \rvert^{p}\,\d \mu.
    \]
\end{lemme}

We insist on the fact that the map \( w \) depends on \( u \).
Even modifying \( u \) on a null set may change the map \( w \) given by Lemma~\ref{lemma:opening_fuglede}.

\begin{proof}
    We choose a sequence \( (\kappa_{i})_{i \in \N} \) diverging to \( +\infty \) such that \( \kappa_{i} \geq 1 \) for every \( i \in \N \).
    We then extract a subsequence \( (u_{n_{i}})_{i \in \N} \) so that
    \[
        \lVert u \rVert_{L^{p}(X,\mu)} + \sum_{i \in \N} \kappa_{i}\lVert u_{n_{i}} - u \rVert_{L^{p}(X,\mu)} < 2^{\frac{1}{p}}\lVert u \rVert_{L^{p}(X,\mu)}
    \]
    and we define \( w \colon X \to [0,+\infty] \) by 
    \[
        w = \biggl(\lvert u \rvert + \sum_{i \in \N} \kappa_{i}\lvert u_{n_{i}} - u \rvert\biggr)^{p}.
    \]
    We deduce from the triangle inequality and Fatou's lemma that \( w \) is summable with
    \begin{equation}
    \label{estimate:int_w}
        \biggl(\int_{X} w \,\d \mu \biggr)^{\frac{1}{p}}
        =
        \lVert w^{\frac{1}{p}} \rVert_{L^{p}(X,\mu)}
        \leq 
        \lVert u \rVert_{L^{p}(X,\mu)} + \sum_{i \in \N} \kappa_{i}\lVert u_{n_{i}} - u \rVert_{L^{p}(X,\mu)} 
        < 
        2^{\frac{1}{p}}\lVert u \rVert_{L^{p}(X,\mu)}.
    \end{equation}
    Since \( \kappa_{i} \geq 1 \), we have
    \[
        \lvert u_{n_{i}} \rvert^{p} \leq \Bigl( \lvert u \rvert + \lvert u_{n_{i}} - u \rvert\Bigr)^{p} \leq w.
    \]
    Hence, if \( \gamma \colon Y \to X \) satisfies \( w \circ \gamma \in L^{1}(Y,\lambda) \), we find that \( u_{n_{i}} \circ \gamma \in L^{p}(Y,\lambda) \), and moreover, we have
    \[
        \int_{Y} \lvert u_{n_{i}} \circ \gamma - u \circ \gamma \rvert^{p}\,\d \lambda
        \leq 
        \frac{1}{\kappa_{i}^{p}}\int_{Y} w \circ \gamma\,\d \lambda.
    \]
    Letting \( i \to +\infty \) allows us to conclude that \( u_{n_{i}} \circ \gamma \to u \circ \gamma \) in \( L^{p}(Y,\lambda) \).
    Furthermore, since \( \lvert u \circ \gamma \rvert^{p} \leq w \circ \gamma \), we obtain
    \[
        \int_{Y} \lvert u \circ \gamma \rvert^{p}\,\d \lambda
        \leq 
        \int_{Y} w \circ \gamma\,\d \lambda.
    \]
    Combining the above inequality with~\eqref{estimate:int_w} provides us with the desired estimate, and therefore concludes the proof.
    \resetconstant
\end{proof}

Using the previous lemma, we may now obtain a criterion to detect a family of maps \( \gamma \) such that composition with \( \gamma \) is compatible with \( W^{k,p} \) regularity, along with the corresponding estimates; see~\cite{BPVS_screening}*{Proposition~2.14}.
Once again, note that the map \( w \) given by the lemma below depends on \( u \).

\begin{lemme}
\label{lemma:Fuglede_Wkp}
    Let \( \Omega \subset \R^{m} \) be an open set and \( u \in W^{k,p}(\Omega) \).
    There exists a summable map \( w \colon \Omega \to [0,+\infty] \) such that 
    \[
        \int_{\Omega} w = 1
    \]
    and such that for every open set \( \omega \subset \R^{M} \) and every map \( \gamma \in \Cc^{\infty}(\omega;\Omega) \) with bounded derivatives,
    if \( w \circ \gamma \) is summable, then we have \( u \circ \gamma \in W^{k,p}(\omega) \), the derivatives \( D^{j}(u \circ \gamma) \) are given by the classical Faà di Bruno formula, and
    \[
        \lVert D^{j}u \circ \gamma \rVert_{L^{p}(\omega)}
        \leq
        C\biggl(\int_{\omega}w\circ\gamma\biggr)^{\frac{1}{p}}\lVert D^{j}u \rVert_{L^{p}(\Omega)}
        \quad 
        \text{for every \( j \in \{0,\dots,k\} \),}
    \]
    for some constant \( C > 0 \) depending on \( m \), \( M \), \( k \), and \( p \).
\end{lemme}

We note for further use that, under the assumptions of Lemma~\ref{lemma:Fuglede_Wkp}, applying the Faà di Bruno formula, we may estimate \( D^{j}(u \circ \gamma) \) as follows:
\begin{multline}
\label{eq:Fuglede_faa_di_bruno}
    \lVert D^{j}(u \circ \gamma) \rVert_{L^{p}(\omega)} \\
    \leq
    C\biggl(\int_{\omega}w\circ\gamma\biggr)^{\frac{1}{p}}\sum_{i=1}^{j}\sum_{\substack{1\leq t_{1}\leq \cdots\leq t_{i} \\ t_{1}+\cdots+t_{i}=j}}
        \lVert D^{t_{1}}\gamma \rVert_{L^{\infty}(\omega)}\cdots \lVert D^{t_{i}}\gamma \rVert_{L^{\infty}(\omega)}\lVert D^{i}u \rVert_{L^{p}(\Omega)}.
\end{multline}

We also make an important remark about a measurability issue.
In Lemma~\ref{lemma:opening_fuglede}, we worked with arbitrary measure spaces.
On the other hand, here we implicitly assume that \( \R \) and \( \R^{m} \) are endowed with the Borel \( \sigma \)-algebra (and not the Lebesgue \( \sigma \)-algebra) in order to ensure that continuous maps are measurable.

\begin{proof}
	We may assume, without loss of generality, that \( u \) and its \( k \) first derivatives are not almost everywhere equal to \( 0 \).
    Let \( (u_{n})_{n \in \N} \) be a sequence of smooth maps converging to \( u \) in \( W^{k,p}(\Omega) \).
    We apply inductively Lemma~\ref{lemma:opening_fuglede} to \( D^{i}u \) for \( i \in \{0,\dots,k\} \) to obtain summable maps \( w_{i} \colon \Omega \to [0,+\infty] \) satisfying \( \int_{\Omega} w_{i} > 0 \) and a subsequence \( (u_{n_{l}})_{l \in \N} \) such that, for every measurable map \( \gamma \colon \omega \to \Omega \) such that \( w_{i} \circ \gamma \) is summable,
    \( D^{i}u_{n_{l}} \circ \gamma \to D^{i}u \circ \gamma \) in \( L^{p}(\omega) \), and 
    \[
        \int_{\omega} \lvert D^{i}u \circ \gamma \rvert^{p}
        \leq
        2\frac{\int_{\omega} w_{i} \circ \gamma}{\int_{\Omega} w_{i}}\int_{\Omega} \lvert D^{i}u \rvert^{p}.
    \]
    Let 
    \[
        w = \frac{1}{k+1}\sum_{i=0}^{k}\frac{w_{i}}{\int_{\Omega}w_{i}}.
    \]
    
    It is readily seen that 
    \[
        \int_{\Omega} w = 1.
    \]
    Observe also that \( w_{i} \leq (k+1)w\int_{\Omega}w_{i} \).
    Therefore, if \( w \circ \gamma \) is summable, we find that \( D^{i}u \circ \gamma \in L^{p}(\omega) \) with
    \begin{equation}
    \label{eq:estimate_Diu_composition_gamma}
        \int_{\omega} \lvert D^{i}u \circ \gamma \rvert^{p}
        \leq
        2(k+1)\biggl(\int_{\omega} w \circ \gamma\biggr)\int_{\Omega} \lvert D^{i}u \rvert^{p}.
    \end{equation}
    If in addition \( \gamma \) is smooth and has bounded derivatives, since \( D^{i}u_{n_{l}} \circ \gamma \to D^{i}u \circ \gamma \) in \( L^{p}(\omega) \), \( D^{i}(u_{n_{l}} \circ \gamma) \) converges in \( L^{p}(\omega) \) to a map which coincides with the function one would obtain by applying the Faà di Bruno formula to compute \( D^{i}(u \circ \gamma) \).
    Hence, the closure property for Sobolev spaces ensures that \( u \circ \gamma \in W^{k,p}(\omega) \) and that the Faà di Bruno formula actually applies.
    The estimates for \( D^{j}u \circ \gamma \) are already contained in inequality~\eqref{eq:estimate_Diu_composition_gamma}, and therefore the proof is complete.
    \resetconstant
\end{proof}

After dealing with integer order Sobolev spaces, we present the next lemma, which contains the construction of a detector for maps preserving fractional Sobolev regularity under composition; see~\cite{BPVS_screening}*{Proposition~2.12}.

\begin{lemme}
\label{lemma:Fuglede_fractional}
    Let \( \Omega \subset \R^{m} \) be an open set and \( u \in W^{\sigma,p}(\Omega) \).
    Define \( w \colon \Omega \to [0,+\infty] \) by 
    \[
        w(x) = \int_{\Omega}\frac{\lvert u(x)-u(y) \rvert^{p}}{\lvert x-y \rvert^{m+\sigma p}}\,\d y.
    \]
    Assume moreover that there exists \( c > 0 \) such that \( \lvert B^{m}_{\lambda}(z) \cap \Omega \rvert \geq c\lambda^{m} \) for every \( z \in \Omega \) and \( 0 < \lambda \leq \frac{1}{2}\diam\Omega \).
    For every open set \( \omega \subset \R^{M} \) and every Lipschitz map \( \gamma \colon \omega \to \Omega \), if \( w \circ \gamma \) is summable, then we have \( u \circ \gamma \in W^{\sigma,p}(\omega) \) with 
    \[
        \lvert u \circ \gamma \rvert_{W^{\sigma,p}(\omega)}
        \leq 
        C \lvert \gamma \rvert_{\Cc^{0,1}(\omega)}^{\sigma} \biggl(\int_{\omega} w \circ \gamma \biggr)^{\frac{1}{p}}
    \]
    for some constant \( C > 0 \) depending on \( m \), \( M \), \( \sigma \), \( p \), and \( c \).
\end{lemme}

We recall that \( \lvert \gamma \rvert_{\Cc^{0,1}(\omega)} \) denotes the Lipschitz seminorm of \( \gamma \), defined by 
\[
	\lvert \gamma \rvert_{\Cc^{0,1}(\omega)}
	=
	\sup_{\substack{x,y \in \omega \\ x \neq y}} \frac{\lvert \gamma(x)-\gamma(y)\rvert}{\lvert x-y \rvert}.
\]

In contrast to what happens for integer order Sobolev spaces, here we have an explicit expression for \( w \) depending on \( u \).
It is useful to observe that 
\[
    \int_{\Omega} w = \lvert u \rvert_{W^{\sigma,p}(\Omega)}^{p}.
\]
We also comment on the assumption on the volume of balls in \( \Omega \), which will be crucial during the proof.
It is in particular satisfied if \( \Omega \) is a cube.
Indeed, in this case, any ball centered at a point of \( \Omega \) with radius less that \( \frac{1}{2}\diam\Omega \) has at least one quadrant in \( \Omega \), which implies that \( \Omega \) satisfies the assumptions of Lemma~\ref{lemma:Fuglede_fractional}.
To prove Proposition~\ref{prop:block_opening}, we only need to apply Lemma~\ref{lemma:Fuglede_fractional} on cubes, but later on in Section~\ref{sect:shrinking}, we will need to use a similar technique on more general domains,
whence our motivation for already presenting a more general statement here.

\begin{proof}
    For every \( x \), \( y \in \omega \), we let \( \Bc_{x,y} = B^{m}_{\lvert \gamma(x)-\gamma(y) \rvert}\Bigl(\frac{\gamma(x)+\gamma(y)}{2}\Bigr) \cap \Omega \).
    We write
    \[
        \lvert u \circ \gamma(x) - u \circ \gamma(y) \rvert^{p}
        \leq
        \C\biggl(\fint_{\Bc_{x,y}} \lvert u \circ \gamma(x) - u(z) \rvert^{p}\,\d z + \fint_{\Bc_{x,y}} \lvert u(z) - u \circ \gamma(y) \rvert^{p}\,\d z\biggr).
    \]
    Note that
    \[
        B^{m}_{\frac{\lvert \gamma(x)-\gamma(y) \rvert}{2}}(\gamma(x)) \cap \Omega \subset \Bc_{x,y}.
    \]
    Since \( \frac{\lvert \gamma(x)-\gamma(y) \rvert}{2} \leq \frac{1}{2}\diam\Omega \), we deduce that \( \lvert \Bc_{x,y} \rvert \geq \C\lvert \gamma(x)-\gamma(y) \rvert^{m} \).
    Moreover, we observe that for every \( z \in \Bc_{x,y} \), we have 
    \[
        \lvert \gamma(x)-z \rvert
        \leq
        \biggl\lvert\frac{\gamma(x)+\gamma(y)}{2}-z \biggr\rvert + \frac{1}{2}\lvert \gamma(x)-\gamma(y) \rvert
        \leq
        \frac{3}{2}\lvert \gamma(x)-\gamma(y) \rvert,
    \]
    and similarly \( \lvert \gamma(y)-z \rvert \leq \frac{3}{2}\rvert \gamma(x)-\gamma(y) \rvert \).
    Hence,
    \[
        \lvert u \circ \gamma(x) - u \circ \gamma(y) \rvert^{p}
        \leq
        \C\biggl(\int_{\Bc_{x,y}} \frac{\lvert u \circ \gamma(x) - u(z) \rvert^{p}}{\lvert \gamma(x)-z \rvert^{m}}\,\d z + \int_{\Bc_{x,y}} \frac{\lvert u(z) - u \circ \gamma(y) \rvert^{p}}{\lvert \gamma(y)-z \rvert^{m}}\,\d z\biggr).
    \]
    Dividing by \( \lvert x-y \rvert^{M+\sigma p} \) and integrating over \( \omega \times \omega \), we deduce that 
    \[
        \lvert u \circ \gamma \rvert_{W^{\sigma,p}(\omega)}^{p}
        \leq
        \Cl{cst:Fuglede_fractional_Tonelli}\int_{\omega}\int_{\omega}\int_{\Bc_{x,y}} \frac{\lvert u \circ \gamma(x) - u(z) \rvert^{p}}{\lvert x-y \rvert^{M+\sigma p} \lvert \gamma(x)-z \rvert^{m}}\,\d z \d y \d x.
    \]
    We use Tonelli's theorem to deduce that
    \[
        \lvert u \circ \gamma \rvert_{W^{\sigma,p}(\omega)}^{p}
        \leq
        \Cr{cst:Fuglede_fractional_Tonelli}\int_{\Omega}\int_{\omega}\int_{\Yc_{x,z}} \frac{\lvert u \circ \gamma(x) - u(z) \rvert^{p}}{\lvert x-y \rvert^{M+\sigma p} \lvert \gamma(x)-z \rvert^{m}}\,\d y \d x \d z,
    \]
    where \( \Yc_{x,z} \) is the set of all \( y \in \omega \) such that \( z \in \Bc_{x,y} \), that is,
    \[
        \Yc_{x,z} = \{ y \in \omega \mathpunct{:} \lvert \gamma(x)+\gamma(y)-2z \rvert < 2\lvert \gamma(x)-\gamma(y) \rvert \}.
    \]
    Observe that 
    \begin{multline*}
        \Yc_{x,z}
        \subset 
        \Bigl\{ y \in \R^{M}\mathpunct{:} \lvert \gamma(x)-z \rvert < \frac{3}{2}\lvert \gamma(x)-\gamma(y) \rvert\Bigr\} \\
        \subset 
        \Bigl\{ y \in \R^{M}\mathpunct{:} \lvert \gamma(x)-z \rvert < \frac{3}{2}\lvert \gamma \rvert_{\Cc^{0,1}(\omega)} \lvert x-y \rvert\Bigr\}
        =
        \R^{M} \setminus B_{r}^{M}(x),
    \end{multline*}
    where 
    \[
    	r = r(x,z) = \frac{2\lvert \gamma(x)-z \rvert}{3\lvert \gamma \rvert_{\Cc^{0,1}(\omega)}}.
    \]
    Hence, integrating with respect to \( y \),
    \begin{align*}
        \lvert u \circ \gamma \rvert_{W^{\sigma,p}(\omega)}^{p}
        &\leq
        \Cr{cst:Fuglede_fractional_Tonelli}\int_{\Omega}\int_{\omega}\int_{\R^{M} \setminus B_{r}^{M}(x)} \frac{\lvert u \circ \gamma(x) - u(z) \rvert^{p}}{\lvert x-y \rvert^{M+\sigma p} \lvert \gamma(x)-z \rvert^{m}}\,\d y \d x \d z \\
        &\leq
        \C\lvert \gamma \rvert_{\Cc^{0,1}(\omega)}^{\sigma p}\int_{\Omega}\int_{\omega}\frac{\lvert u \circ \gamma(x) - u(z) \rvert^{p}}{\lvert \gamma(x)-z \rvert^{m+\sigma p}}\,\d x \d z,
    \end{align*}
    which concludes the proof.
    \resetconstant
\end{proof}

Now that we have at our disposal a criterion to detect a family of maps that preserve membership in Sobolev spaces after composition, it would be useful to know if, given a detector \( w \) associated to a fixed map \( u \in W^{s,p} \), we may actually construct many smooth maps \( \gamma \) such that \( w \circ \gamma \) is summable.
This is based on a genericity argument, and is the purpose of the next lemma, whose proof relies on an averaging argument initially due to Federer and Fleming~\cite{FF_currents}.
Our presentation and proof are taken from~\cite{BPVS_density_higher_order}*{Lemma~2.5}.

\begin{lemme}
\label{lemma:averaging_argument_FF}
    Let \( \omega \), \( \Omega \), and \( P \subset \R^{m} \) be measurable sets, with \( 0 < \lvert P \rvert < +\infty \).
    Let \( \upPhi \colon \omega+P \to \Omega \) and \( w \colon \Omega \to [0,+\infty] \) be measurable maps.
    For every \( a \in P \), define the map \( \upPhi_{a} \colon \omega \to \R^{m} \) by \( \upPhi_{a}(x) = \upPhi(x-a)+a \). 
    Assume that, for every \( a \in P \) and \( x \in \omega \), \( \upPhi_{a}(x) \in \Omega \).
    There exists a subset \( A \subset P \) of positive measure such that, for every \( a \in A \), we have 
    \begin{equation*}
        \int_{\omega} w \circ \upPhi_{a}
        \leq
        C\frac{\lvert \omega+P \rvert}{\lvert P \rvert}\int_{\Omega} w
    \end{equation*}
    for some constant \( C > 0 \).
\end{lemme}

The constant \( C \) in the above estimate does not depend on the different parameters involved in the statement of the lemma.
However, as we shall see in the proof, the measure of the set \( A \) may be taken arbitrarily close to \( \lvert P \rvert \) provided that we enlarge \( C \) accordingly.

\begin{proof}
    We are going to estimate the average 
    \[
        \fint_{P}\biggl(\int_{\omega} w \circ \upPhi_{a}\biggr)\,\d a.
    \]
    By a change of variable by translation and Tonelli's theorem, we compute that 
    \begin{multline*}
        \int_{P}\biggl(\int_{\omega} w \circ \upPhi_{a}\biggr)\,\d a
        =
        \int_{P}\biggl(\int_{\omega+a} w(\upPhi(y)+a)\,\d y\biggr)\,\d a \\ 
        \leq 
        \int_{\omega+P}\biggl(\int_{P\cap (y-\omega)} w(\upPhi(y)+a)\,\d a\biggr)\,\d y 
        \leq
        \int_{\omega+P}\biggl(\int_{\Omega} w(x)\,\d x\biggr)\,\d y
        = 
        \lvert \omega+P \rvert\int_{\Omega} w.
    \end{multline*}
    Therefore,
    \[
        \fint_{P}\biggl(\int_{\omega} w \circ \upPhi_{a}\biggr)\,\d a
        \leq 
        \frac{\lvert \omega+P \rvert}{\lvert P\rvert}\int_{\Omega}w.
    \]
    Hence, for every \( 0 < \theta < 1 \), there exists a subset \( A \subset P \) with measure \( \lvert A \rvert \geq \theta\lvert P \rvert \) such that, for every \( a \in A \), we have 
    \[
        \int_{\omega} w \circ \upPhi_{a}
        \leq
        \frac{1}{1-\theta}\frac{\lvert \omega+P \rvert}{\lvert P\rvert}\int_{\Omega} w,
    \]
    and the proof of the lemma is complete.
    \resetconstant
\end{proof}

With all these tools at our disposal, we are now ready to prove Proposition~\ref{prop:block_opening}.
We start by constructing one model map \( \upPhi \) satisfying the geometric properties in the conclusion of the proposition.
Then we use the previous lemmas to show that \( \upPhi_{a} \) satisfies all the conclusions of Proposition~\ref{prop:block_opening} for some \( a \in \R^{m} \).

\begin{proof}[Proof of Proposition~\ref{prop:block_opening}]
	We use the notation introduced in~\eqref{eq:def_rectangles_opening}.
    We start with the construction of the model map \( \upPhi \).
    Let \( \lambda > 0 \) be such that 
    \[
        \lambda < \min\biggl(\frac{\ulr-\ulrho}{2},\frac{\olr-\ulr}{2},\frac{\olrho-\olr}{2}\biggr).
    \]
    We define \( \upPhi \colon Q_{4,1} \to Q_{4,1} \) by 
    \[
        \upPhi(x',x'') = \Bigl(x',\phi\Bigl(\frac{x'}{\eta},\frac{x''}{\eta}\Bigr)x''\Bigr),
    \]
    where \( \phi \colon Q_{4,1} \to [0,1] \) is a smooth function such that
    \begin{enumerate}[label=(\alph*)]
        \item for \( x \in Q_{1,1}+B^{m}_{\lambda} \), \( \phi(x) = 0 \);
        \item for \( x \in (Q_{4,1} \setminus Q_{2,1})+B^{m}_{\lambda} \), \( \phi(x) = 1 \).
    \end{enumerate}
    Recall that the \( Q_{i,1} \) are the rectangles defined in~\eqref{eq:def_rectangles_opening} with parameter \( \eta = 1 \).
    By scaling, we have 
    \[
        \lVert D^{j}\upPhi \rVert_{L^{\infty}(Q_{4})} \leq \C\eta^{1-j}
        \quad
        \text{for every \( j \in \{1,\dots,k+1\} \).}
    \]
    Now we set \( \upPhi_{a}(x) = \upPhi(x-a)+a \) for every \( a \in B^{m}_{\lambda\eta} \).
    By construction, \( \upPhi_{a} \) satisfies the geometric properties~\ref{item:form_block_opening} to~\ref{item:support_block_opening} for every \( a \in B^{m}_{\lambda\eta} \).

    We now turn to the Sobolev estimates~\ref{item:estimates_block_opening}.
    In the case where \( k = 0 \), we apply Lemma~\ref{lemma:Fuglede_fractional} to \( u \), with \( \Omega = Q_{4} \).
    Let \( w \colon Q_{4} \to [0,+\infty] \) be the corresponding detector.
    By Lemma~\ref{lemma:averaging_argument_FF} with \( \omega = Q_{3} \), \( \Omega = Q_{4} \), and \( P = B^{m}_{\lambda\eta} \),
    there exists \( a \in B^{m}_{\lambda\eta} \) such that 
    \[
        \int_{Q_{3}} w \circ \upPhi_{a} 
        \leq 
        \C\frac{\lvert Q_{3}+B^{m}_{\lambda\eta} \rvert}{\lvert B^{m}_{\lambda\eta} \rvert}\int_{Q_{4}} w.
    \]
    Since \( Q_{3} \) has sides whose length is proportional to \( \eta \), this implies that 
    \begin{equation}
    \label{eq:estimate_detector_sle1}
        \int_{Q_{3}} w \circ \upPhi_{a} 
        \leq 
        \C\int_{Q_{4}} w.
    \end{equation}
    Therefore, \( u \circ \upPhi_{a} \in W^{s,p}(Q_{4}) \) and 
    \[
        \lvert u \circ \upPhi_{a} \rvert_{W^{s,p}(Q_{3})}
        \leq 
        \C \lvert \upPhi_{a} \rvert_{\Cc^{0,1}(Q_{3})}^{s} \biggl( \int_{Q_{3}} w \circ \upPhi_{a} \biggr)^{\frac{1}{p}}.
    \]
    Combining the estimate on the derivative of \( \upPhi_{a} \), equation~\eqref{eq:estimate_detector_sle1} and the remark following Lemma~\ref{lemma:Fuglede_fractional}, we conclude that 
    \[
        \lvert u \circ \upPhi_{a} \rvert_{W^{s,p}(Q_{3})}
        \leq 
        \C \lvert u \rvert_{W^{s,p}(Q_{4})}.
    \]
    The \( L^{p} \) estimate is obtained as in the case \( k \geq 1 \) below, and this concludes the proof when \( 0 < s < 1 \).

    If now \( k \geq 1 \), we apply Lemma~\ref{lemma:Fuglede_Wkp} to \( u \) to obtain a detector \( w_{0} \colon Q_{4} \to [0,+\infty] \) 
    and we apply Lemma~\ref{lemma:Fuglede_fractional} to \( D^{j}u \) for every \( j \in \{1,\dots,k\} \) to obtain a detector \( w_{j} \colon Q_{4} \to [0,+\infty] \).
    (In the case where \( \sigma = 0 \), we skip this second step and only construct \( w_{0} \).
    In the sequel we continue to speak about \( w_{j} \) for \( j \in \{0,\dots,k\} \), it is implicit that when \( \sigma = 0 \) we only consider \( w_{0} \).)
    Then we invoke Lemma~\ref{lemma:averaging_argument_FF} to find some \( a \in B^{m}_{\lambda\eta} \) such that 
    \begin{equation}
    \label{eq:choice_a_block_opening_kge1}
        \int_{Q_{3}} w_{j} \circ \upPhi_{a} 
        \leq 
        \C\frac{\lvert Q_{3}+B^{m}_{\lambda\eta} \rvert}{\lvert B^{m}_{\lambda\eta} \rvert}\int_{Q_{4}} w_{j}
        \quad
        \text{for every \( j \in \{0,\dots,k\} \).}
    \end{equation}
    It is indeed possible to choose the same \( a \) simultaneously for each \( w_{j} \) since the set \( A \) in Lemma~\ref{lemma:averaging_argument_FF}
    can be chosen of measure arbitrarily close of \( \lvert B^{m}_{\lambda\eta} \rvert \).

    For the integer order derivatives, using the estimates on the derivatives of \( \upPhi_{a} \) and the fact that \( \int_{Q_{4}} w_{0} = 1 \), we immediately deduce that \( u \circ \upPhi_{a} \in W^{k,p}(Q_{3}) \),
    \[
        \lVert D^{j}(u \circ \upPhi_{a}) \rVert_{L^{p}(Q_{3})}
        \leq 
        \C\sum_{i=1}^{j}\eta^{i-j}\lVert D^{i}u \rVert_{L^{p}(Q_{4})},
    \]
    and 
    \[
        \lVert u \circ \upPhi_{a} \rVert_{L^{p}(Q_{3})}
        \leq 
        \lVert u \rVert_{L^{p}(Q_{4})}.
    \]
    In the case \( k = 0 \), we may still obtain the \( L^{p} \) estimate at order \( 0 \) above, 
    constructing the detector \( w_{0} \) with the help of Lemma~\ref{lemma:opening_fuglede} instead of Lemma~\ref{lemma:Fuglede_Wkp}
    and using again the fact that we may choose a suitable \( a \) for several detectors simultaneously.

    Dealing with fractional order derivatives requires additional computations.
    We continue to work with \( a \) as in~\eqref{eq:choice_a_block_opening_kge1}.
    Using the Faà di Bruno formula -- which is indeed valid for \( u \circ \upPhi_{a} \) by Lemma~\ref{lemma:Fuglede_Wkp} -- and the multilinearity of the differential, we find 
    \begin{multline}
    \label{eq:faa_di_bruno_block_opening}
        \lvert D^{j}(u \circ \upPhi_{a})(x) - D^{j}(u \circ \upPhi_{a})(y) \rvert^{p} 
        \leq 
        \C\sum_{i = 1}^{j}\Bigl(\lvert D^{i}u \circ \upPhi_{a}(x) - D^{i}u \circ \upPhi_{a}(y) \rvert^{p}\eta^{(i-j)p} \\
             + \sum_{t = 1}^{j} \lvert D^{i}u \circ \upPhi_{a}(x)\rvert^{p}\eta^{(i-1-j+t)p}\lvert D^{t}\upPhi_{a}(x)-D^{t}\upPhi_{a}(y) \rvert^{p}\Bigr).
    \end{multline}
    When dividing~\eqref{eq:faa_di_bruno_block_opening} by \( \lvert x-y \rvert^{m+\sigma p} \) and integrating over \( Q_{3} \times Q_{3} \), the first term on the right-hand side gives 
    \( \eta^{(i-j)p}\lvert D^{i}u \circ \upPhi_{a} \rvert_{W^{\sigma,p}(Q_{3})}^{p} \).
    As in the case \( 0 < s < 1 \), we may estimate it as 
    \[
        \eta^{(i-j)p}\lvert D^{i}u \circ \upPhi_{a} \rvert_{W^{\sigma,p}(Q_{3})}^{p}
        \leq 
        \C \eta^{(i-j)p}\lvert D^{i}u \rvert_{W^{\sigma,p}(Q_{4})}^{p}.
    \]
    For the second term on the right-hand side of~\eqref{eq:faa_di_bruno_block_opening}, we use an optimization argument. 
    For every \( r > 0 \), we write
    \begin{multline*}
        \int_{Q_{3}} \frac{\lvert D^{t}\upPhi_{a}(x) - D^{t}\upPhi_{a}(y) \rvert^{p}}{\lvert x-y \rvert^{m+\sigma p}}\,\d y
        \leq 
        \C\biggl(\int_{B^{m}_{r}(x)} \eta^{-tp}\frac{1}{\lvert x-y \rvert^{m+\sigma p-p}}\,\d y \\
            + \int_{\R^{m} \setminus B^{m}_{r}(x)} \eta^{(1-t)p}\frac{1}{\lvert x-y \rvert^{m+\sigma p}}\,\d y\biggr)
        \leq
        \C(\eta^{-tp}r^{p-\sigma p} + \eta^{(1-t)p}r^{-\sigma p}).
    \end{multline*}
    Letting \( r = \eta \), we find
    \[
        \int_{Q_{3}} \frac{\lvert D^{t}\upPhi_{a}(x) - D^{t}\upPhi_{a}(y) \rvert^{p}}{\lvert x-y \rvert^{m+\sigma p}}\,\d y
        \leq \C\eta^{(1-t-\sigma)p}.
    \]
    Therefore,
    \begin{multline*}
        \int_{Q_{3}}\int_{Q_{3}} 
            \frac{\lvert D^{i}u \circ \upPhi_{a}(x) \rvert^{p} \eta^{(i-1-j+t)p} \lvert D^{t}\upPhi_{a}(x) - D^{t}\upPhi_{a}(y) \rvert^{p}}{\lvert x-y \rvert^{m+\sigma p}}\,\d x\d y \\
        \leq
        \C\eta^{(i-j-\sigma)p}\int_{Q_{3}} \lvert D^{i}u \circ \upPhi_{a}(x) \rvert^{p}\,\d x 
        \leq
        \C\eta^{(i-j-\sigma)p}\int_{Q_{4}} \lvert D^{i}u \rvert^{p},
    \end{multline*}
    where the last inequality follows from Lemma~\ref{lemma:Fuglede_Wkp}.
    Gathering the estimates for both terms in~\eqref{eq:faa_di_bruno_block_opening} yields the desired fractional estimate and concludes the proof.
    \resetconstant
\end{proof}

We conclude this section with two additional results which are the counterparts of~\cite{BPVS_density_higher_order}*{Addendum~1 and~2 to Proposition~2.1} in the context of fractional order estimates.
From now on, we place ourselves under the assumptions of Proposition~\ref{prop:main_opening}.
The first proposition ensures that the opening procedure does not increase too much the energy on one given cube.

\begin{prop}
\label{prop:addendum1_opening}
    Let \( \Kc^{m} \) be a cubication containing \( \Uc^{m} \).
    \begin{enumerate}[label=(\alph*)]
    	\item\label{item:addendum1_sgeq1} If \( s \geq 1 \) and if \( u \in W^{1,sp}(K^{m}+Q^{m}_{2\rho\eta};\R^{\nu}) \), then the map \( \upPhi \colon \R^{m} \to \R^{m} \) provided by Proposition~\ref{prop:main_opening} can be chosen with the additional property that \( u \circ \upPhi \in W^{1,sp}(K^{m}+Q^{m}_{\rho\eta};\R^{\nu}) \), and for every \( \sigma^{m} \in \Kc^{m} \),
	    	\[
		    	\lVert D(u \circ \upPhi)\rVert_{L^{sp}(\sigma^{m}+Q^{m}_{\rho\eta})} 
		    	\leq
		    	C'\lVert Du \rVert_{L^{sp}(\sigma^{m}+Q^{m}_{2\rho\eta})}
	    	\]
    	for some constant \( C' > 0 \) depending on \( m \), \( s \), \( p \), and \( \rho \).
    	\item\label{item:addendum1_sle1} If \( 0 < s < 1 \), then the map \( \upPhi \colon \R^{m} \to \R^{m} \) provided by Proposition~\ref{prop:main_opening} can be chosen with the additional property that \( u \circ \upPhi \in W^{s,p}(K^{m}+Q^{m}_{\rho\eta};\R^{\nu}) \), and for every \( \sigma^{m} \in \Kc^{m} \),
	    	\[
		    	\lvert u \circ \upPhi \rvert_{W^{s,p}(\sigma^{m}+Q^{m}_{\rho\eta})}
		    	\leq 
		    	C'\lvert u \rvert_{W^{s,p}(\sigma^{m}+Q^{m}_{2\rho\eta})}
	    	\]
    		for some constant \( C' > 0 \) depending on \( m \), \( s \), \( p \), and \( \rho \). 
    \end{enumerate}
\end{prop}
\begin{proof}
    \ref{item:addendum1_sgeq1}~In the case \( s \geq 1 \), since the choice of the parameter \( a \) involved in the construction of the map provided by Proposition~\ref{prop:block_opening}
    is made over a set of positive measure, according to the remark following Lemma~\ref{lemma:averaging_argument_FF}, 
    we may assume that the maps \( \upPhi^{\sigma^{d}} \) involved in the construction of the map \( \upPhi \) satisfy in addition the conclusion of Proposition~\ref{prop:block_opening} with parameters \( 1 \) and \( sp \).

    We keep the notation used in the proof of Proposition~\ref{prop:main_opening}.
    Let \( d \in \{0,\dots,\ell\} \).
    We are going to prove that
    \[
        \lVert D(u^{d} \circ \upPhi_{d})\rVert_{L^{sp}(\sigma^{m}+Q^{m}_{\rho_{d}\eta})} 
        \leq
        C'\lVert Du^{d} \rVert_{L^{sp}(\sigma^{m}+Q^{m}_{\rho_{d-1}\eta})},
    \]
    and the conclusion will follow by induction.
    By our additional assumption on the maps \( \upPhi^{\sigma^{d}} \), we have
    \[
        \lVert D(u^{d} \circ \upPhi_{d})\rVert_{L^{sp}(T_{\sigma^{d}}(Q_{3}))} 
        \leq
        \lVert D(u^{d} \circ \upPhi_{d})\rVert_{L^{sp}(T_{\sigma^{d}}(Q_{4}))} 
        \leq
        \C\lVert Du^{d}\rVert_{L^{sp}(T_{\sigma^{d}}(Q_{4}))}
    \]
    for every \( \sigma^{d} \in \Uc^{d} \).
    We conclude by using the fact that
    \[
        \Supp\upPhi_{d} \subset \bigcup_{\sigma^{d} \in \Uc^{d}}T_{\sigma^{d}}(Q_{2})
        \subset 
        \bigcup_{\sigma^{d} \in \Uc^{d}}T_{\sigma^{d}}(Q_{3})
    \]
    along with the additivity of the integral.

    \ref{item:addendum1_sle1}~The proof of the case \( 0 < s < 1 \) is identical, except that we replace the additivity of the integral by Lemma~\ref{lemma:finite_fractional_additivity}.
    Here we use the fact that the number of \( d \)-faces of a given cube depends only on \( d \) and \( m \), 
    and that the geometric support of \( \upPhi^{\sigma^{d}} \) is contained in \( T_{\sigma^{d}}(Q_{2}) \), which is slightly smaller than \( T_{\sigma^{d}}(Q_{3}) \).
    This justifies the application of Lemma~\ref{lemma:finite_fractional_additivity}.
    \resetconstant
\end{proof}

The second proposition gives \( \VMO \)-type estimates for the opened map.
As we mentioned in our informal presentation, such estimates are one of the main features of the opening procedure, 
and they follow from the fact that \( u \circ \upPhi \) behaves locally as a map of \( \ell \) variables in \( U^{\ell}+Q^{m}_{\rho\eta} \).

\begin{prop}
\label{prop:addendum2_opening}
    Under the assumptions of Propositions~\ref{prop:main_opening} and~\ref{prop:addendum1_opening}, the map \( \upPhi \colon \R^{m} \to \R^{m} \) satisfies the following estimates:
    \begin{enumerate}[label=(\alph*)]
    	\item\label{item:addendum2_sgeq1} if \( s \geq 1 \), then 
    		\begin{enumerate}[label=(\roman*)]
    			\item\label{item:addendum_2_sgeq1_limit} it holds that
	    			\[
		    			\lim_{r \to 0}\sup_{Q^{m}_{r}(a)\subset U^{\ell}+Q^{m}_{\rho\eta}}r^{\frac{\ell}{sp}-1}
		    			\fint_{Q^{m}_{r}(a)}\fint_{Q^{m}_{r}(a)}\lvert u \circ \upPhi(x) - u \circ \upPhi(y) \rvert\,\d x\d y
		    			=
		    			0;
	    			\]
    			\item\label{item:addendum_2_sgeq1_estimate} for every \( \sigma^{m} \in \Uc^{m} \) and every \( Q^{m}_{r}(a) \subset U^{\ell}+Q^{m}_{\rho\eta} \) with \( a \in \sigma^{m} \),
	    			\[
		    			\fint_{Q^{m}_{r}(a)}\fint_{Q^{m}_{r}(a)} \lvert u \circ \upPhi(x) - u \circ \upPhi(y) \rvert\,\d x\d y
		    			\leq
		    			C''\frac{r^{1-\frac{\ell}{sp}}}{\eta^{\frac{m-\ell}{sp}}}\lVert Du \rVert_{L^{sp}(\sigma^{m}+Q^{m}_{2\rho\eta})};
	    			\]
    		\end{enumerate}
    	\item\label{item:addendum2_sle1} if \( 0 < s < 1 \), then 
    		\begin{enumerate}[label=(\roman*)]
    			\item\label{item:addendum_2_sle1_limit} it holds that
    				\[
	    				\lim_{r \to 0}\sup_{Q^{m}_{r}(a)\subset U^{\ell}+Q^{m}_{\rho\eta}}r^{\frac{\ell}{p}-s}
	    				\fint_{Q^{m}_{r}(a)}\fint_{Q^{m}_{r}(a)}\lvert u \circ \upPhi(x) - u \circ \upPhi(y) \rvert\,\d x\d y
	    				=
	    				0;
    				\]
    			\item\label{item:addendum_2_sle1_estimate} for every \( \sigma^{m} \in \Uc^{m} \) and every \( Q^{m}_{r}(a) \subset U^{\ell}+Q^{m}_{\rho\eta} \) with \( a \in \sigma^{m} \),
    				\[
	    				\fint_{Q^{m}_{r}(a)}\fint_{Q^{m}_{r}(a)} \lvert u \circ \upPhi(x) - u \circ \upPhi(y) \rvert\,\d x\d y
	    				\leq
	    				C''\frac{r^{s-\frac{\ell}{p}}}{\eta^{\frac{m-\ell}{p}}}\lvert u \rvert_{W^{s,p}(\sigma^{m}+Q^{m}_{2\rho\eta})};
    				\]
    		\end{enumerate}
    \end{enumerate}
	for some constant \( C'' > 0 \) depending on \( m \), \( s \), \( p \), and \( \rho \).
\end{prop}
\begin{proof}
	We start with the proof of items~\ref{item:addendum_2_sgeq1_limit}.
    Let \( a \in \R^{m} \) and \( r > 0 \) be such that \( Q^{m}_{r}(a)\subset U^{\ell}+Q^{m}_{\rho\eta} \)\,, and write \( Q^{m}_{r}(a) = Q^{\ell}_{r}(a') \times Q^{m-\ell}_{r}(a'') \).
    Then \( a \in U^{\ell}+Q^{m}_{\rho\eta-r} \), and hence there exists \( \tau^{\ell} \in \Uc^{\ell} \) such that \( Q^{m}_{r}(a) \subset \tau^{\ell}+Q^{m}_{\rho\eta} \).
    We may assume that \( \tau^{\ell} = Q^{m}_{\eta} \times \{0\}^{m-\ell} \).
    Recall that the map \( \upPhi \) is constant on each \( (m-\ell) \)-dimensional cube orthogonal to \( Q^{\ell}_{(1+\rho)\eta} \times \{0\}^{m-\ell} \).
    Hence we may define \( v \colon Q^{\ell}_{(1+\rho)\eta} \to \R^{\nu} \) by
    \[
        v(x') = u \circ \upPhi(x',a'').
    \]
    Using Proposition~\ref{prop:addendum1_opening}, we deduce that \( v \in W^{1,sp}(Q^{\ell}_{(1+\rho)\eta};\R^{\nu}) \) in the \( s \geq 1 \) case, respectively \( v \in W^{s,p}(Q^{\ell}_{(1+\rho)\eta};\R^{\nu}) \) in the \( 0 < s < 1 \) case.
    We next note that
    \[
        \fint_{Q^{m}_{r}(a)}\fint_{Q^{m}_{r}(a)} \lvert u \circ \upPhi(x) - u \circ \upPhi(y) \rvert\,\d x\d y
        = 
        \fint_{Q^{\ell}_{r}(a')}\fint_{Q^{\ell}_{r}(a')} \lvert v(x') - v(y') \rvert\,\d x'\d y'.
    \]
    The Poincaré--Wirtinger inequality implies that
    \[
        \fint_{Q^{\ell}_{r}(a')}\fint_{Q^{\ell}_{r}(a')} \lvert v(x') - v(y') \rvert\,\d x'\d y'
        \leq
        \C r^{1-\frac{\ell}{sp}}\lVert Dv \rVert_{L^{sp}(Q^{\ell}_{r}(a'))}
    \]
    if \( s \geq 1 \), respectively 
    \[
        \fint_{Q^{\ell}_{r}(a')}\fint_{Q^{\ell}_{r}(a')} \lvert v(x') - v(y') \rvert\,\d x'\d y'
        \leq
        \C r^{s-\frac{\ell}{p}}\lvert v \rvert_{W^{s,p}(Q^{\ell}_{r}(a'))}
    \]
    if \( 0 < s < 1 \).
    It then suffices to invoke Lebesgue's lemma to obtain both items~\ref{item:addendum_2_sgeq1_limit}.
    
    We now turn to the proof of items~\ref{item:addendum_2_sgeq1_estimate}. 
    When \( s \geq 1 \), we observe that
    \begin{equation}
    \label{eq:slicing_opening_addendum_2}
        \lVert D(u \circ \upPhi) \rVert_{L^{sp}(Q^{\ell}_{r}(a')\times Q^{m-\ell}_{\rho\eta}(a''))}
        =
        (2\rho\eta)^{\frac{m-\ell}{sp}}\lVert Dv \rVert_{L^{sp}(Q^{\ell}_{r}(a'))},
    \end{equation}
    and hence, assuming in addition that \( a \in \sigma^{m} \) with \( \sigma^{m} \in \Uc^{m} \), we have
    \[
        \lVert Dv \rVert_{L^{sp}(Q^{\ell}_{r}(a'))}
        =
        \frac{1}{(2\rho\eta)^{\frac{m-\ell}{sp}}}\lVert D(u \circ \upPhi) \rVert_{L^{sp}(Q^{\ell}_{r}(a')\times Q^{m-\ell}_{\rho\eta}(a''))}
        \leq
        \frac{1}{(2\rho\eta)^{\frac{m-\ell}{sp}}}\lVert D(u \circ \upPhi) \rVert_{L^{sp}(\sigma^{m}+Q^{m}_{\rho\eta})}.
    \]
    Thus,
    \[
        \fint_{Q^{m}_{r}(a)}\fint_{Q^{m}_{r}(a)} \lvert u \circ \upPhi(x) - u \circ \upPhi(y) \rvert\,\d x\d y
        \leq
        \C\frac{r^{1-\frac{\ell}{sp}}}{(2\rho\eta)^{\frac{m-\ell}{sp}}}\lVert D(u \circ \upPhi) \rVert_{L^{sp}(\sigma^{m}+Q^{m}_{\rho\eta})}.
    \]
    Proposition~\ref{prop:addendum1_opening} implies that
    \[
        \lVert D(u \circ \upPhi) \rVert_{L^{sp}(\sigma^{m}+Q^{m}_{\rho\eta})}
        \leq
        \C\lVert Du \rVert_{L^{sp}(\sigma^{m}+Q^{m}_{2\rho\eta})},
    \]
    which yields the desired conclusion.

    For the case \( 0 < s < 1 \), we follow the same path, replacing inequality~\eqref{eq:slicing_opening_addendum_2} by the fact that 
    \begin{multline*}
        \lvert v \rvert_{W^{s,p}(Q^{\ell}_{r}(a'))}
        =
        \frac{1}{(2\rho\eta)^{\frac{m-\ell}{p}}}\biggl(\int_{Q^{m-\ell}_{\rho\eta}(a'')}\lvert u \circ \upPhi(\cdot,x'') \rvert_{W^{s,p}(Q^{\ell}_{r}(a'))}^{p}\,\d x''\biggr)^{\frac{1}{p}} \\
        \leq 
        \C \frac{1}{\eta^{\frac{m-\ell}{p}}}\lvert u \rvert_{W^{s,p}(Q^{\ell}_{r}(a')\times Q^{m-\ell}_{\rho\eta}(a''))}.
    \end{multline*}
    This concludes the proof of the proposition.
    \resetconstant
\end{proof}

\section{Adaptive smoothing}
\label{sect:adaptive_smoothing}

In this section, we present the adaptive smoothing, which consists in a regularization by convolution, \( x \mapsto \varphi_{\psi(x)} \ast u(x) \), where the parameter \( \psi \) of convolution is allowed to depend on the point \( x \) where the regularized map is calculated.
Already implicit in the proof of the \( H = W \) theorem~\cite{H=W}, this method was made popular by Schoen and Uhlenbeck~\cite{SU_regularity_theory_harmonic_maps}*{Section~3}.
The approach we follow here is an adaptation, suited to fractional Sobolev spaces, of the one in~\cite{BPVS_density_higher_order}*{Section~3}.

We now become more specific.
Let \( \varphi \) be a \emph{mollifier}, i.e.,
\[
	\varphi \in \Cc^{\infty}_{\mathrm{c}}(B^{m}_{1}),
	\quad 
	\varphi \geq 0 \quad \text{in \( B^{m}_{1} \),}
	\quad
	\varphi \text{ is radial,}
	\quad
	\text{and}
	\quad
	\int_{B^{m}_{1}} \varphi = 1.
\]

Let \( u \in L^{1}_{\mathrm{loc}}(\Omega)\), and consider a map \( \psi \in \Cc^{\infty}(\Omega;(0,+\infty)) \).
For every \( x \in \Omega \) satisfying \( \dist{(x,\partial\Omega)} \geq \psi(x) \), we may define
\[
	\varphi_{\psi} \ast u(x) = \int_{B^{m}_{1}} \varphi(z)u(x+\psi(x)z)\,\d z.
\]
A change of variable yields
\begin{equation}
\label{eq:reformulation_adaptive_smoothing}
	\varphi_{\psi} \ast u(x)
	=
	\frac{1}{\psi(x)^{m}}\int_{B^{m}_{\psi(x)}(x)} \varphi\Bigl(\frac{y-x}{\psi(x)}\Bigr)u(y)\,\d y.
\end{equation}
In particular, \( \varphi_{\psi} \ast u \) is smooth.

Let us first note a straightforward inequality.
Let \( \omega \subset \{ x \in \Omega \mathpunct{:} \dist{(x,\partial\Omega)} \geq \psi(x) \} \), so that \( \varphi_{\psi} \ast u \) is well-defined on \( \omega \).
For any \( x \in \omega \), we write
\[
	\varphi_{\psi} \ast u(x) - u(x)
	=
	\int_{B^{m}_{1}} \varphi(z)(u(x+\psi(x)z)-u(x))\,\d z.
\]
Therefore, by Minkowski's inequality, we find
\begin{multline}
\label{eq:Lp_smooth_leq_translation}
	\lVert \varphi_{\psi} \ast u - u \rVert_{L^{p}(\omega)}
	\leq
	\int_{B^{m}_{1}}\varphi(z)\biggl(\int_{\omega} \lvert u(x+\psi(x)z)-u(x) \rvert^{p}\,\d x\biggr)^{\frac{1}{p}}\,\d z \\
	\leq
	\sup_{v \in B^{m}_{1}} \lVert \tau_{\psi v}(u)-u \rVert_{L^{p}(\omega)},
\end{multline}
where \( \tau_{\psi v}(u)(x) = u(x+\psi(x)v) \).
Our main task in this section will be to obtain estimates in the spirit of~\eqref{eq:Lp_smooth_leq_translation} for maps in \( W^{s,p}(\Omega;\R^{\nu}) \).

Before stating the main result of this section, we pause to explain the role of an important assumption.
In the sequel, we will assume that \( \lVert D\psi \rVert_{L^{\infty}(\Omega)} < 1 \).
We illustrate the usefulness of this condition in the simpler context of \( L^{p} \) estimates.
We start by using Minkowski's inequality to write
\[
	\lVert \varphi_{\psi} \ast u \rVert_{L^{p}(\omega)}
	\leq
	\int_{B^{m}_{1}} \varphi(z)\biggl(\int_{\omega} \lvert u(x+\psi(x)z) \rvert^{p}\,\d x\biggr)^{\frac{1}{p}}\,\d z.
\]
Next we use the change of variable \( w = x+\psi(x)z \).
We note that the map \( \upPsi \colon \omega \to \Omega \) defined by \( \upPsi(x) = x+\psi(x)z \) satisfies \( D\upPsi(x) = \id + D\psi(x) \otimes z \).
Therefore, by rank-one perturbation of the identity (see e.g.~\cite{serre_matrices}*{Section~3.8}), we deduce that
\[
    \jac \upPsi = \lvert \det{(\id+D\psi \otimes z)} \rvert = \lvert 1+D\psi \cdot z \rvert \geq 1 - \lVert D\psi \rVert_{L^{\infty}(\Omega)} 
    \quad\text{for \( z \in B^{m}_{1} \)}
\]
(where \(\jac \upPsi = \lvert\det (D\upPsi)\rvert\) is the Jacobian of \( \upPsi \)).
Thanks to the assumption \( \lVert D\psi \rVert_{L^{\infty}(\Omega)} < 1 \), the linear map \( D\Psi(x) \) is invertible for \( z \in B^{m}_{1} \), so that the above change of variable is well-defined with Jacobian
less than \( \frac{1}{1 - \lVert D\psi \rVert_{L^{\infty}(\omega)}} \).
We conclude that
\begin{equation}
\label{eq:Lp_bound_convolution}
	\lVert \varphi_{\psi} \ast u \rVert_{L^{p}(\omega)}
	\leq
	\frac{1}{(1 - \lVert D\psi \rVert_{L^{\infty}(\omega)})^{\frac{1}{p}}}\lVert u \rVert_{L^{p}(\Omega)}.
\end{equation}

We now state the main result of this section, which is the counterpart of~\cite{BPVS_density_higher_order}*{Proposition~3.2} in the context of fractional Sobolev spaces.

\begin{prop}
\label{prop:adaptive_smoothing}
Let \( \varphi \in \Cc^{\infty}_{\mathrm{c}}(B^{m}_{1}) \) be a mollifier and let \( \psi \in \Cc^{\infty}(\Omega) \) be a nonnegative function such that \( \lVert D\psi \rVert_{L^{\infty}(\Omega)} < 1 \).
For every \( u \in W^{s,p}(\Omega;\R^{\nu}) \) and every open set \( \omega \subset \{ x \in \Omega\mathpunct{:} \dist{(x,\partial\Omega)} > \psi(x) \} \), we have \( \varphi_{\psi} \ast u \in W^{s,p}(\omega;\R^{\nu}) \), and moreover, the following estimates hold: 
\begin{enumerate}[label=(\roman*)]
    \item\label{item:first_estimates_convolution}
    \begin{enumerate}[label=(\alph*)]
    	\item\label{item:first_estimates_convolution_sle1} if \( 0 < s < 1 \), then 
	    	\[
		    	\lvert \varphi_{\psi}\ast u\lvert_{W^{s,p}(\omega)}
		    	\leq 
		    	C\frac{1}{(1-\lVert D\psi \rVert_{L^{\infty}(\omega)})^{\frac{2}{p}}}\lvert u\lvert_{W^{s,p}(\Omega)};
	    	\]
	    \item\label{item:first_estimates_convolution_sgeq1_integer} if \( s \geq 1 \), then for every \( j \in \{1,\dots,k\} \),
	    	\[
		    	\eta^{j}\lVert D^{j}(\varphi_{\psi}\ast u)\lVert_{L^{p}(\omega)}
		    	\leq
		    	C\frac{1}{(1-\lVert D\psi \rVert_{L^{\infty}(\omega)})^{\frac{1}{p}}}\sum_{i=1}^{j}\eta^{i}\lVert D^{i}u \rVert_{L^{p}(\Omega)};
	    	\]
	    \item\label{item:first_estimates_convolution_sgeq1_frac} if \( s \geq 1 \) and \( \sigma \neq 0 \), then for every \( j \in \{1,\dots,k\} \),
	    	\begin{multline*}
		    	\eta^{j+\sigma}\lvert D^{j}(\varphi_{\psi}\ast u)\lvert_{W^{\sigma,p}(\omega)} \\
		    	\leq
		    	C\frac{1}{(1-\lVert D\psi \rVert_{L^{\infty}(\omega)})^{\frac{2}{p}}}\sum_{i=1}^{j}\Bigl(\eta^{i}\lVert D^{i}u \rVert_{L^{p}(\Omega)} 
		    	+ \eta^{i+\sigma}\lvert D^{i}u\lvert_{W^{\sigma,p}(\Omega)}\Bigr);
	    	\end{multline*}
    \end{enumerate} 
    \item\label{item:second_estimates_convolution}
    \begin{enumerate}[label=(\alph*)]
    	\item\label{item:second_estimates_convolution_sle1} if \( 0 < s < 1 \), then 
	    	\[
		    	\lvert \varphi_{\psi}\ast u-u\lvert_{W^{s,p}(\omega)}
		    	\leq 
		    	\sup_{v \in B^{m}_{1}}\lvert \tau_{\psi v}(u)-u \rvert_{W^{s,p}(\omega)};
	    	\]
    	\item\label{item:second_estimates_convolution_sgeq1_integer} if \( s \geq 1 \), then for every \( j \in \{1,\dots,k\} \),
	    	\begin{multline*}
	    		\eta^{j}\lVert D^{j}(\varphi_{\psi}\ast u)-D^{j}u\lVert_{L^{p}(\omega)}
	    		\leq
	    		\sup_{v \in B^{m}_{1}}\eta^{j}\lVert \tau_{\psi v}(D^{j}u)-D^{j}u \rVert_{L^{p}(\omega)} \\
	    		+ C\frac{1}{(1-\lVert D\psi \rVert_{L^{\infty}(\omega)})^{\frac{1}{p}}}\sum_{i=1}^{j}\eta^{i}\lVert D^{i}u \rVert_{L^{p}(A)};
	    	\end{multline*}
    	\item\label{item:second_estimates_convolution_sgeq1_frac} if \( s \geq 1 \) and \( \sigma \neq 0 \), then for every \( j \in \{1,\dots,k\} \),
    		\begin{multline*}
    			\eta^{j+\sigma}\lvert D^{j}(\varphi_{\psi}\ast u)-D^{j}u\lvert_{W^{\sigma,p}(\omega)}
    			\leq
    			\sup_{v \in B^{m}_{1}}\eta^{j+\sigma}\lvert \tau_{\psi v}(D^{j}u)-D^{j}u \rvert_{W^{\sigma,p}(\omega)} \\
    			+ C\frac{1}{(1-\lVert D\psi \rVert_{L^{\infty}(\omega)})^{\frac{2}{p}}}\sum_{i=1}^{j}\Bigl(\eta^{i}\lVert D^{i}u \rVert_{L^{p}(A)}
    			+ \eta^{i+\sigma}\lvert D^{i}u\lvert_{W^{\sigma,p}(A)}\Bigr);
    		\end{multline*}
    \end{enumerate} 
\end{enumerate}
for some constant \( C > 0 \) depending on \( m \), \( s \), and \( p \), where
\[
	A = \bigcup_{x \in \omega \cap \supp D\psi} B^{m}_{\psi(x)}(x)
\]
and \( \eta > 0 \) satisfies
\[
	\eta^{j}\lVert D^{j}\psi \rVert_{L^{\infty}} \leq \eta
	\quad
	\text{for every \( j \in \{2,\dots,k+1\} \).}
\]
\end{prop}
\begin{proof}
    The proof of item~\ref{item:first_estimates_convolution} is completely analogous to the proof of item~\ref{item:second_estimates_convolution} and uses the same ingredients.
    Hence we focus on item~\ref{item:second_estimates_convolution}, and we explain in the end what should be changed in order to get~\ref{item:first_estimates_convolution}.

    We start with the integer order estimate in the case \( s \geq 1 \).
    By the Faà di Bruno formula, for every \( x \in \omega \), we have
    \begin{equation}
    \label{eq:Faa_di_Bruno_adaptive_convolution}
    \begin{split}
        D^{j}(\varphi_{\psi}\ast u)(x)
        =&
        \int_{B^{m}_{1}}\varphi(z)D^{j}u(x+\psi(x)z)[(\id+D\psi(x)\otimes z)^{j}]\,\d z \\
            &+ \sum_{i=1}^{j-1}\sum_{l=1}^{n(i,j)}\int_{B^{m}_{1}}\varphi(z)D^{i}u(x+\psi(x)z)[L_{i,l,j}(x,z)]\,\d z,
    \end{split}
	\end{equation}
    where \( n(i,j) \in \N_{\ast} \) and \( L_{i,l,j}(x,z) \) is a linear mapping \( \R^{j\times m} \to \R^{i\times m} \) depending on \( \psi \) and its derivatives.
    More precisely, each entry of \( L_{i,l,j}(x,z) \) is either \( \id+D\psi(x)\otimes z \) or \( D^{t}\psi(x) \otimes z \) for some \( t \in \{2,\dots,j\} \), and the sum over all \( i \) components of \( L_{i,l,j}(x,z) \) of the order of the derivative of \( \psi \) appearing in this component is \( j \).
    Moreover, since \( i < j \), at least one entry of \( L_{i,j,l}(x,z) \) has the form \( D^{t}\psi(x) \), and thus the second integral in~\eqref{eq:Faa_di_Bruno_adaptive_convolution} lives only on \( \supp D\psi \).
    Therefore, taking into account the assumption \( \lVert D^{t}u \rVert_{L^{\infty}} \leq \eta^{1-t} \), we deduce that 
    \[
        \lvert D^{i}u(x+\psi(x)z)[L_{i,l,j}(x,z)] \rvert 
        \leq 
        \C \lvert D^{i}u(x+\psi(x)z) \rvert \eta^{i-j}\chi_{\supp D\psi}(x).
    \]

    On the other hand, we note that, by \( j \)-linearity of \( D^{j}u \), we may write \( D^{j}u(x+\psi(x)z)[(\id+D\psi(x)\otimes z)^{j}] \) as the sum of \( D^{j}u(x+\psi(x)z) \) and \( 2^{j}-1 \) terms
    which are the composition of \( D^{j}u(x+\psi(x)z) \) with a \( j \)-linear map \( \R^{j\times m} \to \R^{j \times m} \) whose entries are either \( \id \) or \( D\psi(x)\otimes z \), with at least one of them
    being the latter.
    Hence, since \( \lVert D\psi \rVert_{L^{\infty}} < 1 \), each of those \( 2^{j}-1 \) last terms is bounded by \( \lvert D^{j}u(x+\psi(x)z) \rvert \chi_{\supp D\psi}(x) \).
    For instance, if \( j = 2 \), then 
    \begin{multline*}
    	D^{2}u(x+\psi(x)z)[(\id+D\psi(x)\otimes z)^{2}]
    	=
    	D^{2}u(x+\psi(x)z) + D^{2}u(x+\psi(x)z)[\id,D\psi(x)\otimes z] \\
    	+ D^{2}u(x+\psi(x)z)[D\psi(x)\otimes z,\id] + D^{2}u(x+\psi(x)z)[D\psi(x)\otimes z, D\psi(x)\otimes z].
    \end{multline*}
    We observe that indeed, the three last terms are obtained by composition of \( D^{2}u(x+\psi(x)z) \) with a bilinear map, at least one of whose entries being \( D\psi(x)\otimes z \).

    As a consequence, for every \( x \in \omega \), we may write 
    \[
    \begin{split}
        \lvert D^{j}(\varphi_{\psi}\ast u)(x) - D^{j}u(x) \rvert 
        \leq& 
        \int_{B^{m}_{1}}\varphi(z)\lvert D^{j}u(x+\psi(x)z) - D^{j}u(x) \rvert\,\d z \\
            &+ \C \sum_{i=1}^{j} \eta^{i-j}\chi_{\supp D\psi}(x)\int_{B^{m}_{1}}\varphi(z)\lvert D^{i}u(x+\psi(x)z) \rvert\,\d z.
    \end{split}
    \]
    Minkowski's inequality ensures that 
    \[
    \begin{split}
        \lVert D^{j}(\varphi_{\psi}\ast u) - D^{j}u \rVert_{L^{p}(\omega)}
        \leq& 
        \int_{B^{m}_{1}}\varphi(z)\biggl[\biggl(\int_{\omega} \lvert D^{j}u(x + \psi(x) z)-D^{j}u(x)\rvert^{p}\,\d x\biggr)^{\frac{1}{p}} \\
            &+ C\sum_{i=1}^{j}\eta^{i-j}\biggl(\int_{\omega \cap \supp D\psi} \lvert D^{i}u(x+\psi(x) z) \rvert^{p}\,\d x\biggr)^{\frac{1}{p}}\biggr]\,\d z.
    \end{split}
    \]
    For the first term, we note that 
    \[
        \biggl(\int_{\omega} \lvert D^{j}u(x + \psi(x) z)-D^{j}u(x)\rvert^{p}\,\d x\biggr)^{\frac{1}{p}}
        \leq 
        \sup_{v \in B^{m}_{1}}\lVert \tau_{\psi v}(D^{j}u)-D^{j}u\rVert_{L^{p}(\omega)}.
    \]
    For the second term, we use the change of variable \( w = x + \psi(x)z \) that we considered before.
    Taking into account the definition of the set \( A \), we have \( w \in B_{\psi(x)}(x) \subset A \), and therefore 
    \[
        \biggl(\int_{\omega \cap \supp D\psi} \lvert D^{i}u(x+\psi(x) z) \rvert^{p}\,\d x\biggr)^{\frac{1}{p}}
        \leq 
        \frac{1}{(1-\lVert D\psi \rVert_{L^{\infty}(\omega)})^{\frac{1}{p}}}\biggl(\int_{A} \lvert D^{i}u(w) \rvert^{p}\,\d w\biggr)^{\frac{1}{p}}.
    \]
    We obtain the desired estimate by using the fact that \( \varphi \) has integral equal to \( 1 \).

    The estimate in the fractional case \( 0 < s < 1 \) is straightforward.
    Indeed, we first write 
    \begin{multline*}
        \lvert \varphi_{\psi}\ast u(x)-u(x)-\varphi_{\psi}\ast u(y)+u(y) \rvert \\
        \leq 
        \int_{B^{m}_{1}}\varphi(z)\lvert u(x+\psi(x)z)-u(x)-u(y+\psi(y)z)+u(y)\rvert\,\d z.
    \end{multline*}
    Minkowski's inequality then implies that 
    \begin{align*}
        &\lvert \varphi_{\psi} \ast u - u \rvert_{W^{s,p}(\omega)} \\
        &\quad\leq 
        \int_{B^{m}_{1}}\varphi(z)\biggl(\int_{\omega}\int_{\omega}\frac{\lvert u(x+\psi(x)z)-u(x)-u(y+\psi(y)z)+u(y)\rvert^{p}}{\lvert x-y \rvert^{m+sp}}\,\d x\d y\biggr)^{\frac{1}{p}}\,\d z \\
        &\quad\leq 
        \sup_{v \in B^{m}_{1}}\lvert \tau_{\psi v}(u)-u \rvert_{W^{s,p}(\omega)}.
    \end{align*}

    For the fractional estimate when \( s \geq 1 \), we again use equation~\eqref{eq:Faa_di_Bruno_adaptive_convolution} and the observations following this equation.
    Let \( x \), \( y \in \omega \).
    We proceed by distinction of cases, using the multilinearity of the differential.
    
    \emph{Case 1: \( x \), \( y \in \supp D\psi \).}
    For the terms with \( i < j \), using the \( j \)-linearity of \( D^{j}u \), we estimate
    \begin{align*}
        &\lvert D^{i}u(x+\psi(x)z)[L_{i,l,j}(x,z)] - D^{i}u(y+\psi(y)z)[L_{i,l,j}(y,z)] \rvert \\
        &\quad\leq \C\biggl(\sum_{t=1}^{j}\eta^{i-1-j+t}\lvert D^{t}\psi(x) - D^{t}\psi(y) \rvert\lvert D^{i}u(x+\psi(x)z) \rvert \\
        &\quad\quad+ \eta^{i-j}\lvert D^{i}u(x+\psi(x)z)-D^{i}u(y+\psi(y)z) \rvert\biggr).
    \end{align*}
    On the other hand, for the term involving the derivative of order \( j \) of \( u \), we have
    \begin{multline*}
	    \lvert D^{j}u(x+\psi(x)z)[(\id+D\psi(x)\otimes z)^{j}] - D^{j}u(x) - D^{j}u(y+\psi(y)z)[(\id+D\psi(y)\otimes z)^{j}] + D^{j}u(y) \rvert \\
	    \leq
	    A_{j}(x,y,z)
	    + \C B_{j}(x,y,z)
	    + \C\lvert D^{j}u(x+\psi(x)z) \rvert \lvert D\psi(x) - D\psi(y) \rvert,
    \end{multline*}
    where
    \begin{gather*}
    	A_{j}(x,y,z) = \lvert D^{j}u(x+\psi(x)z) - D^{j}u(x) - D^{j}u(y+\psi(y)z) + D^{j}u(y) \rvert \\
    	\intertext{and}
    	B_{j}(x,y,z) = \lvert D^{j}u(x+\psi(x)z) - D^{j}u(y+\psi(y)z) \rvert. 
    \end{gather*}
    Now, for \( t \in \{1,\dots,j\} \), we estimate
    \begin{align*}
	    &\int_{\omega}\frac{\lvert D^{t}\psi(x) - D^{t}\psi(y) \rvert^{p}}{\lvert x-y \rvert^{m+\sigma p}}\,\d y \\
	    &\quad\leq 
	    \eta^{-tp}\int_{B^{m}_{r}(x)}\frac{1}{\lvert x-y \rvert^{m+(\sigma-1)p}}\,\d y
	    + \C\eta^{(1-t)p}\int_{\R^{m}\setminus B^{m}_{r}(x)}\frac{1}{\lvert x-y \rvert^{m+\sigma p}}\,\d y \\
	    &\quad\leq
	    \C\Bigl(\eta^{-tp}r^{(1-\sigma)p} + \eta^{(1-t)p}r^{-\sigma p}\Bigr)
    \end{align*}
    for every \( r > 0 \).
    Letting \( r = \eta \) yields
    \begin{equation}
    \label{eq:optimization_final_adaptive_smoothing}
	    \int_{\omega}\frac{\lvert D^{t}\psi(x) - D^{t}\psi(y) \rvert^{p}}{\lvert x-y \rvert^{m+\sigma p}}\,\d y
	    \leq
	    \C\eta^{(1-t-\sigma)p}.
    \end{equation}
    Therefore, using Minkowski's inequality on the expression obtained from~\eqref{eq:Faa_di_Bruno_adaptive_convolution}, we deduce that 
    \begin{align*}
	    &\biggl(\int_{\omega\cap\supp D\psi}\int_{\omega\cap\supp D\psi}\frac{\lvert D^{j}(\varphi_{\psi}\ast u)(x) - D^{j}u(x) - D^{j}(\varphi_{\psi}\ast u)(y) + D^{j}u(y) \rvert^{p}}{\lvert x-y \rvert^{m+\sigma p}}\,\d x\d y\biggr)^{\frac{1}{p}} \\
	    &\quad\leq
	    \int_{B_{1}^{m}}\varphi(z)\biggl[\biggl(\int_{\omega\cap\supp D\psi}\int_{\omega\cap\supp D\psi}\frac{A_{j}(x,y,z)^{p}}{\lvert x-y \rvert^{m+\sigma p}}\,\d x\d y\biggr)^{\frac{1}{p}} \\
	    &\quad\quad+ \C\sum_{i=1}^{j}\eta^{i-j}\biggl(\int_{\omega \cap\supp D\psi}\int_{\omega\cap\supp D\psi}
	    \frac{B_{j}(x,y,z)^{p}}{\lvert x-y \rvert^{m+\sigma p}}\,\d x\d y\biggr)^{\frac{1}{p}} \\
	    &\quad\quad+\C \sum_{i=1}^{j}\eta^{i-j-\sigma}\biggl(\int_{\omega\cap\supp D\psi} \lvert D^{i}u(x+\psi(x)z)\rvert^{p}\,\d x\biggr)^{\frac{1}{p}}\biggr]\,\d z.
    \end{align*}
    
    \emph{Case 2: without loss of generality,  \( x \in \supp D\psi \) and \( y \notin \supp D\psi \).}
    In this case, since each \( L_{i,l,j}(y,z) \) has at least one entry equal to \( D^{t}\psi(y) \), we find
    \[
    	 D^{i}u(y+\psi(y)z)[L_{i,l,j}(y,z)] = 0 = D^{i}u(x+\psi(x)z)[L_{i,l,j}(y,z)].
    \]
    Hence,
    \begin{multline*}
	    \lvert D^{i}u(x+\psi(x)z)[L_{i,l,j}(x,z)] - D^{i}u(y+\psi(y)z)[L_{i,l,j}(y,z)] \rvert \\
	    \leq \sum_{t=1}^{j}\eta^{i-1-j+t}\lvert D^{t}\psi(x) - D^{t}\psi(y) \rvert\lvert D^{i}u(x+\psi(x)z) \rvert.
    \end{multline*}
    On the other hand, we have 
    \begin{multline*}
	    \lvert D^{j}u(x+\psi(x)z)[(\id+D\psi(x)\otimes z)^{j}] - D^{j}u(x) - D^{j}u(y+\psi(y)z)[(\id+D\psi(y)\otimes z)^{j}] + D^{j}u(y) \rvert \\
	    \leq
	    A_{j}(x,y,z)
	    + \C\lvert D^{j}u(x+\psi(x)z) \rvert \lvert D\psi(x) - D\psi(y) \rvert.
    \end{multline*}
    We then argue as in Case~1, using~\eqref{eq:optimization_final_adaptive_smoothing} to deal with the terms containing \( \lvert D^{t}\psi(x)-D^{t}\psi(y) \rvert \), and we deduce that 
    \begin{align*}
	    &\biggl(\int_{\omega\setminus\supp D\psi}\int_{\omega\cap\supp D\psi}\frac{\lvert D^{j}(\varphi_{\psi}\ast u)(x) - D^{j}u(x) - D^{j}(\varphi_{\psi}\ast u)(y) + D^{j}u(y) \rvert^{p}}{\lvert x-y \rvert^{m+\sigma p}}\,\d x\d y\biggr)^{\frac{1}{p}} \\
	    &\quad\leq
	    \int_{B_{1}^{m}}\varphi(z)\biggl[\biggl(\int_{\omega\setminus\supp D\psi}\int_{\omega\cap\supp D\psi}\frac{A_{j}(x,y,z)^{p}}{\lvert x-y \rvert^{m+\sigma p}}\,\d x\d y\biggr)^{\frac{1}{p}} \\
	    &\quad\quad+\C \sum_{i=1}^{j}\eta^{i-j-\sigma}\biggl(\int_{\omega\cap\supp D\psi} \lvert D^{i}u(x+\psi(x)z)\rvert^{p}\,\d x\biggr)^{\frac{1}{p}}\biggr]\,\d z.
    \end{align*}
    
    \emph{Case 3: \( x \), \( y \notin \supp D\psi \).}
    In this case, for \( i < j \), we observe that
	    \[
	    \lvert D^{i}u(x+\psi(x)z)[L_{i,l,j}(x,z)] - D^{i}u(y+\psi(y)z)[L_{i,l,j}(y,z)] \rvert = 0.
	    \]
    Moreover, 
    \begin{multline*}
	    \lvert D^{j}u(x+\psi(x)z)[(\id+D\psi(x)\otimes z)^{j}] - D^{j}u(x) - D^{j}u(y+\psi(y)z)[(\id+D\psi(y)\otimes z)^{j}] + D^{j}u(y) \rvert \\
	    =
	    A_{j}(x,y,z).
    \end{multline*}
    Hence, unlike in the previous cases, estimate~\eqref{eq:optimization_final_adaptive_smoothing} is not needed, and a simple application of Minkowski's inequality yields
    \begin{multline*}
	    \biggl(\int_{\omega\setminus \supp D\psi}\int_{\omega\setminus \supp D\psi}\frac{\lvert D^{j}(\varphi_{\psi}\ast u)(x) - D^{j}u(x) - D^{j}(\varphi_{\psi}\ast u)(y) + D^{j}u(y) \rvert^{p}}{\lvert x-y \rvert^{m+\sigma p}}\,\d x\d y\biggr)^{\frac{1}{p}} \\
	    \leq
	    \int_{B_{1}^{m}}\varphi(z)\biggl[\biggl(\int_{\omega\setminus \supp D\psi}\int_{\omega\setminus \supp D\psi}\frac{A_{j}(x,y,z)^{p}}{\lvert x-y \rvert^{m+\sigma p}}\,\d x\d y\biggr)^{\frac{1}{p}}\,\d z.
    \end{multline*}

    Gathering the estimates obtained in Cases~1,~2, and~3, we deduce that
    \begin{align*}
        &\biggl(\int_{\omega}\int_{\omega}\frac{\lvert D^{j}(\varphi_{\psi}\ast u)(x) - D^{j}u(x) - D^{j}(\varphi_{\psi}\ast u)(y) + D^{j}u(y) \rvert^{p}}{\lvert x-y \rvert^{m+\sigma p}}\,\d x\d y\biggr)^{\frac{1}{p}} \\
        &\quad\leq
        \C\int_{B_{1}^{m}}\varphi(z)\biggl[\biggl(\int_{\omega}\int_{\omega}\frac{A_{j}(x,y,z)^{p}}{\lvert x-y \rvert^{m+\sigma p}}\,\d x\d y\biggr)^{\frac{1}{p}} \\
        &\quad\quad+ \sum_{i=1}^{j}\eta^{i-j}\biggl(\int_{\omega \cap\supp D\psi}\int_{\omega\cap\supp D\psi}
                \frac{B_{j}(x,y,z)^{p}}{\lvert x-y \rvert^{m+\sigma p}}\,\d x\d y\biggr)^{\frac{1}{p}} \\
        &\quad\quad+ \sum_{i=1}^{j}\eta^{i-j-\sigma}\biggl(\int_{\omega\cap\supp D\psi} \lvert D^{i}u(x+\psi(x)z)\rvert^{p}\,\d x\biggr)^{\frac{1}{p}}\biggr]\,\d z.
    \end{align*}
    For the first term, we observe once again that
    \begin{equation*}
        \biggl(\int_{\omega}\int_{\omega}\frac{A_{j}(x,y,z)^{p}}{\lvert x-y \rvert^{m+\sigma p}}\,\d x\d y\biggr)^{\frac{1}{p}} \\
        \leq
        \sup_{v \in B^{m}_{1}}\lvert \tau_{\psi v}(D^{j}u)-D^{j}u \rvert_{W^{\sigma,p}(\omega)}.
    \end{equation*}
    For the third term, we use the change of variable \( w = x+\psi(x)z \), and we find
    \begin{align*}
        \int_{\omega\cap\supp D\psi} \lvert D^{i}u(x+\psi(x)z)\rvert^{p}\,\d x
        &\leq
        \frac{1}{(1-\lVert D\psi \rVert_{L^{\infty}(\omega)})}\int_{A} \lvert D^{i}u(w) \rvert^{p} \,\d w \\
        &\leq
        \frac{1}{(1-\lVert D\psi \rVert_{L^{\infty}(\omega)})^{2}}\lVert D^{i}u \rVert_{L^{p}(A)}^{p}.
    \end{align*}
    For the second term, we make use of the change of variable \( w = x+\psi(x)z \) and \( \tilde{w} = y+\psi(y)z \).
    Observe that \( \lvert w-\tilde{w} \rvert \leq 2\lvert x-y \rvert \), and hence
    \begin{multline*}
        \int_{\omega \cap\supp D\psi}\int_{\omega\cap\supp D\psi}\frac{B_{j}(x,y,z)^{p}}{\lvert x-y \rvert^{m+\sigma p}}\,\d x\d y \\
        \leq
        \C\frac{1}{(1-\lVert D\psi \rVert_{L^{\infty}(\omega)})^{2}}\int_{A}\int_{A}\frac{\lvert D^{i}u(w)-D^{i}u(\tilde{w}) \rvert^{p}}{\lvert w-\tilde{w} \rvert^{m+\sigma p}}\,\d w\d \tilde{w}.
        \end{multline*}	
    Using the fact that \( \varphi \) has integral equal to \( 1 \), this concludes the proof of the fractional estimate when \( s \geq 1 \).
    
    The proof of assertion~\ref{item:first_estimates_convolution} follows the same strategy.
    The only change is that, instead of grouping the term \( D^{j}u(x+\psi(x)z) \) coming from the first term in~\eqref{eq:Faa_di_Bruno_adaptive_convolution} with the \( D^{j}u \),
    we have to estimate it as all the other terms.
    Unlike the \( 2^{j}-1 \) terms involving a derivative of order \( j \) of \( u \), this term does not vanish outside of the support of \( D\psi \).
    This explains the presence of the norm on the whole \( \Omega \) in estimates~\ref{item:first_estimates_convolution}.
    \resetconstant
\end{proof}

Adaptive smoothing is a very useful tool to approximate a \( W^{s,p} \) map by smooth maps, but it has a major drawback in the context of Sobolev spaces with values into manifolds.
Indeed, if \( u \in W^{s,p}(\Omega;\Nc) \), in general \( \varphi_{\psi} \ast u \) does not take values into \( \Nc \), 
since the convolution product is in general not compatible with the constraint.
Therefore, it will be crucial in the proof of Theorem~\ref{thm:density_class_R} to be able to estimate the distance between the smoothed maps and the manifold.
We close this section by a discussion devoted to this purpose, which also sheds light on how to use the estimates obtained during the opening procedure in the previous section. 

We work in a slightly more general setting, assuming that \( u \in W^{s,p}(\Omega;\R^{\nu}) \) is such that \( u(x) \in F \) for almost every \( x \in \Omega \), where \( F \subset \R^{\nu} \) is an arbitrary closed set.
We place ourselves under the assumptions of Propositions~\ref{prop:main_opening},~\ref{prop:addendum1_opening}, and~\ref{prop:addendum2_opening}.
We denote by \( \upPhi^{\op}_{\eta} \) the map provided by Proposition~\ref{prop:main_opening} and we set \( u^{\op}_{\eta} = u \circ \upPhi^{\op}_{\eta} \).
Let \( u^{\sm}_{\eta} = \varphi_{\psi_{\eta}} \ast u^{\op}_{\eta} \),
where \( \varphi \) is a fixed mollifier, and the variable regularization parameter \( \psi_{\eta} \) is to be chosen later on, depending on \( \eta \).

Let \( 0 < \ulrho < \rho \) be fixed, and assume that \( \Uc^{m}_{\eta} \) is a subskeleton of some skeleton \( \Kc^{m}_{\eta} \) such that \( K^{m}_{\eta} \subset \omega \).
To fix the ideas, one may keep in mind that \( K^{m}_{\eta} = \omega \) in the case where \( \omega \) can be decomposed as a finite union of cubes of radius \( \eta \).
We consider a subskeleton \( \Ec^{m}_{\eta} \) of \( \Uc^{m}_{\eta} \) such that 
\begin{equation}
\label{eq:Em_int_Um}
	E^{m}_{\eta} \subset \Int U^{m}_{\eta}
\end{equation}
in the relative topology of \( K^{m}_{\eta} \).
(Later on in the proof of Theorem~\ref{thm:density_class_R}, \( \Ec^{m}_{\eta} \) will be the class of all bad cubes.)

Given a set \( S \subset \R^{\nu} \), the \emph{directed Hausdorff distance from \( S \) to \( F \)} is defined as 
\[
	\Dist_{F}{(S)} 
	=
	\sup\{\dist{(x,F)}\mathpunct{:} x \in S\}.
\]
Our objective is to show that, for a suitable choice of \( \psi_{\eta} \) and \( r > 0 \), we have 
\begin{multline}
\label{eq:main_estimate_dist_sm_F_sgeq1}
	\Dist_{F}{(u^{\sm}_{\eta}((K^{m}\setminus U^{m}_{\eta})\cup (U^{\ell}_{\eta}+Q^{m}_{\ulrho\eta})))}
	\leq 
	\max\biggl\{\max_{\sigma^{m} \in \Kc^{m}_{\eta} \setminus \Ec^{m}_{\eta}} \Cl{cst:dist1_sm}\frac{1}{\eta^{\frac{m}{sp}-1}}\lVert Du \rVert_{L^{sp}(\sigma^{m}+Q^{m}_{2\rho\eta})}, \\
	\sup_{x \in U^{\ell}_{\eta}+Q^{m}_{\ulrho\eta}} \Cl{cst:dist2_sm}\fint_{Q^{m}_{r}(x)}\fint_{Q^{m}_{r}(x)} \lvert u^{\op}_{\eta}(y)-u^{\op}_{\eta}(z)\rvert\,\d y\d z\biggr\}
\end{multline}
if \( s \geq 1 \), respectively 
\begin{multline}
\label{eq:main_estimate_dist_sm_F_sle1}
	\Dist_{F}{(u^{\sm}_{\eta}((K^{m}\setminus U^{m}_{\eta})\cup (U^{\ell}_{\eta}+Q^{m}_{\ulrho\eta})))}
	\leq 
	\max\biggl\{\max_{\sigma^{m} \in \Kc^{m}_{\eta} \setminus \Ec^{m}_{\eta}} \Cr{cst:dist1_sm}\frac{1}{\eta^{\frac{m}{p}-s}}\lvert u \rvert_{W^{s,p}(\sigma^{m}+Q^{m}_{2\rho\eta})}, \\
	\sup_{x \in U^{\ell}_{\eta}+Q^{m}_{\ulrho\eta}} \Cr{cst:dist2_sm}\fint_{Q^{m}_{r}(x)}\fint_{Q^{m}_{r}(x)} \lvert u^{\op}_{\eta}(y)-u^{\op}_{\eta}(z)\rvert\,\d y\d z\biggr\}
\end{multline}
if \( 0 < s < 1 \).
We note that, in order to make the right-hand side of~\eqref{eq:main_estimate_dist_sm_F_sgeq1}, respectively~\eqref{eq:main_estimate_dist_sm_F_sle1}, small, we need to take \( r \) sufficiently small, 
and also to have control on the \( L^{sp} \) norm of \( Du \), respectively the \( W^{s,p} \) norm of \( u \), on the cubes in \( \Kc^{m}_{\eta} \setminus \Ec^{m}_{\eta} \).
This will be our motivation for the choice of good and bad cubes in the proof of Theorem~\ref{thm:density_class_R}.

We proceed with the proof of~\eqref{eq:main_estimate_dist_sm_F_sgeq1}, respectively~\eqref{eq:main_estimate_dist_sm_F_sle1}.
Since \( u^{\op}_{\eta} \) takes its values into \( F \), for almost every \( z \in Q^{m}_{\psi_{\eta}(x)}(x) \), we have
\[
	\dist{(u^{\sm}_{\eta}(x),F)} \leq \lvert u^{\sm}_{\eta}(x)-u^{\op}_{\eta}(z) \rvert.
\] 
Averaging over \( Q^{m}_{\psi_{\eta}(x)}(x) \), we find 
\[
    \dist{(u^{\sm}_{\eta}(x),F)}
    \leq 
    \fint_{Q^{m}_{\psi_{\eta}(x)}(x)} \lvert u^{\sm}_{\eta}(x)-u^{\op}_{\eta}(z) \rvert\,\d z.
\]
Using the rewriting~\eqref{eq:reformulation_adaptive_smoothing}, we deduce that, for every \( x \in \omega \),
\begin{equation}
\begin{aligned}
\label{eq:dist_sm_F_by_oscillation}
    \dist{(u^{\sm}_{\eta}(x),F)}
    &\leq 
    \fint_{Q^{m}_{\psi_{\eta}(x)}(x)}\fint_{Q^{m}_{\psi_{\eta}(x)}(x)}\varphi\Bigl(\frac{y-x}{\psi(x)}\Bigr)\lvert u^{\op}_{\eta}(y)-u^{\op}_{\eta}(z) \rvert\,\d y\d z \\
    &\leq 
    \Cl{cst:bound_mollifier}\fint_{Q^{m}_{\psi_{\eta}(x)}(x)}\fint_{Q^{m}_{\psi_{\eta}(x)}(x)}\lvert u^{\op}_{\eta}(y)-u^{\op}_{\eta}(z) \rvert\,\d y\d z.
\end{aligned}
\end{equation} 

If \( Q^{m}_{\psi_{\eta}(x)}(x) \subset U^{\ell}_{\eta}+Q^{m}_{\rho\eta} \) and \( \ell \leq sp \), 
Proposition~\ref{prop:addendum2_opening} ensures that the right-hand side of~\eqref{eq:dist_sm_F_by_oscillation} can be made arbitrarily small if we take \( \psi_{\eta}(x) \) sufficiently small.
This invites us to choose \( \psi_{\eta} \) to be very small in a neighborhood of \( U^{\ell} \).

On the other hand, the Poincaré--Wirtinger inequality ensures that 
\begin{equation}
\label{eq:dist_sm_F_good_cubes_sgeq1}
    \dist{(u^{\sm}_{\eta}(x),F)}
    \leq 
    \C\frac{1}{\psi_{\eta}(x)^{\frac{m}{sp}-1}}\lVert Du^{\op}_{\eta} \rVert_{L^{sp}(Q^{m}_{\psi(x)}(x))}
\end{equation}
if \( s \geq 1 \), respectively 
\begin{equation}
\label{eq:dist_sm_F_good_cubes_sle1}
    \dist{(u^{\sm}_{\eta}(x),F)}
    \leq 
    \C\frac{1}{\psi_{\eta}(x)^{\frac{m}{p}-s}}\lvert Du^{\op}_{\eta} \rvert_{W^{s,p}(Q^{m}_{\psi(x)}(x))}
\end{equation}
if \( 0 < s < 1 \).
These estimates are only useful in the region where we can control the \( L^{sp} \) norm of \( Du \) or the \( W^{s,p} \) norm of \( u \), 
that is, on the good cubes.
On the other hand, since \( sp < m \),~\eqref{eq:dist_sm_F_good_cubes_sgeq1} and~\eqref{eq:dist_sm_F_good_cubes_sle1} suggest that \( \psi_{\eta} \) should not be too small.

We now pause to explain the construction of a function \( \psi_{\eta} \) suited for our approximation results. 
As explained in Section~\ref{sect:sketch_proof}, we distinguish between three regimes.
In \( U^{\ell}_{\eta}+Q^{m}_{\rho\eta} \), we take \( \psi_{\eta} \) very small, according to Proposition~\ref{prop:addendum2_opening}.
On the good cubes, we take \( \psi_{\eta} \) of order \( \eta \), in order to apply~\eqref{eq:dist_sm_F_good_cubes_sgeq1}, respectively~\eqref{eq:dist_sm_F_good_cubes_sle1}.
Between these two regimes, we need a transition region in order for \( \psi_{\eta} \) to change of magnitude.
Here the second part of Proposition~\ref{prop:addendum2_opening} comes into play.
Indeed, if \( x \in \sigma^{m} \) for some \( \sigma^{m} \in \Uc^{m}_{\eta} \) and \( Q^{m}_{\psi_{\eta}(x)} \subset U^{\ell}_{\eta}+Q^{m}_{\rho\eta} \)\,,
we have 
\begin{equation}
\label{eq:dist_sm_F_transition_sgeq1}
    \dist{(u^{\sm}_{\eta}(x),F)}
    \leq 
    \C\frac{\psi_{\eta}(x)^{1-\frac{\ell}{sp}}}{\eta^{\frac{m-\ell}{sp}}}\lVert Du \rVert_{L^{sp}(\sigma^{m}+Q^{m}_{2\rho\eta})}
\end{equation}
if \( s \geq 1 \), respectively 
\begin{equation}
\label{eq:dist_sm_F_transition_sle1}
    \dist{(u^{\sm}_{\eta}(x),F)}
    \leq 
    \C\frac{\psi_{\eta}(x)^{s-\frac{\ell}{p}}}{\eta^{\frac{m-\ell}{p}}}\lvert u \rvert_{W^{s,p}(\sigma^{m}+Q^{m}_{2\rho\eta})}
\end{equation}
if \( 0 < s < 1 \).
Again, this inequality is only useful on good cubes, but now it requires an upper bound on \( \psi_{\eta} \) instead if we take \( \ell \leq sp \).

Hence, we proceed with the following construction.
Assumption~\eqref{eq:Em_int_Um} ensures that we have enough room for the transition region for \( \psi_{\eta} \): we have \( \dist{(E^{m}_{\eta},K^{m}_{\eta} \setminus U^{m}_{\eta})} \geq \eta \). 
Therefore, we may find \( \zeta_{\eta} \in \Cc^{\infty}(\Omega) \) such that 
\begin{enumerate}[label=(\alph*)]
\label{list:assumptions_zeta_smoothing}
	\item\label{item:zeta_between_0_and_1} \( 0 \leq \zeta_{\eta} \leq 1 \) in \( \Omega \);
	\item\label{item:zeta_eq1} \( \zeta_{\eta} = 1 \) in \( K^{m}_{\eta} \setminus U^{m}_{\eta} \);
	\item\label{item:zeta_eq0} \( \zeta_{\eta} = 0 \) in \( E^{m}_{\eta} \);
	\item\label{item:D_zeta} for every \( j \in \{1,\dots,k+1\} \), 
		\[
			\eta^{j}\lVert D^{j}\zeta_{\eta} \rVert_{L^{\infty}} \leq \Cl{cst:Dzeta}
		\]
		for some constant \( \Cr{cst:Dzeta} > 0 \) depending only on \( m \).
\end{enumerate}
Now we pick \( 0 < r < t \) and we let 
\[
    \psi_{\eta} = t\zeta_{\eta} + r(1-\zeta_{\eta}).
\]
Therefore, \( \psi_{\eta} = r \) on \( E^{m}_{\eta} \) and \( \psi_{\eta} = t \) on \( K^{m}_{\eta} \setminus U^{m}_{\eta} \).
As we observed, we will need to take \( r \) very small, while keeping \( t \) of order \( \eta \).
We choose 
\begin{equation}
\label{eq:definition_t_smoothing}
	t = \min\Bigl\{\frac{\kappa}{\Cr{cst:Dzeta}},\rho-\ulrho\Bigr\}\eta
\end{equation}
for some fixed \( 0 < \kappa < 1 \).
Therefore, 
\[
    \eta^{j}\lVert D^{j}\psi_{\eta} \rVert_{L^{\infty}}
    \leq \kappa\eta
    \quad
    \text{for every \( j \in \{1,\dots,k+1\} \),}
\]
which ensures that the assumptions of Proposition~\ref{prop:adaptive_smoothing} are satisfied.
Moreover, we have \( 0 < \psi_{\eta} \leq (\rho-\ulrho)\eta \), 
which implies that, if \( x \in U^{\ell}_{\eta} + Q^{m}_{\ulrho\eta} \), then \( Q^{m}_{\psi_{\eta}(x)}(x) \subset U^{\ell}_{\eta}+Q^{m}_{\rho\eta} \).
Estimate~\eqref{eq:main_estimate_dist_sm_F_sgeq1}, respectively~\eqref{eq:main_estimate_dist_sm_F_sle1}, is a straightforward consequence of estimate~\eqref{eq:dist_sm_F_by_oscillation} for \( x \in E^{m}_{\eta} \cap (U^{\ell}+Q^{m}_{\ulrho\eta}) \),
estimate~\eqref{eq:dist_sm_F_good_cubes_sgeq1}, respectively~\eqref{eq:dist_sm_F_good_cubes_sle1}, for \( x \in K^{m}_{\eta} \setminus U^{m}_{\eta} \), 
and estimate~\eqref{eq:dist_sm_F_transition_sgeq1}, respectively~\eqref{eq:dist_sm_F_transition_sle1}, for \( x \in (U^{m}_{\eta} \setminus E^{m}_{\eta}) \cap (U^{\ell}_{\eta}+Q^{m}_{\ulrho\eta}) \).

Before closing this section, we summarize what we have obtained so far.
Given a map \( u \in W^{s,p}(\Omega;F) \), we have constructed a smooth map \( u^{\sm}_{\eta} \) for which we may estimate its distance to \( u \) in \( W^{s,p} \).
Moreover, even though \( u^{\sm}_{\eta} \) does not necessarily take values into \( F \), we are able to control the distance between \( u^{\sm}_{\eta} \) and \( F \) everywhere on the cubication \( K^{m}_{\eta} \), except on the cubes in \( \Uc^{m}_{\eta} \), far from their \( \ell \)-skeleton.
Therefore, our next step is to be able to modify \( u^{\sm}_{\eta} \) into a new map which, on the cubes in \( \Uc^{m}_{\eta} \), depends only on the values of \( u^{\sm}_{\eta} \) near the \( \ell \)-skeleton of the cubes,
while controlling the \( W^{s,p} \) distance between \( u^{\sm}_{\eta} \) and this new map.
\resetconstant

\section{Thickening}
\label{sect:thickening}

This section is devoted to the thickening procedure.
As we explained in Section~\ref{sect:sketch_proof}, this technique is reminiscent of the homogeneous extension method,
which was used by Bethuel to deal with the case \( s = 1 \); see~\cite{bethuel_approximation}.
This approach is valid for \( W^{s,p} \) maps with \( s < 1+\frac{1}{p} \) (but not beyond \( s = 1+\frac{1}{p} \)).
In order to deal with \( W^{s,p} \) maps with arbitrary \( s \), a new tool, thickening, is needed.
Its construction was performed by Bousquet, Ponce, and Van Schaftingen in~\cite{BPVS_density_higher_order}*{Section~4}, which also contains the analytic estimates that make thickening instrumental in the proof of Theorem~\ref{thm:density_class_R} for integer \( s \).
In this section, we establish the fractional counterparts of the estimates in~\cite{BPVS_density_higher_order}.
The main feature of this section is the need for new techniques, taking into account the geometry of the thickening maps, that we develop in order to obtain fractional estimates.
This will become transparent, e.g., in the proof of estimates~\ref{item:block_thickening_geometric_sle1} and~\ref{item:block_thickening_geometric_sgeq1_frac} in Proposition~\ref{prop:block_thickening_analytic}, relying crucially on estimate~\eqref{eq:mean_value_thickening}.

\begin{prop}
\label{prop:main_thickening}
    Let \( \Omega \subset \R^{m} \) be open, \( \ell \in \{0,\dots,m-1\} \), \( \eta > 0 \), \( 0 < \rho < 1 \), \( \Kc^{m} \) be a cubication in \( \R^{m} \) of radius \( \eta \), 
    \( \Uc^{m} \) be a subskeleton of \( \Kc^{m} \) such that \( U^{m}+Q^{m}_{\rho\eta} \subset \Omega \), and \( \Tc^{\ell^{\ast}} \) be the dual skeleton of \( \Uc^{\ell} \).
    There exists a smooth map \( \upPhi \colon \R^{m} \setminus T^{\ell^{\ast}} \to \R^{m} \) such that
    \begin{enumerate}[label=(\roman*)]
        \item\label{item:main_thickining_injective} \( \upPhi \) is injective;
        \item\label{item:main_thickening_geometry} for every \( \sigma^{m} \in \Kc^{m} \), \( \upPhi(\sigma^{m} \setminus T^{\ell^{\ast}}) \subset \sigma^{m} \setminus T^{\ell^{\ast}} \);
        \item\label{item:main_thickening_support} \( \Supp\upPhi \subset U^{\ell}+Q^{m}_{\rho\eta} \) and \( \upPhi(U^{m}\setminus T^{\ell^{\ast}}) \subset U^{\ell}+Q^{m}_{\rho\eta} \);
        \item\label{item:main_thickening_singularity} for every \( j \in \N_{\ast} \) and for every \( x \in (U^{m}+Q^{m}_{\rho\eta}) \setminus T^{\ell^{\ast}} \),
            \[
                \lvert D^{j}\upPhi(x) \rvert \leq C\frac{\eta}{\dist(x,T^{\ell^{\ast}})^{j}}
            \]
            for some constant \( C > 0 \) depending on \( j \), \( m \) and \( \rho \).
    \end{enumerate}
    If in addition \( \ell + 1 > sp \), then for every \( u \in W^{s,p}(\Omega;\R^{\nu}) \), we have \( u \circ \upPhi \in W^{s,p}(\Omega;\R^{\nu}) \),
    and moreover, the following estimates hold:
    \begin{enumerate}[label=(\alph*)]
    	\item\label{item:main_thickening_estimate_sle1} if \( 0 < s < 1 \), then
	    	\[
		    	\eta^{s}\lvert u \circ \upPhi - u \rvert_{W^{s,p}(\Omega)}
		    	\leq 
		    	C'\Bigl(\eta^{s}\lvert u \rvert_{W^{s,p}(U^{m}+Q^{m}_{\rho\eta})} + \lVert u \rVert_{L^{p}(U^{m}+Q^{m}_{\rho\eta})}\Bigr);
	    	\]
        \item\label{item:main_thickening_estimate_sgeq1_integer} if \( s \geq 1 \), then for every \( j \in \{1,\dots,k\} \),
            \[
                \eta^{j}\lVert D^{j}(u \circ \upPhi)-D^{j}u \rVert_{L^{p}(\Omega)}
                \leq
                C'\sum_{i=1}^{j}\eta^{i}\lVert D^{i}u \rVert_{L^{p}(U^{m}+Q^{m}_{\rho\eta})};
            \]
        \item\label{item:main_thickening_estimate_sgeq1_frac} if \( s \geq 1 \) and \( \sigma \neq 0 \), then for every \( j \in \{1,\dots,k\} \),
            \[
                \eta^{j+\sigma}\lvert D^{j}(u \circ \upPhi)-D^{j}u \rvert_{W^{\sigma,p}(\Omega)}
                \leq
                C'\sum_{i=1}^{j}\Bigl(\eta^{i}\lVert D^{i}u \rVert_{L^{p}(U^{m}+Q^{m}_{\rho\eta})} + \eta^{i+\sigma}\lvert D^{i}u \rvert_{W^{\sigma,p}(U^{m}+Q^{m}_{\rho\eta})}\Bigr);
            \]
        \item\label{item:main_thickening_estimate_all} for every \( 0 < s < +\infty \),
            \[
                \lVert u \circ \upPhi - u \rVert_{L^{p}(\Omega)}
                \leq 
                C'\lVert u \rVert_{L^{p}(U^{m}+Q^{m}_{\rho\eta})};
            \]
    \end{enumerate}
    for some constant \( C' > 0 \) depending on \( m \), \( s \), \( p \), and \( \rho \).
\end{prop}

We emphasize that, unlike for opening in Section~\ref{sect:opening}, the map \( \upPhi \) constructed in Proposition~\ref{prop:main_thickening} above is independent of the map \( u \in W^{s,p} \) it shall be composed with.

Similarly to opening, crucial to the proof of Theorem~\ref{thm:density_class_R} is the fact that the thickening procedure increases the energy of the map \( u \) at most by a constant factor in the region where \( u \) is modified.
This, in turn, implies that the distance between \( u \) and \( u \circ \upPhi \) is controlled by the energy of \( u \) on \( U^{m}+Q^{m}_{\rho\eta} \)\,, as stated in the conclusion of Proposition~\ref{prop:main_thickening}.
In the proof of Theorem~\ref{thm:density_class_R}, this will be used in combination with the fact that the measure of the set \( U^{m}+Q^{m}_{\rho\eta} \) tends to \( 0 \) as \( \eta \to 0 \) in order to ensure that \( u \circ \upPhi \) is close to \( u \) when \( \eta \) is sufficiently small.

As for the opening, Proposition~\ref{prop:main_thickening} is proved blockwise: we first construct, in Proposition~\ref{prop:block_thickening_geometric}, a map, still denoted \( \upPhi \), which thickens each face of \( \Tc^{\ell^{\ast}} \). 
We then suitably glue those maps to obtain a thickening map as in Proposition~\ref{prop:main_thickening}.
Before giving the description of the building blocks used in the proof of Proposition~\ref{prop:main_thickening}, we introduce some additional notation similarly to what we did for opening.
The construction of the map in Proposition~\ref{prop:block_thickening_geometric} below involves three parameters \( 0 < \ulrho < \rho < \olrho < 1 \).
These parameters being fixed, we define the rectangles 
\begin{equation}
\label{eq:definition_rectangles_thickening}
    Q_{1} = Q^{d}_{(1-\olrho)\eta} \times Q^{m-d}_{\ulrho\eta},
    \quad
    Q_{2} = Q^{d}_{(1-\rho)\eta} \times Q^{m-d}_{\ulrho\eta},
    \quad
    Q_{3} = Q^{d}_{(1-\rho)\eta} \times Q^{m-d}_{\rho\eta}.
\end{equation}
We note that \( Q_{1} \subset Q_{2} \subset Q_{3} \).
We also set \( T = \{0\}^{d} \times Q^{m-d}_{\rho\eta} \), the part of the dual skeleton contained in \( Q_{3} \).
The rectangle \( Q_{3} \) contains the geometric support of \( \upPhi \), that is, \( \upPhi = \id \) outside of \( Q_{3} \).
The rectangle \( Q_{2} \) is the region where the thickening procedure is fully performed: 
the set \( T \cap Q_{2} \) is entirely mapped outside of \( Q_{1} \), in \( Q_{2} \setminus Q_{1} \).
The region \( Q_{3} \setminus Q_{2} \) serves as a transition, on which the map \( \upPhi \) becomes less and less singular, 
until it reaches the exterior of \( Q_{3} \) where it coincides with the identity.

This section is organized as follows.
First we describe the geometric construction of the building blocks for thickening.
Then we prove the analytic estimates satisfied by the composition of a map \( u \in W^{s,p} \) with those building blocks.
Finally, we explain the construction of the global thickening map based on the aforementioned building blocks, and we prove all properties stated in the conclusion of Proposition~\ref{prop:main_thickening}.

We start by stating the geometric properties satisfied by the building blocks, which do not depend on the map \( u \) to which thickening is applied.
The map \( \upPhi \) constructed in Proposition~\ref{prop:block_thickening_geometric} is exactly the map given by~\cite{BPVS_density_higher_order}*{Proposition~4.3}.
Hence, we shall not give a complete proof of Proposition~\ref{prop:block_thickening_geometric}, but we will limit ourselves to recall, for the convenience of the reader, the main steps in the construction of the map \( \upPhi \).

The main difference with~\cite{BPVS_density_higher_order} is the proof of the Sobolev estimates.
In~\cite{BPVS_density_higher_order}, they were obtained on the whole \( \Omega \) as a corollary of the geometric properties of the map \( \upPhi \),
by the use of the change of variable theorem.
This approach does not seem to work to deal with the Gagliardo seminorm.
Hence, we first establish the estimates for the building blocks, and we deduce the global estimates by gluing, as for opening.
To do so, we need to take into account some additional features of the map \( \upPhi \), that are part of its construction in~\cite{BPVS_density_higher_order}*{Proof of Proposition~4.3}
but do not appear in the conclusion of Proposition~4.3 in~\cite{BPVS_density_higher_order}.

The construction of the map \( \upPhi \) involves another map \( \zeta \colon \R^{m} \to \R \), which we describe hereafter.
For \( (x',x'') \in \R^{d}\times \R^{m-d} \), we define 
\begin{equation}
\label{eq:definition_zeta_thickening}
    \zeta(x',x'') = \sqrt{\lvert x' \rvert^{2}+\eta^{2}\theta\Bigl(\frac{x''}{\eta}\Bigr)}.
\end{equation}
In~\cite{BPVS_density_higher_order}, \( \theta \colon \R^{m-d} \to [0,1] \) is an arbitrary smooth map such that 
\( \theta(x'') = 0 \) if \( x'' \in Q^{m-d}_{\ulrho} \) and \( \theta(x'') = 1 \) if \( x'' \in \R^{m-d} \setminus Q^{m-d}_{\rho} \).
For our purposes, we need to be more precise in our choice of \( \theta \).
We would like to choose \( \theta \) to be nondecreasing with respect to cubes, that is, \( \theta(x'') \) depends only on the \( \infty \)-norm of \( x'' \) 
and \( \theta(x'') \leq \theta(y'') \) if \( \lvert x'' \rvert_{\infty} \leq \lvert y'' \rvert_{\infty} \).
However, this is not possible if we want \( \theta \) to be smooth, since the \( \infty \)-norm is not differentiable.
Nevertheless, we may choose \( \theta \) sufficiently close to be nondecreasing with respect to cubes for our purposes, by replacing the \( \infty \)-norm by some \( q \)-norm for \( q \) sufficiently large.

More precisely, we take \( 1 < q < +\infty \) sufficiently large, depending on \( \ulrho \) and \( \rho \), so that there exists \( 0 < r_{1} < r_{2} \) satisfying 
\[
    Q^{m-d}_{\ulrho} 
    \subset 
    \{ x'' \in \R^{m-d} \mathpunct{:} \lvert x'' \rvert_{q} < r_{1} \}
    \subset
    \{ x'' \in \R^{m-d} \mathpunct{:} \lvert x'' \rvert_{q} < r_{2} \}
    \subset
    Q^{m-d}_{\rho}.
\]
This is indeed possible since \( Q^{m-d}_{\ulrho} \) and \( Q^{m-d}_{\rho} \) are respectively the balls of radius \( \ulrho \) and \( \rho \) with respect to the \( \infty \)-norm in \( \R^{m-d} \),
and since the \( q \)-norm converges uniformly on compact sets to the \( \infty \)-norm as \( q \to +\infty \).
We then pick a nondecreasing smooth map \( \tilde{\theta} \colon \R_{+} \to [0,1] \) such that \( \tilde{\theta}(r) = 0 \) if \( 0 \leq r \leq r_{1} \) and \( \tilde{\theta}(r) = 1 \) if \( r \geq r_{2} \).
We finally set \( \theta(x'') = \tilde{\theta}(\lvert x'' \rvert_{q}) \).
Since \( 1 < q < +\infty \), \( \theta \) is smooth on \( \R^{m-d} \), and, by our choice of \( q \), \( r_{1} \), and \( r_{2} \), we indeed have 
\( \theta(x'') = 0 \) if \( x'' \in Q^{m-d}_{\ulrho} \) and \( \theta(x'') = 1 \) if \( x'' \in \R^{m-d} \setminus Q^{m-d}_{\rho} \).

With the description of the map \( \zeta \) at our disposal, we are now ready to state Proposition~\ref{prop:block_thickening_geometric}.
We recall that the rectangles \( Q_{i} \) in~\eqref{eq:definition_rectangles_thickening} depend on \( d \) and \( \eta \).

\begin{prop}
\label{prop:block_thickening_geometric}
    Let \( d \in \{1,\dots,m\} \), \( \eta > 0 \), and \( 0 < \ulrho < \rho < \olrho \).
    There exists a smooth function \( \upPhi \colon \R^{m} \setminus T \to \R^{m} \) such that
    \begin{enumerate}[label=(\roman*)]
        \item\label{item:block_thickening_geometric_injective} \( \upPhi \) is injective;
        \item\label{item:block_thickening_geometric_support} \( \Supp\upPhi \subset Q_{3} \);
        \item\label{item:block_thickening_geometric_geometry} \( \upPhi(Q_{2}\setminus T) \subset Q_{2} \setminus Q_{1} \);
        \item\label{item:block_thickening_geometric_singularity} for every \( x \in Q_{3} \setminus T \),
            \[
                \lvert D^{j}\upPhi(x) \rvert \leq C\frac{\eta}{\zeta^{j}(x)}
                \quad
                \text{for every \( j \in \N_{\ast} \)}
            \]
            for some constant \( C > 0 \) depending on \( j \), \( m \), \( \ulrho \), \( \rho \), and \( \olrho \);
        \item\label{item:block_thickening_geometric_jacobian} for every \( x \in \R^{m} \setminus T \),
            \[
                \jac\upPhi(x) \geq C'\frac{\eta^{\beta}}{\zeta^{\beta}(x)}
                \quad 
                \text{for every \( 0 < \beta < d \),}
            \]
            for some constant \( C' > 0 \) depending on \( \beta \), \( m \), \( \ulrho \), \( \rho \), and \( \olrho \).
    \end{enumerate}
\end{prop}
\begin{proof}
    As we already mentioned, the desired map \( \upPhi \) is provided by~\cite{BPVS_density_higher_order}*{Proposition~4.3}.
    Hence, we limit ourselves to briefly recall its construction for the convenience of the reader, and we refer to~\cite{BPVS_density_higher_order} for the 
    complete proof of its properties.
    For technical reasons, we start by constructing an intermediate map \( \upPsi \colon \R^{m} \setminus T \to \R^{m} \) as follows; see~\cite{BPVS_density_higher_order}*{Lemma~4.5}.
    We define
    \[
        B_{1} = B^{d}_{(1-\olrho)\eta} \times Q^{m-d}_{\ulrho\eta},
        B_{2} = B^{d}_{(1-\rho)\eta} \times Q^{m-d}_{\ulrho\eta},
        B_{3} = B^{d}_{(1-\rho)\eta} \times Q^{m-d}_{\rho\eta}.
    \]
    The map \( \upPsi \) is constructed to satisfy the conclusion of Proposition~\ref{prop:block_thickening_geometric} with the rectangles \( Q_{i} \) replaced by the corresponding cylinders \( B_{i} \) for \( i \in \{1,2,3\} \).
    Since \( B_{i} \subset Q_{i} \), it will then suffice to compose \( \upPsi \) with a suitable diffeomorphism \( \upTheta \colon \R^{m} \to \R^{m} \) 
    that dilates \( B_{1} \) to a set containing \( Q_{1} \) in order to obtain the desired map \( \upPhi \).

    We choose a smooth map \( \varphi \colon (0,+\infty) \to [1,+\infty) \) such that 
    \begin{enumerate}[label=(\alph*)]
        \item\label{item:phi_small_r_thickening} for \( 0 < r \leq 1-\olrho \), 
            \[ 
                \varphi(r) = \frac{1-\olrho}{r}\Bigl(1+\frac{b}{\ln(\frac{1}{r})}\Bigr);
            \]
        \item\label{item:phi_large_r_thickening} for \( r \geq 1-\rho \), \( \varphi(r) = 1 \);
        \item\label{item:phir_increasing_thickening} the function \( r \in (0,+\infty) \mapsto r\varphi(r) \) is increasing.
    \end{enumerate}
    This is possible provided that we choose \( b > 0 \) such that 
    \[
        (1-\olrho)\Bigl(1+\frac{b}{\ln\frac{1}{1-\olrho}}\Bigr) < 1-\rho.
    \]
    Then we define \( \lambda \colon \R^{m} \setminus T \to [1,+\infty) \) by 
    \[
        \lambda(x) = \varphi\Bigl(\frac{\zeta(x)}{\eta}\Bigr),
    \]
    and finally 
    \[
        \upPsi(x',x'') = (\lambda(x',x'')x',x'').
    \]

    The injectivity of \( \upPsi \) is a consequence of assumption~\ref{item:phir_increasing_thickening} on \( \varphi \).
    The fact that \( \Supp\upPsi \subset B_{3} \) relies on assumption~\ref{item:phi_large_r_thickening} on \( \varphi \), since we may observe that \( \zeta(x) \geq (1-\rho)\eta \) if \( x \in \R^{m}\setminus B_{3} \),
    and therefore \( \lambda(x) = 1 \).
    Combining the observation that, using~\ref{item:phir_increasing_thickening} again,
    \[
        r\varphi(r) \geq \lim_{r \to 0} r\varphi(r) = 1-\olrho
    \]
    with the fact that \( \zeta(x) = \lvert x' \rvert \) if \( x = (x',x'') \in B_{2} \), we find that \( \upPsi(B_{2}\setminus T) \subset B_{2}\setminus B_{1} \).
    In order to obtain~\ref{item:block_thickening_geometric_singularity} on \( B_{3} \setminus T \), we estimate \( \lvert D^{j}\lambda(x) \rvert \) with the help of the Faà di Bruno formula, and then conclude using Leibniz's rule.
    The proof of estimate~\ref{item:block_thickening_geometric_jacobian} is more delicate.
    The Jacobian of \( \upPsi \) may be explicitly evaluated as the determinant of a rank-one perturbation of a diagonal map, as we did in the proof of~\eqref{eq:Lp_bound_convolution}, and one then uses the properties of \( \varphi \) and \( \zeta \) to get the required lower bound on the obtained expression.
    We refer the reader to~\cite{BPVS_density_higher_order}*{Lemma~4.5} for the details.

    It remains to correct the fact that we worked with the cylinders \( B_{i} \) instead of the rectangles \( Q_{i} \).
    This essentially amount to construct a suitable deformation of \( \R^{m} \) with bounded derivatives and a suitable lower bound on the Jacobian. 
    We let \( \upTheta \colon \R^{m} \to \R^{m} \) be a diffeomorphism whose geometric support is contained in \( Q_{3} \), 
    which maps \( B_{2} \setminus B_{1} \) on a set contained in \( Q_{2} \setminus Q_{1} \) -- that is, \( \upTheta \) dilates \( B_{1} \) on a set containing \( Q_{1} \) --
    and satisfies the estimates 
    \[
        \eta^{j-1}\lvert D^{j}\upTheta \rvert 
        \leq 
        \C
        \quad 
        \text{and}
        \quad
        0 < \C \leq \jac\upTheta \leq \C 
        \quad 
        \text{on \( \R^{m} \).}
    \]
    We refer the reader to~\cite{BPVS_density_higher_order}*{Lemma~4.4} for the precise construction of this diffeomorphism.

    Finally, we let \( \upPhi = \upTheta \circ \upPsi \).
    We observe that this construction satisfies the geometric properties~\ref{item:block_thickening_geometric_injective} to~\ref{item:block_thickening_geometric_geometry}.
    The estimate~\ref{item:block_thickening_geometric_jacobian} on the Jacobian readily follows from the composition formula for the Jacobian.
    To get~\ref{item:block_thickening_geometric_singularity}, we invoke the Faà di Bruno formula to compute 
    \begin{align*}
        \lvert D^{j}\upPhi(x) \rvert 
        &\leq 
        \C\sum_{i=1}^{j}\sum_{\substack{1\leq t_{1} \leq \cdots \leq t_{i} \\ t_{1}+\cdots+t_{i}=j}}
            \lvert D^{i}\upTheta(x) \rvert \lvert D^{t_{1}}\upPsi(x)\rvert \cdots \lvert D^{t_{i}}\upPsi(x)\rvert \\
        &\leq 
        \C\sum_{i=1}^{j}\sum_{\substack{1\leq t_{1} \leq \cdots \leq t_{i} \\ t_{1}+\cdots+t_{i}=j}}
            \eta^{1-i}\frac{\eta}{\zeta^{t_{1}}(x)}\cdots\frac{\eta}{\zeta^{t_{i}}(x)}
        \leq
        \C\frac{\eta}{\zeta^{j}(x)}.
    \end{align*}
    This concludes the proof of the proposition.
    \resetconstant
\end{proof}

Now that we have the building block \( \upPhi \), we move to the Sobolev estimates satisfied by the composition \( u \circ \upPhi \).

\begin{prop}
\label{prop:block_thickening_analytic}
    Let \( d > sp \).
    Let \( \upPhi \) be as in Proposition~\ref{prop:block_thickening_geometric}.
    Let \( \omega \subset \R^{m} \) be such that \( Q_{3} \subset \omega \subset B^{m}_{c\eta} \) for some \( c > 0 \), and assume that there exists \( c' > 0 \) such that 
    \begin{equation}
    \label{eq:geometric_assumptions_block_thickening}
    	 \lvert B^{m}_{\lambda}(z) \cap \omega \rvert \geq c'\lambda^{m}
    	 \quad
    	 \text{for every \( z \in \omega \) and \( 0 < \lambda \leq \frac{1}{2}\diam\omega \).}
    \end{equation}
    For every \( u \in W^{s,p}(\omega;\R^{\nu}) \), we have \( u \circ \upPhi \in W^{s,p}(\omega;\R^{\nu}) \), and moreover, the following estimates hold:
    \begin{enumerate}[label=(\alph*)]
        \item\label{item:block_thickening_geometric_sle1} if \( 0 < s < 1 \), then
	        \[
		        \lvert u \circ \upPhi \rvert_{W^{s,p}(\omega)}
		        \leq 
		        C\lvert u \rvert_{W^{s,p}(\omega)};
	        \]
        \item\label{item:block_thickening_geometric_sgeq1_integer} if \( s \geq 1 \), then for every \( j \in \{1,\dots,k\} \),
            \[
                \eta^{j}\lVert D^{j}(u \circ \upPhi) \rVert_{L^{p}(\omega)}
                \leq
                C\sum_{i=1}^{j}\eta^{i}\lVert D^{i}u \rVert_{L^{p}(\omega)};
            \]
        \item\label{item:block_thickening_geometric_sgeq1_frac} if \( s \geq 1 \) and \( \sigma \neq 0 \), then for every \( j \in \{1,\dots,k\} \),
            \[
                \eta^{j+\sigma}\lvert D^{j}(u \circ \upPhi) \rvert_{W^{\sigma,p}(\omega)}
                \leq
                C\sum_{i=1}^{j}\Bigl(\eta^{i}\lVert D^{i}u \rVert_{L^{p}(\omega)} 
                    + \eta^{i+\sigma}\lvert D^{i}u \rvert_{W^{\sigma,p}(\omega)}\Bigr);
            \]
        \item\label{item:block_thickening_geometric_all} for every \( 0 < s < +\infty \),
            \[
                \lVert u \circ \upPhi \rVert_{L^{p}(\omega)}
                \leq 
                C\lVert u \rVert_{L^{p}(\omega)};
            \]
    \end{enumerate}
    for some constant \( C > 0 \) depending on \( s \), \( m \), \( p \), \( c \), \( c' \), \( \ulrho \), \( \rho \), and \( \olrho \).
\end{prop}

We comment on the assumptions~\eqref{eq:geometric_assumptions_block_thickening} on \( \omega \).
In this section, \( \omega \) will be a rectangle whose sidelengths are constant multiples of \( \eta \).
However, in Section~\ref{sect:shrinking}, we will use Proposition~\ref{prop:block_shrinking_analytic}, which is very similar to Proposition~\ref{prop:block_thickening_analytic}, with a more complicated \( \omega \).
This is why we stated Proposition~\ref{prop:block_thickening_analytic} in a rather general form.

The assumption that \( \omega \) is contained in some ball having radius of order \( \eta \) is purely technical.
It ensures that estimate~\ref{item:block_thickening_geometric_singularity} of Proposition~\ref{prop:block_thickening_geometric} still applies for \( x \in \omega \setminus T \), possibly increasing the constant.
Indeed, since \( \Supp\upPhi \subset Q_{3} \), this estimate clearly still applies when \( x \notin Q_{3} \), but the constant deteriorates as \( \lvert x' \rvert \to +\infty \), since then \( \upPhi = \id \) while \( \zeta(x) \to +\infty \).
We could bypass this restriction, but it would not be useful since anyway we intend to apply Proposition~\ref{prop:block_thickening_analytic} to domains \( \omega \) satisfying this requirement.

To prove Proposition~\ref{prop:block_thickening_analytic}, we need a technical lemma.
As it was the case for opening, in the proof of the fractional Sobolev estimates, we will need to estimate terms of the form \( \lvert D^{t}\upPhi(x)-D^{t}\upPhi(y) \rvert \).
However, unlike for the opening, we cannot upper bound such terms by a simple application of the mean value theorem along a segment connecting \( x \) and \( y \), since such a segment 
could potentially get very close to -- or even cross -- the dual skeleton \( T \) where \( \upPhi \) is singular.
The following lemma provides us with a suitable path along which to apply the mean value theorem.

\begin{lemme}
\label{lemma:path_mean_value}
    For every \( x \), \( y \in \R^{m} \setminus T \), there exists a Lipschitz path \( \gamma \colon [0,1] \to \R^{m} \setminus T \) from \( x \) to \( y \) such that 
    \[ 
        \lvert \gamma \rvert_{\Cc^{0,1}([0,1])} \leq C\lvert x-y \rvert
    \] 
    for some constant \( C > 0 \) depending only on \( m \), and such that \( \zeta \geq \min(\zeta(x),\zeta(y)) \) along \( \gamma \), where \( \zeta \) is the map defined in~\eqref{eq:definition_zeta_thickening}.
\end{lemme}
\begin{proof}
    We recall the well-known fact that, given \( x \), \( y \) on a sphere, there exists a Lipschitz path on this sphere connecting those two points,
    with Lipschitz constant less that \( \C\lvert x-y \rvert \).
    Indeed, it suffices to take the shortest arc of great circle joining \( x \) to \( y \).
    The same fact holds for any \( q \)-sphere with \( 1 \leq q \leq +\infty \).
    This can be deduced from the Euclidean case using the \emph{changing norm projection} defined by \( x \mapsto \frac{\lvert x \rvert_{q}}{\lvert x \rvert_{2}}x \),
    which is a Lipschitz map.

    The desired path is then obtained as follows.
    If \( x = (x',x'') \) and \( y = (y',y'') \), we first go from \( x \) to \( (y',x'') \) by following successively an arc of great circle and a straight line in the first \( d \) components, while keeping the \( m-d \) last components fixed.
    Then we go from \( (y',x'') \) to \( y \) by following a path on a \( q \)-sphere as above, where \( q \) is the parameter used in the definition of \( \zeta \), followed by a straight line in the \( m-d \) last components, while keeping the first \( d \) components fixed.
    By construction, using the observations above, this path has Lipschitz constant less than \( \C\lvert x-y \rvert \).
    Moreover, since \( \zeta \) only depends on the \( 2 \)-norm of the \( d \) first components and on the \( q \)-norm of the \( m-d \) last components and is increasing with respect to both these parameters,
    we conclude that the constructed path has all the expected properties.
    \resetconstant
\end{proof}

We may now prove Proposition~\ref{prop:block_thickening_analytic}.

\begin{proof}[Proof of Proposition~\ref{prop:block_thickening_analytic}]
    The integer order estimates were obtained in~\cite{BPVS_density_higher_order}*{Corollary~4.2}.
    Since the proof in the fractional case relies, in part, on the calculations in the integer case, we reproduce here, for the convenience of the reader, the proof in~\cite{BPVS_density_higher_order}.
    When \( s \geq 1 \), we have \( d \geq sp \geq 1 \), and hence the dimension of \( T \) is less than \( m-d-1 \leq m-2 \).
 	Therefore, in order to prove that \( u \circ \upPhi \in W^{k,p}(\omega;\R^{\nu}) \), it suffices to prove that 
 	\[
 		\int_{\omega \setminus T} \lvert D^{j}(u \circ \upPhi) \rvert^{p}
 		< 
 		+\infty
 		\quad
 		\text{for every \( j \in \{0,\dots,k\} \).}
 	\]
    By the Faà di Bruno formula, we estimate for every \( j \in \{1,\dots,k\} \) and \( x \in \omega \setminus T \)
    \[	
        \lvert D^{j}(u \circ \upPhi)(x) \rvert^{p}
        \leq
        \C\sum_{i=1}^{j}\sum_{\substack{1\leq t_{1}\leq \cdots \leq t_{i} \\ t_{1} + \cdots + t_{i} = j}} \lvert D^{i}u(\upPhi(x)) \rvert^{p} \lvert D^{t_{1}}\upPhi(x) \rvert^{p} \cdots \lvert D^{t_{i}}\upPhi(x) \rvert^{p}.
    \]
    Let \( 0 < \beta < d \).
    Using the estimates on the derivatives and the Jacobian of \( \upPhi \), we find
    \[ 
    	\lvert D^{t_{l}}\upPhi \rvert \leq \C\frac{(\jac\upPhi)^{\frac{t_{l}}{\beta}}}{\eta^{t_{l}-1}},
    \] 
    and therefore 
    \begin{align*}
        \lvert D^{j}(u \circ \upPhi)(x) \rvert^{p}
        &\leq
        \C\sum_{i=1}^{j}\sum_{\substack{1\leq t_{1}\leq \cdots \leq t_{i} \\ t_{1} + \cdots + t_{i} = j}}
        \lvert D^{i}u(\upPhi(x)) \rvert^{p}\frac{(\jac\upPhi(x))^{\frac{t_{1}p}{\beta}}}{\eta^{(t_{1}-1)p}}\cdots\frac{(\jac\upPhi(x))^{\frac{t_{i}p}{\beta}}}{\eta^{(t_{i}-1)p}} \\
        &\leq \Cl{cst:fixed_block_thickening}\sum_{i=1}^{j}\lvert D^{i}u(\upPhi(x)) \rvert^{p}\frac{(\jac\upPhi(x))^{\frac{jp}{\beta}}}{\eta^{(j-i)p}}.
    \end{align*}
    Since \( jp \leq sp < d \), we may choose \( \beta = jp \).
    Hence,
    \[
        \lvert D^{j}(u \circ \upPhi)(x) \rvert^{p}
        \leq
        \Cr{cst:fixed_block_thickening}\sum_{i=1}^{j}\lvert D^{i}u(\upPhi(x)) \rvert^{p}\frac{\jac\upPhi(x)}{\eta^{(j-i)p}}.
    \]
    Since \( \upPhi \) is injective and \( \Supp\upPhi \subset Q_{3} \), 
    we have \( \upPhi(\omega\setminus T) \subset \omega \).
    Hence, the change of variable theorem ensures that
    \[
        \int_{\omega\setminus T} \eta^{jp}\lvert D^{j}(u \circ \upPhi)\rvert^{p}
        \leq
        \int_{\omega\setminus T}\Cr{cst:fixed_block_thickening}\sum_{i=1}^{j}\eta^{ip}\lvert D^{i}u(\upPhi(x)) \rvert^{p}\jac\upPhi(x)\,\d x 
        \leq
        \Cr{cst:fixed_block_thickening}\sum_{i=1}^{j}\int_{\omega} \eta^{ip}\lvert D^{i}u \rvert^{p}.
    \]
    
    The proof of the zero order estimate (valid in the full range \( 0 < s < +\infty \)) is straightforward using the same change of variable, noting that in particular, \( \jac \upPhi \geq \C > 0 \).
    In particular, we have \( u \circ \upPhi \in W^{k,p}(\omega;\R^{\nu}) \).

    We now turn to the proof of the fractional estimate in the case \( 0 < s < 1 \).
    
    \emph{Step~1: Mean value-type estimate.}
    We prove that, for every \( x \), \( y \in \omega \setminus T \),
    \begin{equation}
    \label{eq:mean_value_thickening}
	    \frac{\lvert \upPhi(x) - \upPhi(y) \rvert}{\lvert x-y \rvert}
	    \leq
	    \C\frac{\eta}{\zeta(y)}.
    \end{equation}
    It suffices to consider the case when \( \zeta(x) \leq \zeta(y) \).
    First assume that \( \zeta(y) \leq 2\zeta(x) \).
    In this case, we use the mean value theorem with the path \( \gamma \) provided by Lemma~\ref{lemma:path_mean_value} along with the estimate satisfied by \( D\upPhi \) to write
    \[
        \frac{\lvert \upPhi(x) - \upPhi(y) \rvert}{\lvert x-y \rvert}
        \leq
        \C\frac{\eta}{\zeta(x)}
        \leq
        \C\frac{\eta}{\zeta(y)}.
    \]
    We consider now the case where \( 2\zeta(x) \leq \zeta(y) \).
    We observe that we have \( \zeta(y) - \zeta(x) \leq \Cl{cst:lipschitz_zeta_eta}\lvert x-y \rvert \) -- this can be seen as a consequence of the triangle inequality for the Euclidean norm.
    Hence,
    \[
        \Cr{cst:lipschitz_zeta_eta}\lvert x-y \rvert \geq \zeta(y) - \zeta(x) \geq \frac{1}{2}\zeta(y).
    \]
    On the other hand, since \( \omega \subset B^{m}_{c\eta} \), we have \( \lvert \upPhi(x) - \upPhi(y) \rvert \leq \C\eta \).
    This concludes the proof of~\eqref{eq:mean_value_thickening}.

	\emph{Step~2: Averaging.}
    We write
    \[
        \iint\limits_{\substack{(\omega\setminus T) \times (\omega\setminus T) \\ \zeta(x) \leq \zeta(y)}}
            \frac{\lvert u\circ\upPhi(x) - u\circ\upPhi(y) \rvert^{p}}{\lvert x-y \rvert^{m+s p}} \,\d x\d y 
        \leq
        \C\int_{\omega\setminus T}\int_{\omega\setminus T}\fint_{\Bc_{x,y}}
            \frac{\lvert u\circ\upPhi(x) - u(z) \rvert^{p}}{\lvert x-y \rvert^{m+sp}} \,\d z\d x\d y,
    \]
    where we have defined 
    \[
            \Bc_{x,y} = B^{m}_{\lvert \upPhi(x)-\upPhi(y) \rvert}\biggl(\frac{\upPhi(x)+\upPhi(y)}{2}\biggr) \cap \omega.
    \]
    We observe that
    \[  
        B^{m}_{\frac{\lvert \upPhi(x)-\upPhi(y) \rvert}{2}}(\upPhi(x)) \cap \omega
        \subset 
        \Bc_{x,y}.
    \]
    Therefore, since \( \frac{\lvert \upPhi(x)-\upPhi(y) \rvert}{2} \leq \frac{1}{2}\diam\omega \), we find 
    \[
        \lvert \Bc_{x,y} \rvert \geq \C\lvert \upPhi(x)-\upPhi(y) \rvert^{m}.
    \]
    Here, we have used the volume assumption~\eqref{eq:geometric_assumptions_block_thickening}.
    Hence,
    \begin{align*}
        &\int_{\omega\setminus T}\int_{\omega\setminus T}\fint_{\Bc_{x,y}}
            \frac{\lvert u\circ\upPhi(x) - u(z) \rvert^{p}}{\lvert x-y \rvert^{m+sp}} \,\d z\d x\d y \\
        &\quad\leq
        \C\int_{\omega\setminus T}\int_{\omega\setminus T}\int_{\Bc_{x,y}}
            \frac{\lvert u\circ\upPhi(x) - u(z) \rvert^{p}}{\lvert \upPhi(x)-\upPhi(y) \rvert^{m}\lvert x-y \rvert^{m+sp}} \,\d z\d x\d y \\
        &\quad\leq
        \C\int_{\omega\setminus T}\int_{\omega\setminus T}\int_{\Bc_{x,y}}
            \frac{\lvert u\circ\upPhi(x) - u(z) \rvert^{p}}{\lvert \upPhi(x)-z \rvert^{m}\lvert x-y \rvert^{m+sp}} \,\d z\d x\d y,
    \end{align*}
    where we made use of the fact that \( \lvert \upPhi(x)-z \rvert \leq \frac{3}{2}\lvert \upPhi(x)-\upPhi(y) \rvert \) whenever \( z \in \Bc_{x,y} \).
    Invoking Tonelli's theorem, we find
    \begin{multline*}
        \int_{\omega\setminus T}\int_{\omega\setminus T}\int_{\Bc_{x,y}}
            \frac{\lvert u\circ\upPhi(x) - u(z) \rvert^{p}}{\lvert \upPhi(x)-z \rvert^{m}\lvert x-y \rvert^{m+sp}} \,\d z\d x\d y \\
        =
        \int_{\omega\setminus T}\int_{\omega\setminus T}\int_{\Yc_{x,z}}
            \frac{\lvert u\circ\upPhi(x) - u(z) \rvert^{p}}{\lvert \upPhi(x)-z \rvert^{m}\lvert x-y \rvert^{m+sp}} \,\d y\d z\d x,
    \end{multline*}
    where \( \Yc_{x,z} \) is the set of all those \( y \in \omega\setminus T \) such that \( z \in \Bc_{x,y} \), that is,
    \[
        \Yc_{x,z} = \{y \in \omega\setminus T \mathpunct{:} \lvert \upPhi(x) + \upPhi(y) - 2z \rvert \leq 2\lvert \upPhi(x)-\upPhi(y) \rvert\}.
    \]
    Since \( \lvert \upPhi(x) + \upPhi(y) - 2z \rvert \geq 2\lvert \upPhi(x)-z \rvert - \lvert \upPhi(x)-\upPhi(y) \rvert \), we find, using~\eqref{eq:mean_value_thickening},
    \[
        \Yc_{x,z}
        \subset
        \Bigl\{y \in \R^{m} \mathpunct{:} \lvert \upPhi(x)-z \rvert \leq \frac{3}{2}\lvert \upPhi(x)-\upPhi(y) \rvert\Bigr\}
        \subset
        \Bigl\{y \in \R^{m} \mathpunct{:} \lvert \upPhi(x)-z \rvert \leq \C\frac{\eta}{\zeta(x)}\lvert x-y \rvert\Bigr\}.
    \]
    Therefore,
    \[
        \int_{\Yc_{x,z}}\frac{1}{\lvert x-y \rvert^{m+sp}}\,\d y
        \leq
        \C\frac{\eta^{sp}}{\zeta(x)^{sp}\lvert \upPhi(x)-z \rvert^{sp}}.
    \]
    We conclude that
    \[
        \iint\limits_{\substack{(\omega\setminus T) \times (\omega\setminus T) \\ \zeta(x) \leq \zeta(y)}}
            \frac{\lvert u\circ\upPhi(x) - u\circ\upPhi(y) \rvert^{p}}{\lvert x-y \rvert^{m+sp}} \,\d x\d y 
        \leq
        \C\int_{\omega\setminus T}\int_{\omega}
            \frac{\lvert u\circ\upPhi(x) - u(z) \rvert^{p}}{\lvert \upPhi(x)-z \rvert^{m+sp}}\frac{\eta^{sp}}{\zeta(x)^{sp}} \,\d z\d x.
    \]
    
    \emph{Step~3: Change of variable.}
    Since \( 0 < sp < d \), we may apply estimate~\ref{item:block_thickening_geometric_jacobian} of Proposition~\ref{prop:block_thickening_geometric} with \( \beta = sp \).
    Taking into account the fact that \( \upPhi \) is injective and \( \upPhi(\omega \setminus T) \subset \omega \), we deduce from the change of variable theorem that 
    \begin{align*}
        \lvert u \circ \upPhi \rvert_{W^{s,p}(\omega)}^{p}
        &\leq 
        \Cl{cst:cst_fixed_block_thickening_sle1}\int_{\omega\setminus T}\int_{\omega}
            \frac{\lvert u\circ\upPhi(x) - u(z) \rvert^{p}}{\lvert \upPhi(x)-z \rvert^{m+sp}}\jac\upPhi(x) \,\d z\d x \\
        &\leq 
        \Cr{cst:cst_fixed_block_thickening_sle1}\int_{\omega}\int_{\omega}
            \frac{\lvert u(y) - u(z) \rvert^{p}}{\lvert y-z \rvert^{m+sp}} \,\d z\d y.
    \end{align*}
    This concludes the proof in the case \( 0 < s < 1 \).

    We finish with the fractional estimate in the case \( s \geq 1 \).
    
    \emph{Step~1: Estimate of \( \lvert D^{j}u(x)-D^{j}u(y) \rvert \).}
    Consider \( x \), \( y \in \omega \setminus T \) such that, without loss of generality, \( \zeta(x) \leq \zeta(y) \).
    As in the previous sections, using the Faà di Bruno formula, the multilinearity of the differential, and the estimates on the derivatives of \( \upPhi \), we write
    \begin{equation}
    \label{eq:Faa_di_bruno_block_thickening_sgeq1}
   	\begin{aligned}
        &\lvert D^{j}(u \circ \upPhi)(x) - D^{j}(u \circ \upPhi)(y) \rvert \\
        &\quad\leq
        \C\sum_{i=1}^{j}\Bigl(\lvert D^{i}u\circ\upPhi(x) - D^{i}u\circ\upPhi(y) \rvert \frac{\eta^{i}}{\zeta(y)^{j}} \\
        &\quad\quad+ \sum_{t=1}^{j} \lvert D^{i}u\circ\upPhi(x) \rvert \lvert D^{t}\upPhi(x) - D^{t}\upPhi(y) \rvert \frac{\eta^{i-1}}{\zeta(x)^{j-t}}\Bigr).
    \end{aligned}
    \end{equation}
    
    \emph{Step~2: Estimate of the second term in~\eqref{eq:Faa_di_bruno_block_thickening_sgeq1}.}
    We proceed as we did in the proofs of Propositions~\ref{prop:block_opening} and~\ref{prop:adaptive_smoothing}, relying on an optimization argument. 
    We split the integral over \( B^{m}_{r}(x) \) and \( \R^{m}\setminus B^{m}_{r}(x) \) and we insert \( r = \zeta(x) \) to arrive at
    \[
        \int\limits_{\substack{\omega\setminus T \\ \zeta(x) \leq \zeta(y)}} 
            \frac{\lvert D^{t}\upPhi(x) - D^{t}\upPhi(y) \rvert^{p}}{\lvert x-y \rvert^{m+\sigma p}}\,\d y
        \leq
        \C\frac{\eta^{p}}{\zeta(x)^{(t+\sigma)p}}.
    \]
    Hence,
    \begin{multline*}
        \iint\limits_{\substack{(\omega\setminus T) \times (\omega\setminus T) \\ \zeta(x) \leq \zeta(y)}} 
        \frac{\lvert D^{i}u\circ\upPhi(x) \rvert^{p} \lvert D^{t}\upPhi(x) - D^{t}\upPhi(y) \rvert^{p} }{\lvert x-y \rvert^{m+\sigma p}}\frac{\eta^{(i-1)p}}{\zeta(x)^{(j-t)p}}\,\d x\d y \\
        \leq
        \C\int_{\omega\setminus T}\lvert D^{i}u\circ\upPhi(x) \rvert^{p}\frac{\eta^{ip}}{\zeta(x)^{(j+\sigma)p}}\,\d x.
    \end{multline*}
    Now, by the estimate satisfied by \( \jac\upPhi \), we have, for \( 0 < \beta < d \), 
    \[
        \int_{\omega\setminus T}\lvert D^{i}u\circ\upPhi(x) \rvert^{p}\frac{\eta^{ip}}{\zeta(x)^{(j+\sigma)p}}\,\d x 
        \leq
        \C\eta^{ip-(j+\sigma)p}\int_{\omega\setminus T}\lvert D^{i}u\circ\upPhi(x) \rvert^{p}(\jac\upTheta(x))^{\frac{(j+\sigma)p}{\beta}}\,\d x.
    \]
    Since \( d > sp \geq (j+\sigma)p \), we may choose \( \beta = (j+\sigma)p \).
    We conclude by using the change of variable theorem that
    \[
        \int_{\omega \setminus T}\lvert D^{i}u\circ\upPhi(x) \rvert^{p}\frac{\eta^{ip}}{\zeta(x)^{(j+\sigma)p}}\,\d x \\
        \leq
        \C\eta^{(i-j-\sigma)p}\int_{\omega} \lvert D^{i}u \rvert^{p}.
    \]

    \emph{Step~3: Estimate of the first term in~\eqref{eq:Faa_di_bruno_block_thickening_sgeq1}: averaging.}
    We use the same methodology as for the case \( 0 < s < 1 \).
    Hence, we only write the main steps of the reasoning.
    We write 
    \begin{align*}
        &\iint\limits_{\substack{(\omega\setminus T)\times(\omega\setminus T) \\ \zeta(x) \leq \zeta(y)}} 
            \frac{\lvert D^{i}u\circ\upPhi(x) - D^{i}u\circ\upPhi(y) \rvert^{p}}{\lvert x-y \rvert^{m+\sigma p}}\frac{\eta^{ip}}{\zeta(y)^{jp}}\,\d x\d y \\
        &\quad\leq 
        \iint\limits_{\substack{(\omega\setminus T)\times(\omega\setminus T) \\ \zeta(x) \leq \zeta(y)}}\fint_{\Bc_{x,y}} 
            \frac{\lvert D^{i}u\circ\upPhi(x) - D^{i}u(z) \rvert^{p}}{\lvert x-y \rvert^{m+\sigma p}}\frac{\eta^{ip}}{\zeta(y)^{jp}}\,\d z\d x\d y \\
        &\quad\quad+
            \iint\limits_{\substack{(\omega\setminus T)\times(\omega\setminus T) \\ \zeta(x) \leq \zeta(y)}}\fint_{\Bc_{x,y}}
                \frac{\lvert D^{i}u(z) - D^{i}u\circ\upPhi(y) \rvert^{p}}{\lvert x-y \rvert^{m+\sigma p}}\frac{\eta^{ip}}{\zeta(y)^{jp}}\,\d z\d x\d y \\
        &\quad\leq 
        \int_{\omega \setminus T}\int_{\omega \setminus T}\fint_{\Bc_{x,y}}
            \frac{\lvert D^{i}u\circ\upPhi(x) - D^{i}u(z) \rvert^{p}}{\lvert x-y \rvert^{m+\sigma p}}\frac{\eta^{ip}}{\zeta(x)^{jp}}\,\d z\d x\d y.
    \end{align*}
    Observe that here it is important that we wrote estimate~\eqref{eq:Faa_di_bruno_block_thickening_sgeq1} with \( \frac{1}{\zeta(y)} \) on the first term in the right-hand side,
    so that we may further upper bound \( \frac{1}{\zeta(y)} \) by \( \frac{1}{\zeta(x)} \).
    We then pursue exactly as in the case \( 0 < s < 1 \).
    Using the volume assumption~\eqref{eq:geometric_assumptions_block_thickening}, we find
    \begin{multline*}
    	\int_{\omega \setminus T}\int_{\omega \setminus T}\fint_{\Bc_{x,y}}
    	\frac{\lvert D^{i}u\circ\upPhi(x) - D^{i}u(z) \rvert^{p}}{\lvert x-y \rvert^{m+\sigma p}}\frac{\eta^{ip}}{\zeta(x)^{jp}}\,\d z\d x\d y \\
	    \leq
	    \C\int_{\omega\setminus T}\int_{\omega\setminus T}\int_{\Bc_{x,y}}
	    \frac{\lvert D^{i}u\circ\upPhi(x) - D^{i}u(z) \rvert^{p}}{\lvert \upPhi(x)-z \rvert^{m}\lvert x-y \rvert^{m+\sigma p}}\frac{\eta^{ip}}{\zeta(x)^{jp}} \,\d z\d x\d y.
    \end{multline*}
    Relying on Tonelli's theorem, we deduce that 
    \begin{multline*}
	    \int_{\omega\setminus T}\int_{\omega\setminus T}\int_{\Bc_{x,y}}
	    \frac{\lvert D^{i}u\circ\upPhi(x) - D^{i}u(z) \rvert^{p}}{\lvert \upPhi(x)-z \rvert^{m}\lvert x-y \rvert^{m+\sigma p}}\frac{\eta^{ip}}{\zeta(x)^{jp}} \,\d z\d x\d y \\
	    =
	    \int_{\omega\setminus T}\int_{\omega\setminus T}\int_{\Yc_{x,z}}
	    \frac{\lvert D^{i}u\circ\upPhi(x) - D^{i}u(z) \rvert^{p}}{\lvert \upPhi(x)-z \rvert^{m}\lvert x-y \rvert^{m+\sigma p}}\frac{\eta^{ip}}{\zeta(x)^{jp}} \,\d y\d z\d x.
    \end{multline*}
    Using the inclusion
    \[
    	\Yc_{x,z} \subset \R^{m} \setminus B^{m}_{r}(x),
    \]
    where 
    \[
    r = r(x,z) = \frac{\C\lvert \upPhi(x)-z \rvert\zeta(x)}{\eta},
    \]
    we conclude that 
    \begin{multline*}
        \iint\limits_{\substack{(\omega\setminus T)\times(\omega\setminus T) \\ \zeta(x) \leq \zeta(y)}} 
            \frac{\lvert D^{i}u\circ\upPhi(x) - D^{i}u\circ\upPhi(y) \rvert^{p}}{\lvert x-y \rvert^{m+\sigma p}}\frac{\eta^{ip}}{\zeta(y)^{jp}}\,\d x\d y \\
        \leq 
        \C\int_{\omega \setminus T}\int_{\omega}\frac{\lvert D^{i}u \circ \upPhi(x)-D^{i}u(z) \rvert^{p}}{\lvert \upPhi(x)-z \rvert^{m+\sigma p}}\frac{\eta^{\sigma p}}{\zeta(x)^{\sigma p}}\frac{\eta^{ip}}{\zeta(x)^{jp}}\,\d z\d x.
    \end{multline*}
    
    \emph{Step~3: Estimate of the first term in~\eqref{eq:Faa_di_bruno_block_thickening_sgeq1}: change of variable.}
    As previously, we use estimate~\ref{item:block_thickening_geometric_jacobian} of Proposition~\ref{prop:block_thickening_geometric} with \( \beta = sp \) and the change of variable theorem to conclude that 
    \begin{multline*}
        \iint\limits_{\substack{(\omega\setminus T)\times(\omega\setminus T) \\ \zeta(x) \leq \zeta(y)}}
            \frac{\lvert D^{i}u\circ\upPhi(x) - D^{i}u\circ\upPhi(y) \rvert^{p}}{\lvert x-y \rvert^{m+\sigma p}}\frac{\eta^{ip}}{\zeta(y)^{jp}}\,\d x\d y \\
        \leq 
        \C\eta^{(i-j)p}\int_{\omega}\int_{\omega}\frac{\lvert D^{i}u(y)-D^{i}u(z) \rvert^{p}}{\lvert x-y \rvert^{m+\sigma p}}\,\d z\d y.
    \end{multline*}
    Gathering the estimates for both terms in~\eqref{eq:Faa_di_bruno_block_thickening_sgeq1} we obtain the desired conclusion, hence finishing the proof of the proposition.
    \resetconstant
\end{proof}

Now that we have constructed the building block for the thickening procedure, we are ready to proceed with the proof of Proposition~\ref{prop:main_thickening}.
We start by presenting an informal explanation of the construction to clarify the method.

We first apply thickening around the vertices of the dual skeleton \( \Tc^{\ell^{\ast}} \), which are actually the centers of the cubes in \( \Uc^{m} \), with parameters \( 0 < \rho_{m} < \tau_{m-1} < \rho_{m-1} \).
This maps the complement of the center of each cube on a neighborhood of the faces of the cube.
Then we apply thickening around the edges of the dual skeleton, which are segments of lines passing through the center of the \( (m-1) \)-faces of \( \Uc^{m} \), with parameters \( \rho_{m-1} < \tau_{m-2} < \rho_{m-2} \).
This maps the part of the complement of the edges of \( \Tc^{\ell^{\ast}} \) lying at distance at most \( \rho_{m-1} \) of the \( (m-1) \)-faces of \( \Uc^{m} \) on a neighborhood of the \( (m-2) \)-faces of \( \Uc^{m} \).
But since at the previous step the complement of the centers of the cubes was already mapped in a neighborhood of the faces of width \( \rho_{m-1} \), 
we deduce that the whole complement of the \( 1 \)-skeleton of \( \Tc^{\ell^{\ast}} \) is mapped on a neighborhood of \( \Uc^{m-2} \).
We pursue this procedure by induction until we reach dimension \( \ell^{\ast} \) with respect to the dual skeleton -- which corresponds to dimension \( \ell \) with respect to \( \Uc^{m} \) --
and this produces the required map \( \upPhi \).

Figures~\ref{fig:thickening_around_vertices},~\ref{fig:thickening_around_edges}, and~\ref{fig:final_thickening} provide an illustration of this procedure on one cube when \( m = 2 \) and \( \ell = 0 \).
This allows us to see the combination of two steps of the induction procedure.
Figure~\ref{fig:thickening_around_vertices} shows thickening around the vertices of the dual skeleton, which correspond to the centers of the cubes of \( \Uc^{m} \).
The values of \( u \) in the blue region on the left part of the figure are propagated into the blue region on the right part of the figure.
This creates a point singularity in the center of each cube, depicted in red.
Figure~\ref{fig:thickening_around_edges} illustrates thickening around the edges of the dual skeleton.
The values of \( u \) in the dark blue region on the left part of the figure are propagated into the dark blue region on the right part of the figure, which creates line singularities in red.
The map \( u \) is left unchanged on the white region, the part in light blue serving as a transition.
The boundaries of the regions in Figure~\ref{fig:thickening_around_vertices} are shown in light colors, to illustrate how all the different regions involved in the construction combine together.
The combination of both steps inside the square is shown in Figure~\ref{fig:final_thickening}.
The values in the blue regions on the corners are propagated inside of the whole square, which creates line singularities in red, forming a cross.

\begin{figure}[ht]
	~
	\hfill
	\includegraphics[page=4]{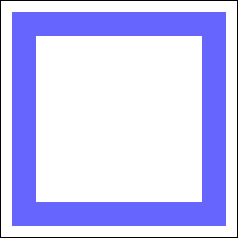}
	\hfill
	\includegraphics[page=5]{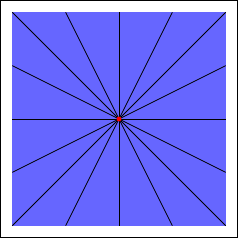}
	\hfill
	~
	\caption{Thickening around vertices}
	\label{fig:thickening_around_vertices}
\end{figure}

\begin{figure}[ht]
	~
	\hfill
	\includegraphics[page=6]{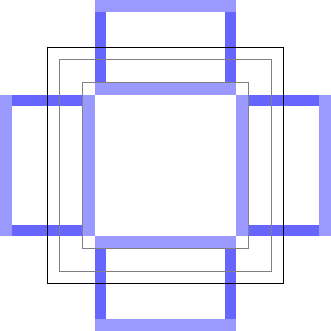}
	\hfill
	\includegraphics[page=7]{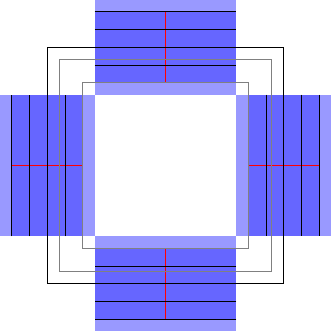}
	\hfill
	~
	\caption{Thickening around edges}
	\label{fig:thickening_around_edges}
\end{figure}

\begin{figure}[ht]
	~
	\hfill
	\includegraphics[page=8]{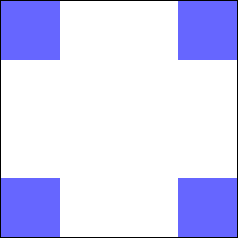}
	\hfill
	\includegraphics[page=9]{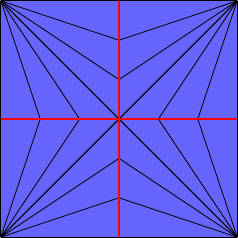}
	\hfill
	~
	\caption{Final thickening at order \( 1 \)}
	\label{fig:final_thickening}
\end{figure}

\begin{proof}[Proof of Proposition~\ref{prop:main_thickening}]
    The map \( \upPhi \) is constructed as follows.
    We first take finite sequences \( (\rho_{i})_{\ell \leq i \leq m} \) and \( (\tau_{i})_{\ell \leq i \leq m} \) such that
    \[
        0 < \rho_{m} < \tau_{m-1} < \rho_{m-1} < \cdots < \rho_{\ell+1} < \tau_{\ell} < \rho_{\ell} = \rho.
    \]
    The map \( \upPhi \) is defined by downward induction.
    For \( d = m \), we let \( \upPhi^{d} = \id \).
    Then, if \( d \in \{\ell+1,\dots,m\} \), given \( \sigma^{d} \in \Uc^{d} \), we identify \( \sigma^{d} \) with \( Q^{d}_{\eta} \times \{0\}^{m-d} \) and \( T^{(d-1)^{\ast}} \cap (\sigma^{d}+Q^{m}_{\tau_{d-1}\eta}) \)
    with \( \{0\}^{d} \times Q^{m-d}_{\tau_{d-1}\eta} \)\,, and we let \( \upPhi_{\sigma^{d}} \) be the map given by~Proposition~\ref{prop:block_thickening_geometric} applied around \( \sigma^{d} \)
    with parameters \( \ulrho = \rho_{d} \), \( \rho = \tau_{d-1} \), and \( \olrho = \rho_{d-1} \).
    We let \( \upPsi^{d} \colon \R^{m} \setminus T^{(d-1)^{\ast}} \to \R^{m} \) be defined by
    \[
        \upPsi^{d}(x) = 
        \begin{cases}
            \upPhi_{\sigma^{d}}(x) & \text{if \( x \in T_{\sigma^{d}}(Q_{3}) \) for some \( \sigma^{d} \in \Uc^{d} \),} \\
            x & \text{otherwise,}
        \end{cases}
    \]
    where \( T_{\sigma^{d}} \) is an isometry mapping \( Q^{d}_{\eta} \times \{0\}^{m-d} \) on \( \sigma^{d} \).
    Finally, we define \( \upPhi^{d-1} = \upPsi^{d} \circ \upPhi^{d} \).
    The desired map is given by \( \upPhi = \upPhi^{\ell} \).
    
    As we mentioned, properties~\ref{item:main_thickining_injective} to~\ref{item:main_thickening_singularity} are already contained in~\cite{BPVS_density_higher_order}*{Proposition~4.1}, so that we only need to prove estimates~\ref{item:main_thickening_estimate_sle1} to~\ref{item:main_thickening_estimate_all}.
    We first prove estimates with \( \upPhi \) replaced by \( \upPsi^{d} \) for every \( d \in \{\ell+1,\dots,m\} \).
    We let \( \omega = Q^{d}_{(1-\rho_{m})\eta} \times Q^{m-d}_{\rho\eta} \).
    We note that \( \omega \) satisfies the assumptions of Proposition~\ref{prop:block_thickening_analytic}.
    In particular, we have \( Q_{3} \subset \omega \subset U^{m}+Q^{m}_{\rho\eta} \) for every \( d \in \{\ell+1,\dots,m\} \).
    We apply Proposition~\ref{prop:block_thickening_analytic} to find that, for every \( \sigma^{d} \in \Uc^{d} \), the following estimates hold:
    \begin{enumerate}[label=(\alph*)]
        \item if \( 0 < s < 1 \), then
	        \[
		        \lvert u \circ \upPhi_{\sigma^{d}} \rvert_{W^{s,p}(T_{\sigma^{d}}(\omega))}
		        \leq 
		        \C\lvert u \rvert_{W^{s,p}(T_{\sigma^{d}}(\omega))};
	        \]
        \item if \( s \geq 1 \), then for every \( j \in \{1,\dots,k\} \),
            \[
                \eta^{j}\lVert D^{j}(u \circ \upPhi_{\sigma^{d}}) \rVert_{L^{p}(T_{\sigma^{d}}(\omega))}
                \leq
                \C\sum_{i=1}^{j}\eta^{i}\lVert D^{i}u \rVert_{L^{p}(T_{\sigma^{d}}(\omega))};
            \]
        \item if \( s \geq 1 \) and \( \sigma \neq 0 \), then for every \( j \in \{1,\dots,k\} \),
            \[
                \eta^{j+\sigma}\lvert D^{j}(u \circ \upPhi_{\sigma^{d}}) \rvert_{W^{\sigma,p}(T_{\sigma^{d}}(\omega))}
                \leq
                \C\sum_{i=1}^{j}\Bigl(\eta^{i}\lVert D^{i}u \rVert_{L^{p}(T_{\sigma^{d}}(\omega))} \\
                    + \eta^{i+\sigma}\lvert D^{i}u \rvert_{W^{\sigma,p}(T_{\sigma^{d}}(\omega))}\Bigr);
            \]
        \item for every \( 0 < s < +\infty \),
            \[
                \lVert u \circ \upPhi_{\sigma^{d}} \rVert_{L^{p}(T_{\sigma^{d}}(\omega))}
                \leq 
                \C\lVert u \rVert_{L^{p}(T_{\sigma^{d}}(\omega))}.
            \]
    \end{enumerate}
    
    Using the additivity of the integral for integer order estimates and Lemma~\ref{lemma:fractional_additivity} for fractional order estimates, we find that
    \begin{enumerate}[label=(\alph*)]
        \item if \( 0 < s < 1 \), then
	        \begin{multline*}
		        \lvert u \circ \upPsi^{d} \rvert_{W^{s,p}(U^{m}+Q^{m}_{\rho\eta})}^{p}
		        \leq 
		        \C\sum_{\sigma^{d} \in \Uc^{d}} \lvert u \circ \upPhi_{\sigma^{d}} \rvert_{W^{s,p}(T_{\sigma^{d}}(\omega))}^{p} \\
		        + \C\lvert u \circ \upPsi^{d} \rvert_{W^{s,p}((U^{m}+Q^{m}_{\rho\eta}) \setminus \Supp\upPsi^{d})}^{p} 
		        + \C\eta^{-sp}\lVert u \circ \upPsi^{d} \rVert_{L^{p}(U^{m}+Q^{m}_{\rho\eta})}^{p};
	        \end{multline*}
        \item if \( s \geq 1 \), then for every \( j \in \{1,\dots,k\} \),
            \begin{multline*}
                \lVert D^{j}(u \circ \upPsi^{d}) \rVert_{L^{p}(U^{m}+Q^{m}_{\rho\eta})}^{p}
                \leq 
                \C\sum_{\sigma^{d} \in \Uc^{d}} \lVert D^{j}(u \circ \upPhi_{\sigma^{d}}) \rVert_{L^{p}(T_{\sigma^{d}}(\omega))}^{p} \\
                    + \C\lVert D^{j}(u \circ \upPsi^{d}) \rVert_{L^{p}((U^{m}+Q^{m}_{\rho\eta}) \setminus \Supp\upPsi^{d})}^{p};
            \end{multline*}
        \item if \( s \geq 1 \) and \( \sigma \neq 0 \), then for every \( j \in \{1,\dots,k\} \),
            \begin{multline*}
                \lvert D^{j}(u \circ \upPsi^{d}) \rvert_{W^{\sigma,p}(U^{m}+Q^{m}_{\rho\eta})}^{p}
                \leq 
                \C\sum_{\sigma^{d} \in \Uc^{d}} \lvert D^{j}(u \circ \upPhi_{\sigma^{d}}) \rvert_{W^{\sigma,p}(T_{\sigma^{d}}(\omega))}^{p} \\
                    + \C\lvert D^{j}(u \circ \upPsi^{d}) \rvert_{W^{\sigma,p}((U^{m}+Q^{m}_{\rho\eta}) \setminus \Supp\upPsi^{d})}^{p} 
                    + \C\eta^{-\sigma p}\lVert D^{j}(u \circ \upPsi^{d}) \rVert_{L^{p}(U^{m}+Q^{m}_{\rho\eta})}^{p};
            \end{multline*}
        \item for every \( 0 < s < +\infty \),
            \[
                \lVert u \circ \upPsi^{d} \rVert_{L^{p}(U^{m}+Q^{m}_{\rho\eta})}^{p}
                \leq 
                \C\sum_{\sigma^{d} \in \Uc^{d}} \lVert u \circ \upPhi_{\sigma^{d}} \rVert_{L^{p}(T_{\sigma^{d}}(\omega))}^{p} 
                    + \C\lVert u \circ \upPsi^{d} \rVert_{L^{p}((U^{m}+Q^{m}_{\rho\eta}) \setminus \Supp\upPsi^{d})}^{p}.
            \]
    \end{enumerate}
    Combining both sets of estimates, by downward induction, we deduce that 
    \begin{enumerate}[label=(\alph*)]
        \item if \( 0 < s < 1 \), then
	        \[
		        \eta^{s}\lvert u \circ \upPhi \rvert_{W^{s,p}(U^{m}+Q^{m}_{\rho\eta})}
		        \leq 
		        \C\Bigl(\eta^{s}\lvert u \rvert_{W^{s,p}(U^{m}+Q^{m}_{\rho\eta})} + \lVert u \rVert_{L^{p}(U^{m}+Q^{m}_{\rho\eta})}\Bigr);
	        \]
        \item if \( s \geq 1 \), then for every \( j \in \{1,\dots,k\} \),
            \[
                \eta^{j}\lVert D^{j}(u \circ \upPhi) \rVert_{L^{p}(U^{m}+Q^{m}_{\rho\eta})}
                \leq
                \C\sum_{i=1}^{j}\eta^{i}\lVert D^{i}u \rVert_{L^{p}(U^{m}+Q^{m}_{\rho\eta})};
            \]
        \item if \( s \geq 1 \) and \( \sigma \neq 0 \), then for every \( j \in \{1,\dots,k\} \),
            \[
                \eta^{j+\sigma}\lvert D^{j}(u \circ \upPhi) \rvert_{W^{\sigma,p}(U^{m}+Q^{m}_{\rho\eta})}
                \leq
                \C\sum_{i=1}^{j}\Bigl(\eta^{i}\lVert D^{i}u \rVert_{L^{p}(U^{m}+Q^{m}_{\rho\eta})} + \eta^{i+\sigma}\lvert D^{i}u \rvert_{W^{\sigma,p}(U^{m}+Q^{m}_{\rho\eta})}\Bigr);
            \]
        \item for every \( 0 < s < +\infty \),
            \[
                \lVert u \circ \upPhi \rVert_{L^{p}(U^{m}+Q^{m}_{\rho\eta})}
                \leq 
                \C\lVert u \rVert_{L^{p}(U^{m}+Q^{m}_{\rho\eta})}.
            \]
    \end{enumerate}
    Conclusion follows by an additional application of the additivity of the integral or Lemma~\ref{lemma:fractional_additivity}, by noting that actually \( \Supp\upPhi \subset U^{m} + Q^{m}_{\tau_{\ell}\eta} \).
    \resetconstant
\end{proof}

We close this section with a discussion about how the thickening technique that we investigated inserts itself in the proof of Theorem~\ref{thm:density_class_R}.
At the end of Section~\ref{sect:adaptive_smoothing}, we obtained an estimate on 
\( \Dist_{F}{(u^{\sm}_{\eta}((K^{m}\setminus U^{m}_{\eta})\cup (U^{\ell}_{\eta}+Q^{m}_{\ulrho\eta}))} \), where we recall that \( u^{\sm}_{\eta} \) is the map obtained by successively opening and smoothing a map \( u \in W^{s,p}(\Omega;F) \), with \( F \subset \R^{\nu} \) being an arbitrary closed set.
Informally, we were able to control the distance between \( u^{\sm}_{\eta} \) and \( F \) except on the cubes in \( \Uc^{m}_{\eta} \),
far from the \( \ell \)-skeleton.
We apply now thickening to the map \( u^{\sm}_{\eta} \).
Let \( \upPhi^{\thck}_{\eta} \) be the map provided by Proposition~\ref{prop:main_thickening} applied to \( \Uc^{m}_{\eta} \) with \( \Kc^{m} = \Kc^{m}_{\eta} \)
and using parameter \( \ulrho \).
We set \( u^{\thck}_{\eta} = u^{\sm}_{\eta} \circ \upPhi^{\thck}_{\eta} \).
To have \( u \in W^{s,p} \) along with the estimates provided by Proposition~\ref{prop:main_thickening}, we need to take \( \ell + 1 > sp \).
Since we already required \( \ell \leq sp \) in Section~\ref{sect:adaptive_smoothing}, this invites us to work with \( \ell = [sp] \).

By inclusion~\ref{item:main_thickening_geometry} in Proposition~\ref{prop:main_thickening}, we have \( \upPhi^{\thck}_{\eta}(K^{m}_{\eta} \setminus (T^{\ell^{\ast}}_{\eta}\cup U^{m}_{\eta})) \subset K^{m}_{\eta} \setminus U^{m}_{\eta} \).
On the other hand, by inclusion~\ref{item:main_thickening_support} in Proposition~\ref{prop:main_thickening}, we have \( \upPhi^{\thck}_{\eta}(U^{m}_{\eta}\setminus T^{\ell^{\ast}}_{\eta}) \subset U^{\ell}_{\eta}+Q^{m}_{\ulrho\eta} \).
Therefore,
\[
    \upPhi^{\thck}_{\eta}(K^{m}_{\eta}\setminus T^{\ell^{\ast}}_{\eta})
    \subset 
    (K^{m}_{\eta}\setminus U^{m}_{\eta}) \cup (U^{\ell}_{\eta}+Q^{m}_{\ulrho\eta}).
\]

Combining this observation with estimate~\eqref{eq:main_estimate_dist_sm_F_sgeq1}, respectively~\eqref{eq:main_estimate_dist_sm_F_sle1}, we deduce that 
\begin{multline}
    \label{eq:main_estimate_dist_thck_F_sgeq1}
        \Dist_{F}{(u^{\thck}_{\eta}(K^{m}_{\eta}\setminus T^{\ell^{\ast}}_{\eta}))}
        \leq 
        \max\biggl\{\max_{\sigma^{m} \in \Kc^{m}_{\eta} \setminus \Ec^{m}_{\eta}} C\frac{1}{\eta^{\frac{m}{sp}-1}}\lVert Du \rVert_{L^{sp}(\sigma^{m}+Q^{m}_{2\rho\eta})}, \\
            \sup_{x \in U^{\ell}_{\eta}+Q^{m}_{\ulrho\eta}} C'\fint_{Q^{m}_{r}(x)}\fint_{Q^{m}_{r}(x)} \lvert u^{\op}_{\eta}(y)-u^{\op}_{\eta}(z)\rvert\,\d y\d z\biggr\}
    \end{multline}
if \( s \geq 1 \), respectively 
\begin{multline}
\label{eq:main_estimate_dist_thck_F_sle1}
    \Dist_{F}{(u^{\thck}_{\eta}(K^{m}_{\eta}\setminus T^{\ell^{\ast}}_{\eta}))}
    \leq 
    \max\biggl\{\max_{\sigma^{m} \in \Kc^{m}_{\eta} \setminus \Ec^{m}_{\eta}} C\frac{1}{\eta^{\frac{m}{p}-s}}\lvert u \rvert_{W^{s,p}(\sigma^{m}+Q^{m}_{2\rho\eta})}, \\
        \sup_{x \in U^{\ell}_{\eta}+Q^{m}_{\ulrho\eta}} C'\fint_{Q^{m}_{r}(x)}\fint_{Q^{m}_{r}(x)} \lvert u^{\op}_{\eta}(y)-u^{\op}_{\eta}(z)\rvert\,\d y\d z\biggr\}
\end{multline}
if \( 0 < s < 1 \).
Moreover, \( u^{\sm}_{\eta} \) being smooth, the map \( u^{\thck}_{\eta} \) is smooth on \( K^{m}_{\eta} \setminus T^{\ell^{\ast}}_{\eta} \).
To summarize, we have obtained a map \( u^{\thck}_{\eta} \) which is smooth on \( K^{m}_{\eta} \setminus T^{\ell^{\ast}}_{\eta} \), 
and whose distance from \( F \) is controlled on the whole \( K^{m}_{\eta} \setminus T^{\ell^{\ast}}_{\eta} \).

Now let us get back to the case we are interested in, that is, where \( F = \Nc \).
In this case, it is well-known that there exists \( \iota > 0 \) such that the nearest point projection \( \upPi \colon \Nc + B^{m}_{\iota} \to \Nc \) is well-defined and smooth.
The open set \( \Nc + B^{m}_{\iota} \) is called a \emph{tubular neighborhood of \( \Nc \)}.
Assume that the right-hand side of~\eqref{eq:main_estimate_dist_thck_F_sgeq1} or~\eqref{eq:main_estimate_dist_thck_F_sle1} is less than \( \iota \).
We note that this requires both to take \( r \) sufficiently small and to choose \( \Ec^{m}_{\eta} \) such that, for every \( \sigma^{m} \in \Kc^{m}_{\eta} \setminus \Ec^{m}_{\eta} \),
\begin{equation}
\label{eq:estimate_good_cubes_thickening}
    \lVert Du \rVert_{L^{sp}(\sigma^{m}+Q^{m}_{2\rho\eta})}
    \leq 
    \frac{\eta^{\frac{m}{sp}-1}}{C}\iota, 
    \quad 
    \text{respectively}
    \quad 
    \lvert u \rvert_{W^{s,p}(\sigma^{m}+Q^{m}_{2\rho\eta})}
    \leq 
    \frac{\eta^{\frac{m}{p}-s}}{C}\iota.
\end{equation}
Under this assumption, the map \( u_{\eta} = \upPi \circ u^{\thck}_{\eta} \) is well-defined and smooth on \( K^{m}_{\eta} \setminus T^{\ell^{\ast}}_{\eta} \),
and takes its values into \( \Nc \).

We next prove that the map \( u_{\eta} \) actually belongs to the class \( \Rc_{m-[sp]-1}(K^{m}_{\eta};\Nc) \).
This follows from property~\ref{item:main_thickening_singularity} in Proposition~\ref{prop:main_thickening}.
Indeed, since \( m-[sp]-1 = \ell^{\ast} \), the singular set of \( u^{\thck}_{\eta} \), and hence of \( u_{\eta}  \), is as in the definition of \( \Rc_{m-[sp]-1}(K^{m}_{\eta};\Nc) \).
Therefore, it only remains to prove the estimates on the derivatives of \( u_{\eta} \).
Since \( u^{\sm}_{\eta} \) and \( \upPi \) are smooth, we deduce from the Faà di Bruno formula that
\begin{align*}
	\lvert D^{j}u_{\eta} \rvert
	&\leq
	\C\sum_{i=1}^{j}\sum_{\substack{1 \leq t_{1} \leq \cdots \leq t_{i} \\ t_{1} + \cdots + t_{i} = j}}
	\lvert D^{i}(\upPi \circ u^{\sm}_{\eta}) \rvert \lvert D^{t_{1}}\upPhi^{\thck}_{\eta} \rvert \cdots \lvert D^{t_{i}}\upPhi^{\thck}_{\eta} \rvert \\
	&\leq
	\C\sum_{i=1}^{j}\sum_{\substack{1 \leq t_{1} \leq \cdots \leq t_{i} \\ t_{1} + \cdots + t_{i} = j}}
	\lvert D^{t_{1}}\upPhi^{\thck}_{\eta} \rvert \cdots \lvert D^{t_{i}}\upPhi^{\thck}_{\eta} \rvert.
\end{align*}
By property~\ref{item:main_thickening_singularity} in Proposition~\ref{prop:main_thickening}, we conclude that, for \( x \in (U^{m}_{\eta}+Q^{m}_{\ulrho\eta}) \setminus T^{\ell^{\ast}}_{\eta} \),
\begin{equation}
\label{eq:decay_derivative_thickened_map}
	\lvert D^{j}u_{\eta}(x) \rvert
	\leq
	\C\sum_{i=1}^{j}\sum_{\substack{1 \leq t_{1} \leq \cdots \leq t_{i} \\ t_{1} + \cdots + t_{i} = j}}
	\frac{\eta}{\dist(x,T^{\ell^{\ast}})^{t_{1}}} \cdots \frac{\eta}{\dist(x,T^{\ell^{\ast}})^{t_{i}}}
	\leq
	\C\frac{\eta^{i}}{\dist(x,T^{\ell^{\ast}})^{j}}.
\end{equation}

Combining~\eqref{eq:decay_derivative_thickened_map} with the fact that, clearly, \( u_{\eta} \) is smooth outside \( U^{m}_{\eta}+Q^{m}_{\ulrho\eta} \), we find that \( u_{\eta} \) belongs indeed to \( \Rc_{m-[sp]-1}(K^{m}_{\eta};\Nc) \).

\resetconstant

With all these observations and tools at our disposal, we are finally ready to proceed with the proof of Theorem~\ref{thm:density_class_R}.
It only remains to explain carefully how to implement the aforementioned steps and to check that the estimates obtained at each step combine to yield \( u_{\eta} \to u \) in \( W^{s,p} \) as \( \eta \to 0 \).

\section{\texorpdfstring{Density of the class \( \Rc \)}{Density of the class R}}
\label{sect:density_class_R}

This section is devoted to the proof of the density of the class \( \Rc_{m-[sp]-1}(\Omega;\Nc) \) in \( W^{s,p}(\Omega;\Nc) \).
For the sake of clarity, we start by proving the result when the domain \( \Omega \) is a cube, which is the case covered by Theorem~\ref{thm:density_class_R} stated in the introduction.
In a second step, we explain how to deal with more general domains.

As we explained, the major part of the work that remains to be done is to suitably estimate the \( W^{s,p} \) distance between the maps \( u_{\eta} \) and \( u \).
As the reader may have noticed, the Sobolev estimates obtained in Sections~\ref{sect:opening} to~\ref{sect:thickening} deteriorate as \( \eta \to 0 \).
For instance, the term involving the \( L^{p} \) norm of \( D^{i}u \) in the estimate of the \( j \)-order derivative blows up at rate \( \eta^{i-j} \).
As we shall see in the proof, this blow-up is compensated by the fact that the measure of the set \( U^{m}_{\eta} + Q^{m}_{2\rho\eta} \) decays sufficiently fast as \( \eta \to 0 \).
For the integer order terms, this is exploited by a combination of the Hölder and Gagliardo--Nirenberg inequalities.
The treatment of fractional order terms is more involved, and we bring ourselves back to the integer order setting with the help of the following lemma.

\begin{lemme}
\label{lemma:interpolation_estimate}
    Let \( \Omega \subset \R^{m} \) be a convex set and let \( \omega \subset \Omega \).
    For every \( q \), \( r \geq p \) and every \( u \in W^{1,1}_{\mathrm{loc}}(\Omega) \),
    \[
        \lvert u \rvert_{W^{\sigma,p}(\omega)}
        \leq 
        C\lvert \omega \rvert^{\frac{1}{p}-\frac{\sigma}{r}-\frac{1-\sigma}{q}}\lVert Du \rVert_{L^{r}(\Omega)}^{\sigma} \lVert u \rVert_{L^{q}(\omega)}^{1-\sigma}
    \]
    for some constant \( C > 0 \) depending only on \( m \).
\end{lemme}
\begin{proof}
    By density, we may assume that \( u \in \Cc^{\infty}(\Omega) \).
    We once again rely on an optimization technique.
    For every \( \rho > 0 \), we write 
    \[
        \lvert u \rvert_{W^{\sigma,p}(\omega)}^{p}
        \leq 
        \int_{\omega}\int_{\omega \setminus B^{m}_{\rho}(x)} \frac{\lvert u(x)-u(y) \rvert^{p}}{\lvert x-y \rvert^{m+\sigma p}}\,\d y\d x 
            + \int_{\omega}\int_{\omega \cap B^{m}_{\rho}(x)} \frac{\lvert u(x)-u(y) \rvert^{p}}{\lvert x-y \rvert^{m+\sigma p}}\,\d y\d x.
    \]
    The first term is readily estimated as 
    \[
        \int_{\omega}\int_{\omega \setminus B^{m}_{\rho}(x)} \frac{\lvert u(x)-u(y) \rvert^{p}}{\lvert x-y \rvert^{m+\sigma p}}\,\d y\d x
        \leq 
        \Cl{cst:fixed_cst_interpolation_estimate} \rho^{-\sigma p}\int_{\omega} \lvert u \rvert^{p}
        \leq 
        \Cr{cst:fixed_cst_interpolation_estimate} \rho^{-\sigma p}\lvert \omega \rvert^{1-\frac{p}{q}}\biggl(\int_{\omega} \lvert u \rvert^{q}\biggr)^{\frac{p}{q}}.
    \]

    For the second term, we start by using the mean value theorem along with Jensen's inequality to find 
    \[
        \int_{\omega}\int_{\omega \cap B^{m}_{\rho}(x)} \frac{\lvert u(x)-u(y) \rvert^{p}}{\lvert x-y \rvert^{m+\sigma p}}\,\d y\d x
        \leq 
        \int_{0}^{1}\int_{\omega}\int_{\omega \cap B^{m}_{\rho}(x)} \frac{\lvert Du(x+t(y-x)) \rvert^{p}}{\lvert x-y \rvert^{m+(\sigma-1)p}}\,\d y \d x \d t.
    \]
    Here we use the convexity of \( \Omega \) to ensure that \( x+t(y-x) \in \Omega \) for every \( x \), \( y \in \omega \) and \( t \in [0,1] \).
    We use the change of variable \( h = y-x \) -- so that \( h \in (\omega - \omega) \cap B^{m}_{\rho} \), where \( \omega - \omega \) is the set of all the differences of two points in \( \omega \) -- and Tonelli's theorem to deduce that 
    \[
        \int_{\omega}\int_{\omega \cap B^{m}_{\rho}(x)} \frac{\lvert u(x)-u(y) \rvert^{p}}{\lvert x-y \rvert^{m+\sigma p}}\,\d y\d x
        \leq 
        \int_{0}^{1} \int_{(\omega - \omega) \cap B^{m}_{\rho}} \int_{\omega \cap (\omega-h)} \frac{\lvert Du(x+th) \rvert^{p}}{\lvert h \rvert^{m+(\sigma-1)p}} \,\d x \d h \d t.
    \]
    By convexity of \( \Omega \), if \( x \in \omega \cap (\omega-h) \), we have \( x+th \in \Omega \) for every \( t \in [0,1] \).
    Moreover, the measure of the set \( (\omega \cap (\omega - h))+th \) is less than \( \lvert \omega \rvert \).
    Hence,
    \begin{multline*}
        \int_{\omega}\int_{\omega \cap B^{m}_{\rho}(x)} \frac{\lvert u(x)-u(y) \rvert^{p}}{\lvert x-y \rvert^{m+\sigma p}}\,\d y\d x \\
        \leq 
        \lvert \omega \rvert^{1-\frac{p}{r}}\int_{0}^{1} \int_{(\omega - \omega) \cap B^{m}_{\rho}}\frac{1}{\lvert h \rvert^{m+(\sigma-1)p}} \biggl(\int_{\Omega} \lvert Du(z) \rvert^{r} \,\d z\biggr)^{\frac{p}{r}}\,\d h \d t.
    \end{multline*}
    We conclude that 
    \[
        \int_{\omega}\int_{\omega \cap B^{m}_{\rho}(x)} \frac{\lvert u(x)-u(y) \rvert^{p}}{\lvert x-y \rvert^{m+\sigma p}}\,\d y\d x
        \leq 
        \C \rho^{(1-\sigma)p}\lvert \omega \rvert^{1-\frac{p}{r}}\biggr(\int_{\Omega}\lvert Du \rvert^{r}\biggr)^{\frac{p}{r}}.
    \]

    We may assume that \( Du \) does not vanish identically, otherwise there is nothing to prove.
    We insert 
    \[
        \rho = \lvert \omega \rvert^{\frac{1}{r}-\frac{1}{q}}\frac{\lVert u \rVert_{L^{q}(\omega)}}{\lVert Du \rVert_{L^{p}(\Omega)}},
    \]
    and we find 
    \[
        \lvert u \rvert_{W^{\sigma,p}(\omega)}
        \leq 
        \C \lvert \omega \rvert^{\frac{1}{p}-\frac{\sigma}{r}-\frac{1-\sigma}{q}}\lVert u \rVert_{L^{q}(\omega)}^{1-\sigma}\lVert Du \rVert_{L^{r}(\Omega)}^{\sigma}.
    \]
    The proof of the lemma is complete.
    \resetconstant
\end{proof}

We finally prove Theorem~\ref{thm:density_class_R}.
Recall that \( \Omega = Q^{m} \).
Note that, in Sections~\ref{sect:opening} to~\ref{sect:thickening}, no assumptions were required on \( \Omega \).
During the proof, we shall carefully indicate whenever restrictions on \( \Omega \) are needed.
Then, we shall explain how the proof should be modified when \( \Omega \) is not a cube, which will lead to a counterpart of Theorem~\ref{thm:density_class_R} for more general domains \( \Omega \).

\begin{proof}[Proof of Theorem~\ref{thm:density_class_R}]
    Let \( u \in W^{s,p}(Q^{m};\Nc) \).
    We note that, for every \( \gamma > 0 \), the map \( u_{\gamma} \) defined by \( u_{\gamma}(x) = u\bigl(\frac{x}{1+2\gamma}\bigr) \) 
    belongs to \( W^{s,p}(Q^{m}_{1+2\gamma}) \) and satisfies \( u_{\gamma} \to u \) in \( W^{s,p}(Q^{m}) \) as \( \gamma \to 0 \).
    Therefore, we may assume that \( u \in W^{s,p}(Q^{m}_{1+2\gamma};\Nc) \).
    Here we used the fact that \( \Omega = Q^{m} \), but we could work instead with any domain on which such a dilation argument may be implemented.

    Let \( 0 < \eta < \gamma \) and \( 0 < \rho < \frac{1}{2} \), so that \( 2\rho\eta < \gamma \).
    Guided by the observations at the end of Sections~\ref{sect:adaptive_smoothing} and~\ref{sect:thickening}, we define the following families of cubes.
    We let \( \Kc^{m}_{\eta} \) be a cubication of \( Q^{m}_{1+\gamma} \), that is, \( K^{m}_{\eta} = Q^{m}_{1+\gamma} \).
    This uses that \( Q^{m}_{1+\gamma} \) is a cube, but the important fact is that \( Q^{m} \subset K^{m}_{\eta} \subset Q^{m}_{1+\gamma} \).
    Then, following~\eqref{eq:estimate_good_cubes_thickening}, we construct the set of bad cubes \( \Ec^{m}_{\eta} \) as the family of all cubes \( \sigma^{m} \in \Kc^{m}_{\eta} \) such that 
    \begin{equation}
    \label{eq:estimate_bad_cubes_class_R_sgeq1}
        \lVert Du \rVert_{L^{sp}(\sigma^{m}+Q^{m}_{2\rho\eta})}
        \leq 
        \frac{\eta^{\frac{m}{sp}-1}}{C}\iota 
        \quad 
        \text{if \( s \geq 1 \),}
	\end{equation}
	respectively
	\begin{equation} 
	\label{eq:estimate_bad_cubes_class_R_sle1}
        \lvert u \rvert_{W^{s,p}(\sigma^{m}+Q^{m}_{2\rho\eta})}
        \leq 
        \frac{\eta^{\frac{m}{p}-s}}{C}\iota     
        \quad 
        \text{if \( 0 < s < 1 \),}
    \end{equation}
    where \( \iota > 0 \) is the radius of a tubular neighborhood of \( \Nc \).
    We also define \( \Uc^{m}_{\eta} \) to be the set of all cubes in \( \Kc^{m}_{\eta} \) intersecting a cube in \( \Ec^{m}_{\eta} \).
    Doing so, we indeed have \( E^{m}_{\eta} \subset \Int U^{m}_{\eta} \) in the relative topology of \( K^{m}_{\eta} \).

    We apply opening to the map \( u \) choosing \( \ell = [sp] \).
    Let \( \upPhi^{\op}_{\eta} \colon \R^{m} \to \R^{m} \) be the smooth map provided by Proposition~\ref{prop:main_opening} applied to \( u \) with \( \Omega = Q^{m}_{1+2\gamma} \), and define 
    \[
        u^{\op}_{\eta} = u \circ \upPhi^{\op}_{\eta}.
    \]
    Hence, we find that
    \begin{enumerate}[label=(\alph*)]
    	\item if \( 0 < s < 1 \), then
	    	\[
		    	\eta^{s}\lvert u^{\op}_{\eta}-u \rvert_{W^{s,p}(Q^{m}_{1+2\gamma})}
		    	\leq 
		    	\C\Bigl(\eta^{s}\lvert u \rvert_{W^{s,p}(U^{\ell}_{\eta} + Q^{m}_{2\rho\eta})}
		    	+ \lVert u \rVert_{L^{p}(U^{\ell}_{\eta} + Q^{m}_{2\rho\eta})}\Bigr);
	    	\]
	    \item if \( s \geq 1 \), then for every \( j \in \{1,\dots,k\} \),
		    \[
			    \eta^{j}\lVert D^{j}u^{\op}_{\eta}-D^{j}u \rVert_{L^{p}(Q^{m}_{1+2\gamma})} \leq \C\sum_{i=1}^{j}\eta^{i}\lVert D^{i}u \rVert_{L^{p}(U^{\ell}_{\eta} + Q^{m}_{2\rho\eta})};
		    \]
		\item if \( s \geq 1 \) and \( \sigma \neq 0 \), then for every \( j \in \{1,\dots,k\} \),
			\[
				\eta^{j+\sigma}\lvert D^{j}u^{\op}_{\eta}-D^{j}u \rvert_{W^{\sigma,p}(Q^{m}_{1+2\gamma})} \leq \C\sum_{i=1}^{j}\Bigl(\eta^{i}\lVert D^{i}u \rVert_{L^{p}(U^{\ell}_{\eta} + Q^{m}_{2\rho\eta})} 
				+ \eta^{i+\sigma}\lvert D^{i}u \rvert_{W^{\sigma,p}(U^{\ell}_{\eta} + Q^{m}_{2\rho\eta})}\Bigr);
			\]
		\item for every \( 0 < s < +\infty \),
			\[
				\lVert u^{\op}_{\eta} - u \rVert_{L^{p}(Q^{m}_{1+2\gamma})}
				\leq 
				\C\lVert u \rVert_{L^{p}(U^{\ell}_{\eta}+Q^{m}_{2\rho\eta})}.
			\]
    \end{enumerate}
   
    Then we apply adaptive smoothing to the map \( u^{\op}_{\eta} \) with \( \Omega = Q^{m}_{1+2\gamma} \).
    Let \( \varphi \in B^{m}_{1} \) be a fixed mollifier.
    Since \( E^{m}_{\eta} \subset \Int U^{m}_{\eta} \), we may define \( \psi_{\eta} \) as at the end of Section~\ref{sect:adaptive_smoothing}.
    Namely, we let 
    \[
        \psi_{\eta} = t\zeta_{\eta} + r(1-\zeta_{\eta}),
    \]
    where \( \zeta_{\eta} \) satisfies assumptions~\ref{item:zeta_between_0_and_1} to~\ref{item:D_zeta} page~\pageref{list:assumptions_zeta_smoothing} and \( 0 < r < t \) 
    with \( t \) defined by~\eqref{eq:definition_t_smoothing}.
    With this choice, \( \psi_{\eta} \) satisfies the assumptions of Proposition~\ref{prop:adaptive_smoothing}, and moreover \( 0 < \psi_{\eta} \leq \rho\eta \).
    This implies that \( Q^{m}_{1+\gamma} \subset \{ x \in Q^{m}_{1+2\gamma} \mathpunct{:} \dist{(x,\partial Q^{m}_{1+2\gamma})} \geq \psi(x) \} \), and hence 
    \( u^{\sm}_{\eta} = \varphi_{\psi_{\eta}} \ast u^{\op}_{\eta} \) is well-defined and smooth on \( Q^{m}_{1+\gamma} \).
    Moreover, recalling that \( \tau_{\psi_{\eta} v}(u)(x) = u(x + \psi_{\eta}(x)v) \), Proposition~\ref{prop:adaptive_smoothing} and equation~\eqref{eq:Lp_smooth_leq_translation} for the zero order case applied with \( \omega = Q^{m}_{1+\gamma} \) ensure that
    \begin{enumerate}[label=(\alph*)]
    	\item if \( 0 < s < 1 \), then 
    		\[
	    		\lvert u^{\sm}_{\eta}-u^{\op}_{\eta}\lvert_{W^{s,p}(Q^{m}_{1+\gamma})}
	    		\leq  
	    		\sup_{v \in B^{m}_{1}}\lvert \tau_{\psi_{\eta} v}(u^{\op}_{\eta})-u^{\op}_{\eta} \rvert_{W^{s,p}(Q^{m}_{1+\gamma})};
    		\]
    	\item if \( s \geq 1 \), then for every \( j \in \{1,\dots,k\} \),
    		\begin{multline*}
	    		\eta^{j}\lVert D^{j}u^{\sm}_{\eta}-D^{j}u^{\op}_{\eta}\lVert_{L^{p}(Q^{m}_{1+\gamma})}
	    		\leq
	    		\sup_{v \in B^{m}_{1}}\eta^{j}\lVert \tau_{\psi_{\eta} v}(D^{j}u^{\op}_{\eta})-D^{j}u^{\op}_{\eta} \rVert_{L^{p}(Q^{m}_{1+\gamma})} \\
	    		+ \C\sum_{i=1}^{j}\eta^{i}\lVert D^{i}u^{\op}_{\eta} \rVert_{L^{p}(A)};
    		\end{multline*}
    	\item if \( s \geq 1 \) and \( \sigma \neq 0 \), then for every \( j \in \{1,\dots,k\} \),
    		\begin{multline*}
	    		\eta^{j+\sigma}\lvert D^{j}u^{\sm}_{\eta}-D^{j}u^{\op}_{\eta}\lvert_{W^{\sigma,p}(Q^{m}_{1+\gamma})}
	    		\leq
	    		\sup_{v \in B^{m}_{1}}\eta^{j+\sigma}\lvert \tau_{\psi_{\eta} v}(D^{j}u^{\op}_{\eta})-D^{j}u^{\op}_{\eta} \rvert_{W^{\sigma,p}(Q^{m}_{1+\gamma})} \\
	    		+ \C\sum_{i=1}^{j}\Bigl(\eta^{i}\lVert D^{i}u^{\op}_{\eta} \rVert_{L^{p}(A)}
	    		+ \eta^{i+\sigma}\lvert D^{i}u^{\op}_{\eta}\lvert_{W^{\sigma,p}(A)}\Bigr);
    		\end{multline*}
    	\item for every \( 0 < s < +\infty \),
    		\[
	    		\lVert u^{\sm}_{\eta} - u^{\op}_{\eta} \rVert_{L^{p}(Q^{m}_{1+\gamma})}
	    		\leq 
	    		\sup_{v \in B^{m}_{1}}\lVert \tau_{\psi_{\eta}v}(u^{\op}_{\eta})-u^{\op}_{\eta}\rVert_{L^{p}(Q^{m}_{1+\gamma})}.
    		\]
    \end{enumerate}
	Here,
    \[
        A = \bigcup_{x \in Q^{m}_{1+\gamma} \cap \supp D\psi_{\eta}} B^{m}_{\psi_{\eta}(x)}(x).
    \]

    By the triangle inequality, for every \( v \in B^{m}_{1} \), we have
    \begin{align*}
        &\lVert \tau_{\psi_{\eta} v}(D^{j}u^{\op}_{\eta})-D^{j}u^{\op}_{\eta} \rVert_{L^{p}(Q^{m}_{1+\gamma})} \\
        &\quad\leq
        \lVert \tau_{\psi_{\eta} v}(D^{j}u^{\op}_{\eta})-\tau_{\psi_{\eta} v}(D^{j}u) \rVert_{L^{p}(Q^{m}_{1+\gamma})} \\
        &\quad\quad+ \lVert \tau_{\psi_{\eta} v}(D^{j}u)-D^{j}u \rVert_{L^{p}(Q^{m}_{1+\gamma})} 
            + \lVert D^{j}u-D^{j}u^{\op}_{\eta} \rVert_{L^{p}(Q^{m}_{1+\gamma})} .
    \end{align*}
    By the change of variable theorem, we find
    \[
        \lVert \tau_{\psi_{\eta} v}(D^{j}u^{\op}_{\eta})-\tau_{\psi_{\eta} v}(D^{j}u) \rVert_{L^{p}(Q^{m}_{1+\gamma})}
        \leq
        \C\lVert D^{j}u^{\op}_{\eta}-D^{j}u \rVert_{L^{p}(Q^{m}_{1+\gamma})}.
    \]
    A similar estimate holds for the Gagliardo seminorm.
    Furthermore, observing that \( \supp D\psi_{\eta} \subset U^{m}_{\eta} \) and using that \( \psi_{\eta} \leq \rho\eta \), 
    we have \( A \subset U^{m}_{\eta}+Q^{m}_{\rho\eta} \).
    Combining this with estimate~\ref{item:first_estimates_main_opening} in Proposition~\ref{prop:main_opening} applied with \( \omega = U^{m}_{\eta}+Q^{m}_{\rho\eta} \)\,,
    we deduce that 
    \begin{enumerate}[label=(\alph*)]
    	\item if \( 0 < s < 1 \), then
    		\begin{multline*}
	    		\eta^{s}\lvert u^{\sm}_{\eta}-u^{\op}_{\eta}\lvert_{W^{s,p}(Q^{m}_{1+\gamma})}
	    		\leq  
	    		\sup_{v \in B^{m}_{1}}\eta^{s}\lvert \tau_{\psi_{\eta} v}(u)-u \rvert_{W^{s,p}(Q^{m}_{1+\gamma})} \\
	    		+ \C\Bigl(\eta^{s}\lvert u \rvert_{W^{s,p}(U^{m}_{\eta} + Q^{m}_{2\rho\eta})}
	    		+ \lVert u \rVert_{L^{p}(U^{m}_{\eta} + Q^{m}_{2\rho\eta})}\Bigr);
    		\end{multline*}
    	\item if \( s \geq 1 \), then for every \( j \in \{1,\dots,k\} \),
    		\begin{multline*}
	    		\eta^{j}\lVert D^{j}u^{\sm}_{\eta}-D^{j}u^{\op}_{\eta}\lVert_{L^{p}(Q^{m}_{1+\gamma})}
	    		\leq
	    		\sup_{v \in B^{m}_{1}}\eta^{j}\lVert \tau_{\psi_{\eta} v}(D^{j}u)-D^{j}u \rVert_{L^{p}(Q^{m}_{1+\gamma})} \\
	    		+ \C\sum_{i=1}^{j}\eta^{i}\lVert D^{i}u \rVert_{L^{p}(U^{m}_{\eta}+Q^{m}_{2\rho\eta})};
    		\end{multline*}
    	\item if \( s \geq 1 \) and \( \sigma \neq 0 \), then for every \( j \in \{1,\dots,k\} \),
    		\begin{multline*}
	    		\eta^{j+\sigma}\lvert D^{j}u^{\sm}_{\eta}-D^{j}u^{\op}_{\eta}\lvert_{W^{\sigma,p}(Q^{m}_{1+\gamma})}
	    		\leq
	    		\sup_{v \in B^{m}_{1}}\eta^{j+\sigma}\lvert \tau_{\psi_{\eta} v}(D^{j}u)-D^{j}u \rvert_{W^{\sigma,p}(Q^{m}_{1+\gamma})} \\
	    		+ \C\sum_{i=1}^{j}\Bigl(\eta^{i}\lVert D^{i}u \rVert_{L^{p}(U^{m}_{\eta}+Q^{m}_{2\rho\eta})}
	    		+ \eta^{i+\sigma}\lvert D^{i}u\lvert_{W^{\sigma,p}(U^{m}_{\eta}+Q^{m}_{2\rho\eta})}\Bigr);
    		\end{multline*}
    	\item for every \( 0 < s < +\infty \),
    		\[
	    		\lVert u^{\sm}_{\eta} - u^{\op}_{\eta} \rVert_{L^{p}(Q^{m}_{1+\gamma})}
	    		\leq 
	    		\sup_{v \in B^{m}_{1}}\lVert \tau_{\psi_{\eta}v}(u)-u \rVert_{L^{p}(Q^{m}_{1+\gamma})}
	    		+ \C\lVert u \rVert_{L^{p}(U^{m}_{\eta}+Q^{m}_{2\rho\eta})}.
    		\]
    \end{enumerate}

    Finally, we apply thickening to the map \( u^{\sm}_{\eta} \).
    We choose \( 0 < \ulrho < \rho \), let \( \upPhi^{\thck}_{\eta} \colon \R^{m} \setminus T^{\ell^{\ast}}_{\eta} \to \R^{m} \) be the smooth map given by Proposition~\ref{prop:main_thickening} 
    applied with parameter \( \ulrho \) and with \( \Omega = Q^{m}_{1+\gamma} \), where we recall that \( \Tc^{\ell^{\ast}}_{\eta} \) is the skeleton dual to \( \Uc^{m}_{\eta} \), and we set
    \[
        u^{\thck}_{\eta} = u^{\sm}_{\eta} \circ \upPhi^{\thck}_{\eta}.
    \]
    This map coincides with \( u^{\sm}_{\eta} \) outside of \( U^{m}_{\eta}+Q^{m}_{\ulrho\eta} \).
    Since \( \ell+1 > sp \), Proposition~\ref{prop:main_thickening} ensures that \( u^{\thck}_{\eta} \in W^{s,p}(Q^{m}_{1+\gamma};\R^{\nu}) \), and moreover, the following estimates hold:
    \begin{enumerate}[label=(\alph*)]
    	\item if \( 0 < s < 1 \), then 
    		\[
	    		\eta^{s}\lvert u^{\thck}_{\eta}-u^{\sm}_{\eta} \rvert_{W^{s,p}(Q^{m}_{1+\gamma})}
	    		\leq 
	    		\C\Bigl(\eta^{s}\lvert u^{\sm}_{\eta} \rvert_{W^{s,p}(U^{m}_{\eta}+Q^{m}_{\ulrho\eta})}
	    		+ \lVert u^{\sm}_{\eta} \rVert_{L^{p}(U^{m}_{\eta}+Q^{m}_{\ulrho\eta})}\Bigr);
    		\]
    	\item if \( s \geq 1 \), then for every \( j \in \{1,\dots,k\} \),
    		\[
	    		\eta^{j}\lVert D^{j}u^{\thck}_{\eta}-D^{j}u^{\sm}_{\eta} \rVert_{L^{p}(Q^{m}_{1+\gamma})}
	    		\leq
	    		\C\sum_{i=1}^{j}\eta^{i}\lVert D^{i}u^{\sm}_{\eta} \rVert_{L^{p}(U^{m}_{\eta}+Q^{m}_{\ulrho\eta})};
    		\]
    	\item if \( s \geq 1 \) and \( \sigma \neq 0 \), then for every \( j \in \{1,\dots,k\} \),
    		\[
	    		\eta^{j+\sigma}\lvert D^{j}u^{\thck}_{\eta}-D^{j}u^{\sm}_{\eta} \rvert_{W^{\sigma,p}(Q^{m}_{1+\gamma})}
	    		\leq
	    		\C\sum_{i=1}^{j}\Bigl(\eta^{i}\lVert D^{i}u^{\sm}_{\eta} \rVert_{L^{p}(U^{m}_{\eta}+Q^{m}_{\ulrho\eta})} + \eta^{i+\sigma}\lvert D^{i}u^{\sm}_{\eta} \rvert_{W^{\sigma,p}(U^{m}_{\eta}+Q^{m}_{\ulrho\eta})}\Bigr);
    		\]
    	\item for every \( 0 < s < +\infty \),
    		\[
    		\lVert u^{\thck}_{\eta}-u^{\sm}_{\eta} \rVert_{L^{p}(Q^{m}_{1+\gamma})}
    		\leq 
    		\C\lVert u^{\sm}_{\eta} \rVert_{L^{p}(U^{m}_{\eta}+Q^{m}_{\ulrho\eta})}.
    		\]
    \end{enumerate}
    Hence, invoking estimate~\ref{item:first_estimates_convolution} in Proposition~\ref{prop:adaptive_smoothing} with \( \Omega = U^{m}_{\eta}+Q^{m}_{(\rho+\ulrho)\eta} \) and \( \omega = U^{m}_{\eta}+Q^{m}_{\ulrho\eta} \), and then estimate~\ref{item:first_estimates_main_opening} in Proposition~\ref{prop:main_opening} with \( \omega = U^{m}_{\eta}+Q^{m}_{2\rho\eta} \), we obtain
    \begin{enumerate}[label=(\alph*)]
    	\item if \( 0 < s < 1 \), then 
    		\[
	    		\eta^{s}\lvert u^{\thck}_{\eta}-u^{\sm}_{\eta} \rvert_{W^{s,p}(Q^{m}_{1+\gamma})}
	    		\leq 
	    		\C\Bigl(\eta^{s}\lvert u \rvert_{W^{s,p}(U^{m}_{\eta}+Q^{m}_{2\rho\eta})} 
	    		+ \lVert u \rVert_{L^{p}(U^{m}_{\eta}+Q^{m}_{2\rho\eta})}\Bigr);
    		\]
    	\item if \( s \geq 1 \), then for every \( j \in \{1,\dots,k\} \),
    		\[
    		\eta^{j}\lVert D^{j}u^{\thck}_{\eta}-D^{j}u^{\sm}_{\eta} \rVert_{L^{p}(Q^{m}_{1+\gamma})}
    		\leq
    		\C\sum_{i=1}^{j}\eta^{i}\lVert D^{i}u \rVert_{L^{p}(U^{m}_{\eta}+Q^{m}_{2\rho\eta})};
    		\]
    	\item if \( s \geq 1 \) and \( \sigma \neq 0 \), then for every \( j \in \{1,\dots,k\} \),
    		\[
	    		\eta^{j+\sigma}\lvert D^{j}u^{\thck}_{\eta}-D^{j}u^{\sm}_{\eta} \rvert_{W^{\sigma,p}(Q^{m}_{1+\gamma})}
	    		\leq
	    		\C\sum_{i=1}^{j}\Bigl(\eta^{i}\lVert D^{i}u \rVert_{L^{p}(U^{m}_{\eta}+Q^{m}_{2\rho\eta})} + \eta^{i+\sigma}\lvert D^{i}u \rvert_{W^{\sigma,p}(U^{m}_{\eta}+Q^{m}_{2\rho\eta})}\Bigr);
    		\]
    	\item for every \( 0 < s < +\infty \),
    		\[
	    		\lVert u^{\thck}_{\eta}-u^{\sm}_{\eta} \rVert_{L^{p}(Q^{m}_{1+\gamma})}
	    		\leq 
	    		\C\lVert u \rVert_{L^{p}(U^{m}_{\eta}+Q^{m}_{2\rho\eta})}.
    		\]
    \end{enumerate}

    Using the triangle inequality, letting \( A_{\mu} = U^{m}_{\eta}+Q^{m}_{2\rho\eta} \), we conclude that
    \begin{enumerate}[label=(\alph*)]
    	\item if \( 0 < s < 1 \), then 
    		\begin{multline*}
	    		\eta^{s}\lvert u^{\thck}_{\eta}-u_{\eta} \rvert_{W^{s,p}(Q^{m}_{1+\gamma})}
	    		\leq 
	    		\sup_{v \in B^{m}_{1}}\eta^{s}\lvert \tau_{\psi_{\eta} v}(u)-u \rvert_{W^{s,p}(Q^{m}_{1+\gamma})} \\
	    		+ \C\Bigl(\eta^{s}\lvert u \rvert_{W^{s,p}(A_{\mu})} 
	    		+ \lVert u \rVert_{L^{p}(A_{\mu})}\Bigr);
    		\end{multline*}
    	\item if \( s \geq 1 \), then for every \( j \in \{1,\dots,k\} \),
    		\begin{multline*}
	    		\eta^{j}\lVert D^{j}u^{\thck}_{\eta}-D^{j}u \rVert_{L^{p}(Q^{m}_{1+\gamma})} \\
	    		\leq
	    		\sup_{v \in B^{m}_{1}}\eta^{j}\lVert \tau_{\psi_{\eta} v}(D^{j}u)-D^{j}u \rVert_{L^{p}(Q^{m}_{1+\gamma})} 
	    		+ \C\sum_{i=1}^{j}\eta^{i}\lVert D^{i}u \rVert_{L^{p}(A_{\mu})};
    		\end{multline*}
    	\item if \( s \geq 1 \) and \( \sigma \neq 0 \), then for every \( j \in \{1,\dots,k\} \),
    		\begin{align*}
	    		&\eta^{j+\sigma}\lvert D^{j}u^{\thck}_{\eta}-D^{j}u \rvert_{W^{\sigma,p}(Q^{m}_{1+\gamma})} \\
	    		&\quad\leq
	    		\sup_{v \in B^{m}_{1}}\eta^{j+\sigma}\lvert \tau_{\psi_{\eta} v}(D^{j}u)-D^{j}u \rvert_{W^{\sigma,p}(Q^{m}_{1+\gamma})} \\
	    		&\quad\quad+ \C\sum_{i=1}^{j}\Bigl(\eta^{i}\lVert D^{i}u \rVert_{L^{p}(A_{\mu})} + \eta^{i+\sigma}\lvert D^{i}u \rvert_{W^{\sigma,p}(A_{\mu})}\Bigr);
    		\end{align*}
    	\item for every \( 0 < s < +\infty \),
    		\[
    		\lVert u^{\thck}_{\eta}-u \rVert_{L^{p}(Q^{m}_{1+\gamma})}
    		\leq 
    		\sup_{v \in B^{m}_{1}} \lVert \tau_{\psi_{\eta}v}(u)-u \rVert_{L^{p}(Q^{m}_{1+\gamma})}
    		+ \C\lVert u \rVert_{L^{p}(A_{\mu})}.
    		\]
	    \end{enumerate}

    Due to our choice of \( \psi_{\eta} \), and since \( \ell \leq sp \) and
    \[
        Q^{m} \subset K^{m}_{\eta} \subset Q^{m}_{1+\gamma} \subset \{ x \in Q^{m}_{1+2\gamma} \mathpunct{:} \dist{(x,\partial Q^{m}_{1+2\gamma})} \geq \psi(x) \}, 
    \]
    according to estimates~\eqref{eq:main_estimate_dist_thck_F_sgeq1} and~\eqref{eq:main_estimate_dist_thck_F_sle1}, we have
    \begin{multline}
    \label{eq:dist_thck_N_thm2_sgeq1}
        \Dist_{\Nc}{(u^{\thck}_{\eta}(K^{m}_{\eta}\setminus T^{\ell^{\ast}}_{\eta}))}
        \leq 
        \max\biggl\{\max_{\sigma^{m} \in \Kc^{m}_{\eta} \setminus \Ec^{m}_{\eta}} C\frac{1}{\eta^{\frac{m}{sp}-1}}\lVert Du \rVert_{L^{sp}(\sigma^{m}+Q^{m}_{2\rho\eta})}, \\
            \sup_{x \in U^{\ell}_{\eta}+Q^{m}_{\ulrho\eta}} C'\fint_{Q^{m}_{r}(x)}\fint_{Q^{m}_{r}(x)} \lvert u^{\op}_{\eta}(y)-u^{\op}_{\eta}(z)\rvert\,\d y\d z\biggr\}
    \end{multline}
    if \( s \geq 1 \), respectively 
    \begin{multline}
    \label{eq:dist_thck_N_thm2_sle1}
        \Dist_{\Nc}{(u^{\thck}_{\eta}(K^{m}_{\eta}\setminus T^{\ell^{\ast}}_{\eta}))}
        \leq 
        \max\biggl\{\max_{\sigma^{m} \in \Kc^{m}_{\eta} \setminus \Ec^{m}_{\eta}} C\frac{1}{\eta^{\frac{m}{p}-s}}\lvert u \rvert_{W^{s,p}(\sigma^{m}+Q^{m}_{2\rho\eta})}, \\
            \sup_{x \in U^{\ell}_{\eta}+Q^{m}_{\ulrho\eta}} C'\fint_{Q^{m}_{r}(x)}\fint_{Q^{m}_{r}(x)} \lvert u^{\op}_{\eta}(y)-u^{\op}_{\eta}(z)\rvert\,\d y\d z\biggr\}
    \end{multline}
    if \( 0 < s < 1 \).
    In~\eqref{eq:dist_thck_N_thm2_sgeq1} and~\eqref{eq:dist_thck_N_thm2_sle1}, we recall that \( r > 0 \) is a number used in the definition of \( \psi_{\eta} \), and chosen sufficiently small to ensure the validity of~\eqref{eq:main_estimate_dist_sm_F_sgeq1} or~\eqref{eq:main_estimate_dist_sm_F_sle1}.
    We note here that, in defining the bad cubes in equation~\eqref{eq:estimate_bad_cubes_class_R_sgeq1}, respectively~\eqref{eq:estimate_bad_cubes_class_R_sle1}, we take the constant \( C > 0 \) which shows up in estimate~\eqref{eq:dist_thck_N_thm2_sgeq1}, respectively~\eqref{eq:dist_thck_N_thm2_sle1}.
    Doing so, by definition of the set of bad cubes \( \Ec^{m}_{\eta} \), the first term in each max is smaller than the radius \( \iota \) of a tubular neighborhood of \( \Nc \).
    Moreover, since \( \ell \leq sp \), Proposition~\ref{prop:addendum2_opening} ensures that we may take \( r > 0 \) so small that the second term in each max is also smaller than \( \iota \).
    Therefore, we deduce that 
    \[
        \Dist_{\Nc}{(u^{\thck}_{\eta}(K^{m}_{\eta}\setminus T^{\ell^{\ast}}_{\eta}))}
        \leq 
        \iota.
    \]
    This enables us to define \( u_{\eta} = \upPi \circ u^{\thck}_{\eta} \), which, as we already explained at the end of Section~\ref{sect:thickening}, is smooth on \( K^{m}_{\eta} \setminus T^{\ell^{\ast}} \), 
    and belongs to \( \Rc_{m-[sp]-1}(K^{m}_{\eta};\Nc) \).
    
    Since \( Q^{m} \subset K^{m}_{\eta} \), to conclude, it only remains to prove that \( u_{\eta} \to u \) in \( W^{s,p}(K^{m}_{\eta}) \) as \( \eta \to 0 \).
    We claim that it suffices to show that \( u^{\thck}_{\eta} \to u \) in \( W^{s,p}(K^{m}_{\eta}) \) as \( \eta \to 0 \).
    Indeed, the map \( \upPi \) is smooth and has uniformly bounded derivatives, and \( \Nc \) is compact.
    Hence, the continuity of the composition operator from \( W^{s,p} \cap L^{\infty} \) to \( W^{s,p} \)
    -- see for instance~\cite{BM_sobolev_maps_to_the_circle}*{Chapter~15.3} -- ensures that, if \( u_{\eta}^{\thck} \to u \) in \( W^{s,p}(K^{m}_{\eta}) \), then \( u_{\eta} = \upPi \circ u^{\thck}_{\eta} \) converges in \( W^{s,p}(K^{m}_{\eta}) \) to \( \upPi \circ u = u \).
    
    We now prove that \( u^{\thck}_{\eta} \to u \).
    We start by noting that the continuity of the translation operator implies that
    \[
        \lim_{\eta\to 0}\sup_{v \in B^{m}_{1}}\lVert \tau_{\psi_{\eta} v}(D^{j}u)-D^{j}u \rVert_{L^{p}(Q^{m}_{1+\gamma})} = 0
    \]
    and
    \[
        \lim_{\eta\to 0}\sup_{v \in B^{m}_{1}}\lvert \tau_{\psi_{\eta} v}(D^{j}u)-D^{j}u \rvert_{W^{\sigma,p}(Q^{m}_{1+\gamma})} = 0
    \]
    for every \( j \in \{0,\dots,k\} \).

    We first deal with the case \( s \geq 1 \).
    By the Gagliardo--Nirenberg interpolation inequality -- see for instance~\cites{BM_fractional_GN, BM_GN_full_story} -- for every \( i \in \{1,\dots,k\} \), 
    we have \( D^{i}u \in L^{\frac{sp}{i}}(Q^{m}_{1+2\gamma}) \). 
    Hölder's inequality ensures that
    \[
        \lVert D^{i}u \rVert_{L^{p}(A_{\mu})}
        \leq
        \lvert A_{\mu} \rvert^{\frac{s-i}{sp}}\lVert D^{i}u \rVert_{L^{sp/i}(A_{\mu})}\,,
    \]
    while Lemma~\ref{lemma:interpolation_estimate} guarantees that 
    \[
        \lvert D^{i}u \rvert_{W^{\sigma,p}(A_{\mu})}
        \leq 
        \Cl{cst:cst_interpolation_lemma}\lvert A_{\mu} \rvert^{\frac{s-i-\sigma}{sp}}\lVert D^{i+1}u \rVert_{L^{sp/(i+1)}(Q^{m}_{1+2\gamma})}^{\sigma}\lVert D^{i}u \rVert_{L^{sp/i}(A_{\mu})}^{1-\sigma}.
    \]
    Here, we use the fact that \( Q^{m}_{1+2\gamma} \) is convex to justify the use of Lemma~\ref{lemma:interpolation_estimate}.

    We now wish to estimate the measure of the set \( A_{\mu} \).
    We first note that
    \begin{equation}
    \label{eq:measure_vs_number_bad_cubes}
        \lvert A_{\mu} \rvert
        \leq
        \C\card{(\Uc^{m}_{\eta})}\eta^{m},
    \end{equation}
    where \( \card{(\Uc^{m}_{\eta})} \) denotes the cardinality of the (finite) set of cubes \( \Uc^{m}_{\eta} \).
    Then, for every \( \sigma^{m} \in \Uc^{m}_{\eta} \), there exists \( \tau^{m} \in \Ec^{m}_{\eta} \) which intersects \( \sigma^{m} \), and thus \( \tau^{m}+Q^{m}_{2\rho\eta} \subset \sigma^{m} + Q^{m}_{2(1+\rho)\eta} \).
    If we write \( \sigma^{m} = Q^{m}_{\eta}(a) \), we find \( \tau^{m}+Q^{m}_{2\rho\eta} \subset Q^{m}_{\alpha\eta}(a) \) with \( \alpha = 3+2\rho \).
    Hence,
    \[
        \tau^{m}+Q^{m}_{2\rho\eta} \subset Q^{m}_{\alpha\eta}(a) \cap Q^{m}_{1+2\gamma}.
    \]
    We deduce from the definition of \( \Ec^{m}_{\eta} \) that
    \[
        \iota 
        <
        C\frac{1}{\eta^{\frac{m}{sp}-1}}\lVert Du \rVert_{L^{sp}(\tau^{m}+Q^{m}_{2\rho\eta})}
        \leq
        C\frac{1}{\eta^{\frac{m}{sp}-1}}\lVert Du \rVert_{L^{sp}(Q^{m}_{\alpha\eta}(a) \cap Q^{m}_{1+2\gamma})}.
    \]
    Since the number of overlaps between one of the cubes \( Q^{m}_{\alpha\eta}(a) \) and all the other ones is bounded from above by a number depending only on \( m \), 
    summing over all cubes in \( \Uc^{m}_{\eta} \) and using the additivity of the integral, we deduce that
    \begin{equation}
    \label{eq:estimate_number_bad_cubes_sgeq1}
    \begin{split}
        \card{(\Uc^{m}_{\eta})}
        &\leq
        \C\frac{1}{\eta^{m-sp}}\sum_{Q^{m}_{\eta}(a)\in \Uc^{m}_{\eta}} \int_{Q^{m}_{\alpha\eta}(a) \cap Q^{m}_{1+2\gamma}}\lvert Du \rvert^{sp} \\
        &\leq
        \C\frac{1}{\eta^{m-sp}}\int_{(U^{m}_{\eta}+Q^{m}_{2(\rho+1)\eta})\cap Q^{m}_{1+2\gamma}} \lvert Du \rvert^{sp}.
    \end{split}
    \end{equation}
    For further use, we already note that, in the case \( 0 < s < 1 \), the exact same reasoning leads to 
    \[
        \iota 
        <
        C\frac{1}{\eta^{\frac{m}{p}-s}}\lvert u \rvert_{W^{s,p}(Q^{m}_{\alpha\eta}(a) \cap Q^{m}_{1+2\gamma})}.
    \]
    As for the case \( s \geq 1 \), replacing the additivity of the integral by the superadditivity of the Gagliardo seminorm, we obtain 
    \begin{equation}
    \label{eq:estimate_number_bad_cubes_sle1}
        \card{(\Uc^{m}_{\eta})}
        \leq
        \C\frac{1}{\eta^{m-sp}}\int_{(U^{m}_{\eta}+Q^{m}_{2(\rho+1)\eta})\cap Q^{m}_{1+2\gamma}}\int_{(U^{m}_{\eta}+Q^{m}_{2(\rho+1)\eta})\cap Q^{m}_{1+2\gamma}} \frac{\lvert u(x)-u(y)\rvert^{p}}{\lvert x-y \rvert^{m+sp}}\,\d x\d y.
    \end{equation}
    In both cases \( s \geq 1 \) and \( 0 < s < 1 \), we conclude that
    \begin{equation}
    \label{eq:decay_measure_bad_cubes}
        \lim_{\eta \to 0}\frac{\lvert A_{\mu} \rvert}{\eta^{sp}}
        =
        0.
    \end{equation}
    Indeed, we first use estimate~\eqref{eq:measure_vs_number_bad_cubes} along with~\eqref{eq:estimate_number_bad_cubes_sgeq1}, respectively~\eqref{eq:estimate_number_bad_cubes_sle1}, to deduce that
    \[
    	\frac{\lvert A_{\mu} \rvert}{\eta^{sp}}
    	\leq 
    	\C\lVert Du \rVert_{L^{p}(Q^{m}_{1+2\gamma})}^{p},
    	\quad
    	\text{respectively}
    	\quad
    	\frac{\lvert A_{\mu} \rvert}{\eta^{sp}}
    	\leq 
    	\C\lvert u \rvert_{W^{s,p}(Q^{m}_{1+2\gamma})}^{p}.
    \]
    In particular, \( \lvert A_{\mu} \rvert \to 0 \).
	Using this information along with Lebesgue's lemma, we invoke again estimate~\eqref{eq:estimate_number_bad_cubes_sgeq1}, respectively~\eqref{eq:estimate_number_bad_cubes_sle1}, to deduce~\eqref{eq:decay_measure_bad_cubes}.

    We next proceed as follows.
    When \( s \geq 1 \), we find 
    \begin{equation}
    \label{eq:Diu_to_0_integer}
        \sum_{i=1}^{j}\eta^{i-j}\lVert D^{i}u \rVert_{L^{p}(A_{\mu})} \\
        \leq 
        \sum_{i=1}^{j}\eta^{s-j}\biggl(\frac{\lvert A_{\mu} \rvert}{\eta^{sp}}\biggr)^{\frac{s-i}{sp}}\lVert D^{i}u \rVert_{L^{sp/i}(A_{\mu})}
        \to 
        0
        \quad 
        \text{as \( \eta \to 0 \)}
    \end{equation}
    and 
    \begin{equation}
    \begin{split}
    \label{eq:Diu_to_0_frac}
        \sum_{i=1}^{j}&\Bigl(\eta^{i}\lVert D^{i}u \rVert_{L^{p}(A_{\mu})} + \eta^{i+\sigma}\lvert D^{i}u \rvert_{W^{\sigma,p}(A_{\mu})}\Bigr) \\
        \leq& 
        \sum_{i=1}^{j}\eta^{s-j-\sigma}\biggl(\frac{\lvert A_{\mu} \rvert}{\eta^{sp}}\biggr)^{\frac{s-i}{sp}}\lVert D^{i}u \rVert_{L^{sp/i}(A_{\mu})} \\
            &+ \Cr{cst:cst_interpolation_lemma}\sum_{i=1}^{j-1}\eta^{s-j-\sigma}\biggl(\frac{\lvert A_{\mu} \rvert}{\eta^{sp}}\biggr)^{\frac{s-i-\sigma}{sp}}\lVert D^{i}u \rVert_{L^{sp/i}(A_{\mu})}^{1-\sigma}\lVert D^{i+1}u \rVert_{L^{sp/(i+1)}(Q^{m}_{1+2\gamma})}^{\sigma} \\
            &+ \lvert D^{j}u \rvert_{W^{\sigma,p}(A_{\mu})} 
            \to 
            0
            \quad
            \text{as \( \eta \to 0 \).}
    \end{split}
    \end{equation}
    For the last term in~\eqref{eq:Diu_to_0_frac}, we use~\eqref{eq:decay_measure_bad_cubes} and the Lebesgue lemma.
    Similarly, we have 
    \[
        \lVert u \rVert_{L^{p}(A_{\mu})}
        \to 
        0
        \quad
        \text{as \( \eta \to 0 \).}
    \]
    This completes the proof that \( u^{\thck}_{\eta} \to u \) in \( W^{s,p}(Q^{m}_{1+\gamma}) \) when \( s \geq 1 \).

    The case \( 0 < s < 1 \) is concluded analogously.
    Note that since \( \Nc \) is compact, we have \( u \in L^{\infty} \).
    Therefore, we have
    \begin{align*}
        \lvert u \rvert_{W^{s,p}(A_{\mu})} 
            + \eta^{-s}\lVert u \rVert_{L^{p}(A_{\mu})}
        &\leq 
        \lvert u \rvert_{W^{s,p}(A_{\mu})}
            + \eta^{-s}\Cl{cst:L_infty_bound_density_class_R}\lvert A_{\mu} \rvert^{\frac{1}{p}} \\
        &\leq 
        \lvert u \rvert_{W^{s,p}(A_{\mu})}
            + \Cr{cst:L_infty_bound_density_class_R}\biggl(\frac{\lvert A_{\mu} \rvert}{\eta^{sp}}\biggr)^{\frac{1}{p}}
        \to 
        0
        \quad 
        \text{as \( \eta \to 0 \).}
    \end{align*}
    Combining this with the fact that \( \lVert u \rVert_{L^{p}(A_{\mu})} \to 0 \) as \( \eta \to 0 \),
    we deduce that \( u^{\thck}_{\eta} \to u \) in \( W^{s,p}(Q^{m}_{1+\gamma}) \) as \( \eta \to 0 \) when \( 0 < s < 1 \).
    This completes the proof of Theorem~\ref{thm:density_class_R}.
    \resetconstant
\end{proof}

We now explain how to deal with more general domains.
The first step is to be able to implement the dilation procedure used at the beginning of the proof.
The method we used adapts without any modification to domains that are starshaped with respect to one of their points.
However, using a more involved technique, it is possible to work with even more general domains.
The reader may consult~\cite{BM_sobolev_maps_to_the_circle}*{Lemma~15.25} for an implementation of this technique on smooth domains using the normal vector,
or~\cite{BPVS_screening} for an argument on continuous bounded domains using local parametrizations.
Here we show that the approach even works under the weaker \emph{segment condition}.

We recall that \( \Omega \) satisfies the \emph{segment condition} whenever, for every \( x \in \partial\Omega \), there exists an open set \( U_{x} \subset \R^{m} \) containing \( x \) 
and a nonzero vector \( z_{x} \in \R^{m} \) such that, if \( y \in U_{x} \cap \overline{\Omega} \), then \( y+tz_{x} \in \Omega \) for every \( 0 < t < 1 \).

\begin{lemme}
\label{lemma:dilation_segment_condition}
    Let \( \Omega \subset \R^{m} \) be a bounded open domain satisfying the segment condition.
    For every \( \gamma > 0 \) sufficiently small, there exists a smooth diffeomorphism \( \upPhi_{\gamma} \colon \R^{m} \to \R^{m} \) such that \( \upPhi_{\gamma}(\overline{\Omega}) \subset \Omega \) and 
    \[
        D^{j}\upPhi_{\gamma} \to \id 
        \quad 
        \text{uniformly on \( \R^{m} \) for every \( j \in \N \) as \( \gamma \to 0 \).}
    \]
\end{lemme}

Geometrically, the segment condition means that \( \Omega \) cannot lie on both sides of \( \partial\Omega \).
A typical example of a domain \( \Omega \) not satisfying this assumption is given by two open cubes whose boundaries share a common face.
It is known -- see for instance~\cite{adams_sobolev_spaces}*{3.17} -- that there exists a \( W^{1,p} \) map on this domain which cannot be approximated by \( \Cc^{\infty}(\overline{\Omega}) \) maps, even in the real valued case.

\begin{proof}[Proof of Lemma~\ref{lemma:dilation_segment_condition}]
    Let \( B_{\delta} = B^{m-1}_{\delta} \times (-\delta,\delta) \) be a cylinder of radius and half-height \( \delta \).
    Since \( \partial\Omega \) is compact, there exists a finite number of points \( x_{1} \), \dots, \( x_{n} \in \R^{m} \) and associated isometries \( T_{1} \), \dots, \( T_{n} \) of \( \R^{m} \) mapping \( 0 \) to \( x_{i} \) such that 
    \begin{equation}
    \label{eq:covering_partial_omega_segment_condition}
        \partial\Omega \subset \bigcup_{i=1}^{n} T_{i}(B_{\delta/2}),
    \end{equation}
    and also associated nonzero vectors \( z_{1} \), \dots, \( z_{n} \in \R^{m} \) such that, if \( y \in T_{i}(B_{\delta}) \cap \overline{\Omega} \), 
    then \( y+tz_{i} \in \Omega \) for every \( 0 < t < 1 \).
    Let \( \psi \colon \R^{m-1} \to [0,1] \) be a smooth map such that \( \psi(x) = 1 \) if \( x \in B_{\delta/2} \) and \( \psi(x) = 0 \) if \( x \in \R^{m-1} \setminus B_{3\delta/4} \).
    For \( 0 < \gamma < 1 \), we define \( \upPhi_{i,\gamma} \colon \R^{m} \to \R^{m} \) by
    \[ 
        \upPhi_{i,\gamma}(x) = x+\gamma\psi(T^{-1}_{i}(x))z_{i}.
    \]
    If \( \gamma < (\lVert D\psi \rVert_{L^{\infty}}\lvert z_{i} \rvert)^{-1} \), we observe that \( \upPhi_{i,\gamma} \) is a smooth diffeomorphism.
    Moreover, by construction of the vectors \( z_{i} \), we have \( \upPhi_{i,\gamma}(\Omega) \subset \Omega \).

    We let 
    \[
        \upPhi_{\gamma} = \upPhi_{n,\gamma} \circ \cdots \circ \upPhi_{1,\gamma}.
    \]
    We observe that 
    \[
        D^{j}\upPhi_{\gamma} \to \id 
        \quad 
        \text{uniformly on \( \R^{m} \) for every \( j \in \N \) as \( \gamma \to 0 \).}
    \]
    By~\eqref{eq:covering_partial_omega_segment_condition} and the construction of the maps \( \upPhi_{i,\gamma} \), 
    for every \( x \in \partial\Omega \), there exists \( i \in \{1,\dots,n\} \) such that \( \upPhi_{i,\gamma}(x) \in \Omega \), 
    and this shows that \( \upPhi_{\gamma}(\overline{\Omega}) \subset \Omega \).
    This proves that the family of maps \( \upPhi_{\gamma} \) satisfies the conclusions of the lemma.
\end{proof}

Using this construction, we observe that, if \( \Omega \subset \R^{m} \) is a bounded domain satisfying the segment condition and \( u \in W^{s,p}(\Omega;\Nc) \),
the map \( u_{\gamma} = u \circ \upPhi_{\gamma} \) belongs to \( W^{s,p}(\Omega_{\gamma};\Nc) \),
where \( \Omega_{\gamma} = \upPhi_{\gamma}^{-1}(\Omega) \) is an open subset of \( \R^{m} \) containing \( \overline{\Omega} \).
Moreover, \( u_{\gamma} \to u \) in \( W^{s,p}(\Omega;\Nc) \) as \( \gamma \to 0 \).

Therefore, we may carry out the same reasoning as in the proof of Theorem~\ref{thm:density_class_R} by choosing a cubication \( \Kc^{m}_{\eta} \) such that 
\( \Omega \subset K^{m}_{\eta} \subset \Omega_{\gamma} \).

The other place in the proof of Theorem~\ref{thm:density_class_R} where we used a specific assumption on the domain is when we applied Lemma~\ref{lemma:interpolation_estimate},
because we needed convexity to justify the use of this lemma.
However, this may easily be bypassed.
Indeed, since we work on a dilated domain, by dilating slightly more if necessary, we may assume that \( u \in W^{s,p}(\widetilde{\Omega}) \) for some open set \( \widetilde{\Omega} \subset \R^{m} \) containing \( \overline{\Omega}_{\gamma} \).
It then suffices to apply instead Lemma~\ref{lemma:interpolation_estimate} to the map \( u\psi \in W^{s,p}(\R^{m}) \), where \( \psi \colon \R^{m} \to [0,1] \) is a smooth map such that \( \psi = 1 \) on \( \Omega_{\gamma} \) 
and \( \psi = 0 \) on \( \R^{m} \setminus \widetilde{\Omega} \).

Taking these modifications into consideration, the proof of Theorem~\ref{thm:density_class_R} above can be carried out the exact same way on any bounded domain \( \Omega \subset \R^{m} \) satisfying the segment condition.
This leads to the following result.

\begin{thm}
\label{thm:density_class_R_segment_condition}
    Let \( \Omega \subset \R^{m} \) be a bounded domain satisfying the segment condition.
    If \( sp < m \), then the class \( \Rc_{m-[sp]-1}(\Omega;\Nc) \) is dense in \( W^{s,p}(\Omega;\Nc) \).
\end{thm}

A second perspective of generalization for Theorem~\ref{thm:density_class_R} consists in replacing \( \Omega \) with a smooth manifold.
From now on, we assume that \( \Mc \) is a smooth, compact, connected Riemannian manifold of dimension \( m \), isometrically embedded in \( \R^{\tilde{\nu}} \) for some \( \tilde{\nu} \in \N \).
In the context where the domain is a smooth manifold, the suitable adaptation of the definition of the class \( \Rc \) is the following. 
We define the class \( \Rc_{i}(\Mc;\Nc) \) as the set of maps \( u \colon \Mc \to \Nc \) which are smooth on \( \Mc \setminus T \), where \( T \) is a finite union of \( i \)-dimensional
submanifolds of \( \Mc \), and such that for every \( j \in \N_{\ast} \) and \( x \in \Mc \setminus T \),
\[
    \lvert D^{j}u(x) \rvert \leq C\frac{1}{\dist{(x,T)}^{j}}
\]
for some constant \( C > 0 \) depending on \( u \) and \( j \).

Our next result is the following counterpart of Theorem~\ref{thm:density_class_R} when the domain is a smooth manifold.

\begin{thm}
\label{thm:density_class_R_manifold}
    If \( sp < m \), then the class \( \Rc_{m-[sp]-1}(\Mc;\Nc) \) is dense in \( W^{s,p}(\Mc;\Nc) \).
\end{thm}

We first prove Theorem~\ref{thm:density_class_R_manifold} when \( \Mc \) has no boundary.
This allows us to rely on the nearest point projection onto \( \Mc \).
In the end, we shall briefly explain how to deduce the case with boundary from the case without boundary. 

Hence, we first assume that \( \Mc \) has no boundary.
Then, Theorem~\ref{thm:density_class_R_manifold} can be deduced from Theorem~\ref{thm:density_class_R_segment_condition}, by extending the function we want to approximate on a tubular neighborhood of \( \Mc \) 
and using a slicing argument.
The key observation is that, if \( \iota > 0 \) is the radius of a tubular neighborhood of \( \Mc \), then for every \( u \in W^{s,p}(\Mc;\Nc) \), the map \( v = u \circ \upPi \) belongs to \( W^{s,p}(\Mc + B^{\tilde{\nu}}_{\iota/2};\Nc) \).
Indeed, for any summable function \( w \colon \Mc \to [0,+\infty] \), we deduce from the coarea formula that
\[
	\int_{\Mc + B^{\tilde{\nu}}_{\iota/2}} w \circ \upPi
	\leq 
	\C\int_{\Mc}\biggl(\int_{\upPi^{-1}(x)\cap (\Mc + B^{\tilde{\nu}}_{\iota/2})} w(\upPi(y))\,\d \Ha^{\tilde{\nu}-m}(y)\biggr)\,\d x
	\leq 
	\C\iota^{\tilde{\nu}-m}\int_{\Mc} w
	<
	+\infty.
\]
The conclusion then follows from the theory of Fuglede maps presented in Section~\ref{sect:opening} (valid also for maps between manifolds, see~\cite{BPVS_screening}).
\resetconstant

To implement this strategy of extension and slicing, we need the following transversality result.

\begin{lemme}
\label{lemma:transversality}
    Let \( \Sigma \subset \R^{\tilde{\nu}} \) be an \( \ell \)-dimensional hyperplane.
    For almost every \( a \in \R^{\tilde{\nu}} \), the set \( \Mc \cap (\Sigma+a) \) is a smooth submanifold of \( \Mc \) of dimension \( m-\tilde{\nu}+\ell \) -- or the empty set if \( \ell < \tilde{\nu}-m \).
    Moreover, if \( \Mc \cap (\Sigma+a) \neq \varnothing \), then for every \( x \in \Mc \) and every \( a \) as above, we have 
    \[
    	\dist{(x,\Mc \cap (\Sigma+a))}
    	\leq 
    	C\dist{(x,\Sigma+a)},
    \]
    for some constant \( C > 0 \) depending on \( \Mc \), \( \Sigma \), and \( a \).
\end{lemme}

Taking Lemma~\ref{lemma:transversality} for granted, we prove Theorem~\ref{thm:density_class_R_manifold}.

\begin{proof}[Proof of Theorem~\ref{thm:density_class_R_manifold} when \( \Mc \) has no boundary]
    Let \( u \in W^{s,p}(\Mc;\Nc) \).
    Let \( \iota > 0 \) be the radius of a tubular neighborhood of \( \Mc \), and let \( \upPi \colon \Mc + B^{\tilde{\nu}}_{\iota} \to \Mc \) be the nearest point projection.
    We define \( \Omega = \Mc + B^{\tilde{\nu}}_{\iota/2} \), which is a smooth bounded open subset of \( \R^{\tilde{\nu}} \).
    As explained above, the map \( v = u \circ \upPi \) belongs to \( W^{s,p}(\Omega;\Nc) \).
    Therefore, Theorem~\ref{thm:density_class_R_segment_condition} ensures the existence of a sequence \( (v_{n})_{n \in \N} \) of maps in \( \Rc_{\tilde{\nu}-[sp]-1}(\Omega;\Nc) \) 
    converging to \( v \) in \( W^{s,p}(\Omega) \) as \( n \to +\infty \).
    Invoking Lemma~\ref{lemma:transversality}, we deduce that for almost every \( a \in B^{\tilde{\nu}}_{\iota/2} \), the map \( u_{n,a} = \tau_{a}(v_{n})_{\vert \Mc} \colon \Mc \to \Nc \) belongs to the class \( \Rc_{m-[sp]-1}(\Mc;\Nc) \).
    Here, we recall that the translation \( \tau_{a}(v_{n}) \) is defined by \( \tau_{a}(v_{n})(x) = v_{n}(x+a) \).
    Indeed, the first part of Lemma~\ref{lemma:transversality} ensures that the singular set of \( u_{n,a} \) is as in the definition of the class \( \Rc_{m-[sp]-1} \).
    On the other hand, the distance estimate in Lemma~\ref{lemma:transversality} implies that the estimates on the derivatives of \( u_{n,a} \) are satisfied.
    
    Moreover, using a slicing argument, we find that for almost every \( a \in B^{\tilde{\nu}}_{\iota/2} \), up to extraction of a subsequence, \( (u_{n,a})_{n \in \N} \) converges in \( W^{s,p}(\Mc) \) to the map \( \tau_{a}(v)_{\vert \Mc} \).
    This can be seen, for instance, using the theory of Fuglede maps presented in Section~\ref{sect:opening}.
    Indeed, consider a summable map \( w \colon \Omega \to [0,+\infty] \), and let \( i \colon \Mc \to \Omega \) be the inclusion map.
    We observe that \( \tau_{a}(v)_{\vert \Mc} = v \circ (i+a) \), and we estimate 
    \[
        \int_{B^{\tilde{\nu}}_{\iota/2}} \int_{\Mc} w(i(x)+a)\,\d x \d a
        = \int_{\Mc} \int_{B^{\tilde{\nu}}_{\iota/2}(x)} w(a)\,\d a\d x
        \leq \lvert \Mc \rvert \lVert w \rVert_{L^{1}(\Omega)}
        < +\infty.
    \]
    Therefore, for almost every \( a \in B^{\tilde{\nu}}_{\iota/2} \), \( w \circ (i+a) \) is summable on \( \Mc \).
    If we now choose \( w \) to be a detector for \( W^{s,p} \) convergence, then, up to a subsequence independent of \( a \), we have 
    \[
    	u_{n,a} = v_{n} \circ (i+a) \to v \circ (i+a) = \tau_{a}(v)_{\vert \Mc}
    	\quad
    	\text{in \( W^{s,p}(\Mc) \) as \( n \to +\infty \).}
    \]
    
    On the other hand, by the continuity of translations in \( W^{s,p} \), we know that \( \tau_{a}(v)_{\vert \Mc} \to v_{\vert \Mc} = u \) in \( W^{s,p}(\Mc) \) as \( a \to 0 \) (more precisely, we should rely on an argument in the spirit of~\cite{BPVS_screening}*{Proposition~2.4}, since there is again a slicing involved here).
    
    We conclude the proof by invoking a diagonal argument: choosing a suitable sequence \( (a_{n})_{n \in \N} \) in \( B^{\tilde{\nu}}_{\iota/2} \) such that \( a_{n} \to 0 \),
    the maps \( u_{n} = u_{n,a_{n}} \) belong to \( \Rc_{m-[sp]-1}(\Mc;\Nc) \) and converge to \( u \) in \( W^{s,p}(\Mc) \) as \( n \to +\infty \).
\end{proof}

We now prove Lemma~\ref{lemma:transversality}.

\begin{proof}[Proof of Lemma~\ref{lemma:transversality}]
    Define \( \upPsi \colon \Mc \times \Sigma \to \R^{\tilde{\nu}} \) by 
    \[
        \upPsi(x,z) = x-z.
    \]
    The map \( \upPsi \) is a smooth map between smooth manifolds.
    Therefore, Sard's lemma ensures that for almost every \( a \in \R^{\tilde{\nu}} \), the linear map 
    \( D\upPsi(x,z) \colon T_{x}\Mc \times T_{z}\Sigma \to \R^{\tilde{\nu}} \) is surjective for every \( (x,z) \in \upPsi^{-1}(\{a\}) \).
    For such \( a \), we compute 
    \[
        \R^{\tilde{\nu}} 
        = D\upPsi(x,z)[T_{x}\Mc \times T_{z}\Sigma]
        = T_{x}\Mc + T_{z}\Sigma.
    \]
    Moreover we observe that \( (x,z) \in \upPsi^{-1}(\{a\}) \) if and only if \( x-z = a \).
    This shows that 
    \[
    	\R^{\tilde{\nu}} = T_{x}\Mc + T_{z}\Sigma
    	\quad
    	\text{for every \( x = z+a \in \Sigma+a \).}
    \]
    Otherwise stated, for almost every \( a \in \R^{\tilde{\nu}} \), \( \Mc \) and \( \Sigma+a \) are transversal,
    which implies that \( \Mc \cap (\Sigma+a) \) is a smooth submanifold of \( \Mc \) of dimension \( m-\tilde{\nu}+\ell \);
    see e.g.~\cite{warner}*{Theorem~1.39}.
    This concludes the proof of the first part of the lemma.
    
    We now turn to the distance estimate.
    Without loss of generality, we may restrict ourselves to prove the estimate when \( a = 0 \).
    Let \( y \in \Mc \cap \Sigma \).
    Since \( \Mc \) and \( \Sigma \) intersect transversely, after a suitable rotation followed by a translation -- which do not modify the distances -- we may assume that \( y = 0 \), and that there exist \( \delta > 0 \) and \( h > 0 \) such that 
    \( \Sigma = \{0\}^{\tilde{\nu}-\ell} \times \R^{\ell} \) and \( \Mc \cap (B_{\delta}^{m} \times (-h,h)^{\tilde{\nu}-m}) \) is the graph of a smooth map \( \phi \colon B_{\delta}^{m} \to (-h,h)^{\tilde{\nu}-m} \).
    Denote by \( \pi_{2} \colon \R^{\tilde{\nu}}  \to \R^{\ell} \) the projection onto the \( \ell \) last variables, which corresponds to the orthogonal projection onto \( \Sigma \), and by \( \pi_{1} \colon \R^{\tilde{\nu}} \to \R^{m} \) the projection onto the \( m \) first variables.
    We observe that, for \( x=(\pi_{1}(x),\phi(\pi_{1}(x))) \in \Mc \cap (B_{\delta}^{m} \times (-h,h)^{\tilde{\nu}-m}) \), we have 
    \[
    	\dist{(x,\Sigma)}
    	= 
    	\lvert x-(0,\pi_{2}(x)) \rvert,
    \]
    while 
    \begin{align*}
    	\dist{(x,\Mc \cap \Sigma)}
    	&\leq
    	\lvert (\pi_{1}(x),\phi(\pi_{1}(x))) - (\pi_{1}((0,\pi_{2}(x))),\phi(\pi_{1}((0,\pi_{2}(x)))))\rvert \\
    	&\leq 
    	(1+\lvert \phi \rvert_{\Cc^{0,1}})\dist{(x,\Sigma)}.
    \end{align*}

	We conclude by using a finite covering argument.
	Indeed, since \( \Mc\cap \Sigma \) is compact, we may cover it by a finite number of cylindrical domains as above.
	We obtain a neighborhood \( U_{\varepsilon} = \Mc\cap\Sigma + B^{\tilde{\nu}}_{\varepsilon} \) of radius \( \varepsilon > 0 \) such that, for every \( x \in U_{\varepsilon} \), we have 
	\[
		\dist{(x,\Mc \cap \Sigma)}
		\leq 
		\C\dist{(x,\Sigma)}.
	\]
	On the other hand, for points \( x \in \Mc \) with \( x \notin U_{\varepsilon} \), we have \( \dist{(x,\Sigma)} \geq \C > 0 \),
	while \( \dist{(x,\Mc \cap \Sigma)} \leq \C \) since \( \Mc \cap \Sigma \neq \varnothing \).
	This completes the proof of the lemma.
\end{proof}

Finally, we give the proof of Theorem~\ref{thm:density_class_R_manifold} in the case where \( \Mc \) has non-empty boundary.

\begin{proof}[Proof of Theorem~\ref{thm:density_class_R_manifold} when \( \Mc \) has non-empty boundary]
	The key idea is to view \( \Mc \), or more precisely any compact subset of the interior of \( \Mc \), as a subset of a smooth manifold \emph{without} boundary, embedded in \( \R^{\tilde{\nu}} \times \R \), identifying \( \R^{\tilde{\nu}} \) with \( \R^{\tilde{\nu}} \times \{0\} \).
	For this, we rely on~\cite{BPVS_density_complete_manifolds}*{Lemma~3.4}, which is a consequence of the collar neighborhood theorem.
	
	Let \( K \) be any compact subset in the relative interior of \( \Mc \).
	From~\cite{BPVS_density_complete_manifolds}*{Lemma~3.4}, we deduce that there exists a smooth compact submanifold \( \tilde{\Mc} \) of \( \R^{\tilde{\nu}} \times \R \), without boundary, such that 
	\[
		K \times \{0\} \subset \tilde{\Mc} 
		\quad
		\text{and}
		\quad
		\pi(\tilde{\Mc}) \subset \Mc,
	\]
	where \( \pi \colon \R^{\tilde{\nu}} \times \R \to \R^{\tilde{\nu}} \) is the projection onto the first \( \tilde{\nu} \) variables.
	
	Let \( u \in W^{s,p}(\Mc;\Nc) \).
	The map \( v = u \circ \pi \) belongs to \( W^{s,p}(\tilde{\Mc},\Nc) \).
	Hence, by Theorem~\ref{thm:density_class_R_manifold} for manifolds without boundary, there exists a sequence \( (v_{n}^{K})_{n \in \N} \) of maps in \( \Rc_{m-[sp]-1}(\tilde{\Mc};\Nc) \) such that \( v_{n}^{K} \to v \) in \( W^{s,p}(\tilde{\Mc}) \).
	In particular, \( (v_{n}^{K})_{\vert K} \to u_{\vert K} \) in \( W^{s,p}(K) \).
	
	Now, we observe that, for every \( \varepsilon > 0 \) sufficiently small, if we take \( K = K_{\varepsilon} \) such that \( \Mc \setminus K_{\varepsilon} \) is contained in a uniform neighborhood of radius \( \varepsilon \) of \( \partial\Mc \), then \( u_{\vert K_{\varepsilon}} \) may be dilated to a map \( u_{\varepsilon} \in W^{s,p}(\Mc;\Nc) \). 
	Moreover, if we denote by \( u_{n,\varepsilon} \) the corresponding dilations of the maps \( (v_{n}^{K_{\varepsilon}})_{\vert K_{\varepsilon}} \), we have both 
	\[
		u_{\varepsilon} \to u
		\quad
		\text{as \( \varepsilon \to 0 \)}
		\quad
		\text{and}
		\quad
		u_{n,\varepsilon} \to u_{\varepsilon}
		\quad
		\text{as \( n \to +\infty \).}
	\]
	We conclude using a diagonal argument.
\end{proof}

\section{Shrinking}
\label{sect:shrinking}

This section is dedicated to the shrinking procedure.
As we explained in Section~\ref{sect:sketch_proof}, shrinking is actually a more involved version of a scaling argument, whose purpose is to modify a given map in order to obtain a better map whose energy is controlled.
The main result of this section is the following proposition, counterpart of~\cite{BPVS_density_higher_order}*{Proposition~8.1} in the fractional setting, which provides the shrinking construction.
We emphasize that, similar to thickening but unlike opening, the map \( \upPhi \) does not depend on the map \( u \in W^{s,p} \) it shall be composed with.

\begin{prop}
\label{prop:main_shrinking}
    Let \( \ell \in \{0,\dots,m-1\} \), \( \eta > 0 \), \( 0 < \mu < \frac{1}{2} \), \( 0 < \tau < \frac{1}{2} \), 
    \( \Kc^{m} \) be a cubication in \( \R^{m} \) of radius \( \eta \), and \( \Tc^{\ell^{\ast}} \) be the dual skeleton of \( \Kc^{\ell} \).
    There exists a smooth map \( \upPhi \colon \R^{m} \to \R^{m} \) such that
    \begin{enumerate}[label=(\roman*)]
        \item\label{item:main_shrinking_injective} \( \upPhi \) is injective;
        \item\label{item:main_shrinking_geometric} for every \( \sigma^{m} \in \Kc^{m} \), \( \upPhi(\sigma^{m}) \subset \sigma^{m} \);
        \item\label{item:main_shrinking_support} \( \Supp\upPhi \subset T^{\ell^{\ast}} + Q^{m}_{2\mu\eta} \) and \( \upPhi(T^{\ell^{\ast}}+Q^{m}_{\tau\mu\eta}) \supset T^{\ell^{\ast}}+Q^{m}_{\mu\eta} \).
    \end{enumerate}
    If in addition \( \ell+1 > sp \), then for every \( u \in W^{s,p}(K^{m};\R^{\nu}) \) and every \( v \in W^{s,p}(K^{m};\R^{\nu}) \) such that 
    \( u = v \) on the complement of \( T^{\ell^{\ast}}+Q^{m}_{\mu\eta} \), we have \( u \circ \upPhi \in W^{s,p}(K^{m};\R^{\nu}) \), and moreover, the following estimates hold:
    \begin{enumerate}[label=(\alph*)]
    	\item\label{item:estimate_main_shrinking_sle1} if \( 0 < s < 1 \), then
	    	\begin{multline*}
		    	(\mu\eta)^{s}\lvert u\circ\upPhi - v \rvert_{W^{s,p}(K^{m})} \\
		    	\leq
		    	C\Bigl((\mu\eta)^{s} \lvert u \rvert_{W^{s,p}(K^{m} \cap (T^{\ell^{\ast}}+Q^{m}_{2\mu\eta}) \setminus (T^{\ell^{\ast}}+Q^{m}_{\mu\eta}))} + \lVert u \rVert_{L^{p}(K^{m} \cap (T^{\ell^{\ast}}+Q^{m}_{2\mu\eta}) \setminus (T^{\ell^{\ast}}+Q^{m}_{\mu\eta}))}\Bigr) \\
		    	+ C\tau^{\frac{\ell+1-sp}{p}}\Bigr((\mu\eta)^{s}\lvert u \rvert_{W^{s,p}(K^{m} \cap (T^{\ell^{\ast}}+Q^{m}_{2\mu\eta}))} + \lVert u \rVert_{L^{p}(K^{m} \cap (T^{\ell^{\ast}}+Q^{m}_{2\mu\eta}))} \Bigr) \\
		    	+  C(\mu\eta)^{s} \lvert v \rvert_{W^{s,p}(K^{m} \cap (T^{\ell^{\ast}}+Q^{m}_{2\mu\eta}))} + C\lVert v \rVert_{L^{p}(K^{m} \cap (T^{\ell^{\ast}}+Q^{m}_{2\mu\eta}))};
	    	\end{multline*}
	    \item\label{item:estimate_main_shrinking_sgeq1_integer} if \( s \geq 1 \), then for every \( j \in \{1,\dots,k\} \),
	    	\begin{multline*}
		    	(\mu\eta)^{j}\lVert D^{j}(u\circ\upPhi) - D^{j}v \rVert_{L^{p}(K^{m})}
		    	\leq
		    	C\sum_{i=1}^{j}(\mu\eta)^{i} \lVert D^{i}u \rVert_{L^{p}(K^{m} \cap (T^{\ell^{\ast}}+Q^{m}_{2\mu\eta}) \setminus (T^{\ell^{\ast}}+Q^{m}_{\mu\eta}))} \\
		    	+ C\tau^{\frac{\ell+1-sp}{p}}\sum_{i=1}^{j} (\mu\eta)^{i}\lVert D^{i}u \rVert_{L^{p}(K^{m} \cap (T^{\ell^{\ast}}+Q^{m}_{2\mu\eta}))}
		    	+ C(\mu\eta)^{j} \lVert D^{j}v \rVert_{L^{p}(K^{m} \cap (T^{\ell^{\ast}}+Q^{m}_{2\mu\eta}))};
	    	\end{multline*}
	    \item\label{item:estimate_main_shrinking_sgeq1_frac} if \( s \geq 1 \) and \( \sigma \neq 0 \), then for every \( j \in \{1,\dots,k\} \),
	    	\begin{multline*}
		    	(\mu\eta)^{j+\sigma}\lvert D^{j}(u\circ\upPhi) - D^{j}v \rvert_{W^{\sigma,p}(K^{m})} \\
		    	\leq
		    	C\sum_{i=1}^{j}\Bigl((\mu\eta)^{i} \lVert D^{i}u \rVert_{L^{p}(K^{m} \cap (T^{\ell^{\ast}}+Q^{m}_{2\mu\eta}) \setminus (T^{\ell^{\ast}}+Q^{m}_{\mu\eta}))} + (\mu\eta)^{i+\sigma} \lvert D^{i}u \rvert_{W^{\sigma,p}(K^{m} \cap (T^{\ell^{\ast}}+Q^{m}_{2\mu\eta}) \setminus (T^{\ell^{\ast}}+Q^{m}_{\mu\eta}))}\Bigr) \\
		    	+ C\tau^{\frac{\ell+1-sp}{p}}\sum_{i=1}^{j}\Bigr((\mu\eta)^{i}\lVert D^{i}u \rVert_{L^{p}(K^{m} \cap (T^{\ell^{\ast}}+Q^{m}_{2\mu\eta}))} + (\mu\eta)^{i+\sigma}\lvert D^{i}u \rvert_{W^{\sigma,p}(K^{m} \cap (T^{\ell^{\ast}}+Q^{m}_{2\mu\eta}))} \Bigr) \\
		    	+ C(\mu\eta)^{j} \lVert D^{j}v \rVert_{L^{p}(K^{m} \cap (T^{\ell^{\ast}}+Q^{m}_{2\mu\eta}))} + C(\mu\eta)^{j+\sigma} \lvert D^{j}v \rvert_{W^{\sigma,p}(K^{m} \cap (T^{\ell^{\ast}}+Q^{m}_{2\mu\eta}))};
	    	\end{multline*}
	    \item\label{item:estimate_main_shrinking_all} for \( 0 < s < +\infty \),
	    	\begin{multline*}
		    	\lVert u \circ \upPhi - v \rVert_{L^{p}(K^{m})}
		    	\leq 
		    	C\lVert u \rVert_{L^{p}(K^{m} \cap (T^{\ell^{\ast}}+Q^{m}_{2\mu\eta}) \setminus (T^{\ell^{\ast}}+Q^{m}_{\mu\eta}))} \\
		    	+ C\tau^{\frac{\ell+1-sp}{p}}\lVert u \rVert_{L^{p}(K^{m} \cap (T^{\ell^{\ast}}+Q^{m}_{2\mu\eta}))} 
		    	+ C\lVert v \rVert_{L^{p}(K^{m} \cap (T^{\ell^{\ast}}+Q^{m}_{2\mu\eta}))};
	    	\end{multline*}
    \end{enumerate}
    for some constant \( C > 0 \) depending on \( m \), \( s \), and \( p \).
\end{prop}

For integer order estimates, we could avoid mentioning the map \( v \) in the statement of Proposition~\ref{prop:main_shrinking}
and only establish energy estimates for \( u \circ \upPhi \) alone on \( K^{m} \cap (T^{\ell^{\ast}}+Q^{m}_{2\mu\eta}) \), 
as in~\cite{BPVS_density_higher_order}, as the estimates above then follow from the assumption \( u = v \) outside of \( T^{\ell^{\ast}}+Q^{m}_{2\mu\eta} \) 
using the additivity of the integral.
However, for fractional order estimates, we face the usual problem linked to the lack of additivity of the Gagliardo seminorm.

We pause here to explain how Proposition~\ref{prop:main_shrinking} will be used in the proof of Theorem~\ref{thm:density_smooth_functions}.
Given \( u \) and \( v \) as above, Proposition~\ref{prop:main_shrinking} allows us to control, \emph{via} a suitable choice of \( \tau > 0 \), the energy of \( u \circ \upPhi \) in terms of the energy of \( v \) alone.
Indeed, given \( \mu > 0 \) and \( \varepsilon > 0 \), if we choose \( \tau_{\mu} \) sufficiently small -- depending on \( u \) and \( v \) -- then, using the fact that \( u = v \) outside of \( T^{\ell^{\ast}} + Q^{m}_{\mu\eta} \),
we find
\begin{enumerate}[label=(\alph*)]
	\item\label{item:below_main_shrinking_sle1} if \( 0 < s < 1 \), then 
		\begin{multline*}
			(\mu\eta)^{s}\lvert u\circ\upPhi - v \rvert_{W^{s,p}(K^{m})} \\
			\leq
			C'\Bigl((\mu\eta)^{s}\lvert v \rvert_{W^{s,p}(K^{m} \cap (T^{\ell^{\ast}}+Q^{m}_{2\mu\eta}))} + \lVert v \rVert_{L^{p}(K^{m} \cap (T^{\ell^{\ast}}+Q^{m}_{2\mu\eta}))} \Bigr) + \varepsilon;
		\end{multline*}
	\item\label{item:below_main_shrinking_sgeq1_integer} if \( s \geq 1 \), then for every \( j \in \{1,\dots,k\} \),
		\[
			(\mu\eta)^{j}\lVert D^{j}(u\circ\upPhi) - D^{j}v \rVert_{L^{p}(K^{m})}
			\leq 
			C'\sum_{i=1}^{j} (\mu\eta)^{i}\lVert D^{i}v \rVert_{L^{p}(K^{m} \cap (T^{\ell^{\ast}}+Q^{m}_{2\mu\eta}))} + \varepsilon;
		\]
	\item\label{item:below_main_shrinking_sgeq1_frac} if \( s \geq 1 \) and \( \sigma \neq 0 \), then for every \( j \in \{1,\dots,k\} \),
		\begin{multline*}
			(\mu\eta)^{j+\sigma}\lvert D^{j}(u\circ\upPhi) - D^{j}v \rvert_{W^{\sigma,p}(K^{m})} \\
			\leq
			C'\sum_{i=1}^{j}\Bigl((\mu\eta)^{i}\lVert D^{i}v \rVert_{L^{p}(K^{m} \cap (T^{\ell^{\ast}}+Q^{m}_{2\mu\eta}))} + (\mu\eta)^{i+\sigma}\lvert D^{i}v \rvert_{W^{\sigma,p}(K^{m} \cap (T^{\ell^{\ast}}+Q^{m}_{2\mu\eta}))} \Bigr) + \varepsilon;
		\end{multline*}
	\item\label{item:below_main_shrinking_all} for every \( 0 < s < +\infty \),
		\[
			\lVert u \circ \upPhi - v \rVert_{L^{p}(K^{m})}
			\leq 
			C'\lVert v \rVert_{L^{p}(K^{m} \cap (T^{\ell^{\ast}}+Q^{m}_{2\mu\eta}))} + \varepsilon;
		\]
\end{enumerate}
for some constant \( C' > 0 \) depending on \( m \), \( s \), and \( p \).
The reason for the extra term \( \varepsilon \) in the right-hand sides of the above estimates is to cover the case where \( u \) is identically \( 0 \) while \( v \) does not vanish on \( T^{d^{\ast}} + Q^{m}_{\mu\eta} \).
Estimates~\ref{item:below_main_shrinking_sle1} to~\ref{item:below_main_shrinking_all} will be used in the proof of Theorem~\ref{thm:density_smooth_functions}.

This section is organized as follows.
In a first time, we explain the construction of the building blocks for shrinking, and we prove their geometric properties.
Then we state the analytic estimates satisfied by the composition of a \( W^{s,p} \) map \( u \) with those building blocks.
Finally, we explain how to suitably combine the building blocks in order to obtain the global shrinking construction, along with the required properties.

We start with the construction of the building blocks for shrinking, which is very similar to thickening.
Therefore, in this section, we shall follow an analogous path to the one in Section~\ref{sect:thickening}.
We start by introducing some additional notation, similar to Sections~\ref{sect:opening} and~\ref{sect:thickening}.
Let \( 0 < \ulmu < \mu < \olmu < 1 \) and \( 0 < \tau < \ulmu/\mu \) be fixed.
We set 
\[
    B_{1} = B^{d}_{\tau\mu\eta} \times Q^{m-d}_{(1-\olmu)\eta},
    \quad 
    Q_{2} = Q^{d}_{\ulmu\eta} \times Q^{m-d}_{(1-\olmu)\eta},
    \quad 
    Q_{3} = Q^{d}_{\mu\eta} \times Q^{m-d}_{(1-\mu)\eta}.
\]
Note that \( B_{1} \subset Q_{2} \subset Q_{3} \).
The rectangle \( Q_{3} \) contains the geometric support of the building block \( \upPhi \), that is, \( \upPhi = \id \) outside of \( Q_{3} \).
The rectangle \( Q_{2} \) is shrinked into the cylinder \( B_{1} \): we have \( \upPhi(B_{1}) \supset Q_{2} \).
As usual, the region in between serves as a transition region.

As we did for thickening, we split the construction of the building block \( \upPhi \) into two parts.
First, we deal with the geometric properties that need to be satisfied by \( \upPhi \) independently of the map \( u \), and then, we move to the Sobolev estimates satisfied by \( u \circ \upPhi \).
We take \( \upPhi \) to be exactly the map given by~\cite{BPVS_density_higher_order}*{Proposition~8.3}, and we therefore only recall briefly how this map is built,
referring the reader to~\cite{BPVS_density_higher_order} for the details.
Once again, the main change in our approach is that we establish the Sobolev estimates first for the building blocks, and then we glue them together in order to obtain the estimates given by Proposition~\ref{prop:main_shrinking}.

Analogously to Section~\ref{sect:thickening}, we define \( \zeta \colon \R^{m} \to \R \) by 
\begin{equation}
\label{eq:definition_zeta_shrinking}
    \zeta(x) = \sqrt{\lvert x' \rvert^{2}+(\mu\eta)^{2}\theta\Bigl(\frac{x''}{\mu\eta}\Bigr)+(\mu\eta)^{2}\varepsilon\tau^{2}}
\end{equation}
for every \( x = (x',x'') \in \R^{d} \times \R^{m-d} \).
Here, \( \theta \colon \R^{m-d} \to \R \) is defined similarly as in Section~\ref{sect:thickening}.
We choose \( 1 < q < +\infty \) sufficiently large so that there exist \( 0 < r_{1} < r_{2} \) satisfying 
\[
	Q^{m-d}_{(1-\olmu)/\mu} 
	\subset 
	\{ x'' \in \R^{m-d} \mathpunct{:} \lvert x'' \rvert_{q} < r_{1} \}
	\subset
	\{ x'' \in \R^{m-d} \mathpunct{:} \lvert x'' \rvert_{q} < r_{2} \}
	\subset
	Q^{m-d}_{(1-\mu)/\mu}.
\] 
Then, we pick a nondecreasing smooth map \( \tilde{\theta} \colon \R_{+} \to [0,1] \) such that \( \tilde{\theta}(r) = 0 \) if \( 0 \leq r \leq r_{1} \) and \( \tilde{\theta}(r) = 1 \) if \( r \geq r_{2} \).
Finally, we let \( \theta(x'')  = \tilde{\theta}(\lvert x'' \rvert_{q}) \).
With this definition, the map \( \theta \) is smooth and satisfies \( \theta(x'') = 0 \) if \( x'' \in Q^{m-d}_{(1-\olmu)/\mu} \) and \( \theta(x'') = 1 \) if \( x'' \in \R^{m-d} \setminus Q^{m-d}_{(1-\mu)/\mu} \).
The number \( \varepsilon > 0 \) is to be determined later on, depending only on \( \ulmu/\mu \).
As we will see in the course of the proof, the extra term involving \( \tau \), which was not present in Section~\ref{sect:thickening}, serves to obtain a desingularized construction.

We are now ready to state the geometric properties of \( \upPhi \), which are the purpose of Proposition~\ref{prop:block_shrinking_geometric} below.

\begin{prop}
\label{prop:block_shrinking_geometric}
Let \( d \in \{1,\dots,m\} \), \( \eta > 0 \), \( 0 < \ulmu < \mu < \olmu < 1 \), and \(0 < \tau < \ulmu/\mu \).
There exists a smooth function \( \upPhi \colon \R^{m} \to \R^{m} \) of the form \( \upPhi(x) = (\lambda(x)x',x'') \), with \( \lambda \colon \R^{m} \to [1,+\infty) \), and such that 
\begin{enumerate}[label=(\roman*)]
    \item\label{item:shrinking_geometric_injective} \( \upPhi \) is injective;
    \item\label{item:shrinking_geometric_support} \( \Supp \upPhi \subset Q_{3} \);
    \item\label{item:shrinking_geometric_geometry} \( \upPhi(B_{1}) \supset Q_{2} \);
    \item\label{item:shrinking_geometric_singularity} for every \( x \in Q_{3} \), 
        \[
            \lvert D^{j}\upPhi(x) \rvert \leq C\frac{\mu\eta}{\zeta^{j}(x)}
            \quad
            \text{for every \( j \in \N_{\ast} \),}
        \]
        and for every \( x \in \R^{m} \), 
        \[
            \lvert D^{j}\upPhi(x) \rvert \leq C\frac{(\mu\eta)^{1-j}}{\tau^{j}}
            \quad
            \text{for every \( j \in \N_{\ast} \),}
        \]  
        for some constant \( C > 0 \) depending on \( j \), \( m \), \( \mu/\olmu \) and \( \ulmu/\mu \);
    \item\label{item:shrinking_geometric_jacobian} for every \( x \in \R^{m} \),
        \[
            \jac\upPhi(x) \geq C'\frac{(\mu\eta)^{\beta}}{\zeta^{\beta}(x)}
            \quad
            \text{for every \( 0 < \beta < d \),}
        \]
        and for every \( x \in B_{1} \),
        \[
            \jac\upPhi(x) \geq C'\frac{1}{\tau^{d}},
        \]
        for some constant \( C' > 0 \) depending on \( \beta \), \( j \), \( m \), \( \mu/\olmu \) and \( \ulmu/\mu \).
\end{enumerate}
\end{prop}
\begin{proof}
    As we announced, we use the same construction as in~\cite{BPVS_density_higher_order}*{Proposition~8.3}.
    Similar to thickening, we start by constructing an intermediate map \( \upPsi \colon \R^{m} \to \R^{m} \) which satisfies 
    the conclusion of Proposition~\ref{prop:block_shrinking_geometric} with the rectangles \( Q_{i} \) replaced by the cylinders \( B_{i} \) defined as 
    \[
        B_{2} = B^{d}_{\ulmu\eta} \times Q^{m-d}_{(1-\olmu)\eta},
        \quad 
        B_{3} = B^{d}_{\mu\eta} \times Q^{m-d}_{(1-\mu)\eta}.
    \]
    It will then suffice to compose \( \upPsi \) with a suitable diffeomorphism \( \upTheta \colon \R^{m} \to \R^{m} \) 
    dilating \( B_{2} \) to a set containing \( Q_{2} \) in order to obtain the desired map \( \upPhi \).
    
    We let \( \varphi \colon (0,+\infty) \to [1,+\infty) \) be a smooth function such that 
    \begin{enumerate}[label=(\alph*)]
        \item\label{item:varphi_shrinking_rleqtau} for \( 0 < r \leq \tau\sqrt{1+\varepsilon} \),
            \[
                \varphi(r) = \frac{\ulmu/\mu}{r}\sqrt{1+\varepsilon}\Bigl(1+\frac{b}{\ln\frac{1}{r}}\Bigr);
            \]
        \item\label{item:varphi_shrinking_rgeq1} for \( r \geq 1 \), \( \varphi(r) = 1 \);
        \item\label{item:varphi_shrinking_increasing} the function \( r \in (0,+\infty) \mapsto r\varphi(r) \) is increasing.
    \end{enumerate}
    This is possible provided that we choose \( \varepsilon \) such that 
    \[
        (\ulmu/\mu)\sqrt{1+\varepsilon} < 1 
    \]
    and then \( b > 0 \) such that 
    \[
        (\ulmu/\mu)\sqrt{1+\varepsilon}\Bigl(1+\frac{b}{\ln\frac{1}{(\ulmu/\mu)\sqrt{1+\varepsilon}}}\Bigr)
        <
        1.
    \]
    Then, we define \( \lambda \colon \R^{m} \to [1,+\infty) \) by 
    \[
        \lambda(x) = \varphi\Bigl(\frac{\zeta(x)}{\mu\eta}\Bigr),
    \]
    and finally 
    \[
        \upPsi(x',x'') = (\lambda(x',x'')x',x'').
    \]

    The injectivity of \( \upPsi \) relies on assumption~\ref{item:varphi_shrinking_increasing} on \( \varphi \).
    The fact that \( \Supp\upPsi \subset B_{3} \) uses assumption~\ref{item:varphi_shrinking_rgeq1} on \( \varphi \), 
    observing that \( \zeta(x) \geq \mu\eta \) whenever \( x \in \R^{m} \setminus B_{3} \),
    and hence \( \lambda(x) = 1 \).
    To prove~\ref{item:shrinking_geometric_geometry}, note that if \( x = (x',x'') \in B_{1} \) and \( t \geq 0 \), we have 
    \[
        \upPsi(tx',x'')
        =
        \biggl(t\varphi\biggl(\sqrt{t^{2}\Bigl\lvert \frac{x'}{\mu\eta}\Bigr\rvert^{2}+\varepsilon\tau^{2}}\biggr)x',x''\biggr),
    \]
    where we used the fact that \( \theta \) vanishes inside of \( Q^{m-d}_{\frac{1-\olmu}{\mu}} \).
    For \( t = 0 \), the factor in front of \( x' \) vanishes, while for \( t = \tau \), it is larger than \( \frac{\ulmu\eta}{\lvert x' \rvert} \geq 1 \).
    We conclude by invoking the intermediate value theorem.
    The proof of~\ref{item:shrinking_geometric_singularity} amounts to estimate \( \lvert D^{j}\lambda \rvert \) using the Faà di Bruno formula, and then conclude using Leibniz's rule.
    We obtain the second estimate from the first one by noting that \( \zeta \geq (\mu\eta)\sqrt{\varepsilon}\tau \).
    The proof of~\ref{item:shrinking_geometric_jacobian} again involves explicitly computing \( \jac\upPsi \) as the determinant of a perturbation of a linear map,
    and then estimating the obtained expression.
    The second estimate relies on the fact that if \( x = (x',x'') \in B_{1} \), then \( \lvert x' \rvert \leq (\mu\eta)\tau \) and \( \theta\Bigl(\frac{x''}{\mu\eta}\Bigr) = 0 \),
    whence \( \zeta(x) \leq (\mu\eta)\sqrt{1+\varepsilon}\tau \).
    We refer the reader to~\cite{BPVS_density_higher_order}*{Lemma~8.5} for the details.

    We then let \( \upTheta \colon \R^{m} \to \R^{m} \) be a smooth diffeomorphism also of the form \( \upTheta(x) = (\tilde{\lambda}(x)x',x'') \), with \( \tilde{\lambda} \colon \R^{m} \to [1,+\infty) \), such that \( \upTheta \) is supported in \( Q_{3} \), maps \( B_{2} \) on a set containing \( Q_{2} \), and satisfies the estimates 
    \[
        (\mu\eta)^{j-1}\lvert D^{j}\upTheta \rvert \leq \C
        \quad 
        \text{and}
        \quad
        0 < \C \leq \jac\upTheta \leq \C 
        \quad
        \text{on \( \R^{m} \);}
    \]
    see~\cite{BPVS_density_higher_order}*{Lemma~8.4}.
    Using the composition formula for the Jacobian and the Faà di Bruno formula, we conclude, as for thickening, that \( \upPhi = \upTheta \circ \upPsi \) is the desired map.
\end{proof}

We now turn to the Sobolev estimates satisfied by \( u \circ \upPhi \).

\begin{prop}
\label{prop:block_shrinking_analytic}
    Let \( d > sp \).
    Let \( \upPhi \) be as in Proposition~\ref{prop:block_shrinking_geometric}.
    Let \( \omega \subset \R^{m} \) be such that \( Q_{2} \subset \omega \subset B^{m}_{c\mu\eta} \) for some \( c > 0 \), and assume that there exists \( c' > 0 \) such that 
    \begin{equation}
    \label{eq:geometric_assumptions_block_shrinking}
    	\lvert B^{m}_{\lambda}(z) \cap (\omega\setminus Q_{2}) \rvert \geq c'\lambda^{m}
    	\quad
    	\text{for every \( z \in \omega \setminus Q_{2} \) and \( 0 < \lambda \leq \frac{1}{2}\diam\omega \).}
    \end{equation}
    For every \( u \in W^{s,p}(\upPhi^{-1}(\omega);\R^{\nu}) \), we have \( u \circ \upPhi \in W^{s,p}(\upPhi^{-1}(\omega);\R^{\nu}) \), and moreover, the following estimates hold:
    \begin{enumerate}[label=(\alph*)]
    	\item\label{item:estimate_block_shrinking_sle1} if \( 0 < s < 1 \), then 
    		\[
	    		\lvert u \circ \upPhi \rvert_{W^{s,p}(\upPhi^{-1}(\omega))}
	    		\leq 
	    		C\lvert u \rvert_{W^{s,p}(\omega\setminus Q_{2})}
	    		+ C\tau^{\frac{d-sp}{p}}\lvert u \rvert_{W^{s,p}(\omega)};
    		\]
    	\item\label{item:estimate_block_shrinking_sgeq1_integer} if \( s \geq 1 \), then for every \( j \in \{1,\dots,k\} \),
    		\[
	    		(\mu\eta)^{j}\lVert D^{j}(u \circ \upPhi) \rVert_{L^{p}(\upPhi^{-1}(\omega))}
	    		\leq
	    		C\sum_{i=1}^{j}(\mu\eta)^{i}\lVert D^{i}u \rVert_{L^{p}(\omega\setminus Q_{2})}
	    		+ C\tau^{\frac{d-jp}{p}}\sum_{i=1}^{j}(\mu\eta)^{i}\lVert D^{i}u \rVert_{L^{p}(\omega)};
    		\]
    	\item\label{item:estimate_block_shrinking_sgeq1_frac} if \( s \geq 1 \) and \( \sigma \neq 0 \), then for every \( j \in \{1,\dots,k\} \),
    		\begin{align*}
	    		(\mu\eta)^{j+\sigma}\lvert D^{j}(u \circ \upPhi) \rvert_{W^{\sigma,p}(\upPhi^{-1}(\omega))}
	    		&\leq
	    		C\sum_{i=1}^{j}\Bigl((\mu\eta)^{i}\lVert D^{i}u \rVert_{L^{p}(\omega\setminus Q_{2})} 
	    		+ (\mu\eta)^{i+\sigma}\lvert D^{i}u \rvert_{W^{\sigma,p}(\omega\setminus Q_{2})}\Bigr) \\
	    		&\quad+ C\tau^{\frac{d-(j+\sigma)p}{p}}\sum_{i=1}^{j}\Bigl((\mu\eta)^{i}\lVert D^{i}u \rVert_{L^{p}(\omega))} 
	    		+ (\mu\eta)^{i+\sigma}\lvert D^{i}u \rvert_{W^{\sigma,p}(\omega)}\Bigr);
    		\end{align*}
    	\item\label{item:estimate_block_shrinking_all} for every \( 0 < s < +\infty \),
    		\[
	    		\lVert u \circ \upPhi \rVert_{L^{p}(\upPhi^{-1}(\omega))}
	    		\leq 
	    		C\lVert u \rVert_{L^{p}(\omega\setminus Q_{2})}
	    		+ C\tau^{\frac{d}{p}}\lVert u \rVert_{L^{p}(\omega)};
    		\]
    \end{enumerate}
	for some constant \( C > 0 \) depending on \( s \), \( m \), \( p \), \( c \), \( c' \), \( \ulmu/\mu \), and \( \mu/\olmu \).
\end{prop}

We encounter again the assumption that balls centered at a point of \( \omega \) significantly intersect \( \omega \).
We call the attention of the reader to the fact that, in the proof of Proposition~\ref{prop:main_shrinking}, Proposition~\ref{prop:block_shrinking_analytic} will be applied with \( \omega \) being a domain more complicated than just a rectangle.
This contrasts with the situation encountered in Sections~\ref{sect:opening} and~\ref{sect:thickening}.

In the proof of Proposition~\ref{prop:block_shrinking_analytic}, we need the counterpart of Lemma~\ref{lemma:path_mean_value} for the map \( \zeta \) used for shrinking.
The proof is the same as the proof of Lemma~\ref{lemma:path_mean_value}, since both constructions are identical up to an additive constant under the square root,
and is therefore omitted.

\begin{lemme}
\label{lemma:path_mean_value_shrinking}
    For every \( x \), \( y \in \R^{m} \), there exists a Lipschitz path \( \gamma \colon [0,1] \to \R^{m} \) from \( x \) to \( y \) such that 
    \[ 
        \lvert \gamma \rvert_{\Cc^{0,1}([0,1])} \leq C\lvert x-y \rvert
    \] 
    for some constant \( C > 0 \) depending only on \( m \), and such that \( \zeta \geq \min(\zeta(x),\zeta(y)) \) along \( \gamma \), where \( \zeta \) is the map defined in~\eqref{eq:definition_zeta_shrinking}.
\end{lemme}

We are now ready to prove Proposition~\ref{prop:block_shrinking_analytic}.

\begin{proof}[Proof of Proposition~\ref{prop:block_shrinking_analytic}]
    As for thickening, the integer order estimate when \( s \geq 1 \) is proved exactly as~\cite{BPVS_density_higher_order}*{Corollary~8.2},
    but is presented here as a prelude for the more involved fractional order case.

    By the Faà di Bruno formula, we estimate
    \[	
        \lvert D^{j}(u \circ \upPhi)(x) \rvert^{p}
        \leq
        \C\sum_{i=1}^{j}\sum_{\substack{1\leq t_{1}\leq \cdots \leq t_{i} \\ t_{1} + \cdots + t_{i} = j}} \lvert D^{i}u(\upPhi(x)) \rvert^{p} \lvert D^{t_{1}}\upPhi(x) \rvert^{p} \cdots \lvert D^{t_{i}}\upPhi(x) \rvert^{p}
    \]
    for every \( j \in \{1,\dots,k\} \) and \( x \in \upPhi^{-1}(\omega) \).
    Let \( 0 < \beta < d \).
    Using the estimates on the derivatives and the Jacobian of \( \upPhi \), we find
    \( \lvert D^{t_{l}}\upPhi \rvert \leq \C\frac{(\jac\upPhi)^{\frac{t_{l}}{\beta}}}{(\mu\eta)^{t_{l}-1}} \), and therefore 
    \begin{multline*}
        \lvert D^{j}(u \circ \upPhi)(x) \rvert^{p}
        \leq
        \C\sum_{i=1}^{j}\sum_{\substack{1\leq t_{1}\leq \cdots \leq t_{i} \\ t_{1} + \cdots + t_{i} = j}}
        \lvert D^{i}u(\upPhi(x)) \rvert^{p}\frac{(\jac\upPhi(x))^{\frac{t_{1}p}{\beta}}}{(\mu\eta)^{(t_{1}-1)p}}\cdots\frac{(\jac\upPhi(x))^{\frac{t_{i}p}{\beta}}}{(\mu\eta)^{(t_{i}-1)p}} \\
        \leq \Cl{cst:fixed_block_shrinking}\sum_{i=1}^{j}\lvert D^{i}u(\upPhi(x)) \rvert^{p}\frac{(\jac\upPhi(x))^{\frac{jp}{\beta}}}{(\mu\eta)^{(j-i)p}}.
    \end{multline*}
    Since \( jp \leq sp < d \), we may choose \( \beta = jp \).
    Hence,
    \[
        \lvert D^{j}(u \circ \upPhi)(x) \rvert^{p}
        \leq
        \Cr{cst:fixed_block_shrinking}\sum_{i=1}^{j}\lvert D^{i}u(\upPhi(x)) \rvert^{p}\frac{\jac\upPhi(x)}{(\mu\eta)^{(j-i)p}}.
    \]
    Since \( \upPhi \) is injective, the change of variable theorem ensures that
    \begin{multline*}
        \int_{\upPhi^{-1}(\omega\setminus Q_{2})} (\mu\eta)^{jp}\lvert D^{j}(u \circ \upPhi)\rvert^{p}
        \leq
        \int_{\upPhi^{-1}(\omega\setminus Q_{2})}\Cr{cst:fixed_block_shrinking}\sum_{i=1}^{j}(\mu\eta)^{ip}\lvert D^{i}u(\upPhi(x)) \rvert^{p}\jac\upPhi(x)\,\d x \\
        \leq
        \Cr{cst:fixed_block_shrinking}\sum_{i=1}^{j}\int_{\omega\setminus Q_{2}} (\mu\eta)^{ip}\lvert D^{i}u \rvert^{p}.
    \end{multline*}

    Combining inclusion~\ref{item:shrinking_geometric_geometry} and estimates~\ref{item:shrinking_geometric_singularity} and~\ref{item:shrinking_geometric_jacobian} in Proposition~\ref{prop:block_shrinking_geometric}, we find 
    \[
        \lvert D^{j}(u \circ \upPhi)(x) \rvert^{p}
        \leq 
        \Cl{cst:fixed_block_shrinking2}\tau^{d-jp}\sum_{i=1}^{j}\lvert D^{i}u(\upPhi(x)) \rvert^{p}\frac{\jac\upPhi(x)}{(\mu\eta)^{(j-i)p}}
    \]
    for every \( x \in \upPhi^{-1}(Q_{2}) \subset B_{1} \).
    Using again the change of variable theorem, we deduce that 
    \begin{multline*}
        \int_{\upPhi^{-1}(Q_{2})} (\mu\eta)^{jp}\lvert D^{j}(u \circ \upPhi)\rvert^{p}
        \leq
        \int_{\upPhi^{-1}(Q_{2})}\Cr{cst:fixed_block_shrinking2}\tau^{d-jp}\sum_{i=1}^{j}(\mu\eta)^{ip}\lvert D^{i}u(\upPhi(x)) \rvert^{p}\jac\upPhi(x)\,\d x \\
        \leq
        \Cr{cst:fixed_block_shrinking2}\tau^{d-jp}\sum_{i=1}^{j}\int_{Q_{2}} (\mu\eta)^{ip}\lvert D^{i}u \rvert^{p}.
    \end{multline*}
    We conclude by additivity of the integral, combining the estimates on \( \upPhi^{-1}(\omega\setminus Q_{2}) \) and on \( \upPhi^{-1}(Q_{2}) \).

    The proof of the estimate at order \( 0 \) relies on the same decomposition and change of variable, noting that in particular \( \jac\upPhi \geq \C \) 
    to handle the region \( \upPhi^{-1}(\omega\setminus Q_{2}) \).

    We now move to the fractional estimate when \( 0 < s < 1 \).
    We observe that, as in~\eqref{eq:mean_value_thickening}, we have 
    \begin{equation}
    \label{eq:mean_value_shrinking}
        \frac{\lvert \upPhi(x)-\upPhi(y) \rvert}{\lvert x-y \rvert} 
        \leq 
        \C\frac{\mu\eta}{\zeta(y)}
        \quad
        \text{for every \( x \), \( y \in \omega \).}
    \end{equation}
    We start by splitting, in the spirit of the proof of Proposition~\ref{prop:block_thickening_analytic},
	\begin{equation}
	\label{eq:shrinking_splitted_in_4}
		\iint\limits_{\substack{\upPhi^{-1}(\omega) \times \upPhi^{-1}(\omega) \\ \zeta(x) \leq \zeta(y)}}
		\frac{\lvert u\circ\upPhi(x)-u\circ\upPhi(y) \rvert^{p}}{\lvert x-y \rvert^{m+sp}} \,\d x \d y
		= 
		I_{1} + I_{2} + I_{3} + I_{4},
	\end{equation}
	where we have set
	\[
		\begin{array}{ll}
		I_{1} = \iint\limits_{\substack{\upPhi^{-1}(\omega \setminus Q_{2}) \times \upPhi^{-1}(\omega \setminus Q_{2}) \\ \zeta(x) \leq \zeta(y)}}
		\frac{\lvert u\circ\upPhi(x)-u\circ\upPhi(y) \rvert^{p}}{\lvert x-y \rvert^{m+sp}} \,\d x \d y,
		&
		I_{2} = \iint\limits_{\substack{\upPhi^{-1}(Q_{2}) \times \upPhi^{-1}(Q_{2}) \\ \zeta(x) \leq \zeta(y)}}
		\frac{\lvert u\circ\upPhi(x)-u\circ\upPhi(y) \rvert^{p}}{\lvert x-y \rvert^{m+sp}} \,\d x \d y, \\
		I_{3} = \iint\limits_{\substack{\upPhi^{-1}(\omega\setminus Q_{2}) \times \upPhi^{-1}(Q_{2}) \\ \zeta(x) \leq \zeta(y)}}
		\frac{\lvert u\circ\upPhi(x)-u\circ\upPhi(y) \rvert^{p}}{\lvert x-y \rvert^{m+sp}} \,\d x \d y, 
		&
		I_{4} = \iint\limits_{\substack{\upPhi^{-1}(Q_{2}) \times \upPhi^{-1}(\omega\setminus Q_{2}) \\ \zeta(x) \leq \zeta(y)}}
		\frac{\lvert u\circ\upPhi(x)-u\circ\upPhi(y) \rvert^{p}}{\lvert x-y \rvert^{m+sp}} \,\d x \d y.
		\end{array}
	\]
    Estimating the right-hand side of~\eqref{eq:shrinking_splitted_in_4} is similar to Step~2 in the case \( 0 < s < 1 \) of Proposition~\ref{prop:block_thickening_analytic}.
    The novelty here is that we need to be more careful with the domains on which the estimates are performed.
    Indeed, in order to obtain~\ref{item:estimate_block_shrinking_sle1}, we need to estimate the right-hand side of~\eqref{eq:shrinking_splitted_in_4} by a sum of terms that are either preceded by a suitable power of \( \tau \), or involve only the energy of \( u \) on \( \omega \setminus Q_{2} \).
    
    We begin with \( I_{2} \).
    We define 
    \[
        \Bc_{x,y} = B^{m}_{\lvert \upPhi(x)-\upPhi(y) \rvert}\biggl(\frac{\upPhi(x)+\upPhi(y)}{2}\biggr) \cap Q_{2},
    \]
    so that
    \begin{equation*}
		I_{2}
        \leq 
        \C\int_{\upPhi^{-1}(Q_{2})}\int_{\upPhi^{-1}(Q_{2})}\fint_{\Bc_{x,y}}
            \frac{\lvert u\circ\upPhi(x)-u(z) \rvert^{p}}{\lvert x-y \rvert^{m+sp}}\,\d z \d y \d x.
    \end{equation*}
    We observe that \( \lvert \Bc_{x,y} \rvert \geq \C\lvert \upPhi(x)-\upPhi(y) \rvert^{m} \) due to the fact that \( Q_{2} \) is a rectangle with comparable sidelengths.
    Moreover, \( \lvert \upPhi(x)-z \rvert \leq \frac{3}{2}\lvert \upPhi(x)-\upPhi(y) \rvert \).
    Hence, using Tonelli's theorem, we find
    \begin{multline*}
        \int_{\upPhi^{-1}(Q_{2})}\int_{\upPhi^{-1}(Q_{2})}\fint_{\Bc_{x,y}}
            \frac{\lvert u\circ\upPhi(x)-u(z) \rvert^{p}}{\lvert x-y \rvert^{m+sp}}\,\d z \d y \d x \\
        \leq 
        \C\int_{\upPhi^{-1}(Q_{2})}\int_{Q_{2}}\int_{\Yc_{x,z}}
            \frac{\lvert u\circ\upPhi(x)-u(z) \rvert^{p}}{\lvert x-y \rvert^{m+sp}\lvert \upPhi(x)-z \rvert^{m}}\,\d y \d z \d x,
    \end{multline*}
    where 
    \[
        \Yc_{x,z} = \{y \in \upPhi^{-1}(Q_{2}) \mathpunct{:} z \in \Bc_{x,y}\}
        \subset
        \{y \in \R^{m} \mathpunct{:} \lvert \upPhi(x)-z \rvert < \C\frac{\mu\eta}{\zeta(x)}\lvert x-y \rvert\}.
    \]
    Therefore,
    \begin{multline}
    \label{eq:intermediate_step_block_shrinking}
        \int_{\upPhi^{-1}(Q_{2})}\int_{Q_{2}}\int_{\Yc_{x,z}}
            \frac{\lvert u\circ\upPhi(x)-u(z) \rvert^{p}}{\lvert x-y \rvert^{m+sp}\lvert \upPhi(x)-z \rvert^{m}}\,\d y \d z \d x \\
        \leq 
        \C\int_{\upPhi^{-1}(Q_{2})}\int_{Q_{2}}
            \frac{\lvert u\circ\upPhi(x)-u(z) \rvert^{p}}{\lvert \upPhi(x)-z \rvert^{m+sp}}\frac{(\mu\eta)^{sp}}{\zeta(x)^{sp}}\,\d z \d x.
    \end{multline}
    Now we use: (i) the fact that \( \zeta(x) \geq \C\mu\eta\tau \), (ii) the second estimate on \( \jac\upPhi \) -- valid on \( \upPhi^{-1}(Q_{2}) \subset B_{1} \) -- 
    and (iii) the change of variable theorem to get
    \[
        \int_{\upPhi^{-1}(Q_{2})}\int_{Q_{2}}
            \frac{\lvert u\circ\upPhi(x)-u(z) \rvert^{p}}{\lvert \upPhi(x)-z \rvert^{m+sp}}\frac{(\mu\eta)^{sp}}{\zeta(x)^{sp}}\,\d z \d x 
        \leq 
        \C\tau^{d-sp}\int_{Q_{2}}\int_{Q_{2}}
            \frac{\lvert u(x)-u(y) \rvert^{p}}{\lvert x-y \rvert^{m+sp}}\,\d x \d y.
    \]
    
    The three other terms are handled similarly, so we only point out the required changes.
    We define instead 
    \[
        \Bc_{x,y} = B^{m}_{\lvert \upPhi(x)-\upPhi(y) \rvert}\biggl(\frac{\upPhi(x)+\upPhi(y)}{2}\biggr) \cap (\omega \setminus Q_{2}).
    \]
    For \( I_{3} \), we split
    \begin{multline*}
		I_{3}
    	\leq 
    	\C\biggl(\int_{\upPhi^{-1}(\omega \setminus Q_{2})}\int_{\upPhi^{-1}(Q_{2})}\fint_{\Bc_{x,y}}
    	\frac{\lvert u\circ\upPhi(x)-u(z) \rvert^{p}}{\lvert x-y \rvert^{m+sp}}\,\d z \d y \d x \\
    	+
    	\int_{\upPhi^{-1}(\omega \setminus Q_{2})}\int_{\upPhi^{-1}(Q_{2})}\fint_{\Bc_{x,y}}
    	\frac{\lvert u\circ\upPhi(y)-u(z) \rvert^{p}}{\lvert x-y \rvert^{m+sp}}\,\d z \d y \d x
    	\biggr).
    \end{multline*}
    We note that we still have \( \lvert \Bc_{x,y} \rvert \geq \C\lvert \upPhi(x)-\upPhi(y) \rvert^{m} \), using this time the assumption on the volume of balls 
    centered in \( \omega \setminus Q_{2} \).
    We then pursue as for the second term in the right-hand side of~\eqref{eq:shrinking_splitted_in_4}: we use Tonelli's theorem, and after that, we integrate with respect to \( y \).
    Similar to~\eqref{eq:intermediate_step_block_shrinking}, we deduce that 
    \begin{multline*}
    	I_{3} 
    	\leq 
    	\C\biggl(\int_{\upPhi^{-1}(\omega \setminus Q_{2})}\int_{\omega \setminus Q_{2}}
    	\frac{\lvert u\circ\upPhi(x)-u(z) \rvert^{p}}{\lvert \upPhi(x)-z \rvert^{m+sp}}\frac{(\mu\eta)^{sp}}{\zeta(x)^{sp}}\,\d z \d x \\
    	+
    	\int_{\upPhi^{-1}(Q_{2})}\int_{\omega \setminus Q_{2}}
    	\frac{\lvert u\circ\upPhi(x)-u(z) \rvert^{p}}{\lvert \upPhi(x)-z \rvert^{m+sp}}\frac{(\mu\eta)^{sp}}{\zeta(x)^{sp}}\,\d z \d x\biggr).
    \end{multline*}
    Invoking the change of variable theorem, using the first estimate on \( \jac\upPhi \) with \( \beta = sp \) for the first term and the second estimate on \( \jac\upPhi \) for the second term, we conclude that
    \begin{equation*}
    	I_{3}
    	\leq 
    	\biggl(\int_{\omega \setminus Q_{2}}\int_{\omega \setminus Q_{2}}
    	\frac{\lvert u(x)-u(y) \rvert^{p}}{\lvert x-y \rvert^{m+sp}}\,\d x \d y
    	+
    	\tau^{d-sp}\int_{\omega}\int_{\omega}
    	\frac{\lvert u(x)-u(y) \rvert^{p}}{\lvert x-y \rvert^{m+sp}}\,\d x \d y\biggr).
    \end{equation*}
    By the exact same reasoning,
    \[
    	I_{4}
    	\leq 
    	\biggl(\int_{\omega \setminus Q_{2}}\int_{\omega \setminus Q_{2}}
    	\frac{\lvert u(x)-u(y) \rvert^{p}}{\lvert x-y \rvert^{m+sp}}\,\d x \d y
    	+
    	\tau^{d-sp}\int_{\omega}\int_{\omega}
    	\frac{\lvert u(x)-u(y) \rvert^{p}}{\lvert x-y \rvert^{m+sp}}\,\d x \d y\biggr),
    \] 
    while
    \[
    	I_{1}
    	\leq 
    	\C\int_{\omega \setminus Q_{2}}\int_{\omega \setminus Q_{2}}
    	\frac{\lvert u(x)-u(y) \rvert^{p}}{\lvert x-y \rvert^{m+sp}}\,\d x \d y.
    \]
    Collecting the estimates for the right-hand side of~\eqref{eq:shrinking_splitted_in_4}, we arrive at estimate~\ref{item:estimate_block_shrinking_sle1} of Proposition~\ref{prop:block_shrinking_analytic}.

    We finish with the estimate for the Gagliardo seminorm in the case \( s \geq 1 \).
    We consider \( x \), \( y \in \upPhi^{-1}(\omega) \) such that, without loss of generality, 
    \( \zeta(x) \leq \zeta(y) \).
    As usual, using the Faà di Bruno formula, the multilinearity of the differential and the estimates on the derivatives of \( \upPhi \), we write
    \begin{multline}
    \label{eq:Faa_di_bruno_block_shrinking_sgeq1}
        \lvert D^{j}(u \circ \upPhi)(x) - D^{j}(u \circ \upPhi)(y) \rvert \\
        \leq
        \C\sum_{i=1}^{j}\Bigl(\lvert D^{i}u\circ\upPhi(x) - D^{i}u\circ\upPhi(y) \rvert \frac{(\mu\eta)^{i}}{\zeta(y)^{j}} \\
            + \sum_{t=1}^{j} \lvert D^{i}u\circ\upPhi(x) \rvert \lvert D^{t}\upPhi(x) - D^{t}\upPhi(y) \rvert \frac{(\mu\eta)^{i-1}}{\zeta(x)^{j-t}}\Bigr).
    \end{multline}
    For the second term in~\eqref{eq:Faa_di_bruno_block_shrinking_sgeq1}, we proceed once again by splitting the integral over \( B^{m}_{r}(x) \) and \( \R^{m}\setminus B^{m}_{r}(x) \) with \( r = \zeta(x) \) to arrive at
    \[
        \int\limits_{\substack{\upPhi^{-1}(\omega) \\ \zeta(x) \leq \zeta(y)}} 
            \frac{\lvert D^{t}\upPhi(x) - D^{t}\upPhi(y) \rvert^{p}}{\lvert x-y \rvert^{m+\sigma p}}\,\d y
        \leq
        \C\frac{(\mu\eta)^{p}}{\zeta(x)^{(t+\sigma)p}}.
    \]
    Hence,
    \begin{multline}
    \label{eq:block_shrinking_second_term_fractional_higher_order}
        \iint\limits_{\substack{\upPhi^{-1}(\omega) \times \upPhi^{-1}(\omega) \\ \zeta(x) \leq \zeta(y)}} 
        \frac{\lvert D^{i}u\circ\upPhi(x) \rvert^{p} \lvert D^{t}\upPhi(x) - D^{t}\upPhi(y) \rvert^{p} }{\lvert x-y \rvert^{m+\sigma p}}\frac{(\mu\eta)^{(i-1)p}}{\zeta(x)^{(j-t)p}}\,\d x\d y \\
        \leq
        \C\int_{\upPhi^{-1}(\omega)}\lvert D^{i}u\circ\upPhi(x) \rvert^{p}\frac{(\mu\eta)^{ip}}{\zeta(x)^{(j+\sigma)p}}\,\d x.
    \end{multline}
    We then argue as for the integer order term. 
    We split the integral in the right-hand side of~\eqref{eq:block_shrinking_second_term_fractional_higher_order} over the regions \( \omega \setminus Q_{2} \) and \( Q_{2} \). 
    Owing to the change of variable theorem, using the first estimate on \( \jac\upPhi\) with \( \beta = (j+\sigma)p \) over \( \omega \setminus Q_{2} \) and the second estimate on \( \jac\upPhi \) over \( Q_{2} \), we obtain
    \begin{multline*}
        \int_{\upPhi^{-1}(\omega)} \lvert D^{i}u \circ \upPhi(x) \rvert^{p}\frac{1}{\zeta(x)^{(j+\sigma)p}}\,\d x \\
        \leq 
        \C(\mu\eta)^{ip-(j+\sigma)p}\int_{\omega \setminus Q_{2}} \lvert D^{i}u \rvert^{p}
            + \C\tau^{d-(j+\sigma)p}(\mu\eta)^{ip-(j+\sigma)p}\int_{Q_{2}} \lvert D^{i}u \rvert^{p}.
    \end{multline*}

    As for thickening, the first term in~\eqref{eq:Faa_di_bruno_block_shrinking_sgeq1} is handled exactly as in the case \( 0 < s < 1 \), taking into account the presence of the factor \( \frac{(\mu\eta)^{ip}}{\zeta(y)^{jp}} \): we use the same splitting as in~\eqref{eq:shrinking_splitted_in_4}, and then the usual averaging argument.
    Doing so, we deduce that the first term in~\eqref{eq:Faa_di_bruno_block_shrinking_sgeq1} is bounded from above by a constant multiple of
    \begin{multline}
    \label{eq:block_shrinking_first_term_fractional_higher_order}
    	\int_{\upPhi^{-1}(\omega \setminus Q_{2})}\int_{\omega \setminus Q_{2}}
    	\frac{\lvert D^{i}u\circ\upPhi(x)-D^{i}u(z) \rvert^{p}}{\lvert \upPhi(x)-z \rvert^{m+\sigma p}}\frac{(\mu\eta)^{(i+\sigma)p}}{\zeta(x)^{(j+\sigma)p}}\,\d z \d x \\
    	+
    	\int_{\upPhi^{-1}(Q_{2})}\int_{\omega}
    	\frac{\lvert D^{i}u\circ\upPhi(x)-D^{i}u(z) \rvert^{p}}{\lvert \upPhi(x)-z \rvert^{m+\sigma p}}\frac{(\mu\eta)^{(i+\sigma)p}}{\zeta(x)^{(j+\sigma)p}}\,\d z \d x.
    \end{multline}
    An additional use of the change of variable theorem shows that~\eqref{eq:block_shrinking_first_term_fractional_higher_order} is estimated, up to a constant factor, by
    \begin{multline*}
    (\mu\eta)^{(i-j)p}\int_{\omega \setminus Q_{2}}\int_{\omega \setminus Q_{2}}
    \frac{\lvert D^{i}u(x)-D^{i}u(y) \rvert^{p}}{\lvert x-y \rvert^{m+\sigma p}}\,\d x \d y \\
    +
    \tau^{d-(j+\sigma)p}(\mu\eta)^{(i-j)p}\int_{\omega}\int_{\omega}
    \frac{\lvert D^{i}u(x)-D^{i}u(y) \rvert^{p}}{\lvert x-y \rvert^{m+\sigma p}}\,\d z \d x.
    \end{multline*}
    
    Gathering the estimates for both terms in~\eqref{eq:Faa_di_bruno_block_shrinking_sgeq1}, we obtain the desired conclusion, hence finishing the proof of Proposition~\ref{prop:block_shrinking_analytic}.
    \resetconstant
\end{proof}

Now that we have at our disposal the building blocks for the shrinking procedure, we are ready to prove Proposition~\ref{prop:main_shrinking}.
As usual, for the convenience of the reader, we start with an informal presentation of the construction.

We first apply shrinking around the vertices of the dual skeleton \( \Tc^{\ell^{\ast}} \),
with parameters \( 0 < \mu_{m-1} < \nu_{m} < \mu_{m} \) and \( \frac{\tau\mu}{\nu_{m}} \), where \( \mu_{m-1} \geq \mu \) and \( \mu_{m} \leq 2\mu \).
This shrinks a neighborhood of size \( \mu_{m-1}\eta \) of these vertices into a neighborhood of size \( \tau\mu\eta \).
We then apply shrinking around the edges of \( \Tc^{\ell^{\ast}} \) with parameters \( 0 < \mu_{m-2} < \nu_{m-1} < \mu_{m-1} \) and \( \frac{\tau\mu}{\nu_{m-1}} \),
where \( \mu_{m-2} \geq \mu \).
This shrinks the part of a neighborhood of size \( \mu_{m-2}\eta \) of the edges of \( \Tc^{\ell^{\ast}} \) 
lying at distance at most \( \mu_{m-1}\eta \) of the \( (m-1) \)-faces of \( \Kc^{m} \) into a neighborhood of size \( \tau\mu\eta \) of those edges.
But since the part of the neighborhood of size \( \mu_{m-2}\eta \) lying at distance more than \( \mu_{m-1}\eta \) of the \( (m-1) \)-faces of \( \Kc^{m} \) 
has already been shrinked during the previous step, we conclude that the whole neighborhood of size \( \mu_{m-2}\eta \) of \( T^{1} \), the union of the vertices in \( \Tc^{1} \), is shrinked into a neighborhood of size \( \tau\mu\eta \).
We continue this procedure by downward induction until we reach the dimension \( \ell^{\ast} \), which produces the desired map \( \upPhi \).

We illustrate this induction procedure in Figures~\ref{fig:shrinking_around_vertices},~\ref{fig:shrinking_around_edges}, and~\ref{fig:final_shrinking}.
Here, we take \( m = 2 \) and \( \ell = 0 \).
In Figure~\ref{fig:shrinking_around_vertices}, which corresponds to the first step of the induction, the values in the gray region around the center of the cube in the left part of the figure are shrinked into the much smaller gray region on the right.
During the next step, depicted in Figure~\ref{fig:shrinking_around_edges}, the values in gray around the edges of the cube on the left are shrinked into the much smaller gray region around the dual skeleton on the right.
The combination of both steps is shown in Figure~\ref{fig:final_shrinking}.
The values in the region in gray on the left are shrinked into the small neighborhood of the dual skeleton in gray on the right.

\begin{figure}[ht]
	~
	\hfill
	\includegraphics[page=10]{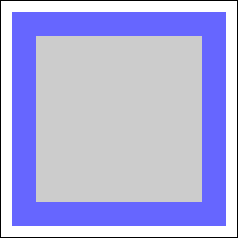}
	\hfill
	\includegraphics[page=11]{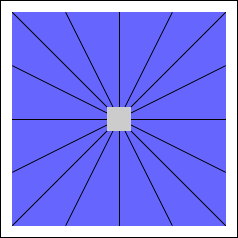}
	\hfill
	~
	\caption{Shrinking around vertices}
	\label{fig:shrinking_around_vertices}
\end{figure}

\begin{figure}[ht]
	~
	\hfill
	\includegraphics[page=12]{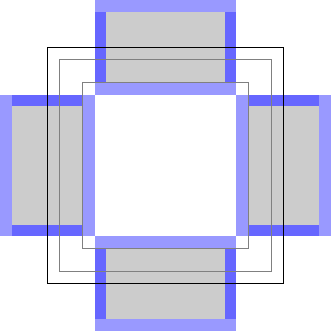}
	\hfill
	\includegraphics[page=13]{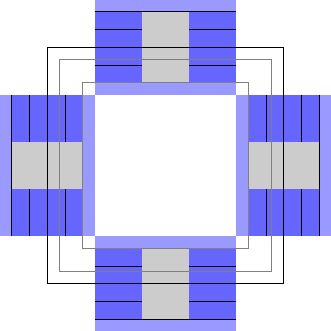}
	\hfill
	~
	\caption{Shrinking around edges}
	\label{fig:shrinking_around_edges}
\end{figure}

\begin{figure}[ht]
	~
	\hfill
	\includegraphics[page=14]{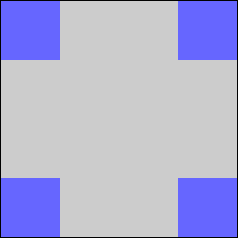}
	\hfill
	\includegraphics[page=15]{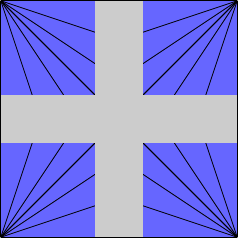}
	\hfill
	~
	\caption{Final shrinking at order \( 1 \)}
	\label{fig:final_shrinking}
\end{figure}

As we will see, the induction procedure is more involved than in the case of thickening, and relies on Proposition~\ref{prop:block_shrinking_analytic} applied with domains more general than rectangles.

\begin{proof}[Proof of Proposition~\ref{prop:main_shrinking}]
    The map \( \upPhi \) is constructed by downward induction.
	We consider finite sequences \( (\mu_{i})_{\ell \leq i \leq m} \) and \( (\nu_{i})_{\ell \leq i \leq m} \) such that 
	\[
		0 < \mu = \mu_{\ell} < \nu_{\ell+1} < \mu_{\ell+1} < \cdots < \mu_{m-1} < \nu_{m} < \mu_{m} \leq 2\mu.
	\]
	We first define \( \upPhi^{m} = \id \).
	Then, assuming that \( \upPhi^{d} \) has been defined for some \( d \in \{\ell+1,\dots,m\} \), 
	we identify any \( \sigma^{d} \in \Kc^{d} \) with \( Q^{d}_{\eta} \times \{0\}^{m-d} \), and we let \( \upPhi_{\sigma^{d}} \) be the map given by Proposition~\ref{prop:block_shrinking_geometric} applied around \( \sigma^{d} \) 
	with parameters \( \ulmu = \mu_{d-1} \), \( \mu = \nu_{d} \), \( \olmu = \mu_{d} \), and \( \frac{\tau\mu}{\nu_{d}} \).
	We define \( \upPsi^{d} \colon \R^{m} \to \R^{m} \) by 
	\[
		\upPsi^{d}(x)
		=
		\begin{cases}
			\upPhi_{\sigma^{d}}(x) & \text{if \( x \in T_{\sigma^{d}}(Q_{3}) \) for some \( \sigma^{d} \in \Kc^{d} \),} \\
			x & \text{otherwise,}
		\end{cases}
	\]
    where \( T_{\sigma^{d}} \) is an isometry of \( \R^{m} \) mapping \( Q^{d}_{\eta} \times \{0\}^{m-d} \) to \( \sigma^{d} \).
	We then let \( \upPhi^{d-1} = \upPsi^{d} \circ \upPhi^{d} \).
	The required map is given by \( \upPhi = \upPhi^{\ell} \).

	Properties~\ref{item:main_shrinking_injective} to~\ref{item:main_shrinking_support} are already contained in~\cite{BPVS_density_higher_order}*{Proposition~8.1}, so it only remains to prove the Sobolev estimates.
	The argument is similar to the one used in the proof of Proposition~\ref{prop:main_thickening}. 
	We proceed by induction.
    One of the issues is how to remove inductively neighborhoods of dual skeletons.
    We let \( Q_{4} = Q^{d}_{2\mu\eta} \times Q^{m-d}_{(1-\mu)\eta} \), so that \( Q_{3} \subset Q_{4} \) for every \( d \in \{\ell+1,\dots,m\} \).
    First, note that invoking Proposition~\ref{prop:block_shrinking_analytic} with \( \omega = Q_{4} \setminus T_{\sigma^{d}}^{-1}(T^{m-d-1}+Q^{m}_{\mu_{d}\eta}) \) ensures that 
    \begin{enumerate}[label=(\alph*)]
    	\item if \( 0 < s < 1 \), then
    		\[
	    		\lvert u \circ \upPhi_{\sigma^{d}} \rvert_{W^{s,p}(T_{\sigma^{d}}(Q_{4}) \setminus (T^{m-d-1}+Q^{m}_{\mu_{d}\eta}))}
	    		\leq 
	    		\C\lvert u \rvert_{W^{s,p}(T_{\sigma^{d}}(Q_{4})\setminus (T^{m-d}+Q^{m}_{\mu_{d-1}\eta}))}
	    		+ \C\tau^{\frac{d-sp}{p}}\lvert u \rvert_{W^{s,p}(T_{\sigma^{d}}(Q_{4}))};
    		\]
    	\item if \( s \geq 1 \), then for every \( j \in \{1,\dots,k\} \),
    		\begin{multline*}
	    		(\mu\eta)^{j}\lVert D^{j}(u \circ \upPhi_{\sigma^{d}}) \rVert_{L^{p}(T_{\sigma^{d}}(Q_{4}) \setminus (T^{m-d-1}+Q^{m}_{\mu_{d}\eta}))}
	    		\leq
	    		\C\sum_{i=1}^{j}(\mu\eta)^{i}\lVert D^{i}u \rVert_{L^{p}(T_{\sigma^{d}}(Q_{4})\setminus (T^{m-d}+Q^{m}_{\mu_{d-1}\eta}))} \\
	    		+ \C\tau^{\frac{d-jp}{p}}\sum_{i=1}^{j}(\mu\eta)^{i}\lVert D^{i}u \rVert_{L^{p}(T_{\sigma^{d}}(Q_{4}))};
    		\end{multline*}
    	\item if \( s \geq 1 \) and \( \sigma \neq 0 \), then for every \( j \in \{1,\dots,k\} \),
    		\begin{multline*}
	    		(\mu\eta)^{j+\sigma}\lvert D^{j}(u \circ \upPhi_{\sigma^{d}}) \rvert_{W^{\sigma,p}(T_{\sigma^{d}}(Q_{4}) \setminus (T^{m-d-1}+Q^{m}_{\mu_{d}\eta}))} \\
	    		\leq
	    		\C\sum_{i=1}^{j}\Bigl((\mu\eta)^{i}\lVert D^{i}u \rVert_{L^{p}(T_{\sigma^{d}}(Q_{4})\setminus (T^{m-d}+Q^{m}_{\mu_{d-1}\eta}))} 
	    		+ (\mu\eta)^{i+\sigma}\lvert D^{i}u \rvert_{W^{\sigma,p}(T_{\sigma^{d}}(Q_{4})\setminus (T^{m-d}+Q^{m}_{\mu_{d-1}\eta}))}\Bigr) \\
	    		+ \C\tau^{\frac{d-(j+\sigma)p}{p}}\sum_{i=1}^{j}\Bigl((\mu\eta)^{i}\lVert D^{i}u \rVert_{L^{p}(T_{\sigma^{d}}(Q_{4})))} 
	    		+ (\mu\eta)^{i+\sigma}\lvert D^{i}u \rvert_{W^{\sigma,p}(T_{\sigma^{d}}(Q_{4}))}\Bigr);
    		\end{multline*}
    	\item for every \( 0 < s < +\infty \),
    		\[
	    		\lVert u \circ \upPhi_{\sigma^{d}} \rVert_{L^{p}(T_{\sigma^{d}}(Q_{4}) \setminus (T^{m-d-1}+Q^{m}_{\mu_{d}\eta}))}
	    		\leq 
	    		\C\lVert u \rVert_{L^{p}(T_{\sigma^{d}}(Q_{4})\setminus (T^{m-d}+Q^{m}_{\mu_{d-1}\eta}))}
	    		+ \C\tau^{d}\lVert u \rVert_{L^{p}(T_{\sigma^{d}}(Q_{4}))}.
    		\]
    \end{enumerate}
    Indeed, we have: (i) \( (T_{\sigma^{d}}(Q_{4}) \setminus (T^{m-d-1}+Q^{m}_{\mu_{d}\eta})) \setminus T_{\sigma^{d}}(Q_{2}) \subset T_{\sigma^{d}}(Q_{4}) \setminus (T^{m-d}+Q^{m}_{\mu_{d-1}\eta}) \), (ii) \( Q_{4} \setminus T_{\sigma^{d}}^{-1}(T^{m-d-1}+Q^{m}_{\mu_{d}\eta}) \subset \upPhi^{-1}(\omega) \), and (iii) \( \omega \) satisfies the condition on the volume of balls required to apply Proposition~\ref{prop:block_shrinking_analytic}.
    Affirmation~(ii) is a consequence of the fact that \( \upPhi \) has the specific form \( \upPhi(x) = (\lambda(x)x',x'') \) with \( \lambda \colon \R^{m} \to [1,+\infty) \).
    Affirmation~(iii) follows from the fact that \( \omega \setminus Q_{2} \) is actually a rectangle to which other rectangles have been removed.
    Note that, for convenience of notation, we let \( T^{-1} = \varnothing \).
    
    Using the additivity of the integral or Lemma~\ref{lemma:fractional_additivity} combined with the usual finite number of overlaps argument, we deduce that 
    \begin{enumerate}[label=(\alph*)]
    	\item if \( 0 < s < 1 \), then 
    		\begin{align*}
	    		&(\mu\eta)^{s}\lvert u\circ\upPsi^{d} \rvert_{W^{s,p}(K^{m}\cap (T^{\ell^{\ast}}+Q^{m}_{2\mu\eta})\setminus (T^{m-d-1}+Q^{m}_{\mu_{d}\eta}))} \\
	    		&\quad\leq
	    		\C\Bigl((\mu\eta)^{s} \lvert u \rvert_{W^{s,p}(K^{m}\cap (T^{\ell^{\ast}}+Q^{m}_{2\mu\eta})\setminus (T^{m-d}+Q^{m}_{\mu_{d-1}\eta}))} + \lVert u \rVert_{L^{p}(K^{m}\cap (T^{\ell^{\ast}}+Q^{m}_{2\mu\eta}) \setminus (T^{m-d}+Q^{m}_{\mu_{d-1}\eta}))}\Bigr) \\
	    		&\quad\quad+ \C\tau^{\frac{\ell+1-sp}{p}}\Bigr((\mu\eta)^{s}\lvert u \rvert_{W^{s,p}(K^{m}\cap (T^{\ell^{\ast}}+Q^{m}_{2\mu\eta}))} + \lVert u \rVert_{L^{p}(K^{m}\cap (T^{\ell^{\ast}}+Q^{m}_{2\mu\eta}))} \Bigr); 
    		\end{align*}
    	\item if \( s \geq 1 \), then for every \( j \in \{1,\dots,k\} \),
    		\begin{align*}
	    		&(\mu\eta)^{j}\lVert D^{j}(u\circ\upPsi^{d}) \rVert_{L^{p}(K^{m}\cap (T^{\ell^{\ast}}+Q^{m}_{2\mu\eta})\setminus (T^{m-d-1}+Q^{m}_{\mu_{d}\eta}))} \\
	    		&\quad\leq
	    		\C\sum_{i=1}^{j}(\mu\eta)^{i} \lVert D^{i}u \rVert_{L^{p}(K^{m}\cap (T^{\ell^{\ast}}+Q^{m}_{2\mu\eta}) \setminus (T^{m-d}+Q^{m}_{\mu_{d-1}\eta}))} \\
	    		&\quad\quad+ \C\tau^{\frac{\ell+1-sp}{p}}\sum_{i=1}^{j} (\mu\eta)^{i}\lVert D^{i}u \rVert_{L^{p}(K^{m}\cap (T^{\ell^{\ast}}+Q^{m}_{2\mu\eta}))};
    		\end{align*}
    	\item if \( s \geq 1 \) and \( \sigma \neq 0 \), then for every \( j \in \{1,\dots,k\} \),
    		\[
    		\begin{split}
	    		(\mu\eta)^{j+\sigma}&\lvert D^{j}(u\circ\upPsi^{d}) \rvert_{W^{\sigma,p}(K^{m}\cap (T^{\ell^{\ast}}+Q^{m}_{2\mu\eta})\setminus (T^{m-d-1}+Q^{m}_{\mu_{d}\eta}))} \\
	    		\leq&
	    		\C\sum_{i=1}^{j}\Bigl((\mu\eta)^{i} \lVert D^{i}u \rVert_{L^{p}(K^{m}\cap (T^{\ell^{\ast}}+Q^{m}_{2\mu\eta})\setminus (T^{m-d}+Q^{m}_{\mu_{d-1}\eta}))} \\ 
	    		&+ (\mu\eta)^{i+\sigma} \lvert D^{i}u \rvert_{W^{\sigma,p}(K^{m}\cap (T^{\ell^{\ast}}+Q^{m}_{2\mu\eta})\setminus (T^{m-d}+Q^{m}_{\mu_{d-1}\eta}))}\Bigr) \\
	    		&+ \C\tau^{\frac{\ell+1-sp}{p}}\sum_{i=1}^{j}\Bigr((\mu\eta)^{i}\lVert D^{i}u \rVert_{L^{p}(K^{m}\cap (T^{\ell^{\ast}}+Q^{m}_{2\mu\eta}))} + (\mu\eta)^{i+\sigma}\lvert D^{i}u \rvert_{W^{\sigma,p}(K^{m}\cap (T^{\ell^{\ast}}+Q^{m}_{2\mu\eta}))} \Bigr); 
	    	\end{split}
    		\]
    	\item for every \( 0 < s < +\infty \), 
    		\begin{multline*}
    		\lVert u\circ\upPsi^{d} \rVert_{L^{p}(K^{m}\cap (T^{\ell^{\ast}}+Q^{m}_{2\mu\eta})\setminus (T^{m-d-1}+Q^{m}_{\mu_{d}\eta}))} \\
    		\leq 
    		\C\lVert u \rVert_{L^{p}(K^{m}\cap (T^{\ell^{\ast}}+Q^{m}_{2\mu\eta}) \setminus (T^{m-d}+Q^{m}_{\mu_{d-1}\eta}))} 
    		+ \C\tau^{\frac{\ell+1-sp}{p}}\lVert u \rVert_{L^{p}(K^{m}\cap (T^{\ell^{\ast}}+Q^{m}_{2\mu\eta}))}
    		\end{multline*}
    \end{enumerate}

    In particular, since \( \tau < 1 \), another application of Proposition~\ref{prop:block_shrinking_analytic} yields the following simpler estimates:
    \begin{enumerate}[label=(\alph*)]
    	\item if \( 0 < s < 1 \), then 
    		\begin{multline*}
	    		(\mu\eta)^{s}\lvert u\circ\upPsi^{d} \rvert_{W^{s,p}(K^{m}\cap (T^{\ell^{\ast}}+Q^{m}_{2\mu\eta}))} \\
	    		\leq
	    		\C\Bigl((\mu\eta)^{s} \lvert u \rvert_{W^{s,p}(K^{m}\cap (T^{\ell^{\ast}}+Q^{m}_{2\mu\eta}))} + \lVert u \rVert_{L^{p}(K^{m}\cap (T^{\ell^{\ast}}+Q^{m}_{2\mu\eta}))}\Bigr); 
    		\end{multline*}
    	\item if \( s \geq 1 \), then for every \( j \in \{1,\dots,k\} \),
    		\[
	    		(\mu\eta)^{j}\lVert D^{j}(u\circ\upPsi^{d}) \rVert_{L^{p}(K^{m}\cap (T^{\ell^{\ast}}+Q^{m}_{2\mu\eta}))}
	    		\leq
	    		\C\sum_{i=1}^{j}(\mu\eta)^{i} \lVert D^{i}u \rVert_{L^{p}(K^{m}\cap (T^{\ell^{\ast}}+Q^{m}_{2\mu\eta}))};
    		\]
    	\item if \( s \geq 1 \) and \( \sigma \neq 0 \), then for every \( j \in \{1,\dots,k\} \),
    		\begin{multline*}
	    		(\mu\eta)^{j+\sigma}\lvert D^{j}(u\circ\upPsi^{d}) \rvert_{W^{\sigma,p}(K^{m}\cap (T^{\ell^{\ast}}+Q^{m}_{2\mu\eta}))} \\
	    		\leq
	    		\C\sum_{i=1}^{j}\Bigl((\mu\eta)^{i} \lVert D^{i}u \rVert_{L^{p}(K^{m}\cap (T^{\ell^{\ast}}+Q^{m}_{2\mu\eta}))} + (\mu\eta)^{i+\sigma} \lvert D^{i}u \rvert_{W^{\sigma,p}(K^{m}\cap (T^{\ell^{\ast}}+Q^{m}_{2\mu\eta}))}\Bigr); 
    		\end{multline*}
    	\item for every \( 0 < s < +\infty \),
    		\[
	    		\lVert u\circ\upPsi^{d} \rVert_{L^{p}(K^{m}\cap (T^{\ell^{\ast}}+Q^{m}_{2\mu\eta}))}
	    		\leq 
	    		\C\lVert u \rVert_{L^{p}(K^{m}\cap (T^{\ell^{\ast}}+Q^{m}_{2\mu\eta}))}.
    		\]
    \end{enumerate} 

    Combining both these sets of estimates through a downward induction procedure on \( d \), we arrive at the following estimates:
    \begin{enumerate}[label=(\alph*)]
    	\item if \( 0 < s < 1 \), then 
    		\begin{align*}
	    		&(\mu\eta)^{s}\lvert u\circ\upPhi \rvert_{W^{s,p}(K^{m}\cap (T^{\ell^{\ast}}+Q^{m}_{2\mu\eta}))} \\
	    		&\quad\leq
	    		\C\Bigl((\mu\eta)^{s} \lvert u \rvert_{W^{s,p}(K^{m} \cap (T^{\ell^{\ast}}+Q^{m}_{2\mu\eta}) \setminus (T^{\ell^{\ast}}+Q^{m}_{\mu\eta}))} + \lVert u \rVert_{L^{p}(K^{m} \cap (T^{\ell^{\ast}}+Q^{m}_{2\mu\eta}) \setminus (T^{\ell^{\ast}}+Q^{m}_{\mu\eta}))}\Bigr) \\
	    		&\quad\quad+ \C\tau^{\frac{\ell+1-sp}{p}}\Bigr((\mu\eta)^{s}\lvert u \rvert_{W^{s,p}(K^{m} \cap (T^{\ell^{\ast}}+Q^{m}_{2\mu\eta}))} + \lVert u \rVert_{L^{p}(K^{m} \cap (T^{\ell^{\ast}}+Q^{m}_{2\mu\eta}))} \Bigr);
    		\end{align*}
    	\item if \( s \geq 1 \), then for every \( j \in \{1,\dots,k\} \),
    		\begin{align*}
	    		(\mu\eta)^{j}\lVert D^{j}(u\circ\upPhi) \rVert_{L^{p}(K^{m}\cap (T^{\ell^{\ast}}+Q^{m}_{2\mu\eta}))}
	    		&\leq
	    		\C\sum_{i=1}^{j}(\mu\eta)^{i} \lVert D^{i}u \rVert_{L^{p}(K^{m} \cap (T^{\ell^{\ast}}+Q^{m}_{2\mu\eta}) \setminus (T^{\ell^{\ast}}+Q^{m}_{\mu\eta}))} \\
	    		&\quad+ \C\tau^{\frac{\ell+1-sp}{p}}\sum_{i=1}^{j} (\mu\eta)^{i}\lVert D^{i}u \rVert_{L^{p}(K^{m} \cap (T^{\ell^{\ast}}+Q^{m}_{2\mu\eta}))};
    		\end{align*}
    	\item if \( s \geq 1 \) and \( \sigma \neq 0 \), then for every \( j \in \{1,\dots,k\} \),
    		\begin{align*}
	    		&(\mu\eta)^{j+\sigma}\lvert D^{j}(u\circ\upPhi) \rvert_{W^{\sigma,p}(K^{m}\cap (T^{\ell^{\ast}}+Q^{m}_{2\mu\eta}))} \\
	    		&\quad\leq
	    		\C\sum_{i=1}^{j}\Bigl((\mu\eta)^{i} \lVert D^{i}u \rVert_{L^{p}(K^{m} \cap (T^{\ell^{\ast}}+Q^{m}_{2\mu\eta}) \setminus (T^{\ell^{\ast}}+Q^{m}_{\mu\eta}))} \\ 
	    		&\quad\quad\quad+ (\mu\eta)^{i+\sigma} \lvert D^{i}u \rvert_{W^{\sigma,p}(K^{m} \cap (T^{\ell^{\ast}}+Q^{m}_{2\mu\eta}) \setminus (T^{\ell^{\ast}}+Q^{m}_{\mu\eta}))}\Bigr) \\
	    		&\quad\quad+ \C\tau^{\frac{\ell+1-sp}{p}}\sum_{i=1}^{j}\Bigr((\mu\eta)^{i}\lVert D^{i}u \rVert_{L^{p}(K^{m} \cap (T^{\ell^{\ast}}+Q^{m}_{2\mu\eta}))} \\
	    		&\quad\quad\quad+ (\mu\eta)^{i+\sigma}\lvert D^{i}u \rvert_{W^{\sigma,p}(K^{m} \cap (T^{\ell^{\ast}}+Q^{m}_{2\mu\eta}))} \Bigr);
    		\end{align*}
    	\item for every \( 0 < s < +\infty \),
    		\begin{multline*}
    		\lVert u\circ\upPhi \rVert_{L^{p}(K^{m}\cap (T^{\ell^{\ast}}+Q^{m}_{2\mu\eta}))} \\
    		\leq 
    		\C\lVert u \rVert_{L^{p}(K^{m} \cap (T^{\ell^{\ast}}+Q^{m}_{2\mu\eta}) \setminus (T^{\ell^{\ast}}+Q^{m}_{\mu\eta}))}
    		+\C\tau^{\frac{\ell+1-sp}{p}}\lVert u \rVert_{L^{p}(K^{m} \cap (T^{\ell^{\ast}}+Q^{m}_{2\mu\eta}))}.
    		\end{multline*}
    \end{enumerate} 

    The conclusion follows from the fact that \( u = v \) outside of \( T^{\ell^{\ast}}+Q^{m}_{\mu\eta} \), by noting that actually \( \Supp\upPhi \subset T^{\ell^{\ast}}+Q^{m}_{\nu_{m}\eta} \), and using once again the additivity of the integral or Lemma~\ref{lemma:fractional_additivity}.
    \resetconstant
\end{proof}

\section{Density of smooth maps}
\label{sect:density_smooth_maps}

In view of Theorem~\ref{thm:density_class_R}, in order to prove Theorem~\ref{thm:density_smooth_functions}, it suffices to show that maps of the class \( \Rc \) may be approximated by smooth maps with values into \( \Nc \).
As we already announced, the basic idea to do so is to remove the singularities of maps in the class \( \Rc \) by filling them with a smooth map.
The key tool in this direction is the following lemma, which relies on the fact that \( K^{\ell} \) is a homotopy retract of the complement \( K^{m} \setminus T^{\ell^{\ast}} \) of the dual skeleton \( T^{\ell^{\ast}} \).
The statement we present is from~\cite{BPVS_density_higher_order}*{Proposition~7.1}, but similar ideas were already used, e.g., in~\cite{white_infima_energy_functionals}*{Section~1},~\cite{hajlasz}*{Section~2}, or~\cite{HL_topology_of_sobolev_mappings_II}*{Section~6}.

\begin{lemme}
\label{lemma:extension_continuous_to_smooth}
    Let \( \Kc^{m} \) be a cubication in \( \R^{m} \) of radius \( \eta > 0 \), \( \ell \in \{0,\dots,m-1\} \), \( \Tc^{\ell^{\ast}} \) the dual skeleton of \( \Kc^{\ell} \), and \( u \in \Cc^{\infty}(K^{m}\setminus T^{\ell^{\ast}};\Nc) \).
    If there exists \( f \in \Cc^{0}(K^{m};\Nc) \) such that \( f_{\vert K^{\ell}} = u_{\vert K^{\ell}} \), then for every \( 0 < \mu < 1 \), there exists \( v \in \Cc^{\infty}(K^{m};\Nc) \) 
    such that \( v = u \) on \( K^{m} \setminus (T^{\ell^{\ast}}+Q^{m}_{\mu\eta}) \).
\end{lemme}

In order to apply Lemma~\ref{lemma:extension_continuous_to_smooth}, it is useful to know when a continuous map from \( K^{\ell} \) to \( \Nc \) may be extended to a continuous map from \( K^{m} \) to \( \Nc \).
Following Hang and Lin~\cite{HL_topology_of_sobolev_mappings_II}, we introduce the notion of \emph{extension property}.

\begin{definition}
	\label{def:extension_property}
	Let \( K^{m} \) be a cubication in \( \R^{m} \) and \( \ell \in \{0,\dots,m-1\} \).
	We say that \( K^{m} \) has the \emph{\( \ell \)-extension property with respect to \( \Nc \)} whenever every continuous map \( f \colon K^{\ell} \to \Nc \) has an extension \( g \in \Cc^{0}(K^{m};\Nc) \).
\end{definition}

This definition is slightly different from the one given by Hang and Lin~\cite{HL_topology_of_sobolev_mappings_II}, and more in the spirit of the recent work by Bousquet, Ponce, and Van Schaftingen~\cite{BPVS_screening}.
In~\cite{HL_topology_of_sobolev_mappings_II}, Hang and Lin required that
\begin{equation}
	\label{eq:41e969e6b54715e5}
	\begin{split}
		&\text{for every continuous map \( f \colon K^{\ell+1} \to \Nc \),} \\
		&\text{\( f_{\vert K^{\ell}} \) has an extension \( g \in \Cc^{0}(K^{m};\Nc) \).}
	\end{split}
\end{equation}
To draw the parallel between both settings, let us mention that the \( \ell \)-extension property as stated in Definition~\ref{def:extension_property} is equivalent to the assumption that~\eqref{eq:41e969e6b54715e5} holds along with \( \pi_{\ell}(\Nc) = \{0\} \).

The identification of the key role played by the extension property in the strong density problem was one of the major contributions of~\cite{HL_topology_of_sobolev_mappings_II}.
In this respect, we start with the following proposition, which provides an approximation result for maps in the class \( \Rc \) as the ones used in the proof of Theorem~\ref{thm:density_class_R}.
All the other results in this section will be deduced from this proposition.

\begin{prop}
\label{prop:density_smooth_maps_cubication}	
	Let \( \Kc^{m} \) be a cubication in \( \R^{m} \).
	Let \( \ell \in \{0,\dots,m-1\} \) be such that \( \ell = [sp] \), and \( \Tc^{\ell^{\ast}} \) the dual skeleton of \( \Kc^{\ell} \).
	If \( K^{m} \) has the \( \ell \)-extension property with respect to \( \Nc \), then \( \Cc^{\infty}(K^{m};\Nc) \) is dense in \( \Cc^{\infty}(K^{m}\setminus T^{\ell^{\ast}};\Nc) \cap W^{s,p}(K^{m};\Nc) \) with respect to the \( W^{s,p} \) distance.
\end{prop}
\begin{proof}
	Let \( u \in \Cc^{\infty}(K^{m}\setminus T^{\ell^{\ast}};\Nc) \cap W^{s,p}(K^{m};\Nc) \).
	We denote by \( \eta \) the radius of the cubication \( \Kc^{m} \).
	The \( \ell \)-extension property of \( \Kc^{m} \) with respect to \( \Nc \) ensures that \( u_{\vert K^{\ell}} \) extends to a continuous map from \( K^{m} \) to \( \Nc \).
	Therefore, Lemma~\ref{lemma:extension_continuous_to_smooth} implies that, for every \( 0 < \mu < 1 \), there exists a map \( u^{\ext}_{\mu} \in \Cc^{\infty}(K^{m};\Nc) \) 
	such that \( u^{\ext}_{\mu} = u \) on \( K^{m} \setminus (T^{\ell^{\ast}}+Q^{m}_{\mu\eta}) \).
	
	We now apply shrinking to this map \( u^{\ext}_{\mu} \).
	More precisely, we assume that \( \mu < \frac{1}{2} \), we take \( 0 < \tau < \frac{1}{2} \) and we define 
	\( u^{\sh}_{\tau,\mu} = u^{\ext}_{\mu} \circ \upPhi^{\sh}_{\tau,\mu} \), 
	where \( \upPhi^{\sh}_{\tau,\mu} \) is provided by Proposition~\ref{prop:main_shrinking}.
	By Proposition~\ref{prop:main_shrinking} and the remark that follows the statement, choosing \( \tau = \tau_{\mu} \) sufficiently small and \( \varepsilon = \mu \), we deduce that, with \( A_{\mu} = K^{m} \cap (T^{\ell^{\ast}}+Q^{m}_{2\mu\eta}) \),
	\begin{enumerate}[label=(\alph*)]
		\item if \( 0 < s < 1 \), then 
			\[
				(\mu\eta)^{s}\lvert u^{\sh}_{\tau_{\mu},\mu} - u \rvert_{W^{s,p}(K^{m})} 
				\leq
				\C\Bigl((\mu\eta)^{s} \lvert u \rvert_{W^{s,p}(A_{\mu})} + \lVert u \rVert_{L^{p}(A_{\mu})}\Bigr) + \mu; 
			\]
		\item if \( s \geq 1 \), then for every \( j \in \{1,\dots,k\} \),
			\[
				(\mu\eta)^{j}\lVert D^{j}u^{\sh}_{\tau_{\mu},\mu} - D^{j}u \rVert_{L^{p}(K^{m})}
				\leq
				\C\sum_{i=1}^{j}(\mu\eta)^{i} \lVert D^{i}u \rVert_{L^{p}(A_{\mu})} + \mu;
			\]
		\item if \( s \geq 1 \) and \( \sigma \neq 0 \), then for every \( j \in \{1,\dots,k\} \),
			\begin{multline*}
				(\mu\eta)^{j+\sigma}\lvert D^{j}u^{\sh}_{\tau_{\mu},\mu} - D^{j}u \rvert_{W^{\sigma,p}(K^{m})} \\
				\leq
				\C\sum_{i=1}^{j}\Bigl((\mu\eta)^{i} \lVert D^{i}u \rVert_{L^{p}(A_{\mu})} 
				+ (\mu\eta)^{i+\sigma} \lvert D^{i}u \rvert_{W^{\sigma,p}(A_{\mu})}\Bigr) + \mu; 
			\end{multline*}
		\item for every \( 0 < s < +\infty \),
			\[
				\lVert u^{\sh}_{\tau_{\mu},\mu}-u \rVert_{L^{p}(K^{m})}
				\leq
				\C\lVert u \rVert_{L^{p}(A_{\mu})} + \mu.
			\]
	\end{enumerate}
	
	Using the compactness of \( \Nc \) and the fact that \( u \in W^{s,p}(K^{m}) \), we deduce from the Gagliardo--Nirenberg inequality that \( D^{i}u \in L^{sp/i}(K^{m}) \).
	Therefore, by Hölder's inequality,
	\[
	\lVert D^{i}u \rVert_{L^{p}(A_{\mu})}
	\leq 
	\lvert A_{\mu} \rvert^{\frac{s-i}{sp}}\lVert D^{i}u \rVert_{L^{sp/i}(A_{\mu})}.
	\]
	Similarly, using Lemma~\ref{lemma:interpolation_estimate}, for every \( i \in \{1,\dots,k-1\} \), we find that 
	\begin{equation*}
	\lvert D^{i}u \rvert_{W^{\sigma,p}(A_{\mu})}
	\leq 
	\C\lvert A_{\mu} \rvert^{\frac{s-i-\sigma}{sp}} \lVert D^{i}u \rVert_{L^{sp/i}(A_{\mu})}^{1-\sigma}\lVert D^{i+1}u \rVert_{L^{sp/(i+1)}(A_{\mu})}^{\sigma}.
	\end{equation*}
	(Strictly speaking, Lemma~\ref{lemma:interpolation_estimate} requires to work with two open sets \( \omega \Subset \Omega \), with \( \Omega \) being convex.
	However, we already saw that this assumption may be easily bypassed in the situation we are facing here.
	Indeed, this can be done, for instance, relying on the existence of a continuous extension operator from \( W^{s,p}(K^{m};\R^{\nu}) \) to \( W^{s,p}(\R^{m};\R^{\nu}) \).)
	Moreover, using the fact that \( u \in L^{\infty}(K^{m}) \) since \( \Nc \) is compact, we have 
	\[
		\lVert u \rVert_{L^{p}(A_{\mu})}
		\leq 
		\C\lvert A_{\mu} \rvert^{\frac{1}{p}}.
	\]
	On the other hand, we observe that \( \lvert A_{\mu} \rvert \leq \C(\mu\eta)^{\ell+1} \).
	Therefore, 
		\begin{enumerate}[label=(\alph*)]
		\item if \( 0 < s < 1 \), then 
		\[
			\lvert u^{\sh}_{\tau_{\mu},\mu} - u \rvert_{W^{s,p}(K^{m})} 
			\leq
			\C \lvert u \rvert_{W^{s,p}(A_{\mu})} + \C(\mu\eta)^{\frac{\ell+1-sp}{p}} + \mu; 
		\]
		\item if \( s \geq 1 \), then for every \( j \in \{1,\dots,k\} \),
		\[
			\lVert D^{j}u^{\sh}_{\tau_{\mu},\mu} - D^{j}u \rVert_{L^{p}(K^{m})}
			\leq
			\C\sum_{i=1}^{j}(\mu\eta)^{i-j+\frac{s-i}{sp}(\ell+1)} \lVert D^{i}u \rVert_{L^{\frac{sp}{i}}(A_{\mu})} + \mu;
		\]
		\item if \( s \geq 1 \) and \( \sigma \neq 0 \), then for every \( j \in \{1,\dots,k\} \),
		\begin{multline*}
			\lvert D^{j}u^{\sh}_{\tau_{\mu},\mu} - D^{j}u \rvert_{W^{\sigma,p}(K^{m})} 
			\leq
			\C\sum_{i=1}^{j}(\mu\eta)^{i-j-\sigma+\frac{s-i}{sp}(\ell+1)} \lVert D^{i}u \rVert_{L^{\frac{sp}{i}}(A_{\mu})} \\
			+ \C\sum_{i=1}^{k-1}(\mu\eta)^{i-j+\frac{s-i-\sigma}{sp}(\ell+1)}\lVert D^{i}u \rVert_{L^{\frac{sp}{i}}(A_{\mu})}^{1-\sigma}\lVert D^{i+1}u \rVert_{L^{\frac{sp}{i+1}}(A_{\mu})}^{\sigma} 
			+
			\C\lvert D^{j}u \rvert_{W^{\sigma,p}(A_{\mu})} + \mu; 
		\end{multline*}
		\item for every \( 0 < s < +\infty \),
		\[
		\lVert u^{\sh}_{\tau_{\mu},\mu}-u \rVert_{L^{p}(K^{m})}
		\leq
		\C(\mu\eta)^{\frac{\ell+1}{p}} + \mu.
		\]
	\end{enumerate}

	Since \( sp < \ell+1 \), we observe that all the powers on \( \mu\eta \) above are positive.
	Moreover, since \( \lvert A_{\mu} \rvert \to 0 \) as \( \mu \to 0 \), we deduce from Lebesgue's lemma that all Lebesgue norms and Gagliardo seminorms above tend to \( 0 \) when \( \mu \to 0 \).
	
	This shows that \( u^{\sh}_{\tau_{\mu},\mu} \to u \) in \( W^{s,p}(K^{m}) \), 
	and since \( u^{\sh}_{\tau_{\mu},\mu} \in \Cc^{\infty}(K^{m};\Nc) \), the proof is complete.
\end{proof}


%

Theorem~\ref{thm:density_smooth_functions} follows from Proposition~\ref{prop:density_smooth_maps_cubication} by using the fact that a cube has the extension property with respect to any manifold.
This was already present in~\cite{HL_topology_of_sobolev_mappings_II}.
For a proof, the reader may also consult~\cite{BPVS_density_higher_order}*{Proposition~7.3}.

\begin{proof}[Proof of Theorem~\ref{thm:density_smooth_functions}]
	From the proof of Theorem~\ref{thm:density_class_R}, for every map \( u \in W^{s,p}(Q^{m};\Nc) \) and every number \( \varepsilon > 0 \) there exists \( v \in \Cc^{\infty}(K^{m}\setminus T^{\ell^{\ast}};\Nc) \cap W^{s,p}(K^{m};\Nc) \) such that \( \lVert u-v \rVert_{W^{s,p}(Q^{m})} \leq \varepsilon \), where \( \ell = [sp] \) and \( \Kc^{m} \) is a cubication (depending on \( v \)) in \( \R^{m} \) slightly larger than \( Q^{m} \).
	Removing cubes that do not intersect \( Q^{m} \) if necessary, we may assume that \( K^{m} \) is also a cube.
	Doing so, \( K^{m} \) has the \( \ell \)-extension property with respect to \( \Nc \).
	Hence, \( \Cc^{\infty}(K^{m};\Nc) \) is dense in \( \Cc^{\infty}(K^{m}\setminus T^{\ell^{\ast}};\Nc) \cap W^{s,p}(K^{m};\Nc) \) with respect to the \( W^{s,p} \) distance.
	This implies that \( \Cc^{\infty}(\overline{Q}^{m};\Nc) \) is dense in \( W^{s,p}(Q^{m};\Nc) \), and finishes the proof of Theorem~\ref{thm:density_smooth_functions}. 
\end{proof}

We now turn to the case of more general domains.
Replacing Theorem~\ref{thm:density_class_R} by Theorem~\ref{thm:density_class_R_segment_condition} in the above proof, we obtain the following counterpart of Theorem~\ref{thm:density_smooth_functions}:
If \( \Omega \) satisfies the segment condition and if we may find a sequence \( (\eta_{n})_{n \in \N} \) of positive real numbers such that \( \eta_{n} \to 0 \) 
and such that for every \( n \in \N \), the cubication \( \Kc^{m}_{\eta_{n}} \) used in the proof of Theorem~\ref{thm:density_class_R_segment_condition} satisfies the \( [sp] \)-extension property with respect to \( \Nc \),
then \( \Cc^{\infty}(\overline{\Omega};\Nc) \) is dense in \( W^{s,p}(\Omega;\Nc) \).


%

Under an assumption as weak as the segment condition, it is not clear how to link the topology of the cubications \( K^{m}_{\eta_{n}} \) containing \( \Omega \) to the topology of \( \Omega \) itself.
However, in the case where \( \Omega \) is a smooth domain, the topological assumption above can be clarified.
Indeed, in this case, using a retraction along the normal vector to \( \partial\Omega \), one may show that, if \( \Kc^{m}_{\eta} \) is a cubication of radius \( \eta > 0 \) in \( \R^{m} \) 
for \( \eta > 0 \) sufficiently small such that \( \Kc^{m}_{\eta} \) is made only of cubes that intersect \( \overline{\Omega} \), then \( K^{m}_{\eta} \) is homotopic to \( \Omega \).
This implies that, if we endow \( \Omega \) with a structure of CW-complex, then the \( \ell \)-extension property of \( \Kc^{m}_{\ell} \) 
is equivalent to the \( \ell \)-extension property of \( \Omega \), and this does not depends on the choice of CW-complex structure on \( \Omega \); see e.g.~\cite{HL_topology_of_sobolev_mappings_II}*{Section~2}.
Here, analogously to the definition on a cubication, we say that \( \Omega \) has the \( \ell \)-extension property with respect to \( \Nc \) whenever any map \( f \in \Cc^{0}(\Omega^{\ell};\Nc) \) has an extension \( f \in \Cc^{0}(\Omega;\Nc) \), where \( \Omega^{\ell} \) denotes the \( \ell \)-skeleton of the CW-complex structure on \( \Omega \).

This leads to the following theorem.

\begin{thm}
\label{thm:density_smooth_maps_smooth_domain}
    Let \( \Omega \subset \R^{m} \) be a smooth bounded open domain.
    If \( sp < m \) and if \( \Omega \) has the \( [sp] \)-extension property with respect to \( \Nc \), 
    then \( \Cc^{\infty}(\overline{\Omega};\Nc) \) is dense in \( W^{s,p}(\Omega;\Nc) \).
\end{thm}

As for Theorem~\ref{thm:density_class_R}, a last perspective of generalisation for Theorem~\ref{thm:density_smooth_functions} consists in allowing the domain to be a smooth compact, connected Riemannian manifold \( \Mc \) of dimension \( m \),
and isometrically embedded in \( \R^{\tilde{\nu}} \).
As we did for Theorem~\ref{thm:density_class_R_manifold}, we may restrict to the case where \( \Mc \) has empty boundary, since the general case reduces to this special case by embedding into a larger manifold without boundary.

In this setting, a tubular neighborhood of \( \Mc \) is homotopic to \( \Mc \) through the nearest point projection, and therefore has the \( \ell \)-extension property if and only if \( \Mc \) has the \( \ell \)-extension property.
We may thus proceed as for Theorem~\ref{thm:density_class_R_manifold} to deduce the following result.

\begin{thm}
\label{thm:density_smooth_maps_manifold}
    If \( sp < m \) and if \( \Mc \) has the \( [sp] \)-extension property with respect to \( \Nc \), 
    then \( \Cc^{\infty}(\Mc;\Nc) \) is dense in \( W^{s,p}(\Mc;\Nc) \).
\end{thm}
\begin{proof}
	First assume that \( \Mc \) has empty boundary.
    Let \( \iota > 0 \) be the radius of a tubular neighborhood of \( \Mc \), let \( \upPi \) denote the nearest point projection onto \( \Mc \), and let \( \Omega = \Mc+B^{\tilde{\nu}}_{\iota/2} \).
    Given \( u \in W^{s,p}(\Mc;\Nc) \), as explained before the proof of Theorem~\ref{thm:density_class_R_manifold}, the map \( v = u \circ \upPi \) belongs to \( W^{s,p}(\Omega;\Nc) \).
    By the observation above, \( \Omega \) has the \( [sp] \)-extension property.
    Therefore, by Theorem~\ref{thm:density_smooth_maps_smooth_domain}, there exists a sequence \( (v_{n})_{n \in \N} \) in \( \Cc^{\infty}(\overline{\Omega};\Nc) \) which converges to \( v \) in \( W^{s,p}(\Omega) \).
    We conclude as in the proof of Theorem~\ref{thm:density_class_R_manifold}, using a slicing argument to find a sequence \( (a_{n})_{n \in \N} \) in \( B^{\tilde{\nu}}_{\iota/2} \) such that \( a_{n} \to 0 \) as \( n \to +\infty \) 
    satisfying \( u_{n} = \tau_{a_{n}}(v_{n})_{\vert \Mc} \to u \) in \( W^{s,p}(\Mc) \). 
    
    The case where \( \Mc \) is allowed to have non-empty boundary is deduced from the empty boundary case exactly as for the proof of Theorem~\ref{thm:density_class_R_manifold}, and we therefore omit the proof.
\end{proof}

\end{document}